\documentclass[12pt]{article}
\usepackage{ulem}
\usepackage{amssymb,amsmath,graphicx,amsthm,mathrsfs}
\usepackage{color}
\usepackage[colorlinks=true]{hyperref}
\usepackage{epstopdf}
\usepackage{subfigure}
\usepackage{galois}
\usepackage{amsfonts,graphicx,amssymb,amsmath,comment}
\usepackage{mathrsfs,galois,booktabs,ctable,threeparttable,tabularx,algorithm,algorithmic}
\usepackage{subfigure}
\usepackage{epstopdf}

\hypersetup{linkcolor=blue, urlcolor=green, citecolor=red}
\newtheorem{Theorem}{Theorem}[section]
\newtheorem{Lemma}[Theorem]{Lemma}
\newtheorem{Proposition}[Theorem]{Proposition}

\newtheorem{Corollary}[Theorem]{Corollary}

\newtheorem{Remark}[Theorem]{Remark}
\newtheorem{Example}[Theorem]{Example}
\numberwithin{equation}{section}

\title{Decompositions and bang-bang properties}

\author{Gengsheng Wang\thanks{School of Mathematics and Statistics, Computational Science Hubei Key Laboratory, Wuhan University, Wuhan, 430072, China (wanggs62@yeah.net). The author was partially supported by the National Natural Science Foundation of China  under grant 11571264. }
\and Yubiao Zhang\thanks{School of
Mathematics and Statistics, Wuhan University,  Wuhan, 430072, China; Center for Applied Mathematics, Tianjin University, Tianjin, 300072,  China
(yubiao\b{ }zhang@whu.edu.cn). The author was partially supported by the National Natural Science Foundation of China  under grants 11571264 and 11371285.} }

\begin{document}

 \date{ }
\maketitle

\begin{abstract}
Two kinds of important optimal control problems for linear controlled systems  are
minimal time control problems and minimal norm control problems. A minimal time control problem is to ask for a control (which takes values in a  control constraint set in a control space) driving  the corresponding solution of
a controlled system from an initial state $y_0$ to a  target set in the shortest time, while a minimal norm control problem is to ask for a control which has the minimal norm among all  controls driving  the corresponding solutions of a controlled system  from an  initial state $y_0$ to a target set at  fixed ending time $T$. In this paper, we focus on the case that
 target sets are the origin of  state spaces, control constraint sets are closed balls $B(0,M)$ in  control spaces (centered at the origin and of radius $M>0$) and controls are $L^\infty$ functions. The bang-bang property for a minimal time control problem means that  any minimal time control, as a function of time,  point-wisely takes its value at the boundary of $B(0,M)$, while the bang-bang property for a minimal norm control problem means that  each minimal norm control, as a function of time, point-wisely takes  the  minimal norm.

 This paper studies the bang-bang properties for the above-mentioned two kinds of problems from  a new perspective.
The motivation of this study is as follows: For a given controlled system, a minimal time control problem  depends only on
 $(M,y_0)$, while a minimal norm control problem depends only on $(T,y_0)$.
When a controlled system reads: $y'(t)=Ay(t)+Bu(t)$, with  $(A,B)$ in $\mathbb{R}^{n\times n}\times (\mathbb{R}^{n\times m}\setminus\{0\})$, an interesting phenomenon related to the bang-bang properties is as follows: The product space consisting of all pairs $(M,y_0)$
can be divided into two disjoint parts so that when $(M,y_0)$ is in one part,
the corresponding minimal time control problem has the bang-bang property; when
  $(M,y_0)$ is in another part,
the corresponding minimal time control problem has no any solution. The same can be said about the product space consisting of all pairs $(T, y_0)$, as well as minimal norm control problems. We call such decompositions as the BBP decompositions for minimal time control problems and minimal norm control problems, respectively.
Though they   can be easily derived from
the Kalman controllability decomposition to the pair of matrices $(A,B)$,
it seems for us that they are new.
 A natural question is how to  extend  the above-mentioned BBP decompositions to the infinitely dimensional setting where  state and control spaces are two real Hilbert spaces,
 $A$ is a generator of a $C_0$-semigroup on the state space and $B$ is a linear operator from the control space to the state space.

  The  purpose of this paper is to build up the BBP decompositions
  in  the infinitely dimensional setting.   The main difficulty to get such extension is the lack of the Kalman controllability decomposition in the infinitely dimensional setting.
 Our first key to overcome this difficulty is to find two   properties held by
 any pair of matrices $(A,B)$ in $\mathbb{R}^{n\times n}\times (\mathbb{R}^{n\times m}\setminus\{0\})$
 so that they have the following functionalities: First, with the aid of these properties, we can prove the BBP decompositions in finitely dimensional setting, without using the Kalman controllability decomposition; Second, these properties can be easily stated  in the infinitely dimensional setting.

 By assuming the above-mentioned  two
 properties  in the infinitely dimensional setting,
 we divide the product space  of $(T,y_0)$ (or $(M,y_0)$) into four disjoint parts, with the aid of several functions and an affiliated minimal norm problem.
 For minimal norm control problems, the first part consists of those pairs $(T,y_0)$  so that the corresponding minimal norm problems  have no any solution.
 The second part consists of those pairs $(T,y_0)$  so that the corresponding minimal norm problems  hold the bang-bang property and  the null control is not their minimal norm control.  The third part consists of those pairs $(T,y_0)$ so that the corresponding minimal norm problems  hold the bang-bang property and the null control is their unique minimal norm control.   The last part is a curve segment $\gamma_1$ in the product space  of $(T,y_0)$.
 For minimal time control problems, the first part consists of those pairs $(M,y_0)$  so that the corresponding minimal time control problems  have no any solution.
  The second part consists of those pairs $(M,y_0)$  so that the corresponding minimal time control problems  hold the bang-bang property. The third part consists of those pairs $(M,y_0)$ so that the corresponding minimal time control problems  have infinitely many different minimal time controls and do not have the bang-bang property.   The last part is a curve segment $\gamma_2$ in the product space  of $(M,y_0)$. The aforementioned two curve segments $\gamma_1$ and $\gamma_2$  are  crucial parts for us in the following sense: First,  we are not sure if these curve segments are empty; Second, when  $(T,y_0)\in\gamma_1$ (or $(M,y_0)\in\gamma_2$), we  know  that the corresponding minimal norm (or minimal time) control problem  has a solution, but we are not sure if it has the bang-bang property.

\end{abstract}

\noindent\textbf{keywords.}   minimal time controls, minimal norm controls, bang-bang properties, decompositions,  evolution systems, reachable subspaces\\


\tableofcontents

\bigskip

\section{Introduction}

\subsection{Motivation}

Two kinds of important optimal control problems for linear controlled systems are
minimal time control problems and minimal norm control problems.
  A minimal time control problem is to ask for a control (taking values from a  control constraint set which is, in general, a closed and bounded subset in a control space) which drives  the corresponding solution of
a controlled system from an initial state  to a  target set in the shortest time, while a minimal norm control problem is to ask for a control which has the minimal norm among all  controls that drive  the corresponding solutions of a controlled systems  from an  initial state  to a target set at  fixed ending time.
 Several important issues on minimal time (or minimal norm) control problems are as follows: The Pontryagin maximum principle of minimal time (or minimal norm) controls (see, for instance, \cite{HOF2, KL2, LM, LWX, WZ}); The existence of minimal time ( or minimal norm) controls  (see, for instance, \cite{Barbu,LZ,PWZ}); Their connections  with  controllabilities (see, for instance \cite{carja,GL,Petrov}); Numerical analyses on minimal time (or minimal norm) controls (see, for instance, \cite{GY, ITO, MZ, C.S.E.T, WZheng}); And the bang-bang property of minimal time  (or minimal norm) controls (see, for instance, \cite{HOF,KL1, KL2, LM, QL, MRT, MS, PW, PWCZ, EJPGS, Wang, WX-2, WXZ, WCZ, Y, CZ-2}).

 In this paper, we concern the bang-bang properties  for  these two kinds of problems
in the case that both  state   and control spaces are real Hilbert spaces,
 controlled systems are linear and time-invariant,
 target sets are the origin of  state spaces, control constraint sets are closed balls  in  control spaces (centered at the origin) and controls are $L^\infty$ functions.  The bang-bang property for a minimal time control problem means that  any minimal time control, as a function of time,  point-wisely takes its value at the boundary of the control constraint set, while the bang-bang property for a minimal norm control problem means that  each minimal norm control, as a function of time, point-wisely takes  the  minimal norm.
  The significance of the bang-bang property for minimal time control problems can be explained from the following aspects:
 (i) Mathematically, the bang-bang property means that  each minimizer of a functional (from $[0,\infty)$ to a bounded and closed subset in a Hilbert space)
   point-wisely takes value
 on the boundary of this subset.   (ii) From application point of view,
  the bang-bang property means that each minimal time control takes the most advantage of possible control actions. For instance,
  controls always have bounds which are designed by peoples. The bigger  bounds are designed, the more costs  peoples pay. If the bang-bang property holds for a minimal time problem, then
  the  designed bound for controls will not be wasted at almost each time.
    (iii) The bang-bang property is  powerful in the studies of minimal time control problems. For instance,
    in many cases, the  uniqueness of minimal time controls follows from this property; in some cases, this property can help people to do more dedicate numerical analyses on minimal time controls (see, for instance,
    \cite{ITO,  C.S.E.T}).
    We can also explain the significance of the bang-bang property for minimal norm control problems from both mathematical and application points of view.
    In most literatures
    on  the  bang-bang property for the minimal time (or minimal norm) control problems, peoples mainly concern about: (i) For a given problem, whether the bang-bang property holds; (ii)  Applications of the bang-bang property (see, for instance,         \cite{HOF,KL1, KL2, LM, QL, MRT, MS, PW, PWCZ, EJPGS, Wang, WX-2, WXZ, WCZ, Y, CZ-2} and the references therein).

In this paper, we study the bang-bang properties of the minimal time control problems and the minimal norm control problems from a different perspective.
The motivation of this study is as follows:
 Two typical minimal time and minimal norm control problems in the finitely dimensional setting  are as follows: Let $\mathbb{R}^n$ and $\mathbb{R}^m$ (with $n,m\geq 1$) be the state space and the control  space. Let $(A,B)$ be a pair of matrices  in $\mathbb R^{n\times n}\times (\mathbb R^{n\times m}\setminus\{0\})$.
 Given $M>0$ and $y_0\in \mathbb R^n\setminus\{0\}$, consider the minimal time control problem:
  \begin{eqnarray}\label{0123-intro-TP}
  (\mathcal{TP})^{M,y_0} \;\;\;\;\;\; \mathcal{T}(M,y_0)  \triangleq  \{ \hat t>0~:~\exists\,u\in\mathbb{U}^M
 \;\;\mbox{s.t.}\;\;  y(\hat t;y_0,u)=0\},
 \end{eqnarray}
 where
 \begin{eqnarray}\label{0125-UM}
  \mathbb{U}^M\triangleq \{u:~\mathbb R^+\triangleq[0,\infty) \rightarrow \mathbb R^m \;\;\mbox{measurable} ~:~  \|u\|_{L^\infty(\mathbb R^+;\mathbb R^m)} \leq M\},
 \end{eqnarray}
 and $y(\cdot;y_0,u)$ is the solution to the equation:
   \begin{eqnarray}\label{0123-intro-1}
         y^\prime(t)=Ay(t)+Bu(t),~t>0;\;\;
         y(0)=y_0.
         \end{eqnarray}
 Given $y_0\in \mathbb R^n\setminus\{0\}$ and $T\in(0,\infty)$, consider the   minimal norm control problem:
 \begin{eqnarray}\label{0123-intro-NP}
 (\mathcal{NP})^{T,y_0}\;\;\;\;\;\; \mathcal{N}(T,y_0)  \triangleq  \inf\{ \|v\|_{L^\infty(0,T;\mathbb R^m)}~:~\hat y(T;y_0,v)=0\},
 \end{eqnarray}
 where $v\in L^\infty(0,T;\mathbb R^m)$
 and $\hat y(\cdot;y_0,v)$ is the solution to the equation:
  \begin{eqnarray}
 \label{0123-intro-2}
         y^\prime(t)=Ay(t)+Bv(t),~0<t\leq T;\;\;
         y(0)=y_0.
         \end{eqnarray}
   In the problem $(\mathcal{TP})^{M,y_0}$, $\mathcal{T}(M,y_0)$ is called the minimal time; $\hat u\in\mathbb{U}^M$  is called an admissible control  if
   $y(\hat t;y_0,\hat u)=0$ for some $\hat t \in(0,\infty)$;
     $u^*\in\mathbb{U}^M$ is called a minimal time control if
 $y(\mathcal{T}(M,y_0);y_0,u^*)=0$ and $u^*=0$ over $(0, \mathcal{T}(M,y_0))$.
   We say that the problem $(\mathcal{TP})^{M,y_0}$ has the bang-bang property if any minimal time control $u^*$ verifies that
   $   \|u^*(t)\|_{\mathbb R^m}=M$ for a.e. $t\in\big(0,\mathcal{T}(M,y_0)\big)$.
       When $(\mathcal{TP})^{M,y_0}$ has no any admissible control, we agree that   it does not hold the bang-bang property and $\mathcal{T}(M,y_0)=\infty$.
    In the  problem $(\mathcal{NP})^{T,y_0}$, $\mathcal{N}(T,y_0)$ is called the minimal norm; $ \hat v\in L^\infty(0,T;\mathbb R^m)$  is called an admissible control  if
    $\hat y(T;y_0,\hat v)=0$;
   $v^*\in L^\infty(0,T;\mathbb R^m)$ is called a minimal norm control  if
 $\|v^*\|_{L^\infty(0,T;\mathbb R^m)}=\mathcal{N}(T,y_0)$ and
 $\hat{y}(T;y_0,v^*)=0$.
    We say that the problem $(\mathcal{NP})^{T,y_0}$ has the bang-bang property if any minimal norm control $v^*$ verifies that
    $\|v^*(t)\|_{\mathbb R^m}=\mathcal{N}(T,y_0)$ for a.e. $t\in(0,T)$.
     When $(\mathcal{NP})^{T,y_0}$ has no any admissible control, we agree that it does not hold the bang-bang property and $\mathcal{N}(T,y_0)=\infty$.

       When $(A,B)$ is fixed  in $\mathbb R^{n\times n}\times (\mathbb R^{n\times m}\setminus\{0\})$, the problem    $(\mathcal{TP})^{M,y_0}$  depends only on the pair $(M,y_0)$ which belongs to the product space:
     \begin{eqnarray}\label{NEWYEAR1.6}
      \mathcal X_1  \triangleq
      \big\{(M,y_0)~:~ 0<M<\infty,~y_0\in \mathbb R^n\setminus\{0\}\big\};
     \end{eqnarray}
    and the problem    $(\mathcal{NP})^{T,y_0}$ depends only on the pair $(T,y_0)$ which belongs to the space:
     \begin{eqnarray}\label{NEWYEAR1.7}
      \mathcal X_2  \triangleq
      \big\{(T,y_0)~:~ 0<T<\infty,~y_0\in \mathbb R^n\setminus\{0\}\big\}.
     \end{eqnarray}
      By applying the Kalman controllability decomposition to the pair $(A,B)$
      (see, for instance, Lemma 3.3.3 and Lemma 3.3.4 in \cite{E.D.Sontag}),
      we can easily divide the space  $\mathcal X_1$ into two disjoint parts so that when
      $(M,y_0)$ is in one part, the corresponding $(\mathcal{TP})^{M,y_0}$ has the bang-bang property; when $(M,y_0)$ is in another part,  the corresponding $(\mathcal{TP})^{M,y_0}$ has no any admissible control (which implies that it does not hold the bang-bang property). The same can be said
      about the space   $\mathcal X_2$. We call such decompositions as
   the BBP decompositions for $(\mathcal{TP})^{M,y_0}$ and
   $(\mathcal{NP})^{T,y_0}$, respectively.
  The exact BBP decompositions for the above two problems are the following (P1) and (P2):
         \begin{eqnarray}\label{0127-intro-P1}
  \mbox{\textbf{(P1)}} \begin{array}{ll}
       &\bullet  ~\,\mbox{When } (M,y_0)\in \mathcal \mathcal D_{bbp}, \;\;(\mathcal{TP})^{M,y_0}
                  \mbox{ has the bang-bang property};  \\
       &\bullet ~\,\mbox{When } (M,y_0)\in \mathcal X_1 \setminus \mathcal D_{bbp}, \;\;(\mathcal{TP})^{M,y_0}
                  \mbox{ has } \mbox{ no }  \mbox{ admissible control. }
       \end{array}
 \end{eqnarray}
 Here,
 \begin{eqnarray}\label{wangyu1.8}
  \mathcal D_{bbp} \triangleq  \big\{(M,y_0)\in (0,\infty)\times (\mathcal R\setminus\{0\})
  ~:~  M> \lim_{T\rightarrow\infty} \mathcal{N}(T,y_0)\big\},
 \end{eqnarray}
 where
         \begin{eqnarray}\label{0125-intro-attain}
   \mathcal R  \triangleq  \mathcal B + A\mathcal B + \cdots + A^{n} \mathcal B,\;\;\mbox{with}\;\;
      \mathcal B  \triangleq  \{ Bx \in \mathbb R^n ~:~ x\in \mathbb R^m\}.
   \end{eqnarray}
     \begin{eqnarray}\label{0127-intro-P2}
  \mbox{\textbf{(P2)}} \begin{array}{ll}
       &\bullet  ~\,\mbox{When } (T,y_0)\in \mathcal{X}_{2,1}, \;\;(\mathcal{NP})^{T,y_0}
                  \mbox{ has the bang-bang property;}  \\
             &\bullet ~\,\mbox{When  } (T,y_0)\in \mathcal{X}_{2,2}, \;\;(\mathcal{NP})^{T,y_0}
                  \mbox{ has } \mbox{ no }  \mbox{ admissible control,}
       \end{array}
 \end{eqnarray}
 where $\mathcal{X}_{2,1}\triangleq (0,\infty) \times (\mathcal R\setminus\{0\})$
 and $\mathcal{X}_{2,2}\triangleq (0,\infty) \times (\mathbb R^n\setminus\mathcal R)$. (Notice that both $\mathcal D_{bbp}$ and $\mathcal R\setminus\{0\}$ are not empty. These are proved in Appendix A, see (\ref{Appendixwang8.17}) and (\ref{newyearyuan1.12}).)

     The proofs of (P1) and (P2), via the  Kalman controllability decomposition,
     are given in Appendix A. Though the proofs are quite simple,
      such BBP decompositions seem to be new. (At least we  do not find them in any published literature.)
 {\it  A natural question is how to extend the above-mentioned BBP decompositions to the infinitely dimensional setting where  state and control spaces are two real Hilbert spaces,
 $A$ is a generator of a $C_0$-semigroup on the state space and $B$ is a linear operator from the control space to the state space.}  The  purpose of this paper is to build up such BBP decompositions
  in  the infinitely dimensional setting.
   The main difficulty to get such extension is the lack of the Kalman controllability decomposition in the
 infinitely dimensional setting.

 Our first key to overcome this difficulty is to find two   properties held by
 any pair of matrices $(A,B)$ in $\mathbb{R}^{n\times n}\times (\mathbb{R}^{n\times m}\setminus\{0\})$
 so that they have the following functionalities: (i) With the aid of these properties, we can get the decompositions (P1) and (P2), without using the Kalman controllability decomposition; (ii) These properties can be easily stated  in the infinitely dimensional setting.
  The first one is a kind of unique continuation property from measurable set
  for functions: $B^*e^{A^*(T-\cdot)}z$, with $T>0$ and $z\in \mathbb{R}^n$.
  This property follows immediately  from the analyticity of the function $t\rightarrow B_1^*e^{A_1^*(T-t)}$, $t\in \mathbb{R}$, in the finitely dimensional setting.  In our infinitely dimensional setting, it is the assumption (H2) given in the next subsection. The second property is quite hidden: {\it For all $t$ and $T$, with $0<t<T<\infty$,  and  $u\in L^2(0,T;\mathbb R^m)$, with supp $u\subset(0,t)$, there is $v_u\in L^\infty(0,T;\mathbb R^m)$, with supp $v\subset(t,T)$, so that
  $\hat y(T;0,u) = \hat y(T;0,v_u)$,
  where   $\hat y(\cdot;0, u)$ and $\hat y(\cdot; 0,v_u)$  denote
  the  solutions of (\ref{0123-intro-2}) with the same initial datum $0$ and controls $u$ and $v_u$, respectively. }
  (Proposition~\ref{Lemma-8ex1-1} in Appendix B proves that each pair of matrices $(A,B)$  in $\mathbb R^{n\times n}\times (\mathbb R^{n\times m}\setminus\{0\})$ holds this property.)
  The assumption (H1) given in the next subsection is exactly the same version of the second property in our finitely dimensional setting.
About (H1), two facts are given in order: First,
for a pair $(A,B)$ in the finitely dimensional setting, it may happen that the
above-mentioned second property  holds but  $(A,B)$ is not controllable. Second,  even in the infinitely dimensional setting, the null controllability of $(A,B)$   implies that the above-mentioned second property (see Proposition~\ref{WGSWGSlemma2.15}).

 \subsection{Problems and assumptions}

Let us first introduce the minimal time and the minimal norm control problems studied in this paper. Let   $X$  be a real Hilbert space (which is our state space), with its  inner product $\langle\cdot,\cdot\rangle_X$ and its  norm $\|\cdot\|_X$. Let   $A:\,D(A)\subset X\rightarrow X$ be a state operator which generates a $C_0$-semigroup $\{S(t)\}_{t\in \mathbb{R}^+}$ on $X$.  Write  $U$ for another  real Hilbert space (which is our control space), with its inner product  $\langle\cdot,\cdot\rangle_U$ and its  norm $\|\cdot\|_U$. Let  $B\in \mathcal L(U,X_{-1})$ be a nontrivial control operator (i.e., $B\neq0$), where $X_{-1}\triangleq D(A^*)'$ is the dual of $D(A^*)$ with respect to the pivot space $X$.
Throughout this paper, we assume that $B$  is   an admissible control operator for $\{S(t)\}_{t\in \mathbb{R}^+}$ (see Section 4.2 in \cite{TW}), i.e.,
   for each $\hat t\in(0,\infty)$, there is a positive constant $C_1(\hat t)$,
    depending on $\hat t$, so that
\begin{eqnarray}\label{admissible-control}
 \big\|\int_0^{\hat t} S_{-1}(\hat t-\tau)B u(\tau) \,\mathrm d\tau \big\|_X \leq C_1(\hat t)\|u\|_{L^2(0,\hat t;U)}\;\;\mbox{for all}\;\;u\in L^2_{{loc}}(\mathbb{R}^+; U),
\end{eqnarray}
  where  $\{S_{-1}(t)\}_{t\in \mathbb{R}^+}$ denotes the extension of $\{S(t)\}_{t\in \mathbb{R}^+}$ on $X_{-1}$. In the finitely dimensional setting where $X=\mathbb{R}^n$, $U=\mathbb{R}^m$, $A\in \mathbb{R}^{n\times n}$ and $B\in \mathbb{R}^{n\times m}\setminus\{0\}$, (\ref{admissible-control}) holds automatically.

Two controlled equations studied in this paper are as follows:
\begin{eqnarray}\label{system-0}
  y^\prime(t) =Ay(t)+B u(t),~t>0;\;\;
  y(0)=y_0;
 \end{eqnarray}
 \begin{eqnarray}\label{system-1}
   y^\prime(t) =Ay(t)+B v(t),~0<t\leq T;\;\;
  y(0)=y_0.
  \end{eqnarray}
Here, $y_0\in X$, $T>0$, controls $u$ and $v$ are taken from  $L^\infty(\mathbb{R}^+;U)$ and  $L^\infty(0,T;U)$, respectively.
For each $T>0$, $y_0\in X$ and $v\in L^2(0,T;U)$, a solution of  the equation
(\ref{system-1}) is defined to be a function  $\hat y(\cdot; y_0,v)\in C([0,T];X)$
satisfying that when $z\in D(A^*)$,
 \begin{eqnarray}\label{NNNWWW2.1}
 \langle \hat y(t;y_0,v),z \rangle_X  -  \langle y_0,S^*(t)z \rangle_X
 = \int_0^t \langle v(s), B^*S^*(t-s)z \rangle_U  \,\mathrm ds,
 ~\forall\,t\in[0,T] .~~
\end{eqnarray}
One can easily see from Lemma \ref{ad-control-ob} that the definition of $\hat y(\cdot; y_0,v)$ is  the same as the definition of a solution to (\ref{system-1}) in \cite[Definition 2.36]{Coron}.
 Thus, it follows from \cite[Theorem 2.37]{Coron} and Lemma \ref{ad-control-ob} that the equation (\ref{system-1}) is well-posed. For each $y_0\in X$ and $u\in L^\infty(\mathbb{R}^+;U)$, a solution of the equation (\ref{system-0}) is defined to be a function $y(\cdot; y_0,u)\in C(\mathbb R^+;X)$ so that
for each $T>0$, $y(\cdot; y_0,u)|_{[0,T]}$ (the restriction of $y(\cdot; y_0,u)$ over $[0,T]$) is the solution to (\ref{system-1}) with $v=u|_{(0,T]}$.  Consequently, the system (\ref{system-0}) is well-posed. Besides, by
Proposition~\ref{huangwanghenproposition2.1}, one can check the following two facts: First,
 for each $y_0\in X$ and $u\in L^\infty(\mathbb{R}^+;U)$, the  solution $y(\cdot; y_0,u)$ to the system (\ref{system-0}) satisfies that
\begin{eqnarray}\label{Changshubianyi1.6}
 y(t;y_0,u)=S(t)y_0+\int_0^t S_{-1}(t-\tau)Bu(\tau) \,\mathrm d\tau,\;\; 0\leq t<\infty.
\end{eqnarray}
Second, if for some $y_0\in X$ and $u\in L^\infty(\mathbb{R}^+;U)$, a function $y(\cdot)\in C(\mathbb{R}^+;X)$ equals to  the
right hand side of  (\ref{Changshubianyi1.6}) point-wisely, then  $y(\cdot)=y(\cdot;y_0,u)$ over $\mathbb{R}^+$.

For each pair $(M,y_0)\in (0,\infty)\times(X\setminus\{0\})$, we define a minimal
time control problem:
 \begin{eqnarray}\label{TP-0}
 (TP)^{M,y_0}\;\;\;\;\;\; T(M,y_0)\triangleq \inf \big\{\hat t\in (0, \infty)\; :\; \exists\,u\in\mathcal U^M \mbox{ s.t. } y(\hat t;y_0,u)=0\big\},
 \end{eqnarray}
  where
 \begin{eqnarray*}
  \mathcal U^M \triangleq \big\{u:\,\mathbb R^+\rightarrow U \mbox{ strongly measurable}\; :\; \|u(t)\|_U\leq M \mbox{ a.e. }t\in\mathbb R^+\big\}.
 \end{eqnarray*}
   In the problem $(TP)^{M,y_0}$,  the minimal time, an admissible control and a minimal time control  can be defined in the same manners as in  $(\mathcal{TP})^{M,y_0}$ (see (\ref{0123-intro-TP})).
   We say that the problem $(TP)^{M,y_0}$ has the bang-bang property if any minimal time control $u^*$ verifies that
   $\|u^*(t)\|_U=M$  for a.e. $t\in\big(0,T(M,y_0)\big)$.
        When $(TP)^{M,y_0}$ has no any admissible control, we agree that   it does not hold the bang-bang property and   $T(M,y_0)=\infty$.

    For each pair $(T,y_0)\in (0,\infty)\times X\setminus\{0\}$, we define a  minimal  norm control problem:
     \begin{eqnarray}\label{NP-0}
 (NP)^{T,y_0}\;\;\;\;\;\; N(T,y_0)\triangleq \inf\big\{\|v\|_{L^\infty(0,T;U)}\; :\;  v\in L^\infty(0,T;U)\mbox{ s.t. }  \hat{y}(T;y_0,v)=0\big\}.
 \end{eqnarray}
   In the  problem $(NP)^{T,y_0}$,   the minimal norm, an admissible control and a minimal norm control can be defined in the same  ways as in $(\mathcal{NP})^{T,y_0}$ (see (\ref{0123-intro-NP})).
    We say that the problem $(NP)^{T,y_0}$ has the bang-bang property if any minimal norm control $v^*$ verifies that
    $\|v^*(t)\|_U=N(T,y_0)$ for a.e. $t\in(0,T)$.
     When $(NP)^{T,y_0}$ has no any admissible control, we agree that it does not hold the bang-bang property and $N(T,y_0)=\infty$.

We say that $(A,B)$ has the $L^\infty$-null controllability if for any $T>0$  and  $y_0\in X$, there is  $v\in L^\infty(0,T;U)$ so that
$\hat y(T;y_0,v)=0$. We say that the semigroup $\{S(t)\}_{t\in \mathbb{R}^+}$  has   the backward uniqueness property if $S(T)y_0=0\Rightarrow y_0=0$.
       In our infinitely setting, we assume neither  the $L^\infty$-null controllability nor the backward uniqueness property.
         To make up the lack of these properties, we define
  the following two functions $T^0(\cdot)$ and $T^1(\cdot)$ (which play important roles in our study):
   \begin{eqnarray}\label{y0-controllable}
  T^0(y_0)\triangleq\inf\left\{\hat t\in\mathbb R^+\; :\; \exists\,u\in L^\infty(\mathbb{R}^+; U) \mbox{ s.t. }  y(\hat t;y_0, u)=0\right\},\;\;y_0\in X;
 \end{eqnarray}
   \begin{eqnarray}\label{Ty0}
   T^1(y_0)\triangleq\inf\left\{\hat t\in\mathbb R^+\; :\; S(\hat t)y_0=0 \right\},\;\;y_0\in X.
  \end{eqnarray}
  When the set on the right hand side of (\ref{y0-controllable}) is empty for some $y_0$, we let  $T^0(y_0)\triangleq\infty$. The same can be said about  $T^1(y_0)$.
  \begin{Remark}\label{YUYONGYUremark1.2}
(i) The pair $(A,B)$ has the $L^\infty$-null controllability if and only if for each $y_0\in X$, $T^0(y_0)=0$.

\noindent (ii) Though  many controlled systems, such as internally
or boundary controlled heat equations, hold the $L^\infty$-null controllability, there are some controlled systems having no the $L^\infty$-null controllability. Among them, it may happen that
$T^0(y_0)\in (0,\infty)$ for some $y_0\in X$ (see Remark~\ref{Remark-minimal-time-infty}).

\noindent (iii) The semigroup $\{S(t)\}_{t\in \mathbb{R}^+}$ has the backward uniqueness property if and only if for each $y_0\in X\setminus\{0\}$, $T^1(y_0)=\infty$.

 \noindent (iv) Though  many semigroups governed by PDEs, such as heat equations and wave equations, hold the backward uniqueness property,  there are some semigroups governed by PDEs having no this property. Among them, it may happen that $T^1(y_0)<\infty$ for all $y_0\in X$. A transport equation over a finite interval
 is one of such examples.
\end{Remark}

 From (\ref{NP-0}), we see that for each $y_0\in X\setminus\{0\}$, $T\rightarrow N(T,y_0)$ defines a function over $(0,\infty)$. Since the quantities $N(T^0(y_0),y_0)$ and $N(T^1(y_0),y_0)$ will appear frequently,   $T^0(\cdot)$ may take values  $0$ and  $\infty$, and  $T^1(\cdot)$ may take value $\infty$, it is necessary for us to give  definitions for
  $N(\infty,y_0)$ and $ N(0,y_0)$. For this purpose,
we notice  that for each $y_0\in X\backslash\{0\}$,
           $T\rightarrow N(T, y_0)$ is a decreasing  function from
           $(0,\infty)$ to $[0,\infty]$.
           (This can be easily obtained from (\ref{NP-0}), see also (i) of Lemma \ref{Lemma-NP-decreasing} for the detailed proof.) Thus, we can extend this function over $[0,\infty]$ in the following  manner:
 \begin{eqnarray}\label{N-infty-y0}
  N(\infty,y_0)\triangleq \lim_{t\rightarrow\infty} N(t,y_0)\;\;\mbox{and}\;\;
   N(0,y_0)\triangleq \lim_{t\rightarrow 0^+} N(t,y_0),\;\;y_0\in X\setminus\{0\}.
\end{eqnarray}

As mentioned in
     Subsection 1.1,  we    impose two assumptions on  $(A,B)$
     as follows:

     \noindent {(H1)}   There is  $ p_0\in [2,\infty)$ so that $\mathcal{A}_{p_0}(T,\hat t)\subset \mathcal{A}_\infty(T,\hat t)$ for all $T$, $\hat t$, with $0<\hat t<T<\infty$, where
      \begin{equation*}\label{assumptionspace1}
  \mathcal{A}_{p_0}(T, \hat t)\triangleq \Big\{\hat{y}(T;0,u)\; :\; u\in L^{p_0}(0,T;U),\;\;\mbox{with}\;\;u|_{(\hat t,T)}=0\Big\};
 \end{equation*}
   \begin{equation*}\label{assumptionspace2}
  \mathcal{A}_\infty(T,\hat t)\triangleq \Big\{\hat{y}(T;0,v)\; :\;v\in L^{\infty}(0,T;U),\;\;\mbox{with}\;\;v|_{(0,\hat t)}=0\Big\}.
 \end{equation*}

     \noindent {(H2}) If  there is  $T\in(0,\infty)$, a subset
      $E\subset (0,T)$ of positive measure and a function
      $f\in Y_T$ so that  $f=0$ over $E$, then $f\equiv 0$ over $(0,T)$.
Here,
    \begin{equation}\label{ob-space}
   Y_T \triangleq \overline{X}_T^{\|\cdot\|_{L^1(0,T;U)}},\;\;\mbox{with the}\;\; L^1(0,T;U)\mbox{-norm},
 \end{equation}
    where
        \begin{equation}\label{assumptionspace3}
  X_T\triangleq \{B^*S^*(T-\cdot)z|_{(0,T)}\; :\; z\in D(A^*)\},\;\;\mbox{with the}\;\; L^1(0,T;U)\mbox{-norm}.
 \end{equation}

\begin{Remark}\label{newwangremark1.1}
(i) The assumption (H1)  says  roughly that the functionality of a control  supported on $(0,\hat t)$  can be replaced by that of a control  supported on $(\hat t,T)$. The assumption (H2) says, in plain language, that any function in $Y_T$ has some unique continuation property from measurable sets.

 \noindent (ii) We do not know if every  function in $Y_T$ can be  expressed as $B^*\varphi$ with $\varphi$ a solution of the adjoint equation over $(0,T)$, even in the case that $B\in \mathcal{L}(U,X)$.  However, if  $(A,B)$ has the  $L^\infty$-null controllability, then the above-mentioned expression holds (see Remark~\ref{YT-characterization}).

 \noindent (iii)   Each pair $(A,B)$ in $\mathbb R^{n\times n}\times (\mathbb R^{n \times m}\setminus\{0\})$ (with $n,m\geq 1$) satisfies both (H1) and (H2)
 (see Proposition~\ref{Lemma-8ex1-1} in Appendix B).
\end{Remark}

Our  studies on the BBP decompositions are based on the assumptions (H1) and (H2).
However, our main  results  can be improved, if  instead of (H1) and (H2), we  impose the following stronger assumptions (H3) and (H4):

 \noindent {(H3)} The pair $(A^*, B^*)$ is  $L^1$-observable over each interval (or simply $L^1$-observable), i.e.,  for each $T\in(0,\infty)$, there exists a
 positive constant $C_1(T)$ so that
 \begin{eqnarray}\label{ob-ineq}
  \|S^*(T)z\|_X\leq C_1(T) \int_0^T \|B^*S^*(T-t)z\|_U \,\mathrm dt\;\;\mbox{for all}\;\; z\in D(A^*).
 \end{eqnarray}

 \noindent {(H4}) If $z\in X$ satisfies that $\widetilde{B^*S^*}(T-\cdot)z= 0$ over  $E$ for some $T\in(0,\infty)$ and some subset $E\subset(0,T)$ of positive measure, then  $\widetilde{B^*S^*}(T-\cdot)z\equiv0$ over $(0,T)$. Here, $\widetilde{B^*S^*}(T-\cdot)$ is the natural extension of
 ${B^*S^*}(T-\cdot)$ over $X$. (It will be explained in the next remark.)

 \begin{Remark} (i) The function  $\widetilde{B^*S^*}(T-\cdot)$ in (H4)  is defined in the following manner:   Since $B\in\mathcal L(U,X_{-1})$ is an admissible control operator for $\{S(t)\}_{t\in\mathbb R^+}$, it follows from Lemma \ref{ad-control-ob}
   that $B^*$ is an admissible observation operator for $\{S^*(t)\}_{t\in\mathbb R^+}$, i.e.,
   for each $T\in(0,\infty)$, there exists a  $C(T)>0$ so that
 $$
   \int_0^{T} \|B^* S^*(T-\tau)z\|_U^2 \,\mathrm d\tau \leq C(T) \|z\|_X^2\;\;\mbox{for all}\;\; z\in D(A^*).
$$
 (Indeed,   \cite[Theorem 4.4.3]{TW} proves that  $B\in\mathcal L(U,X_{-1})$ is an admissible control operator for $\{S(t)\}_{t\in\mathbb R^+}$ if and only if
 $B^*$ is an admissible observation operator for $\{S^*(t)\}_{t\in\mathbb R^+}$ in  the case where $X$ and $U$ are complex Hilbert spaces.)
  Thus, for each $T\in(0,\infty)$,  the operator
 $ B^* S^*(T-\cdot): D(A^*) \rightarrow  L^2(0,T;U) $
   can be uniquely extended  to a linear bounded operator $\widetilde{B^*S^*}(T-\cdot)\;$ from $X$ to $L^2(0,T;U)$. More precisely, for each
 $z\in X$,
 \begin{equation}\label{wang1.15}
 \widetilde{B^*S^*}(T-\cdot)z \triangleq \lim_{n\rightarrow\infty}
 B^*S^*(T-\cdot)z_n\;\;\mbox{in}\;\;L^2(0,T; U),
 \end{equation}
where $\{z_n\}\subset D(A^*)$, with $\lim_{n\rightarrow\infty} z_n = z$ in $X$.

 \noindent (ii) The  condition (H3) is an $L^1$-observability estimate for the pair $(A^*,B^*)$, which is
  equivalent to the $L^\infty$-null controllability for the pair $(A,B)$.
(See  Proposition~\ref{newhuangproposition2.15}.)

 \noindent (iii) The condition (H4) is a kind of unique continuation property of the dual equation over  $(0,T)$ for each $T\in(0,\infty)$.

  \noindent (iv)  The condition (H1) can be implied by  (H3) (see Proposition~\ref{WGSWGSlemma2.15}). However, (H1) may hold when (H3) does not stand. For instance, when $X=\mathbb{R}^n$ and $U=\mathbb{R}^m$ (with $n,m\in \mathbb{N}^+$), any pair $(A,B)\in \mathbb{R}^{n\times n}\times (\mathbb{R}^{n\times m}\setminus\{0\})$ satisfies not only the condition (H1) but also the condition (H2) (see Proposition~\ref{Lemma-8ex1-1} in Appendix B). On the other hand, it is well known that  $(A,B)$ is $L^\infty$-null controllable if and only if it is controllable, and the later holds if and only if
  $(A,B)$ satisfies the Kalman rank condition. Thus any $(A,B)\in \mathbb{R}^{n\times n}\times (\mathbb{R}^{n\times m}\setminus\{0\})$ that does not satisfy the Kalman rank condition has the property (H1) but does not hold the property (H3).

    \noindent (v) The condition  (H2) can be derived from (H3) and (H4) (see Proposition \ref{Lemma-YT-bounded}).

 \end{Remark}

\subsection{Main results}

The main results of this paper concern with  the BBP decompositions for $(TP)^{M,y_0}$ and $(NP)^{T,y_0}$. To state them,
we  notice that  the domain $\mathcal{W}$ of the pairs $(T,y_0)$ for  $(NP)^{T,y_0}$ and the domain $\mathcal V$ of  the pairs $(M,y_0)$ for $(TP)^{M,y_0}$ are the following spaces:
  \begin{equation}\label{zhangjinchu1}
   \mathcal{W}=
   \big\{(T,y_0) ~:~ 0<T<\infty,~y_0\in X\setminus\{0\}\big\},
   \end{equation}
   and
    \begin{eqnarray}\label{TP-divide-domain}
 \mathcal V =
 \big\{(M,y_0) ~:~ 0<M<\infty,~y_0\in X\setminus\{0\}\big\}.
\end{eqnarray}
  In the domain $\mathcal{W}$, we define the following subsets:
  \begin{eqnarray}\label{zhangjinchu5}
\mathcal{W}_{1,1}&\triangleq&\{(T,y_0)\in \mathcal{W}_1 ~:~ T< T^0(y_0)\},\nonumber\\
\mathcal{W}_{1,2}&\triangleq&\{(T,y_0)\in \mathcal{W}_1 ~:~ T\geq T^0(y_0)\},
\end{eqnarray}
where
\begin{eqnarray}\label{zhangjinchu3-1}
\mathcal{W}_1\triangleq\{(T,y_0)\in \mathcal{W} ~:~ N(T^0(y_0),y_0)=0\};
\end{eqnarray}
\begin{eqnarray}\label{zhangjinchu7}
\mathcal{W}_{2,1}&\triangleq&\{(T,y_0)\in \mathcal{W}_2 ~:~ T< T^0(y_0)\},\nonumber\\
\mathcal{W}_{2,2}&\triangleq&\{(T,y_0)\in \mathcal{W}_2 ~:~ T= T^0(y_0)\},\\
\mathcal{W}_{2,3}&\triangleq&\{(T,y_0)\in \mathcal{W}_2 ~:~ T^0(y_0)<T<T^1(y_0)\},\nonumber\\
\mathcal{W}_{2,4}&\triangleq&\{(T,y_0)\in \mathcal{W}_2 ~:~ T^0(y_0)<T,\,T\geq T^1(y_0)\},\nonumber
\end{eqnarray}
where
\begin{eqnarray}\label{zhangjinchu3-2}
  \mathcal{W}_2 \triangleq\{(T,y_0)\in \mathcal{W} ~:~ 0<N(T^0(y_0),y_0)<\infty\};
\end{eqnarray}
\begin{eqnarray}\label{zhangjinchu9}
\mathcal{W}_{3,1}&\triangleq&\{(T,y_0)\in \mathcal{W}_3 ~:~ T^0(y_0)<\infty,\,T\leq T^0(y_0)\},\nonumber\\
\mathcal{W}_{3,2}&\triangleq&\{(T,y_0)\in \mathcal{W}_3 ~:~ T^0(y_0)<\infty,\,T^0(y_0)<T<T^1(y_0)\},\\
\mathcal{W}_{3,3}&\triangleq&\{(T,y_0)\in \mathcal{W}_3 ~:~ T^0(y_0)<\infty,\,T^0(y_0)<T,\,T\geq T^1(y_0)\},\nonumber\\
\mathcal{W}_{3,4}&\triangleq&\{(T,y_0)\in \mathcal{W}_3 ~:~ T^0(y_0)=\infty\},\nonumber
\end{eqnarray}
where
\begin{eqnarray}\label{zhangjinchu3}
  \mathcal{W}_3\triangleq\{(T,y_0)\in \mathcal{W} ~:~ N(T^0(y_0),y_0)=\infty\}.
\end{eqnarray}
 In the domain $\mathcal{V}$, we define the following subsets:
 \begin{eqnarray}\label{Lambda-di-0-1-1}
 \mathcal V_1  &\triangleq&  \{(M,y_0)\in \mathcal V  ~:~  N(T^0(y_0),y_0)=0\};
 \end{eqnarray}
 \begin{eqnarray}\label{Lambda-di-2-1}
 \mathcal V_{2,1}  &\triangleq&  \{(M,y_0)\in \mathcal V_2  ~:~  M\leq N(T^1(y_0),y_0)\},
      \nonumber\\
 \mathcal V_{2,2}  &\triangleq&  \{(M,y_0)\in \mathcal V_2  ~:~
 N(T^1(y_0),y_0) < M < N(T^0(y_0),y_0) \},
     \nonumber \\
 \mathcal V_{2,3}  &\triangleq&  \{(M,y_0)\in \mathcal V_2  ~:~  N(T^1(y_0),y_0) < M,\,M= N(T^0(y_0),y_0)\},
      \\
  \mathcal V_{2,4}  &\triangleq&  \{(M,y_0)\in \mathcal V_2  ~:~ N(T^1(y_0),y_0) < M,\, M> N(T^0(y_0),y_0)\},
      \nonumber
\end{eqnarray}
 where
 \begin{eqnarray}\label{Lambda-di-0-1-2}
 \mathcal V_2  \triangleq  \{(M,y_0)\in \mathcal V  ~:~  0<N(T^0(y_0),y_0)<\infty\};
 \end{eqnarray}
 \begin{eqnarray}\label{Lambda-di-3-1}
 \mathcal V_{3,1}  &\triangleq&  \{(M,y_0)\in \mathcal V_3  ~:~  T^0(y_0)<\infty,\,M\leq N(T^1(y_0),y_0)\},
      \nonumber\\
 \mathcal V_{3,2}  &\triangleq&  \{(M,y_0)\in \mathcal V_3  ~:~  T^0(y_0)<\infty,\,M> N(T^1(y_0),y_0)\},
      \\
 \mathcal V_{3,3}  &\triangleq&  \{(M,y_0)\in \mathcal V_3  ~:~  T^0(y_0) =\infty\},
     \nonumber
\end{eqnarray}
where
\begin{eqnarray}\label{Lambda-di-0-1}
  \mathcal V_3  \triangleq  \{(M,y_0)\in \mathcal V  ~:~  N(T^0(y_0),y_0)=\infty\}.
  \end{eqnarray}

The main results of this paper are presented in the following two theorems:

   \begin{Theorem}\label{Proposition-NTy0-partition}
Let $\mathcal{W}$ be given by (\ref{zhangjinchu1}). Let $\mathcal{W}_{1,j}$ ($j=1,2$), $\mathcal{W}_{2,j}$ ($j=1,2,3,4$), and $\mathcal{W}_{3,j}$ ($j=1,2,3,4$)
be given by (\ref{zhangjinchu5}), (\ref{zhangjinchu7}) and (\ref{zhangjinchu9}),
respectively.
Then the following conclusions  are true:

\noindent(i) The set $\mathcal{W}$ is the disjoint union of the above mentioned
subsets $\mathcal{W}_{i,j}$.


 \noindent(ii) For each $(T,y_0)\in \mathcal{W}_{1,2}\cup\mathcal{W}_{2,4}\cup\mathcal{W}_{3,3}$, $(NP)^{T,y_0}$
 has the bang-bang property and
 the null control is its unique minimal norm control.

\noindent(iii) Suppose that (H1) and (H2) hold. Then for each $(T,y_0)\in \mathcal{W}_{2,3}\cup\mathcal{W}_{3,2}$, $(NP)^{T,y_0}$ has  the bang-bang property and the null control is not a minimal norm control to this problem.

 \noindent(iv) For each $(T,y_0)\in \mathcal{W}_{1,1}\cup\mathcal{W}_{2,1}\cup\mathcal{W}_{3,1}\cup\mathcal{W}_{3,4}$, $(NP)^{T,y_0}$ has no any admissible control and does not hold the bang-bang property.

 \noindent (v) For each $(T,y_0)\in \mathcal{W}_{2,2}$,  $(NP)^{T,y_0}$ has at least one minimal norm control.


\end{Theorem}

\begin{Theorem}\label{Proposition-TMy0-partition}
Let $\mathcal V$ be given by (\ref{TP-divide-domain}). Let $\mathcal V_{1}$, $\mathcal V_{2,j}$ ($j=1,2,3,4$) and $\mathcal V_{3,j}$ ($j=1,2,3$) be given by (\ref{Lambda-di-0-1-1}), (\ref{Lambda-di-2-1}) and (\ref{Lambda-di-3-1}), respectively. Then the following conclusions are true:

\noindent (i) The set $\mathcal V$ is the disjoint union of $\mathcal V_1$ and the above mentioned subsets $\mathcal V_{i,j}$.


\noindent(ii) Suppose that (H1) and (H2) hold. Then for each $(M,y_0)\in \mathcal V_{2,2}\cup\mathcal V_{3,2}$, $(TP)^{M,y_0}$ has the bang-bang property.

\noindent(iii)  Suppose that (H1) holds. Then for each $(M,y_0)\in \mathcal V_{2,4}$, $(TP)^{M,y_0}$ has infinitely many different minimal time controls (not including the null control), and does not hold the bang-bang property.

\noindent(iv)  Suppose that (H1) holds. Then for each $(M,y_0)\in \mathcal V_{1}$, $(TP)^{M,y_0}$ has infinitely many different minimal time controls (including the null control), and does not hold the bang-bang property.

 \noindent(v) For each $(M,y_0)\in \mathcal V_{3,3}$, $(TP)^{M,y_0}$ has no any admissible control
 and does not hold the bang-bang property. If  assume that  (H1)  holds, then for each $(M,y_0)\in \mathcal V_{2,1}\cup\mathcal V_{3,1}$, $(TP)^{M,y_0}$ has no any admissible control
 and does not hold the bang-bang property.

 \noindent (vi) For each $(M,y_0)\in \mathcal{V}_{2,3}$,  $(TP)^{M,y_0}$ has at least one minimal time control.

\end{Theorem}

\vskip 5 pt

\begin{Remark}\label{1217-Remark-critical-point}

\begin{figure}[!thb]\label{Fig.Process}
 \includegraphics[scale=.6]{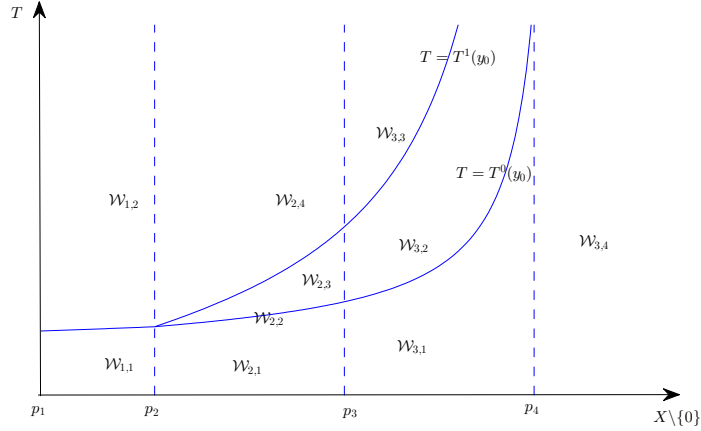}
\centering\caption{The BBP decomposition for $(NP)^{T,y_0}$}
\end{figure}

  To make the BBP decomposition for $(NP)^{T,y_0}$ (i.e., the decomposition of $\mathcal{W}$ given by Theorem~\ref{Proposition-NTy0-partition})  understood better,  a draft is given in  Figure 1. We  explain it as follows: The abscissa axis   denotes the set $X\setminus\{0\}$, while the
      ordinates axis  denotes  the set of time variables $T>0$. Each  $p_i$ (with $i=1,2,3,4$) on the abscissa axis is a  ``point" of the set $X\setminus\{0\}$.

In figure 1, some notations are explained as follows:
\begin{itemize}
  \item  $(p_1,p_2]$ denotes the set:
   $\{y_0\in X\setminus\{0\}:\,N(T^0(y_0),y_0)=0\}$.
     \item  $(p_2,p_3)$ denotes the  set:
   $\{y_0\in X\setminus\{0\}:\,0<N(T^0(y_0),y_0)<\infty\}$.
     \item  $[p_3,p_4)$ denotes the set:
  $\{y_0\in X\setminus\{0\}:\,N(T^0(y_0),y_0)=\infty,\,T^0(y_0)<\infty\}$.
    \item  $[p_4,\infty)$ denotes the  set:
$\{y_0\in X\setminus\{0\}:\,N(T^0(y_0),y_0)=\infty,\,T^0(y_0)=\infty\}$.
  \item The two curves above the abscissa axis (from the left to the right) respectively  denote the graph of the  functions:
      $y_0\rightarrow T^1(y_0),\;\;y_0\in X\setminus\{0\}$ and
      $y_0\rightarrow T^0(y_0),\;\;y_0\in X\setminus\{0\}$.
          These two curves coincide  over $(p_1,p_2]$.
\end{itemize}
Let $\mathcal{W}_{1,j}$ ($j=1,2$), $\mathcal{W}_{2,j}$ ($j=1,2,3,4$), and $\mathcal{W}_{3,j}$ ($j=1,2,3,4$)
be given by (\ref{zhangjinchu5}), (\ref{zhangjinchu7}) and (\ref{zhangjinchu9}),
respectively. Then we conclude that
\begin{itemize}
  \item The set $\mathcal W_{1,1}$ is  the region $\{(T,y_0)~:~y_0\in (p_1,p_2],~0<T<T^0(y_0) \}$;
  \item The set $\mathcal W_{1,2}$ is  the region $\{(T,y_0)~:~y_0\in (p_1,p_2],~T^0(y_0)\leq T<\infty \}$;
  \item The set $\mathcal W_{2,1}$ is  the region $\{(T,y_0)~:~y_0\in (p_2,p_3),~0<T<T^0(y_0) \}$;
  \item The set $\mathcal W_{2,2}$ is  the region $\{(T,y_0)~:~y_0\in (p_2,p_3),~T=T^0(y_0) \}$;
  \item The set $\mathcal W_{2,3}$ is  the region $\{(T,y_0)~:~y_0\in (p_2,p_3),~T^0(y_0)<T<T^1(y_0) \}$;
  \item The set $\mathcal W_{2,4}$ is  the region $\{(T,y_0)~:~y_0\in (p_2,p_3),~T^1(y_0)\leq T<\infty \}$;
  \item The set $\mathcal W_{3,1}$ is  the region $\{(T,y_0)~:~y_0\in [p_3,p_4),~0<T\leq T^0(y_0) \}$;
  \item The set $\mathcal W_{3,2}$ is  the region $\{(T,y_0)~:~y_0\in [p_3,p_4),~T^0(y_0)<T<T^1(y_0) \}$;
  \item The set $\mathcal W_{3,3}$ is  the region $\{(T,y_0)~:~y_0\in [p_3,p_4),~T^1(y_0)\leq T<\infty \}$;
  \item The set $\mathcal W_{3,4}$ is  the region $\{(T,y_0)~:~y_0\in [p_4,\infty),~0<T<\infty \}$;
  \item When $\{S(t)\}_{t\in\mathbb R^+}$ has the backward uniqueness property, we have that $T^1(y_0)=\infty$ for all $y_0\in X\setminus\{0\}$. In this case,  the curve:  $\{(y_0, T^1(y_0))\;:\; y_0\in X\setminus\{0\}\}$ will not appear in Figure 1;  $\mathcal W_{1,1} \cup \mathcal W_{1,2} \cup \mathcal W_{2,4} \cup \mathcal W_{3,3}=\emptyset$ (see (iv) of Lemma \ref{Lemma-NT0-T0-T1}).
\end{itemize}

\begin{figure}[!thb]
 \includegraphics[scale=.6]{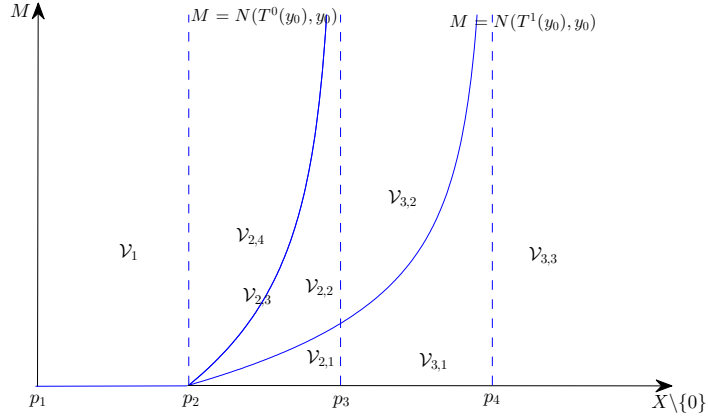}
\centering\caption{The BBP decomposition for  $(TP)^{M,y_0}$}
\end{figure}

 To make the BBP decomposition for $(TP)^{M,y_0}$ (i.e., the decomposition of $\mathcal{V}$ given by Theorem~\ref{Proposition-TMy0-partition})  understood better,  a draft is given  in Figure 2. We  explain this figure as follows: The abscissa axis   denotes the set $X\setminus\{0\}$, while the
      ordinates axis  denotes  the  variables $M>0$. Each  $p_i$, with $i=1,2,3,4$, on the abscissa axis is a ``point" of the set $X\setminus\{0\}$.

In figure 2, some notations are given in order.
\begin{itemize}
  \item  $(p_1,p_2]$ denotes  the set:
   $\{y_0\in X\setminus\{0\}:\,N(T^0(y_0),y_0)=0\}$.
     \item  $(p_2,p_3)$ denotes the set:
  $\{y_0\in X\setminus\{0\}:\,0<N(T^0(y_0),y_0)<\infty\}$.
   \item  $[p_3,p_4)$ denotes the set:
$\{y_0\in X\setminus\{0\}:\,N(T^0(y_0),y_0)=\infty,\,T^0(y_0)<\infty\}$.
  \item  $[p_4,\infty)$ denotes the set:
$\{y_0\in X\setminus\{0\}:\,N(T^0(y_0),y_0)=\infty,\,T^0(y_0)=\infty\}$.
  \item The two curves above the abscissa axis (denoted by $F_0$ and $F_1$ from the left to the right) respectively denote  the graphs of the functions:
       $y_0\rightarrow N(T^0(y_0),y_0),\;y_0\in X\setminus\{0\}$
       and $y_0\rightarrow N(T^1(y_0),y_0),\; y_0\in X\setminus\{0\}$.
      These two curves are identically zero  over $(p_1,p_2]$.
      \end{itemize}
Let $\mathcal{V}_{1}$, $\mathcal{V}_{2,j}$ ($j=1,2,3,4$), and $\mathcal{V}_{3,j}$ ($j=1,2,3$)
be given by (\ref{Lambda-di-0-1-1}), (\ref{Lambda-di-2-1}) and (\ref{Lambda-di-3-1}),
respectively. Then we have the following conclusions:
\begin{itemize}
  \item The set $\mathcal V_{1}$ is  the region $\{(M,y_0)~:~ y_0\in (p_1,p_2],~0<M<\infty\}$;
  \item The set $\mathcal V_{2,1}$ is  the region $\{(M,y_0)~:~ y_0\in (p_2,p_3),~0<M \leq F_1(y_0)\}$;
  \item The set $\mathcal V_{2,2}$ is  the region $\{(M,y_0)~:~ y_0\in (p_2,p_3),~F_1(y_0)<M<F_0(y_0)\}$;
  \item The set $\mathcal V_{2,3}$ is  the region $\{(M,y_0)~:~ y_0\in (p_2,p_3),~M=F_0(y_0),~M\neq F_1(y_0)\}$;
  \item The set $\mathcal V_{2,4}$ is  the region $\{(M,y_0)~:~ y_0\in (p_2,p_3),~F_0(y_0)<M<\infty\}$;
  \item The set $\mathcal V_{3,1}$ is  the region $\{(M,y_0)~:~ y_0\in [p_3,p_4),~0<M \leq F_1(y_0)\}$;
  \item The set $\mathcal V_{3,2}$ is  the region $\{(M,y_0)~:~ y_0\in [p_3,p_4),~F_1(y_0)<M<\infty\}$;
  \item The set $\mathcal V_{3,3}$ is  the region $\{(M,y_0)~:~ y_0\in [p_4,\infty),~0<M <\infty\}$.
\end{itemize}

 \end{Remark}

 \begin{Remark}
 (i) The decomposition given by  Theorem \ref{Proposition-NTy0-partition}
 is comparable with the decomposition (P2) in Subsection 1.1, except for the part
 $\mathcal{W}_{2,2}$, which is indeed the following ``curve" in the product space
  $\mathcal{W}$:
 \begin{equation}\label{zhangjinchu1-5}
    \gamma_1\triangleq \big\{(T^0(y_0), y_0)\in \mathcal{W}
    \; :\; 0<N(T^0(y_0),y_0)<\infty\big\}.
   \end{equation}
   It is a critical curve  in the following sense: First, we do not know if it is empty. Second, when $(T,y_0)\in\gamma_1$,  we know the corresponding $(NP)^{T,y_0}$ has
    at least one minimal norm control, but we are not sure if it has
    the bang-bang property. It deserves to mention that when $(A,B)$ is  $L^\infty$-null controllable, this curve is empty (see Theorem \ref{Proposition-Corollary1.7}).

 (ii) The decomposition given by  Theorem \ref{Proposition-TMy0-partition}
 is comparable with the decomposition (P1) in Subsection 1.1, except for the part
 $\mathcal{V}_{2,3}$, which is indeed the following ``curve" in the product space
  $\mathcal{V}$:
 \begin{eqnarray}\label{TP-divide-domain-left}
 \gamma_2  \triangleq
 \left\{\big(N(T^0(y_0),y_0),y_0\big)\in \mathcal{V} ~:~
 \begin{array}{l}
  0<N(T^0(y_0),y_0)<\infty,\\
  N(T^0(y_0),y_0)\neq N(T^1(y_0),y_0)
  \end{array}
  \right\}.
\end{eqnarray}
   It is a critical curve in the BBP decomposition for $(TP)^{M,y_0}$ in the following sense: First, we do not know if it is empty. Second, when $(M,y_0)\in\gamma_2$, we  know the corresponding $(TP)^{M,y_0}$ has a minimal time control, but we are not sure if it has  the bang-bang property. It deserves to mention that when $(A,B)$ is  $L^\infty$-null controllable, this curve is empty (see Theorem \ref{Proposition-Corollary1.7}).
 \end{Remark}

 \begin{Remark}
 In the finitely dimensional setting where $(A,B)$ is a pair in $\mathbb R^{n\times n}\times (\mathbb R^{n \times m}\setminus\{0\})$, (with $n,m\in\mathbb N^+$),  the BBP decompositions for $(TP)^{M,y_0}$ and $(NP)^{T,y_0}$, obtained in Theorem \ref{Proposition-NTy0-partition} and Theorem \ref{Proposition-TMy0-partition}, are exactly the same as (P1) and (P2) in Subsection 1.1. This is proved in Appendix C
(see Proposition~\ref{NEWYearprop8.4}).

 \end{Remark}

Under the assumptions (H3) and (H4), the main results obtained in Theorem \ref{Proposition-NTy0-partition} and Theorem \ref{Proposition-TMy0-partition} can be improved as follows:

%
%

\begin{Theorem}\label{Proposition-Corollary1.7}
 Let $\mathcal{W}$  and $\mathcal W_{3,j}$ ($j=2,3$) be given by (\ref{zhangjinchu1}) and (\ref{zhangjinchu9}), respectively. Let $\mathcal{V}$  and $\mathcal{V}_{3,j}$ ($j=1,2$) be given by    (\ref{TP-divide-domain})  and  (\ref{Lambda-di-3-1}), respectively. Then
 the following conclusions are true:

\noindent(i) Suppose that (H3) holds. Then
\begin{eqnarray}\label{no-critical-point-2}
 \mathcal W=\mathcal W_{3,2} \cup \mathcal{W}_{3,3}   \;\;\mbox{and}\;\;
 \mathcal V=\mathcal V_{3,1} \cup \mathcal V_{3,2}.
\end{eqnarray}
In particular,
\begin{eqnarray}\label{no-critical-point-3}
 \gamma_1=\mathcal{W}_{2,2}=\emptyset
 \;\;\mbox{and}\;\;   \gamma_2=\mathcal{V}_{2,3}=\emptyset.
\end{eqnarray}
where $\gamma_1$, $\gamma_2$, $\mathcal{W}_{2,2}$ and $\mathcal{V}_{2,3}$ are given respectively by (\ref{zhangjinchu1-5}), (\ref{TP-divide-domain-left}), (\ref{zhangjinchu7}) and (\ref{Lambda-di-2-1}).

\noindent(ii) Suppose that (H3) holds.
Then for each $(M,y_0)\in  \mathcal V_{3,1}$,   $(TP)^{M,y_0}$ has no any admissible control and does not hold the bang-bang property. If further assume that (H4) holds, then
 for each $(M,y_0)\in\mathcal V_{3,2}$, $(TP)^{M,y_0}$ has the bang-bang property.

\noindent(iii) For each $(T,y_0)\in\mathcal W_{3,3}$, the null control is the unique minimal norm control to $(NP)^{T,y_0}$ and this problem has the bang-bang property. If further assume  that (H3) and (H4) hold,  then for each $(T,y_0)\in\mathcal W_{3,2}$, $(NP)^{T,y_0}$ has the bang-bang property and the null control is not a minimal norm control to this problem.

\end{Theorem}

\subsection{The  ideas to get the main results}

The main difficulty to get the BBP decompositions of $(TP)^{M,y_0}$ and $(NP)^{T,y_0}$ is  the lack of the Kalman controllability decomposition. The first key to overcome this difficulty is to find assumptions (H1) and (H2). Then with the aid of functions $T^0(\cdot)$, $T^1(\cdot)$ and $N(\cdot, y_0)$, we
get the conclusions (i) in both Theorem~\ref{Proposition-NTy0-partition} and
Theorem \ref{Proposition-TMy0-partition}. In the decomposition of $\mathcal{W}$, the part
 $\mathcal W_{2,2}=\gamma_1$ is a critical curve for us; the studies for the problem $(NP)^{T,y_0}$, with $(T,y_0)\in\mathcal W_{2,3}\cup\mathcal W_{3,2}$, are not easy for us; when $(T,y_0)$ is in  the rest parts,
 it is easy to prove the corresponding conclusions in Theorem \ref{Proposition-NTy0-partition} for  $(NP)^{T,y_0}$, through using properties of  functions $T^0(\cdot)$, $T^1(\cdot)$ and $N(\cdot, y_0)$. The proof of the  corresponding conclusion in Theorem \ref{Proposition-NTy0-partition}  for  $(NP)^{T,y_0}$, with $(T,y_0)\in \mathcal W_{2,3}\cup \mathcal W_{3,2}$,
 is mainly based on a maximum principle for $(NP)^{T,y_0}$, as well as (H2). To get the maximum principle, we  build up the   following affiliated minimal norm problems:
 \begin{eqnarray}\label{attianable-space-norm}
  (NP)^{y_T}\;\;\;\;\;\; \|y_T\|_{\mathcal R_T} \triangleq
  \inf\big\{\|v\|_{L^\infty(0,T;U)}\; :\; \hat{y}(T;0,v)=y_T\big\},
 \end{eqnarray}
where $T\in(0,\infty)$ and $y_T$ is in the   reachable  subspace
 \begin{eqnarray}\label{attainable-space}
   \mathcal R_T \triangleq \big\{\hat{y}(T;0,v)\; :\; v\in L^\infty(0,T;U) \big\}.
 \end{eqnarray}
 (In the problem $(NP)^{y_T}$, we can define the minimal norm,  an admissible control, a  minimal norm control and the bang-bang property in the  similar manner as in $({NP})^{T,y_0}$ (see  (\ref{NP-0})).
  By the connection between  $(NP)^{y_T}$ and $(NP)^{T,y_0}$ built up in Proposition~\ref{Lemma-NP-yT-y0-eq}, we realized that the maximum principle for
  $(NP)^{T,y_0}$ can be derived from a maximum principle for $(NP)^{y_T}$.
  Though we are not able to get a maximum principle of $(NP)^{y_T}$ for all  $y_T\in \mathcal R_T$, we  get a maximum principle for $(NP)^{y_T}$, with $y_T$ in the subspace:
\begin{eqnarray}\label{R0T}
  \mathcal R^0_T\triangleq
  \big\{\hat{y}(T;0,v)\; :\; v\in L^\infty(0,T;U),\,\lim_{s\rightarrow T}\|v\|_{L^\infty(s,T;U)}=0 \big\},
  \;\;T\in(0,\infty).
 \end{eqnarray}
More precisely, we obtain that if (H1) holds, then for each $y_T\in \mathcal{R}_T^0\setminus\{0\}$, there exists a vector $ f^*\in Y_T\setminus\{0\}$ so that  each minimal norm control $v^*$ to $(NP)^{y_T}$ verifies that
  \begin{eqnarray}\label{wang1.17}
   \langle v^*(t),f^*(t) \rangle_U=\max_{\|w\|_U\leq \|y_T\|_{\mathcal R_T}} \langle w,f^*(t) \rangle_U~\mbox{ a.e. }t\in (0,T).
  \end{eqnarray}
(This is exactly Theorem~\ref{Lemma-maximum-NP-yT}.) About (\ref{wang1.17}), we would like to mention two facts: First, it is not the standard Pontryagin maximum principle, since we are not sure if the function $f^*$ in (\ref{wang1.17})  can   be expressed as $B^*\varphi$ with $\varphi$ a solution of the adjoint equation, even in the case that $B\in \mathcal{L}(U,X)$. Second, the proof of (\ref{wang1.17}) is the most difficult part in this paper. It is based on two  representation theorems (Theorem~\ref{Theorem-the-representation-theorem} and Theorem~\ref{Theorem-the-second-representation-theorem}). From  (\ref{wang1.17}) and the connection between $(NP)^{y_T}$ and $(NP)^{T,y_0}$ built up in Proposition~\ref{Lemma-NP-yT-y0-eq}, we get  the  maximum principle for $(NP)^{T,y_0}$, with
 $(T,y_0)\in \mathcal W_{2,3}\cup \mathcal W_{3,2}$, which along with (H2), yields
 that when $(T,y_0)\in \mathcal W_{2,3}\cup \mathcal W_{3,2}$, $(NP)^{T,y_0}$ has the bang-bang property.

 Regarding the  decomposition of $\mathcal{V}$, the part
 $\mathcal V_{2,3}=\gamma_2$ is a critical curve for us; the studies for the problem $(TP)^{M,y_0}$, with $(M,y_0)\in\mathcal V_{2,2}\cup \mathcal V_{3,2}$, are not easy for us; when $(M,y_0)$ is in  the rest parts,
 it is easy to prove the corresponding conclusions in Theorem \ref{Proposition-TMy0-partition} for  $(TP)^{M,y_0}$, through using properties of  functions $T^0(\cdot)$, $T^1(\cdot)$ and $N(\cdot, y_0)$, as well as the assumption (H1). The proof of the  corresponding conclusion in Theorem \ref{Proposition-TMy0-partition}  for  $(TP)^{M,y_0}$, with $(M,y_0)\in \mathcal V_{2,2}\cup \mathcal V_{3,2}$,
 is mainly based on a maximum principle for $(TP)^{M,y_0}$, as well as (H2). This maximum principle  follows from the above-mentioned maximum principle for $(NP)^{T,y_0}$, with
 $(T,y_0)\in \mathcal W_{2,3}\cup \mathcal W_{3,2}$, as well as the connection between
 $(TP)^{M,y_0}$ and $(NP)^{T,y_0}$ built up in Lemma~\ref{the-equivalence}.

 \begin{Remark}\label{yubiaoremark1.10}
 The reason to cause curves $\gamma_1$ and $\gamma_2$ to be critical is that
 in general,
 we do not know if $(NP)^{y_T}$, with $y_T\in \mathcal{R}_T\setminus \mathcal{R}_T^0$, has the maximum principle (\ref{wang1.17}), under the assumption (H1).
   \end{Remark}

\subsection{More about the bang-bang properties}

To the best of our knowledge,  there are two ways to derive the bang-bang property for  minimal time control problems governed by linear evolution systems, in general.    The first one is the use of the $L^\infty$-null controllability from measurable sets.
      In \cite[Section 2.1]{HOF}, H. O. Fattorini
   studied the minimal time control problem for the  abstract system:
    \begin{equation}\label{HF1.15}
    y^\prime(t)=Ay(t)+u(t),~t> 0,
    \end{equation}
   with $A$ generating a $C_0$-semigroup in a Banach space.
   This corresponds to (\ref{system-0}) with $U=X$ and $B=Id_X$ (the identity operator on a Banach space $X$).
   By a constructive method, he proved that the reachable sets of (\ref{HF1.15}) have the following property: For any subset $E\subset(0,\infty)$ of positive measure,
  $\mathcal R_{T,E}=\mathcal R_T$ for a.e. $T\in E$,
   where
   $ \mathcal R_{T,E} \triangleq \{ y(T;0,\chi_E u) \,:\,  u\in L^\infty(\mathbb R^+;U)\}$.    From this property, he proved the bang-bang property by a contradiction argument.  In \cite{MS}, V. Mizel and T. Seidman pointed out  that the bang-bang property of minimal time control problems for linear time-invariant evolution systems can be derived by the $L^\infty$-null controllability from measurable sets. Indeed, by this controllability and by a translation invariance which holds only for  time invariant systems, one can  use a contradiction argument to prove the bang-bang property. However, it seems for us that this way does not work for the case where controlled systems are time-varying.  In \cite{WXZ}, the authors proved the bang-bang property of minimal time control problems for some very special time-varying heat equations.  To our best knowledge, how to study the bang-bang property of minimal time control problems for general time-varying systems is still a quite open problem.
    For studies on the $L^\infty$-null controllability from measurable sets, we  would like to mention the literatures \cite{AEWZ,MS,PW,PW1,PWCZ,Wang,WCZ,CZ-1} and the references therein.

    The second way is the use of the Pontryagin maximum principle and the unique  continuation property  from measurable sets in time. The key  is to derive the Pontryagin maximum principle. We would like to mention that the Pontryagin maximum principle may not hold for  some cases (see Example 1.4 on Page 132 in \cite{LY}). In \cite[Chapter 2]{HOF},  H. O. Fattorini studied the Pontryagin maximum principle for both minimal time and minimal norm control problems, with an initial state $\zeta$ and a target state $\bar y$, for the system (\ref{HF1.15}). He first proved the property that for each $T>0$,
     $ D(A)$ is continuously embedded into $\mathcal R_T$.
     Then, with the aid of this property, he divided the dual space of $\mathcal R_T$ into  {\it the regular part} and {\it the singular part}. After that, he proved that if
     $\bar y-S(T^*)\zeta\in \overline{D(A)}$,
       then
       $\bar y-S(T^*)\zeta$  and $B_{\mathcal{R}_{T^*}}(0,1)$
                can be separated by a hyperplane (in $\mathcal R_{T^*}$), with a regular normal vector. (Here,  $T^*$ is the minimal time,  $B_{\mathcal{R}_{T^*}}(0,1)$ is the closed unit ball in $\mathcal R_{T^*}$ and the controls for the minimal time control problem are within $L^\infty$-norm not larger than 1.) Finally, with the help of the aforementioned separating property, he obtained the Pontryagin maximum principle.
         By the second way, one might get the bang-bang property of minimal time control problems for the linear time-varying evolution systems which hold some unique continuation property.

         For the minimal norm control problems governed by linear time-varying evolution systems,
         the $L^\infty$-null controllability from measurable sets implies the bang-bang property. Though the paper \cite{PW} proves this only for heat equations with time-varying lower terms, the method in \cite{PW} works for general linear time-varying evolution systems.

   About studies  on
minimal time and minimal norm control problems, we would like to mention the literatures
\cite{ Arada-Raymond,   Barbu, carja,  HOF, HOF1, HOF2, HOF3, HOF4, HOF5, GL, ITO, KL1, KL2,  LW, LM,  LZ, LWX, QL, MRT, MZ,  MS,  Petrov, PW, PWCZ, PWZ, EJPGS, C.S.E.T, Wang, WX-1, WXZ, WCZ, WZheng, WZ, Y,
CZ-1, CZ-2, Guo-Zh} and the references therein.

    The rest part of this paper is organized as follows: Section 2
     studies some properties on the subspaces $\mathcal{R}_T$ and $\mathcal{R}_T^0$.
      Section 3 shows some properties of functions $N(\cdot,y_0)$, $T^0(\cdot)$ and $T^1(\cdot)$. Section 4 studies the existence of minimal time and minimal norm controls. Section 5 studies maximum principles and  bang-bang properties.
       Section 6 proves the main results.  Section 7  gives some applications.
Section 8 provides several appendixes.

 \bigskip
 \section{Properties on attainable subspaces}

 In this section, we  mainly  study the properties on the subspaces
 $\mathcal R_T$ and $\mathcal R_T^0$ given by (\ref{attainable-space}) and (\ref{R0T}), respectively. These properties mainly help us to get a maximum principle for   the affiliated minimal norm problem $(NP)^{y_T}$, with $y_T\in \mathcal R_T^0$.
   The later is the base in the proofs of (iii) of Theorem~\ref{Proposition-NTy0-partition} and (ii) of Theorem~\ref{Proposition-TMy0-partition}.

 \subsection{The first representation theorem}

 In this subsection, we will present a  representation theorem  on the space $Y_T^*$ which is the dual space of $Y_T$ (defined by (\ref{ob-space})). This theorem was built up for heat equations in  \cite[(i) of Theorem 1.4]{WXZ}.
 To prove it, we need the following two results: Proposition~\ref{huangwanghenproposition2.1} and Lemma~\ref{ad-control-ob}.
 Very similar versions of these two results are given in \cite[Section 2.3.1]{Coron}.
 For the sake of the completeness of the paper, we give their  proofs in Appendix D.

\begin{Proposition}\label{huangwanghenproposition2.1}
The following  equality is valid:
\begin{eqnarray}\label{NNNWWW2.2}
\big \langle \int_0^T S_{-1}(T-t)Bv(t) \,\mathrm dt, z \big\rangle _X
=  \int_0^T \langle v(t), B^*S^*(T-t)z \rangle_U  \,\mathrm dt
\end{eqnarray}
 for all $T\in(0,\infty)$,  $v\in L^\infty(0,T;U)$ and $z\in D(A^*)$.

\end{Proposition}

\begin{Lemma}\label{ad-control-ob}
  For each $T\in(0,\infty)$, there exists a positive constant $C(T)$ so that
  \begin{eqnarray}\label{admissible-observable}
   \int_0^{T} \|B^* S^*(T-\tau)z\|_U^2 \,\mathrm d\tau \leq C(T) \|z\|_X^2\;\;\mbox{for all}\;\; z\in D(A^*).
 \end{eqnarray}
 \end{Lemma}

 \begin{Theorem}\label{Theorem-the-representation-theorem}
  ~For each $T\in(0,\infty)$, there is a linear isomorphism $\Phi_T$ from $\mathcal R_T$ to $Y_T^*$ so that for all $y_T\in\mathcal R_T$ and $f\in Y_T$,
  \begin{eqnarray}\label{ob-attain-0}
   \langle y_T,f \rangle_{\mathcal R_T,Y_T}\triangleq  \langle \Phi_T(y_T),f\rangle_{Y_T^*,Y_T}=\int_0^T \langle v(t),f(t) \rangle_U \,\mathrm dt,
  \end{eqnarray}
  where $v$ is any  admissible control to $(NP)^{y_T}$.
 \end{Theorem}

 \begin{proof}
 Arbitrarily fix a $T\in (0,\infty)$.  It follows from (\ref{admissible-observable}) that
 \begin{equation}\label{huairenhuaishi2.8}
 B^*S^*(T-\cdot)z\in L^1(0,T;U)\;\;\mbox{for each}\;\;z\in D(A^*).
  \end{equation}
 For each $y_T\in \mathcal{R}_T$, define the following set of admissible controls to $(NP)^{y_T}$:
 \begin{equation}\label{HuangHuangHuang2.9}
 \mathcal U_{ad}^{y_T}\triangleq\{v\in L^\infty(0,T;U)\; :\; \hat y(T;0,v)=y_T\}.
 \end{equation}
 Observe from (\ref{attainable-space}) and (\ref{HuangHuangHuang2.9}) that
  $  \mathcal U_{ad}^{y_T}\neq\emptyset$ for each $y_T\in \mathcal R_T$,
  and that
  $  y_T=\hat y(T;0,v)$ for each $y_T\in \mathcal R_T$ and each $v\in\mathcal U_{ad}^{y_T}$.
   These, along with   (\ref{NNNWWW2.1})  and (\ref{huairenhuaishi2.8}), yields  that for each $y_T\in \mathcal R_T$, $z\in D(A^*)$ and each
  $v\in\mathcal U_{ad}^{y_T}$,
   \begin{eqnarray}\label{ob-attain-11}
  \langle y_T,z \rangle_X&=&\int_0^T \langle v(t),B^*S^*(T-t)z \rangle_U \,\mathrm dt \nonumber\\
     &\leq& \|v(\cdot)\|_{L^\infty(0,T;U)} \|B^*S^*(T-\cdot)z\|_{L^1(0,T;U)}.
 \end{eqnarray}

 Arbitrarily fix a $y_T\in \mathcal R_T$ and then fix a $v_1\in\mathcal U_{ad}^{y_T}$.
 Then we define a map $\mathcal{F}_{y_T}: X_T\rightarrow \mathbb{R}$ in the following manner:
  \begin{eqnarray}\label{ob-attain-12}
  \mathcal F_{y_T}\big(B^*S^*(T-\cdot)z|_{(0,T)}\big) &\triangleq&\int_0^T \langle v_1(t),B^*S^*(T-t)z \rangle_U \,\mathrm dt\;\;\mbox{for each}\;\;z\in D(A^*),~~~
 \end{eqnarray}
 where $X_T$ is given by  (\ref{assumptionspace3}).
  Because of the first equality in (\ref{ob-attain-11}), we see from (\ref{ob-attain-12})  that the definition of $\mathcal{F}_{y_T}$ is
  independent of the choice of $v_1\in\mathcal U_{ad}^{y_T}$. Thus it is well-defined.
   From (\ref{ob-attain-12}), the  inequality in (\ref{ob-attain-11}) and (\ref{ob-space}), we find that $\mathcal F_{y_T}$ can be uniquely extended to be an element  $\widetilde{\mathcal F}_{y_T}\in Y_T^*$. Furthermore, we have that
   $\|\widetilde{\mathcal F}_{y_T}\|_{Y_T^*} \leq \|v\|_{L^\infty(0,T;U)}$ for all $v\in\mathcal U_{ad}^{y_T}$.
  Since $y_T\in \mathcal R_T$ was arbitrarily fixed, the above estimate, along with (\ref{attianable-space-norm}), yields that
 \begin{eqnarray}\label{ob-attain-13}
  \|\widetilde{\mathcal F}_{y_T}\|_{Y_T^*} \leq \inf\{\|v\|_{L^\infty(0,T;U)}\; :\; v\in\mathcal U_{ad}^{y_T}\} =
  \|y_T\|_{\mathcal R_T}\;\;\mbox{for all}\;\;y_T\in\mathcal R_T.
 \end{eqnarray}

 We now define a map $\Phi_T:~\mathcal R_T \longrightarrow Y_T^*$  in the following manner:
 \begin{eqnarray}\label{ob-attain-14}
      \Phi_T(y_T)= \widetilde{\mathcal F}_{y_T}
      \;\;\mbox{for each}\;\; y_T\in\mathcal R_T.
 \end{eqnarray}
It is clear that $\Phi_T$ is well defined and linear.  We claim that
 $\Phi_T$ is surjective. Arbitrarily take
   $g\in Y_T^*$. Since $Y_T\subset L^1(0,T;U)$ (see (\ref{ob-space})),  according to  the Hahn-Banach theorem, there exists a $\widetilde g\in \big(L^1(0,T;U)\big)^*$ so that
$\widetilde g(\psi)=g(\psi)$ for all $\psi\in Y_T$; and so that $\|\widetilde g\|_{\mathcal L(L^1(0,T;U);\mathbb R)}=\|g\|_{Y_T^*}$.
 Then by the Riesz representation theorem, there is $\widehat v$ in $L^\infty(0,T;U)$ so that
 \begin{eqnarray}\label{ob-attain-15}
  \int_0^T \langle \widehat v(t),B^*S^*(T-t)z\rangle_U \,\mathrm dt= g\big(B^*S^*(T-\cdot)z|_{(0,T)}\big)
  \;\;\mbox{for all}\;\;z\in D(A^*)~~~
 \end{eqnarray}
 and so that
 \begin{eqnarray}\label{ob-attain-16}
  \|\widehat v\|_{L^\infty(0,T;U)}= \|g\|_{Y_T^*}.
 \end{eqnarray}
 Write $\hat y_T\triangleq \hat{y}(T;0,\widehat v)$.
 Then  $\hat v\in \mathcal{U}_{ad}^{\hat y_T}$ (see (\ref{HuangHuangHuang2.9})). This, together with  (\ref{ob-attain-15}), (\ref{ob-attain-12}) and (\ref{ob-space}), indicates that
 $  g=\widetilde{\mathcal{F}}_{\hat y_T}$  in $Y_T^*$,
  which, along with (\ref{ob-attain-14}), shows  that $\Phi_T$ is surjective.

  We now show that $\Phi_T$ is    injective. Let $y_T\in \mathcal R_T$ satisfy that
    $  \widetilde{\mathcal{F}}_{y_T}=0$ in $Y_T^*$.
     Then by (\ref{ob-attain-12}) and (\ref{ob-attain-11}), we find that
     $ \langle y_T,z\rangle_X=0$ for all $z\in D(A^*)$.
      Since $D(A^*)$ is dense in $X$, the above yields that $y_T=0$, which implies that  $\Phi_T$ is injective.

      We next show that $\Phi_T$ preserves norms. Let $g\in Y_T^*$. Then we have that
      $ g=\widetilde{\mathcal{F}}_{\hat y_T}$  in $Y_T^*$,
        where $\hat y_T=\hat y(T; 0,\hat v)$, with $\hat v\in L^\infty(0,T;U)$ satisfying (\ref{ob-attain-15}) and
      (\ref{ob-attain-16}). This, along with  (\ref{attianable-space-norm}) yields that
 $\|\hat y_T\|_{\mathcal R_T}\leq \|\widetilde{\mathcal{F}}_{\hat y_T}\|_{Y_T^*}$.
  From this and (\ref{ob-attain-13}), we see that $\Phi_T$ preserves norms.

 Finally, (\ref{ob-attain-0}) follows from  (\ref{ob-attain-14}), (\ref{ob-attain-12}) and (\ref{ob-space}). This ends the proof of this theorem.

 \end{proof}

 \begin{Remark}
   Since $Y_T^*$ is complete, it follows from  Theorem~\ref{Theorem-the-representation-theorem} that  the normed space  $(\mathcal R_T,\|\cdot\|_{\mathcal R_T})$ is complete.
   \end{Remark}

 \subsection{The second representation theorem }

  This subsection mainly presents a representation theorem on $(\mathcal{R}^0_T)^*$, the dual space of the space $\mathcal{R}^0_T$ (defined  by (\ref{R0T})).
 This theorem  gives an important property of $Y_T$ (which is defined by (\ref{ob-space})). For this purpose, we need  three lemmas.

  \begin{Lemma}\label{Lemma-H3-eq}
      The following propositions are equivalent:

 \noindent   $(i)$  The condition (H1) holds.

 \noindent  $(ii)$ There is a $p_1\in [2,\infty)$ so that for each $T\in(0,\infty)$,  each $u\in L^{p_1}(0,T;U)$ and each  $t\in (0,T)$,
    there exists a control   $v\in L^{\infty}(0,T;U)$ satisfying that
      $$
      \hat{y}(T;0,\chi_{(t,T)}v)=\hat{y}(T;0,\chi_{(0,t)}u)\;\;\mbox{and}\;\;
      \|v\|_{L^{\infty}(0,T;U)}\leq C_1\|u\|_{L^{p_1}(0,T;U)}
      $$
   for some $C_1\triangleq C_1(T,t)>0$ (independent of $u$).

  \noindent $(iii)$  There is a $p_2\in (1,2]$ so that when $0<t<T<\infty$,
 $$
 \|g\|_{L^{p_2}(0,t;U)} \leq C_2 \|g\|_{L^{1}(t,T;U)}\;\;\mbox{for all}\;\;g\in Y_T
 $$
 for some $C_2\triangleq C_2(T,t)>0$ (independent of $g$).

 Furthermore, when one of the above three propositions is valid,  the  constants $p_0$ (given in (H1)), $p_1$ and $p_2$ (given in  (ii) and (iii), respectively)  can be chosen so that $p_0=p_1=p_2/(p_2-1)$.

  \end{Lemma}

  \begin{proof}
  Our proof is organized by  several steps as follows:

\noindent   \textit{Step 1. To show that (i) $\Rightarrow$ (ii)}

  Suppose that (H1) holds for some  $p_0\in [2,\infty)$. Let $T$ and $t$ satisfy that  $0<t<T<\infty$. Define two maps as follows:
  $$
  L_1:\, Y\triangleq L^{p_0}(0,T;U) \rightarrow X,\;\;\;\;L_1(u) = \hat{y}(T;0,\chi_{(0,t)}u),\;\; u\in Y;
  $$
     $$
 L_2:\, Z\triangleq L^{\infty}(0,T;U) \rightarrow X,\;\;\;\;  L_2(v) = \hat y(T;0,\chi_{(t,T)}v),\;\;v\in Z.
 $$
 By (\ref{Changshubianyi1.6}), we have  that
 $$
 \hat y(T; 0, \chi_{(0,t)}u)=\int_0^T S_{-1}(T-\tau)B\chi_{(0,t)}(\tau)u(\tau) \,\mathrm d\tau;
 $$
$$
 \hat y(T; 0, \chi_{(t,T)}v)=\int_0^T S_{-1}(T-\tau)B\chi_{(t,T)}(\tau)v(\tau) \,\mathrm d\tau.
$$
 These, together with  (\ref{admissible-control}), indicate that both $L_1$ and $L_2$ are bounded. Moreover, by (H1),  we find that
    \begin{equation}\label{wangyuan2.10}
    \mbox{Range}\, L_1\subset \mbox{Range}\, L_2.
    \end{equation}
      Let   $\pi:~Z\rightarrow\widehat Z\triangleq Z/\mbox{Ker}\,L_2$ be the  quotient map. Then  $\pi$ is surjective and it stands that
  \begin{equation}\label{wang2.10-0}
   \|\pi(v)\|_{\widehat Z}=\inf\big\{\|w\|_Z\; :\; \,w\in v+\mbox{Ker}\,L_2\big\}
   \;\;\mbox{for each}\;\; v\in Z.
  \end{equation}
           Define a map $ \widehat L_2:$ $\widehat Z\rightarrow X$ in the following manner:
  \begin{eqnarray}\label{L2-widehatL2-1207}
  \widehat L_2(\pi(v))=L_2(v),\;\;\;\; \pi(v)\in \widehat Z.
  \end{eqnarray}
 One can easily check that $\widehat L_2$ is well defined, linear and bounded.  By (\ref{wangyuan2.10}) and (\ref{L2-widehatL2-1207}),   we see that $\mbox{Range}\,L_1   \subset   \mbox{Range}\,\widehat L_2$.
    Thus, given $u\in Y$, there is a unique $\pi(v_u)\in \widehat Z$ so that
 \begin{equation}\label{wang2.10}
 L_1 (u)=\widehat L_2 (\pi(v_u)).
 \end{equation}
 We now define another map $\mathcal T:~Y \rightarrow \widehat Z$ by
  \begin{eqnarray}\label{T-def}
     \mathcal T(u) = \pi(v_u)
     \;\;\mbox{for each}\;\; u\in Y.
  \end{eqnarray}
 One can easily check that  $\mathcal T$ is well defined and linear. We next use the closed graph theorem to show that $\mathcal T$ is bounded. For this purpose, we let
  $\{u_n\}\subset Y$ satisfy  that
   $$
   u_n\rightarrow u_0\;\;\mbox{ in}\;\; Y\;\;\mbox{and}\;\; \mathcal T(u_n)\rightarrow h_0\;\;\mbox{ in}\;\;\widehat Z,   \;\;\mbox{as}\;\; n\rightarrow\infty.
    $$
    Because $L_1$ and $\widehat L_2$ are bounded, we find from (\ref{T-def})
     and (\ref{wang2.10})  that
  \begin{eqnarray*}
   \widehat L_2 h_0=\lim_{n\rightarrow \infty}\widehat L_2(\mathcal T(u_n))
   =\lim_{n\rightarrow \infty}\widehat L_2(\pi(v_{u_n}))
   =\lim_{n\rightarrow \infty} L_1(u_n)=L_1 u_0.
  \end{eqnarray*}
  This, together with  (\ref{T-def}),  indicates that
  $h_0=\mathcal T(u_0)$. Then by  the closed graph theorem, we see that $\mathcal T$ is bounded. Thus, by (\ref{T-def}), there exists a $C\triangleq C(T,t)>0$ so that
   \begin{eqnarray}\label{i-ii-zy-y}
    \|\pi(v_u)\|_{\widehat Z}=\|\mathcal T(u)\|_{\widehat Z}\leq C\|u\|_Y
    \;\;\mbox{for each}\;\;    u\in Y.
   \end{eqnarray}
  Meanwhile, it follows from  (\ref{wang2.10-0}) that
     for each $v\in Z$, there is a $v^\prime\in v+\mbox{Ker\,}L_2$ so that
     $\|v^\prime\|_Z\leq 2 \|\pi(v)\|_{\widehat Z}$.
     Thus, by   (\ref{wang2.10}), (\ref{L2-widehatL2-1207}) and
   (\ref{i-ii-zy-y}), we find that for each $u\in Y$, there is a $v_u^\prime\in Z$
   so that
   $L_1(u)=L_2(v_u^\prime)$ and $\|v_u^\prime\|_Z\leq 2 C \|u\|_{Y}$.
  Hence, by the definitions of $L_1$ and $L_2$,  we obtain (ii), with $C_1=2C$ and $p_1=p_0$.

   \vskip 5pt
 \noindent \textit{Step 2. To show that $(ii)$ $\Rightarrow$ $(iii)$}

  Suppose that (ii) holds for some $p_1\in[2,\infty)$.  Arbitrarily fix $T$ and $t$,
  with  $0<t<T<\infty$.  Then for each $u\in L^{p_1}(0,T;U)$,
  there is a control $v_u\in L^{\infty}(0,T;U)$ so that
  \begin{eqnarray*}
   \hat{y}(T;0,\chi_{(0,t)}u)=\hat{y}(T;0,\chi_{(t,T)}v_u)
   \;\;\mbox{and}\;\;
   \|v_u\|_{L^{\infty}(0,T;U)}\leq C_1 \|u\|_{L^{p_1}(0,T;U)},
  \end{eqnarray*}
  where   $C_1\triangleq C_1(T,t)$ is given by (ii).
  These, along with  (\ref{NNNWWW2.1}), yield that  for  each $z\in D(A^*)$,
  \begin{eqnarray*}
   & & \int_0^t \langle B^*S^*(T-\eta)z,u(\eta) \rangle_U   \,\mathrm d\eta
   =\int_0^T  \big\langle  B^*S^*(T-\eta)z,\chi_{(0,t)}(\eta)u(\eta)  \big\rangle_U   \,\mathrm d\eta
   \nonumber\\
   &=& \big\langle  z,\hat{y}(T;0,\chi_{(0,t)}u)  \big\rangle_X
   = \big\langle  z,\hat{y}(T;0,\chi_{(t,T)}v_u)  \big\rangle_X
   \nonumber\\
   &=& \int_0^T \big\langle  B^*S^*(T-\eta)z,\chi_{(t,T)}(\eta)v_u(\eta)  \big\rangle_U \,\mathrm d\eta
   =\int_t^T  \big\langle  B^*S^*(T-\eta)z,v_u(\eta)  \big\rangle_U   \,\mathrm d\eta
    \nonumber\\
   &\leq& \|B^*S^*(T-\cdot)z\|_{L^{1}(t,T;U)} \|v_u(\cdot)\|_{L^{\infty}(t,T;U)}
   \leq  C_1 \|B^*S^*(T-\cdot)z\|_{L^{1}(t,T;U)}  \|u(\cdot)\|_{L^{p_1}(0,t;U)}.
  \end{eqnarray*}
  Let  $p_1^\prime$ be the  conjugate index of $p_1$, i.e.,  $1/p_1+1/p_1^\prime=1$.
    Then we find that
  \begin{eqnarray*}
   \|B^*S^*(T-\cdot)z\|_{L^{p_1^\prime}(0,t;U)} \leq C_1 \|B^*S^*(T-\cdot)z\|_{L^{1}(t,T;U)}
   \;\;\mbox{for all}\;\;  z\in D(A^*).
  \end{eqnarray*}
   The above, as well as (\ref{ob-space}), leads to (iii), with $p_2=p_1^\prime$ and $C_2=C_1$.

   \vskip 5pt
  \noindent \textit{Step 3. (iii) $\Rightarrow$ $(i)$}

    Suppose that (iii) holds for some $p_2\in(1,2]$. Let $p_2^\prime$ be the  conjugate index of $p_2$, i.e.,
     $ 1/p_2 + 1/p_2^\prime  =1$.
      Arbitrarily fix  $T$ and $t$, with  $0<t<T<\infty$.
   Define the following subspace of $L^1(t,T;U)$:
  \begin{eqnarray*}
   \mathcal O\triangleq \big\{B^*S^*(T-\cdot)z|_{(t,T)}\in L^1(t,T;U)~:\;\;z\in D(A^*)\big\}.
  \end{eqnarray*}
  Let  $u\in L^{p_2^\prime}(0,T;U)$.  We define  a linear map $L_3:\,\mathcal O\rightarrow\mathbb R$ by
  \begin{eqnarray}\label{H3-eq-1}
   L_3\big(B^*S^*(T-\cdot)z|_{(t,T)}\big)=\int_0^t \langle B^*S^*(T-s)z,u(s) \rangle_U   \,\mathrm ds,
   \;\;z\in D(A^*).
  \end{eqnarray}
  Since
  $ B^*S^*(T-\cdot)z |_{(0,T)} \in Y_T$ for all $z\in D(A^*)$,   it follows from (iii) that $L_3$ is well defined. Then by  (\ref{H3-eq-1}) and (iii), we find  that for each $z\in D(A^*)$,
  \begin{eqnarray*}
   \big|L_3\big(B^*S^*(T-\cdot)z|_{(t,T)}\big)\big| &\leq&  \|B^*S^*(T-\cdot)z\|_{L^{p_2}(0,t;U)}  \|u(\cdot)\|_{L^{p_2^\prime}(0,t;U)}   \nonumber\\
   &\leq&   C_2 \|B^*S^*(T-\cdot)z\|_{L^{1}(t,T;U)}  \|u(\cdot)\|_{L^{p_2^\prime}(0,t;U)},
  \end{eqnarray*}
  where $C_2\triangleq C_2(T,t)$ is given by (iii).
  This implies that $L_3$ is bounded from $\mathcal O$ to $\mathbb R$. Thus, by the Hahn-Banach theorem, $L_3$ can be extended from $L^1(t,T;U)$ to $\mathbb R$ and there exists   $g\in \big(L^1(t,T;U)\big)^*$ so that
$$
L_3(\psi)=g(\psi)\;\;\mbox{ for all}\;\;\psi\in \mathcal O; \;\;
\mbox{and}\;\;
 \|g\|_{\mathcal L(L^1(t,T;U);\mathbb R)}\leq C_2\|u\|_{L^{p_2^\prime}(0,t;U)}.
$$
Then by the Riesz representation theorem and (\ref{H3-eq-1}), there is  $v_u\in L^\infty(t,T;U)$ so that
 \begin{eqnarray}\label{H3-eq-2}
  \int_t^T \langle v_u(s),\psi(s)\rangle_U   \,\mathrm ds
  = g(\psi)=\int_0^t \langle \psi(s),u(s) \rangle_U   \,\mathrm ds
  \;\;\mbox{for all}\;\;  \psi\in \mathcal O
 \end{eqnarray}
 and so that
 \begin{eqnarray}\label{H3-eq-3}
  \|v_u\|_{L^\infty(t,T;U)}=\|g\|_{\mathcal L(L^1(t,T;U);\mathbb R)}\leq C_2\|u\|_{L^{p_2^\prime}(0,t;U)}.
 \end{eqnarray}
 Write $\widetilde v_u$ for the zero extension of $v_u$ over $(0,T)$. Then we see from (\ref{NNNWWW2.1}) and (\ref{H3-eq-2})  that for all $z\in D(A^*)$,
  \begin{eqnarray*}
   & & \langle z,\hat{y}(T;0,\chi_{(0,t)}u)\rangle_X
   = \int_0^T \big\langle B^*S^*(T-s)z,\chi_{(0,t)}(s)u(s) \big\rangle_U   \,\mathrm ds
    \nonumber\\
   &=&  \int_0^t \langle B^*S^*(T-s)z,u(s) \rangle_U   \,\mathrm ds
    =\int_t^T \langle v_u(s),B^*S^*(T-s)z\rangle_U   \,\mathrm ds
    \nonumber\\
   &=& \int_0^T \langle \widetilde v_u(s),B^*S^*(T-s)z\rangle_U   \,\mathrm ds
    = \langle z,\hat{y}(T;0,\widetilde v_u)\rangle_X.
  \end{eqnarray*}
  Since $D(A^*)$ is dense in $X$, the above leads to (H1), with $p_0=p_2^\prime$.

  \vskip 3pt
  \noindent \textit{Step 4. About the constants  $p_0$, $p_1$ and $p_2$}

  From  the proofs in Step 1-Step 3, we find that $p_0$, $p_1$ and $p_2$ can be chosen so that $p_0=p_1=p_2/(p_2-1)$, provided that one of the propositions (i)-(iii) holds.

  \vskip 3pt
  In summary, we finish the  proof of this lemma.

  \end{proof}

  \begin{Lemma}\label{Lemma-left-continuity-YT}
  Let $T\in(0,\infty)$. The following conclusions are true:

  \noindent (i) If $f\in Y_T$, then $f|_{(0,S)}\in Y_S$ for each $S\in (0,T)$.

 \noindent  (ii) Suppose that (H1) holds. If  $f\in L^1(0,T;U)$ and $f|_{(0,S)}\in Y_S$ for each $S\in (0,T)$, then $f\in Y_T$.

  \end{Lemma}

  \begin{proof}
  (i) Let $f\in Y_T$. Then by (\ref{ob-space}), there exists a subsequence $\{w_n\}\subset D(A^*)$ so that
  \begin{eqnarray}\label{YT-left-cont-1}
   B^*S^*(T-\cdot)w_n \rightarrow f(\cdot)
   \;\;\mbox{in}\;\;  L^1(0,T;U).
  \end{eqnarray}
 Arbitrarily fix an $S\in (0,T)$. Since
 $ S^*(T-S)w_n\in D(A^*)$ for  all $ n$,
    by making use of (\ref{ob-space}) again, we find that
   \begin{eqnarray*}
    B^*S^*(T-\cdot)w_n|_{(0,S)} = B^*S^*(S-\cdot) \big(S^*(T-S)w_n\big)|_{(0,S)} \in Y_S.
  \end{eqnarray*}
  Since $Y_S$ is closed in $L^1(0,S;U)$, the above, as well as
         (\ref{YT-left-cont-1}), yields that $f|_{(0,S)} \in Y_S$.

\vskip 5pt
  (ii) Suppose that (H1) holds. We  organize the  proof by the following  steps:

  \vskip 5pt
  \noindent  \textit{Step 1. To show  that for each $s\in(0,\infty)$ and $g^s\in Y_s$, there is a  unique  function $\widetilde g^s$  over $(-1,s)$ so that
  \begin{eqnarray}\label{YT-left-cont-2}
   \widetilde g^s(\tau)=g^s(\tau)\;\;\mbox{for all}\;\; \tau\in (0,s), \;\;\mbox{ and }\;\; \widetilde g^s(\cdot-1)\in Y_{s+1}
  \end{eqnarray}}
 ~~~\,Let $0<s<\infty$ and  $g^s\in Y_s$.  We first show the existence of such $\widetilde g^s$.
   For this purpose,  we define the following subspace:
  $$
     X_s\triangleq \big\{ g_z(\cdot)\in L^1(0,s; U) ~:~ z\in D(A^*)\big\},
      $$
      where
      $ g_z(\cdot)\triangleq B^*S^*(s-\cdot)z$ over $(0,s)$.
  Then define a map   $\mathcal F_s:\, X_s\rightarrow Y_{s+1}$ in the following manner: For each $z\in D(A^*)$,
  \begin{eqnarray}\label{def-F}
   \left(\mathcal F_s g_z\right)(\tau)\triangleq B^*S^*(s+1-\tau)z,\;\;\tau\in (0,s+1).
  \end{eqnarray}
   From  (\ref{def-F}), we find  that for each $z\in D(A^*)$,
 \begin{eqnarray}\label{wangyuan2.21}
  (\mathcal F_s g_z)(\tau+1) =g_z(\tau),\; \tau\in (0,s).
 \end{eqnarray}
  Meanwhile, by (H1) and Lemma \ref{Lemma-H3-eq},  we have the assertion  (iii) of Lemma \ref{Lemma-H3-eq}, which, together with  (\ref{def-F}), yields that  when $z\in D(A^*)$,
  \begin{eqnarray*}
   & & \big\| \mathcal F_s\big(B^*S^*(s-\cdot)z|_{(0,s)}\big)\big\|_{L^1(0,s+1;U)}  \nonumber\\
   &=& \int_1^{s+1} \|B^*S^*(s+1-\tau)z\|_U \,\mathrm d\tau
   + \int_0^1 \|B^*S^*(s+1-\tau)z\|_U \,\mathrm d\tau \nonumber\\
   &\leq& (1+C_2)\int_1^{s+1} \|B^*S^*(s+1-\tau)z\|_U \,\mathrm d\tau=(1+C_2)\|B^*S^*(s-\cdot)z\|_{L^1(0,s;U)}
  \end{eqnarray*}
 for some $C_2>0$ independent of $z$. (Here we used the time-invariance of  the controlled system). Hence, $\mathcal F_s$ is linear and bounded from $X_s$ to $Y_{s+1}$. Since $X_s$ is dense in $Y_s$ (see (\ref{ob-space})), $\mathcal F_s$ can be uniquely extended to be  a linear and bounded operator $\widetilde {\mathcal F_s} $ from $Y_s$ to $Y_{s+1}$.
This, along with (\ref{wangyuan2.21}),
  yields that
   \begin{equation}\label{wyuan2.20}
   (\widetilde{\mathcal F_s} g^s)(\tau+1) =g^s(\tau),\; \tau\in (0,s).
   \end{equation}
   We now define
 \begin{equation}\label{wyuan2.21}
 \widetilde g^s(\tau) \triangleq (\widetilde{\mathcal F_s}g^s)(\tau+1),\;\; \tau\in (-1,s).
 \end{equation}
 It follows from (\ref{wyuan2.21}) and (\ref{wyuan2.20}) that
  $\widetilde g^s$ satisfies (\ref{YT-left-cont-2}).

  We next show the  uniqueness of   such $\widetilde g^s$. Let $\widehat g^s$ be another extension of $g^s$ (over $(-1,s)$) satisfying (\ref{YT-left-cont-2}).   Then we see from  (\ref{YT-left-cont-2}) that
 \begin{eqnarray}\label{YT-left-cont-unique-1}
  (\widetilde g^s-\widehat g^s)(\tau)=0
  \;\;\mbox{for all}\;\;   \tau\in (0,s)
 \end{eqnarray}
 and
 \begin{eqnarray}\label{YT-left-cont-unique-2}
  (\widetilde g^s-\widehat g^s)(\cdot-1)\in Y_{s+1}.
 \end{eqnarray}
 From (\ref{YT-left-cont-unique-1}), we see that
 \begin{eqnarray}\label{YT-left-cont-unique-3}
  (\widetilde g^s-\widehat g^s)(\tau-1)=0
  \;\;\mbox{for all}\;\; \tau\in (1,s+1).
 \end{eqnarray}
 By (H1) and Lemma \ref{Lemma-H3-eq},  we have   (iii) of Lemma \ref{Lemma-H3-eq}. This, along with (\ref{YT-left-cont-unique-2}), yields that
 \begin{eqnarray*}
  \|(\widetilde g^s-\widehat g^s)(\cdot-1)\|_{L^{p_2}(0,1;U)} \leq C_2 \|(\widetilde g^s-\widehat g^s)(\cdot-1)\|_{L^{1}(1,s+1;U)},
 \end{eqnarray*}
 where $p_2$ and $C_2$ are given by (iii) of Lemma \ref{Lemma-H3-eq}. This, together with (\ref{YT-left-cont-unique-3}), implies that
 $(\widetilde g^s-\widehat g^s)(\cdot-1)=0$ over $(0,s+1)$.
  Hence, we have that
  $\widetilde g^s(\cdot)=\widehat g^s(\cdot)$ over $(-1,s)$.
   This shows the  uniqueness of such $\widetilde g^s(\cdot)$ that satisfies (\ref{YT-left-cont-2}).
 We call the above $\widetilde g^s(\cdot)$ the $Y$-extension of $g^s(\cdot)$.

\vskip 5pt
\noindent  \textit{Step 2. To show that $f\in Y_T$, when  $f\in L^1(0,T;U)$ and $f|_{(0,S)}\in Y_S$ for each $S\in (0,T)$}

\vskip 5pt

    Let $f\in L^1(0,T;U)$ satisfy that
    $   f|_{(0,S)}\in Y_S $ for each $S\in (0,T)$. Given
   $S\in (0,T)$, we write $f_S$ for the $Y$-extension of $f|_{(0,S)}$ over $(-1,S)$ (see the conclusion of Step 1).
   We claim that
   $$
   f_{S_1}=f_{S_2}\;\;\mbox{over}\;\; (-1,0),\;\;\mbox{when}\;\;0<S_1<S_2<T.
   $$
      Here is the argument: on one hand, we let
     \begin{equation}\label{yuan2.23}
     \bar f(\tau)\triangleq f_{S_2}(\tau-1),\; \tau\in (0, S_2+1).
     \end{equation}
     By (\ref{yuan2.23}) and the definition of $f_{S_2}$ (see (\ref{YT-left-cont-2})), we find that
     $  \bar f\in Y_{S_2+1}$.
    This, as well as  (i) in this lemma, yields that
\begin{equation}\label{yuan2.25}
\bar f|_{(0,S_1+1)}\in Y_{S_1+1}.
\end{equation}
By making use of (\ref{yuan2.23}) again, we see that
$\bar f|_{(0, S_1+1)}(\tau)=f_{S_2}|_{(-1,S_1)}(\tau-1)\;\;\mbox{for each}\;\; \tau\in (0,S_1+1)$.
This, along with (\ref{yuan2.25}), indicates that
\begin{equation}\label{yuan2.26}
f_{S_2}|_{(-1,S_1)}(\cdot-1)\in Y_{S_1+1}.
\end{equation}
Meanwhile, since $f_{S_2}=f$ over $(0,S_2)$, we have that
$f_{S_2}|_{(-1,S_1)}(\tau)=f|_{(0,S_1)}(\tau)$ for all $\tau\in (0,S_1)$.
This, along with (\ref{yuan2.26}), indicates that  $f_{S_2}|_{(-1,S_1)}(\cdot)$ is the $Y$-extension of $f|_{(0,S_1)}(\cdot)$ over $(-1,S_1)$.

On the other hand, $f_{S_1}$ is also the $Y$-extension of $f|_{(0,S_1)}(\cdot)$ over $(-1,S_1)$. By the uniqueness of the $Y$-extension, we see that
$f_{S_1}=f_{S_2}|_{(-1,S_1)}$ over $(-1,S_1)$,
  which leads to that
 $ f_{S_1}=f_{S_2}$ over $(-1,0)$.
  This ends the proof of the above claim.

  Now we arbitrarily fix  an $S_0\in (0,T)$. Define a function   $\hat f: (-1, T)\rightarrow U$ by setting
     \begin{eqnarray}\label{g-ex-3}
  \hat f(\cdot)= f(\cdot)\;\;\mbox{over}\;\; (0,T); \;\;
  \hat f(\cdot)=f_{S_0}(\cdot)\;\; \mbox{ over }\;\; (-1,0].
  \end{eqnarray}
   Because of the above-mentioned claim, we find that
   \begin{eqnarray}\label{f-extension-independent}
     \hat f \;\;\mbox{is independent of the choice of}\;\; S_0.
   \end{eqnarray}
    It is clear that
 $   \hat f\in L^1(-1,T;U)$. Take a sequence $\{T_k\}\subset(0,T)$ so that $T_k\nearrow T$.
   Then  we see from the first equality in (\ref{g-ex-3}) that
  \begin{eqnarray}\label{g-ex-5}
   \hat f(\cdot+T_k-T)|_{(0,T)} \rightarrow \hat f(\cdot)\mid_{(0,T)}=f(\cdot) \;\;\mbox{ in }\;\; L^1(0,T;U),\;\; \mbox{ as }\;\; k\rightarrow\infty.
  \end{eqnarray}
  Meanwhile,  for each $k$, since $f_{T_k}(\cdot-1)\in Y_{T_k+1}$ (see (\ref{YT-left-cont-2})), by (\ref{ob-space}), there exists a sequence $\{w_{k,n}\}\subset D(A^*)$ so that
  \begin{eqnarray*}
  \int_0^{T_k+1} \|B^*S^*(T_k+1-t)w_{k,n}  -  f_{T_k}(t-1)\|_U   \,\mathrm dt \rightarrow 0,
  \;\;\mbox{as}\;\; n\rightarrow\infty.
  \end{eqnarray*}
  Since $f_{T_k}=\hat f$ over $(-1,T_k)$  for each $k$ (see (\ref{f-extension-independent}) and (\ref{g-ex-3})), the above yields that for all $k$, with $T_k+1\geq T$,
  \begin{eqnarray*}
   \int_{T_k+1-T}^{T_k+1} \|B^*S^*(T_k+1-t)w_{k,n}  -  \hat f(t-1)\|_U   \,\mathrm dt \rightarrow 0,
   \;\;\mbox{as}\;\; n\rightarrow\infty,
  \end{eqnarray*}
  which implies that for all $k$, with $T_k+1\geq T$,
  \begin{eqnarray*}
   \int_{0}^{T} \|B^*S^*(T-t)w_{k,n}  -  \hat f(t+T_k-T)\|_U  \,\mathrm dt \rightarrow 0,
   \;\;\mbox{as}\;\; n\rightarrow\infty.
  \end{eqnarray*}
  This, along with (\ref{ob-space}), indicates that
   \begin{eqnarray*}
   \hat f(\cdot+T_k-T)|_{(0,T)}\in Y_{T}
   \;\;\mbox{for all}\;\; k \;\;\mbox{with}\;\; T_k+1\geq T.
   \end{eqnarray*}
    Since $Y_T$ is closed in $L^1(0,T;U)$, the above, together with  (\ref{g-ex-5}), implies that $f\in Y_{T}$.

    \vskip 3pt

    In summary, we complete the proof of this lemma.

  \end{proof}

  \begin{Lemma}\label{Lemma-Linfty-kernal}
  Let $T\in(0,\infty)$. If $f\in L^1(0,T;U)$ satisfies that
  \begin{eqnarray}\label{kernal-YT}
  \int_0^T \langle f(t),u(t) \rangle_U  \,\mathrm dt=0
  \;\;\mbox{for all}\;\;u\in \big\{v\in L^\infty(0,T;U)\;:\; \hat y(T;0,v)=0 \big\},
  \end{eqnarray}
   then  $f\in Y_T$.

  \end{Lemma}

  \begin{proof}
  By contradiction, we suppose that for some $T\in (0,\infty)$, there were a function  $f\in L^1(0,T;U)$, with the property  (\ref{kernal-YT}), so that  $f\not\in Y_T$. Then, by the Hahn-Banach theorem, we could find a function $\widehat u \in L^\infty(0,T;U)$ so that
  \begin{eqnarray}\label{wang2.26}
   0=\int_0^T \langle g(t),\widehat u(t) \rangle_U \,\mathrm  dt
   < \int_0^T \langle f(t),\widehat u(t) \rangle_U \,\mathrm dt
   \;\;\mbox{for each}\;\;   g\in Y_T.
  \end{eqnarray}
  (Here, we noticed that   $Y_T$ is a closed subspace of $L^1(0,T;U)$.) From
  Theorem~\ref{Theorem-the-representation-theorem} and
  the first assertion in (\ref{wang2.26}), we find that $\hat y(T;0,\widehat u)=0$,
  which, along with (\ref{kernal-YT}) and the second assertion in  (\ref{wang2.26}), leads to a contradiction.  This ends  the proof.

  \end{proof}

The following result is a representation theorem on $\mathcal (\mathcal R_T^0)^*$, which plays an important role in our study.
  \begin{Theorem}\label{Theorem-the-second-representation-theorem}
   Suppose that (H1) holds. Then for each $T\in(0,\infty)$, there is a linear isomorphism $\Psi_T$ from $Y_T$ to $(\mathcal R_T^0)^*$ so that for all $g\in Y_T$ and $y_T\in\mathcal R_T^0$,
  \begin{eqnarray}\label{ob-attain-R0}
   \langle g,y_T \rangle_{Y_T,\mathcal R_T^0}
   \triangleq\langle \Psi_T(g),y_T \rangle_{(\mathcal R_T^0)^*,\mathcal R_T^0}
   =\int_0^T \langle g(t),v(t) \rangle_U \, \mathrm dt,
  \end{eqnarray}
  where $v$ is any admissible control to $(NP)^{y_T}$.
 \end{Theorem}

 \begin{proof}
 Let $0<T<\infty$. Recall that $\mathcal{R}^0_T$, with the norm $\|\cdot\|_{\mathcal R_T}$, is a subspace of $\mathcal{R}_T$ (see (\ref{attainable-space}) and (\ref{R0T})). According to Theorem \ref{Theorem-the-representation-theorem}, each  $g\in Y_T$ defines  a linear bounded functional $\widehat{\mathcal F}_g$ over $\mathcal R_T^0$ (i.e.,
 $\widehat{\mathcal F}_g\in (\mathcal R_T^0)^*$), via
 \begin{equation}\label{wang2.28}
 \widehat{\mathcal F}_g(y_T)\triangleq  \langle g, y_T \rangle_{Y_T, \mathcal{R}_T},\;\; y_T\in \mathcal{R}^0_T,
 \end{equation}
 where  $\langle \cdot, \cdot \rangle_{Y_T, \mathcal{R}_T}$ is given by (\ref{ob-attain-0}). Then we define a map $\Psi_T$ from  $Y_T$ to $(\mathcal{R}^0_T)^*$
 by
  \begin{eqnarray}\label{wang2.29}
     \Psi_T(g)\triangleq \widehat{\mathcal F}_g,\;\; g\in Y_T.
 \end{eqnarray}
 One can easily check that  $\Psi_T$ is linear.
 The rest of the proof is organized by three steps.

 \vskip 5pt

 \noindent {\it Step 1. To show that $\|g\|_{Y_T}=\|\Psi_T(g)\|_{\mathcal{R}_T} $ for all $g\in \mathcal{R}^0_T$ }

 \vskip 5 pt

 Let $g\in Y_T$ be given.  On one hand, from (\ref{wang2.28}), we see that
  \begin{eqnarray}\label{second-dual-2}
   \|\widehat{\mathcal F}_g\|_{(\mathcal R_T^0)^*}    = \sup_{y_T\in B_{\mathcal{R}^0_T}(0,1)} \langle g,y_T \rangle_{Y_T,\mathcal R_T}     \leq  \|g\|_{Y_T},
  \end{eqnarray}
  where $B_{\mathcal{R}^0_T}(0,1)$ is the closed unit ball in $\mathcal{R}^0_T$.
 On the other hand, we arbitrarily fix $S\in (0,T)$. Then according to  the Hahn-Banach theorem,
 there is a control $\hat u_S\in L^\infty(0,S;U)$
 so that
 \begin{equation}\label{wang2.31}
  \|g\|_{L^1(0,S;U)}= \langle g, \hat u_S \rangle_{L^1(0,S;U),L^\infty(0,S;U)}
  \;\;\mbox{and}\;\;\|\hat u_S\|_{L^\infty(0,S;U)}=1.
 \end{equation}
 Write $\widetilde{u}_S$ for the zero extension of $\hat u_S$ over $(0,T)$. Then it follows from (\ref{R0T}) that $\hat y(T;0,\widetilde{u}_S) \in \mathcal R_T^0$.
Now, by (\ref{wang2.31}),
 (\ref{ob-attain-0}), (\ref{wang2.28}) and (\ref{attianable-space-norm}),
 one can directly  check that
  \begin{eqnarray*}
  \|g\|_{L^1(0,S;U)}
  &=& \langle g, \widetilde u_S \rangle_{L^1(0,T;U),L^\infty(0,T;U)}
  = \langle g, \hat{y}(T;0,\widetilde u_S)\rangle_{Y_T,\mathcal R_T}
  \nonumber\\
  &=&  \widehat{\mathcal F}_g \big(\hat{y}(T;0,\widetilde u_S)\big)
  \leq \|\widehat{\mathcal F}_g\|_{(\mathcal R_T^0)^*} \|\hat{y}(T;0,\widetilde u_S)\|_{\mathcal R_T} \nonumber\\
  &\leq& \|\widehat{\mathcal F}_g\|_{(\mathcal R_T^0)^*} \|\widetilde u_S\|_{L^\infty(0,S;U)}=\|\widehat{\mathcal F}_g\|_{(\mathcal R_T^0)^*},
 \end{eqnarray*}
which yields that $ \|g\|_{Y_T}=\|g\|_{L^1(0,T;U)} \leq \|\widehat{\mathcal F}_g\|_{(\mathcal R_T^0)^*}$ (since $S$ was arbitrarily taken from $(0,T)$).
 This, along with   (\ref{second-dual-2}), leads to that $\|g\|_{Y_T}=\|\Psi_T(g)\|_{\mathcal{R}_T} $.

 \vskip 5pt

 \noindent {\it Step 2. To show that $\Psi_T$ is surjective}

  \vskip 5pt
Let $\hat f\in (\mathcal{R}^0_T)^*$. We aim to find a $\hat g\in Y_T$ so that
\begin{equation}\label{wang2.32}
\hat f= \Psi_T(\hat g)
\;\;\mbox{in}\;\; (\mathcal{R}^0_T)^*.
\end{equation}
In what follows, for each $u\in L^\infty(0,S;U)$, with $S\in (0,T)$, we denote by
$\widetilde{u}$ the zero extension of $u$ over $(0,T)$. Then it follows from (\ref{R0T}) that $\hat y(T;0,\widetilde{u}) \in \mathcal R_T^0$.
 We define, for each $S\in (0,T)$, a map $G_{\hat f, S}$ from
$L^\infty(0,S;U)$ to $\mathbb{R}$ by setting
 \begin{eqnarray}\label{second-dual-3}
  \mathcal G_{\hat f, S}(u)\triangleq \langle \hat f, \hat{y}(T;0,\widetilde u)\rangle_{(\mathcal{R}^0_T)^*, \mathcal{R}^0_T}
  \;\;\mbox{for each}\;\; u\in L^\infty(0,S;U).
 \end{eqnarray}
 From (\ref{second-dual-3}), we see that for each $S\in (0,T)$,
 \begin{equation}\label{second-dual-4}
 |\mathcal G_{\hat f, S}(u)|\leq \|\hat f\|_{(\mathcal{R}^0_T)^*}\|\hat y(T;0,\widetilde{u})\|_{\mathcal{R}_T}\;\;\mbox{for each}\;\; u\in L^\infty(0,S;U).
 \end{equation}
 Arbitrarily fix an $S\in (0,T)$. By (H1) and Lemma \ref{Lemma-H3-eq},  we have  (ii) of Lemma \ref{Lemma-H3-eq}.  Thus,  there exists a $C_1(T,S)>0$ so that for each $u\in
   L^\infty(0,S;U)$, there is a control $\hat v_u \in L^\infty(0,T;U)$ verifying that
 \begin{equation}\label{wang2.35}
 \hat y(T;0, \widetilde{u})=\hat y(T; 0,\chi_{(S,T)}\hat v_u)
 \;\;\mbox{and}\;\; \|\hat v_u\|_{L^\infty(0,T;U)}\leq C_1(T,S)\|\widetilde{u}\|_{L^{p_1}(0,T;U)}
 \end{equation}
 for some   $p_1\in[2,\infty)$.
 From the first assertion in (\ref{wang2.35}) and (\ref{attianable-space-norm}), we find that
 $$
  \|\hat{y}(T;0,\widetilde u)\|_{\mathcal R_T} \leq \|\hat v_u\|_{L^{\infty}(0,T;U)},
  $$
 which, together with  the second assertion in (\ref{wang2.35}), indicates that
 $$
 \|\hat{y}(T;0,\widetilde u)\|_{\mathcal R_T}\leq C_1(T,S)\|\widetilde{u}\|_{L^{p_1}(0,S;U)}\;\;\mbox{for all}\;\; u\in L^\infty(0,S;U).
 $$
 This, as well as (\ref{second-dual-4}), yields that for each $S\in (0,T)$,
 \begin{eqnarray}\label{wang3.36}
   |\mathcal G_{\hat f,S}(u) | \leq C_1(T,S)\|\hat f\|_{(\mathcal R_T^0)^*} \|u\|_{L^{p_1}(0,S;U)}\;\;\mbox{for all}\;\; u\in L^\infty(0,S;U).
 \end{eqnarray}
 By (\ref{wang3.36}) and  the Hahn-Banach theorem, we can uniquely extend $\mathcal G_{\hat f,S}$ to be an element in $\big(L^{p_1}(0,S;U)\big)^*$, denoted in the same manner, so that
  \begin{eqnarray}\label{wang3.37}
   |\mathcal G_{\hat f,S}(u) | \leq C_1(T,S)\|\hat f\|_{(\mathcal R_T^0)^*} \|u\|_{L^{p_1}(0,S;U)}\;\;\mbox{for all}\;\; u\in L^{p_1}(0,S;U).
 \end{eqnarray}
  From (\ref{wang3.37}), using  the Riesz representation theorem, we find that for each $S\in (0,T)$, there exists a $g_S\in L^{p_1^\prime}(0,S;U)$, with $1/p_1+1/p_1^\prime=1$,  so that
 \begin{eqnarray}\label{second-dual-6}
    \mathcal G_{\hat f,S}(u)=\int_0^S \langle g_S(t),u(t) \rangle_U \, \mathrm dt
    \;\;    \mbox{for all}\;\;   u\in L^{p_1}(0,S;U).
 \end{eqnarray}

 Next, arbitrarily fix an $S\in (0,T)$. Then take $v\in L^\infty(0,S;U)$ so that
 $\hat y(T; 0,\widetilde{v})=0$.
 (Here, $\widetilde{v}$ is the zero extension of $v$ over
 $(0,T)$.)
  By (\ref{second-dual-6}) and
 (\ref{second-dual-3}), we see that
  \begin{eqnarray*}
 \int_0^S \langle g_S(t),v(t) \rangle_U \, \mathrm dt=\mathcal G_{\hat f,S}(v)=0.
 \end{eqnarray*}
 This, along with Lemma \ref{Lemma-Linfty-kernal}, yields that
 \begin{eqnarray}\label{second-dual-7}
  g_S \in Y_S \;\;\mbox{ for each }\;\;S\in (0,T).
 \end{eqnarray}
 Meanwhile, from (\ref{second-dual-6}), (\ref{second-dual-4}) and (\ref{attianable-space-norm}), one can easily check  that for each $u\in L^\infty(0,S;U)$,
  \begin{eqnarray*}
  \int_0^S \langle g_S(t),u(t) \rangle_U  \,\mathrm dt
  \leq   \|\hat f\|_{(\mathcal R_T^0)^*}\|\hat y(T;0,\widetilde u)\|_{\mathcal{R}_T}
  \leq \|\hat f\|_{(\mathcal R_T^0)^*}\|u\|_{L^\infty(0,S;U)}.
 \end{eqnarray*}
 This, together with (\ref{second-dual-7}), implies that
 \begin{eqnarray}\label{second-dual-6-1}
    \|g_S\|_{Y_S}=\|g_S\|_{L^1(0,S;U)} \leq \|\hat f\|_{(\mathcal R_T^0)^*}\;\;\mbox{for all}\;\; S\in (0,T).
 \end{eqnarray}
 We now define a function $\hat g: (0,T)\rightarrow U$ in the following manner: For each $S\in (0,T)$,
 \begin{eqnarray}\label{second-dual-8}
  \hat g(t)\triangleq g_S(t)\;\;\mbox{for all}\;\;t\in (0,S).
 \end{eqnarray}
 The map $\hat g$ is well defined on $(0,T)$. In fact, when
   $0<S_1<S_2<T$, it follows from  (\ref{second-dual-6}) and (\ref{second-dual-3}) that
   for each $u\in L^\infty(0,S_1;U)$,
   \begin{eqnarray*}
    \int_0^{S_1}  \langle g_{S_1}(t),u(t) \rangle_U \,\mathrm dt
    &=& \mathcal G_{\hat f,S_1} (u)
    =  \langle \hat f, \hat y(T;0,\widetilde u) \rangle_{(\mathcal R_T^0)^*,\mathcal R_T^0}
    =  \mathcal G_{\hat f,S_2} (\widetilde u|_{(0,S_2)})
    \nonumber\\
    &=& \int_0^{S_2}  \langle g_{S_2}(t), \widetilde u(t) \rangle_U \,\mathrm dt
    = \int_0^{S_1}  \langle g_{S_2}(t),u(t) \rangle_U \,\mathrm dt,
   \end{eqnarray*}
   which implies that
 $ g_{S_1}(\cdot)=g_{S_2}(\cdot)$ over $(0,S_1)$.
  So one can check from (\ref{second-dual-8}) that $\hat g$ is well defined. By  (\ref{second-dual-8}) and (\ref{second-dual-6-1}), we see that
 \begin{equation}\label{wang2.42}
 \|\hat g\|_{L^1(0,T;U)} \leq \|\hat f\|_{(\mathcal R_T^0)^*}.
 \end{equation}
   Since (H1) was assumed, from  (\ref{wang2.42}), (\ref{second-dual-8}), (\ref{second-dual-7}) and (ii) of  Lemma \ref{Lemma-left-continuity-YT}, we find that
 \begin{eqnarray}\label{second-dual-8-1}
    \hat g\in Y_T\;\;\mbox{and}\;\; \|\hat g\|_{Y_T}\leq \|\hat f\|_{(\mathcal R_T^0)^*}.
 \end{eqnarray}
  By (\ref{second-dual-3}),  (\ref{second-dual-6}) and (\ref{second-dual-8}), we deduce that for each $S\in (0,T)$,
  \begin{eqnarray}\label{second-dual-9}
 \langle \hat f, \hat{y}(T;0,\widetilde u)\rangle_{(\mathcal{R}^0_T)^*, \mathcal{R}^0_T}=\int_0^T \langle \hat g(t),\widetilde u(t) \rangle_U \mathrm dt\;\;\mbox{for all}\;\; u\in L^\infty(0,S;U).
 \end{eqnarray}

 Now, for each $y_T\in \mathcal{R}^0_T$, it follows by (\ref{R0T}) that there is an $u_{y_T}\in L^\infty(0,T;U)$ so that $$
 y_T=\hat y(T;0,u_{y_T})
 \;\;\mbox{and}\;\;   \lim_{S\rightarrow T} \|u_{y_T}\|_{L^\infty(S,T;U)}=0.
 $$
 From these and (\ref{attianable-space-norm}), one can check that
 \begin{eqnarray*}
  \big\| \hat y(T;0,\chi_{(0,S)}u_{y_T})  - y_T  \big\|_{\mathcal R_T}
  =\| \hat y(T;0,\chi_{(S,T)}u_{y_T}) \|_{\mathcal R_T}
  \leq \|u_{y_T}\|_{L^\infty(S,T;U)}
  \rightarrow 0,
  \;\;\mbox{as}\;\; S\rightarrow T,
 \end{eqnarray*}
 which implies that
 \begin{equation}\label{wang2.45}
 \hat y(T;0,\chi_{(0,S)}u_{y_T})\rightarrow y_T\;\;\mbox{in}\;\;\mathcal R_T,
 \;\;\mbox{as}\;\; S\rightarrow T.
 \end{equation}
 Notice that  $\hat y(T; 0,\chi_{(0,S)}u_{y_T})\in \mathcal{R}^0_T$ and $\hat g\in Y_T$. Thus,  from (\ref{wang2.45}), (\ref{second-dual-9}) and  (\ref{ob-attain-0}), using the dominated convergence theorem, we find that for each $y_T\in \mathcal{R}^0_T$,
 \begin{eqnarray*}
  \langle \hat f, y_T\rangle_{(\mathcal{R}^0_T)^*, \mathcal{R}^0_T}
  &=& \lim_{S\rightarrow T}
  \big\langle \hat f, \hat{y}(T;0,\chi_{(0,S)}u_{y_T}) \big\rangle
  _{(\mathcal{R}^0_T)^*, \mathcal{R}^0_T}
  \nonumber\\
  &=&  \lim_{S\rightarrow T}  \int_0^T  \big\langle \hat g(t),\chi_{(0,S)}(t)u_{y_T}(t) \big\rangle_U \,\mathrm dt \nonumber\\
  &=& \int_0^T \langle \hat g(t),u_{y_T}(t) \rangle_U  \,\mathrm dt= \langle \hat g,y_T\rangle_{Y_T,\mathcal R_T}.
 \end{eqnarray*}
 This, along with (\ref{wang2.28}), yields that
   $$
   \langle \hat f, y_T\rangle_{(\mathcal{R}^0_T)^*, \mathcal{R}^0_T}
   =\widehat{\mathcal{F}}_{\hat g}(y_T)\;\;\mbox{for all}\;\; y_T\in \mathcal{R}^0_T,\;\;\mbox{i.e.,}\;\;\hat f= \widehat{\mathcal{F}}_{\hat g}
   \;\;\mbox{in}\;\;  (\mathcal R_T^0)^*,
   $$
    which, together with (\ref{wang2.29}), leads to  (\ref{wang2.32}). So $\Psi_T$ is surjective.

    \vskip 5pt

    \noindent {\it Step 3. To show the second equality in (\ref{ob-attain-R0})}

The second equality in (\ref{ob-attain-R0}) follows from (\ref{wang2.29}), (\ref{wang2.28}) and (\ref{ob-attain-0}) (in Theorem \ref{Theorem-the-representation-theorem}).

\vskip 5pt

In summary, we finish the proof of this theorem.
 \end{proof}

 \begin{Remark}
  We do not know whether  $\mathcal R_T^0$ is a  closed subspace of $\mathcal R_T$
  in general.

 \end{Remark}

 \vskip 5pt

 \begin{Corollary}\label{Corollary-weakly-star-compact}
 Suppose that (H1) holds. Then for each $T\in(0,\infty)$, $B_{Y_T}$ (the closed unit ball in  $ Y_T$)
 is compact in the topology $\sigma(Y_T,\mathcal R_T^0)$  (i.e., weak star compact). \end{Corollary}

 \begin{proof}

 By Theorem \ref{Theorem-the-second-representation-theorem}, we have that $Y_T=(\mathcal R_T^0)^*$. Then by the Banach-Alaoglu theorem, $B_{Y_T}$ is compact in the topology $\sigma(Y_T,\mathcal R_T^0)$. This ends the proof.\\
 \end{proof}

 \vskip 5pt
 \subsection{Further studies on attainable subspaces}

  The following Lemma presents the non-triviality of the  subspaces $Y_T$ and $\mathcal R_T^0$, with $T\in(0,\infty)$. (Consequently, $\mathcal R_T$ is also non trivial.)  Here, we will use the assumption
  that the control operator $B$ is non-trivial.

 \begin{Lemma}\label{Lemma-non-trivial}
 Let $0<T<\infty$. Then the sets $Y_T\setminus\{0\}$ and $\mathcal R_T^0\setminus\{0\}$   are nonempty.

 \end{Lemma}

 \begin{proof}
 Arbitrarily fix a $T\in(0,\infty)$. We first show that  $Y_T\setminus\{0\}\neq\emptyset$. Seeking for a  contradiction, we suppose that  $Y_T\setminus\{0\}=\emptyset$. Since $X_T\subset Y_T$ (see (\ref{ob-space})), we could derive from (\ref{assumptionspace3}) that for each $z\in D(A^*)$,
 $ B^*S^*(T-\cdot)z=0 $ over $(0,T)$.
  Since $\{S^*(t)|_{D(A^*)}\}_{t\in\mathbb R^+}$ is a $C_0$-semigroup on $D(A^*)$ and $B^*\in \mathcal L(D(A^*),U)$, the above yields that for each $t\in[0,T]$
  and each $z\in D(A^*)$, $ B^*S^*(T-t)z=0$.
Taking $t=T$ in above equality leads to that
$ B^*z=0 $ for all $z\in D(A^*)$, i.e., $ B^*=0$,
  which contradicts  the assumption that $B\neq 0$. Thus we have proved that $Y_T\setminus\{0\} \neq \emptyset$.

 We next verify that the set $\mathcal R_T^0\setminus\{0\}$ is nonempty.
  By contradiction, suppose that it was not true. Then we would have that
 \begin{eqnarray}\label{1230-first-one}
 \mathcal R_T^0\setminus\{0\}=\emptyset, \;\;\mbox{i.e.,}\;\;  \mathcal R_T^0 = \{0\}.
 \end{eqnarray}
 Arbitrarily fix an $\varepsilon\in(0,T)$. We find from (\ref{R0T}) that
 $\hat y(T;0,\widetilde v) \in\mathcal R_T^0$ for all $ v\in L^\infty(0,\varepsilon;U)$,
  where $\widetilde v$ denotes the zero extension of $v$ over $(0,T)$.
 This, together with  (\ref{1230-first-one}) and (\ref{NNNWWW2.1}), yields  that for all $z\in D(A^*)$ and $v\in L^\infty(0,\varepsilon;U)$,
$$
  \int_0^T \langle B^*S^*(T-t)z, \widetilde v(t) \rangle_U  \,\mathrm dt
  =\langle z, \hat y(T;0,\widetilde v) \rangle_X=0.
$$
 From the above, we find  that for each $z\in D(A^*)$,
 $B^*S^*(T-\cdot)z=0$  over $(0,\varepsilon)$.
    Since $\varepsilon$ was arbitrarily taken from $(0,T)$, the above indicates that $B^*S^*(T-\cdot)z=0$ over $(0,T)$,
    for each $z\in D(A^*)$.
 From this and  (\ref{assumptionspace3}), we find that $X_T=\{0\}$, which, along with  (\ref{ob-space}), indicates that $Y_T=\{0\}$. This  leads to  a contradiction, since we have proved that $Y_T\setminus\{0\}\neq \emptyset$. Therefore, $\mathcal R_T^0\setminus\{0\} \neq \emptyset$. Thus, we end the proof of this lemma.

 \end{proof}

 \vskip 5pt
  The next  result presents  an expression on the norm $\|\cdot\|_{\mathcal{R}_T}$.

 \begin{Proposition}\label{proposition2.3}
  Let $0<T<\infty$. Write
  \begin{equation*}
  \hat Z_T\triangleq \big\{z\in D(A^*)\; :\; B^*S^*(T-\cdot)z\neq 0\;\;\mbox{in}\;\;L^1(0,T;U) \big\}.
  \end{equation*}
  Then it stands that
  \begin{eqnarray}\label{norm-dual}
   \|y_T\|_{\mathcal R_T}=
   \sup_{z\in \hat Z_T} \frac{\langle y_T,z\rangle_X}{\|B^*S^*(T-\cdot)z\|_{L^1(0,T;U)}}
   \;\;\mbox{for all}\;\; y_T\in \mathcal R_T.
  \end{eqnarray}
   \end{Proposition}

 \begin{proof}
 Arbitrarily fix $T\in (0,\infty)$.
 First of all, we notice that $\hat Z_T\neq \emptyset$.
 Indeed, if it was not true, then by (\ref{assumptionspace3}),
  we would have that $X_T=\{0\}$, which, along with  (\ref{ob-space}), yields that $Y_T=\{0\}$. This contradicts  Lemma \ref{Lemma-non-trivial}. So we have proved that $\hat Z_T\neq \emptyset$.

 Recall (\ref{ob-attain-14}) for the linear isomorphism $\Phi_T$
 from $\mathcal R_T$ to $Y_T^*$. It is clear that
  \begin{eqnarray}\label{gengsheng2.9}
    \|y_T\|_{\mathcal R_T}=\|\Phi_T(y_T)\|_{Y_T^*}=\|\widetilde{\mathcal{F}}_{y_T}\|_{Y_T^*}
    =\sup_{f\in Y_T\setminus\{0\}}
    \frac{\langle \widetilde{\mathcal{F}}_{y_T},f\rangle_{Y_T^*,Y_T}}{\|f\|_{Y_T}}
    \;\;\mbox{for each}\;\; y_T\in \mathcal R_T.
   \end{eqnarray}
 We claim that for each $y_T\in \mathcal R_T$,
 \begin{eqnarray}\label{gengsheng2.10}
 \sup_{f\in Y_T\setminus\{0\}}
    \frac{\langle \widetilde{\mathcal{F}}_{y_T},f\rangle_{Y_T^*,Y_T}}{\|f\|_{Y_T}}
    =\sup_{z\in \hat Z_T} \frac{\langle y_T,z\rangle_X}{\|B^*S^*(T-\cdot)z\|_{L^1(0,T;U)}}.
     \end{eqnarray}
 To this end, we arbitrarily take $y_T\in \mathcal{R}_T$ and then fix $v\in \mathcal{U}_{ad}^{y_T}$. (Since $y_T\in \mathcal{R}_T$, it follows by (\ref{HuangHuangHuang2.9}) that
  $\mathcal{U}_{ad}^{y_T}\neq
 \emptyset$.)
 On one hand, given $f\in Y_T\setminus\{0\}$, it follows by (\ref{ob-space}) that
 there is a sequence $\{z_n\}$ in $D(A^*)$ so that
 \begin{equation}\label{gengsheng2.11}
 B^*S^*(T-\cdot)z_n\rightarrow f(\cdot)\;\;\mbox{in}\;\; L^1(0,T;U).
 \end{equation}
  Since $f\neq 0$, we see from (\ref{gengsheng2.11}) that when $n$ is large enough,
  \begin{equation}\label{gengsheng2.12}
  B^*S^*(T-\cdot)z_n\neq 0\;\;\mbox{in}\;\; L^1(0,T;U),\;\;\mbox{i.e.,}\;\;z_n\in \hat Z_T.
  \end{equation}
  From (\ref{gengsheng2.11}), the definition of $\widetilde{\mathcal{F}}_{y_T}$ (see (\ref{ob-attain-12}))
    and the first equality in (\ref{ob-attain-11}), we find that
  \begin{eqnarray*}
  \langle \widetilde{\mathcal{F}}_{y_T}, f    \rangle_{Y_T^*,Y_T}
   =\lim_{n\rightarrow\infty}  \int_0^T
  \langle v(t), B^*S^*(T-t)z_n\rangle_U \,\mathrm dt
  =\lim_{n\rightarrow\infty}
  \langle y_T, z_n\rangle_X.
  \end{eqnarray*}
 This, together with (\ref{gengsheng2.11}) and (\ref{gengsheng2.12}), yields that
 \begin{eqnarray*}
 \frac{\langle \widetilde{\mathcal{F}}_{y_T},f\rangle_{Y_T^*,Y_T}}{\|f\|_{Y_T}}
 =\lim_{n\rightarrow\infty}\frac{\langle y_T,z_n\rangle_X}{\|B^*S^*(T-\cdot)z_n\|_{L^1(0,T;U)}}
 \leq \sup_{z\in \hat Z_T} \frac{\langle y_T,z\rangle_X}{\|B^*S^*(T-\cdot)z\|_{L^1(0,T;U)}}.
  \end{eqnarray*}
 Since $f$ was arbitrarily taken from $Y_T\setminus\{0\}$, the above leads to
 \begin{eqnarray}\label{gengsheng2.13}
 \sup_{f\in Y_T\setminus\{0\}}
    \frac{\langle \widetilde{\mathcal{F}}_{y_T},f\rangle_{Y_T^*,Y_T}}{\|f\|_{Y_T}}
    \leq\sup_{z\in \hat Z_T} \frac{\langle y_T,z\rangle_X}{\|B^*S^*(T-\cdot)z\|_{L^1(0,T;U)}}.
  \end{eqnarray}
 On the other hand, let $z\in \hat Z_T$ be arbitrarily fixed. It is clear that
 \begin{eqnarray}\label{gengsheng2.14}
 B^*S^*(T-\cdot)z\in Y_T\setminus\{0\}.
  \end{eqnarray}
 Moreover, it follows from the first equality in (\ref{ob-attain-11}) and
 (\ref{ob-attain-12}) that
 \begin{eqnarray}\label{gengsheng2.15}
 \langle y_T,z\rangle_X=\int_0^T\langle v(t), B^*S^*(T-t)z\rangle_U \,\mathrm dt
 =\widetilde{\mathcal{F}}_{y_T} \big(B^*S^*(T-\cdot)z|_{(0,T)}\big).
 \end{eqnarray}
 By (\ref{gengsheng2.15}) and (\ref{gengsheng2.14}). we find that
 \begin{eqnarray*}
 \frac{\langle y_T,z\rangle_X}{\|B^*S^*(T-\cdot)z\|_{L^1(0,T;U)}}
 = \frac{\widetilde{\mathcal{F}}_{y_T}\left(B^*S^*(T-\cdot)z|_{(0,T)}\right)}
 {\|B^*S^*(T-\cdot)z\|_{L^1(0,T;U)}}\leq
 \sup_{f\in Y_T\setminus\{0\}}
    \frac{\langle \widetilde{\mathcal{F}}_{y_T},f\rangle_{Y_T^*,Y_T}}{\|f\|_{Y_T}}.
 \end{eqnarray*}
 Since $z$ was arbitrarily taken from $\hat Z_T$, the above leads to that
 \begin{eqnarray}\label{gengsheng2.16}
     \sup_{z\in \hat Z_T} \frac{\langle y_T,z\rangle_X}{\|B^*S^*(T-\cdot)z\|_{L^1(0,T;U)}}
     \leq
     \sup_{f\in Y_T\setminus\{0\}}
    \frac{\langle \widetilde{\mathcal{F}}_{y_T},f\rangle_{Y_T^*,Y_T}}{\|f\|_{Y_T}}.
  \end{eqnarray}

  Finally, (\ref{gengsheng2.10}) follows from (\ref{gengsheng2.13}) and (\ref{gengsheng2.16}). This, along with (\ref{gengsheng2.9}), proves    (\ref{norm-dual}).
   We end the proof of this proposition.
 \end{proof}

 The following proposition is about the relation between $(NP)^{y_T}$ and $(NP)^{T,y_0}$ with $y_T=-S(T)y_0$.
 \begin{Proposition}\label{Lemma-NP-yT-y0-eq}
Let $y_0\in X$ and $T\in(0,\infty)$ satisfy  that $-S(T)y_0\in\mathcal R_T$. Then
the following conclusions are valid:

\noindent (i) Any admissible control to  $(NP)^{y_T}$ (with $y_T\triangleq-S(T)y_0$) is an admissible control to  $(NP)^{T,y_0}$. And the reverse is also true.

\noindent (ii) $\|-S(T)y_0\|_{\mathcal R_T}=N(T,y_0)$.

\noindent (iii) Any minimal norm  control to  $(NP)^{y_T}$ (with $y_T=-S(T)y_0$) is a minimal norm  control to  $(NP)^{T,y_0}$. And the reverse is also true.
 \end{Proposition}

 \begin{proof}
  (i) Let  $\hat v$ be an admissible control to $(NP)^{y_T}$, with $y_T \triangleq -S(T)y_0$. Then it follows from (\ref{attianable-space-norm}) that
  $ \hat y(T;0,\hat v)=-S(T)y_0$,
    which yields that
   $\hat{y}(T;y_0,\hat v)=0$.
      This, along with  (\ref{NP-0}), implies that $\hat v$ is an admissible control to $(NP)^{T,y_0}$.

   Conversely, if $\widetilde{v}$ is an admissible control to $(NP)^{T,y_0}$, then by (\ref{NP-0}), we see that
   $   \hat y(T;y_0, \widetilde{v})=0$,
      which yields that
   $   \hat y(T;0,\widetilde{v})=-S(T)y_0$.
       This, along with (\ref{attianable-space-norm}), indicates that $\widetilde{v}$
    is an  admissible control to $(NP)^{y_T}$, with $y_T = -S(T)y_0$.

  (ii) By   (\ref{NP-0}) and (\ref{attianable-space-norm}), one can directly check that $N(T,y_0)=\|-S(T)y_0\|_{\mathcal R_T}$.

   (iii) Let $v^*$ be a minimal norm control to $(NP)^{y_T}$, with $y_T = -S(T)y_0$.
  Then by (i) of this proposition, $v^*$ is an admissible control to $(NP)^{T,y_0}$,
    i.e.,
  \begin{equation}\label{wuji2.17}
  \hat y(T;y_0,v^*)=0.
  \end{equation}
  Meanwhile, by the optimality of $v^*$, we have that
  $  \|v^*\|_{L^\infty(0,T;U)}=\|-S(T)y_0\|_{\mathcal{R}_T}$,
    which, along with (ii) of this proposition, shows that
  \begin{equation}\label{wuji2.18}
   \|v^*\|_{L^\infty(0,T;U)}=N(T,y_0).
  \end{equation}
  By (\ref{wuji2.17}) and (\ref{wuji2.18}), we see that $v^*$ is a minimal norm control to $(NP)^{T,y_0}$.

  Similarly, we can show the reverse. Thus, we finish the proof of this proposition.

 \end{proof}

 \begin{Corollary}\label{Co-NP-duality}
  Let $y_0\in X\setminus\{0\}$ satisfy that $T^0(y_0)<\infty$. Write
  \begin{equation*}
  \hat Z_T\triangleq \big\{z\in D(A^*)\; :\; B^*S^*(T-\cdot)z\neq 0\;\;\mbox{in}\;\;L^1(0,T;U) \big\},~0<T<\infty.
  \end{equation*}
      Then for each $T\in \big(T^0(y_0),\infty\big)$,
  \begin{eqnarray}\label{norm-dual-y0}
   N(T,y_0)=\sup_{z\in \hat Z_T} \frac{\langle S(T)y_0,z\rangle_X}{\|B^*S^*(T-\cdot)z\|_{L^1(0,T;U)}}<\infty.
  \end{eqnarray}
   \end{Corollary}

 \begin{proof}
Arbitrarily fix $T\in \big(T^0(y_0),\infty\big)$. At the start  of the proof of Proposition~\ref{proposition2.3}, we already proved that   $\hat Z_{s}\neq \emptyset$ for each $s\in (0,\infty)$.
    Since $T>T^0(y_0)$, by (\ref{y0-controllable}),  there exists a control $u\in L^\infty(0,T;U)$ so that $\hat y(T;y_0,u)=0$. This, along with (\ref{attainable-space}),  yields that
   $   -S(T)y_0= \hat y(T;0,u)  \in\mathcal R_T$,
      which, together with (ii) of Proposition~\ref{Lemma-NP-yT-y0-eq} and Proposition~\ref{proposition2.3}, leads to (\ref{norm-dual-y0}). We end the proof.

 \end{proof}

 The  property on $\mathcal{R}_T^0$  presented in the following Proposition~\ref{Proposition-ball-RT0-RT} plays another important role in the studies of a
  maximum principle for $(NP)^{y_T}$, with $y_T\in\mathcal{R}_T^0$. In what follows, we denote by $B_{\mathcal R_T^0}$ and  $B_{\mathcal R_T}$ the closed unit balls  in $\mathcal R_T^0$ and $\mathcal R_T$, respectively.
 \begin{Proposition}\label{Proposition-ball-RT0-RT}
 For each $T\in(0,\infty)$, it holds that $B_{\mathcal R_T}=\overline{B}_{\mathcal R_T^0}^{\sigma(\mathcal R_T,Y_T)}$. Here, the set
 $\overline{B}_{\mathcal R_T^0}^{\sigma(\mathcal R_T,Y_T)}$
 is  the closure of $B_{\mathcal R_T^0}$ in the space $\mathcal R_T$, under the topology
 $\sigma(\mathcal R_T,Y_T)$.
  \end{Proposition}

 \begin{proof}
 Let $0<T<\infty$. We first prove that
 \begin{eqnarray}\label{RT0-RT-4-2}
   B_{\mathcal R_T}\subset\overline{B}_{\mathcal R_T^0}^{\sigma(\mathcal R_T,Y_T)}.
 \end{eqnarray}
  Let $y_T\in B_{\mathcal R_T}$. From (\ref{attianable-space-norm}), there exists a sequence $\{v_k\}$ so that for all $k\in \mathbb N^+$,
  \begin{eqnarray}\label{1021-1}
   y_T=\hat y(T;0,v_k)\;\;\mbox{and}\;\;\|y_T\|_{\mathcal R_T}\leq\|v_k\|_{L^\infty(0,T;U)}\leq \|y_T\|_{\mathcal R_T}+1/k.
  \end{eqnarray}
  For each $k\in \mathbb{N}^+$, we set
  \begin{eqnarray}\label{RT0-RT-2}
  \lambda_k \triangleq \frac{\|y_T\|_{\mathcal R_T}}{\|y_T\|_{\mathcal R_T}+1/k}  \;\;\mbox{and}\;\;  u_k\triangleq\chi_{(0,T-1/k)} \lambda_k v_k.
  \end{eqnarray}
  It is clear that
  \begin{eqnarray}\label{RT0-RT-3}
  \|u_k\|_{L^\infty(0,T;U)}\leq \|y_T\|_{\mathcal R_T}\leq 1\;\;\mbox{for all}\;\; k\in\mathbb N^+.
  \end{eqnarray}
  From (\ref{R0T}), (\ref{RT0-RT-2}), (\ref{attianable-space-norm})  and (\ref{RT0-RT-3}), we can easily check that
  \begin{eqnarray}\label{RT0-RT-4}
  \hat y(T;0,u_k)\in B_{\mathcal R_T^0}\;\;\mbox{for all}\;\; k\in\mathbb N^+.
  \end{eqnarray}
  Meanwhile, from  (\ref{1021-1}), (\ref{ob-attain-0}) and (\ref{RT0-RT-2}),  we find that for each
  $f\in Y_T$,
  \begin{eqnarray*}
    \big\langle \hat y(T;0,u_k)-y_T,f \big\rangle_{\mathcal R_T,Y_T}
    &=& \big\langle \hat y(T;0,u_k-v_k),f \big\rangle_{\mathcal R_T,Y_T}
      \nonumber\\
    &=& \int_0^T  \big\langle u_k(t)-v_k(t),f(t) \big\rangle_U \,\mathrm dt
  \rightarrow  0,   \;\;\mbox{as}\;\;     k\rightarrow\infty.
  \end{eqnarray*}
  This, along with (\ref{RT0-RT-4}), yields that
   $   y_T\in  \overline B_{\mathcal R_T^0}^{\sigma(\mathcal R_T,Y_T)}$.
       Since $y_T$ was arbitrarily taken from $B_{\mathcal R_T}$, the above leads (\ref{RT0-RT-4-2}).

  We next show that
  \begin{eqnarray}\label{RT0-RT-4-1}
   B_{\mathcal R_T}\supseteq\overline{B}_{\mathcal R_T^0}^{\sigma(\mathcal R_T,Y_T)}.
  \end{eqnarray}
  For this purpose, we let $y_T\in\mathcal R_T$ and $\{y_n\}\subset B_{\mathcal R_T^0}$ so that
  $$
  y_n\rightarrow y_T\;\;\mbox{in the topology}\;\;\sigma(\mathcal R_T,Y_T),\;\;\mbox{as}\;\;n\rightarrow\infty.
   $$
   Since $\mathcal R_T=Y_T^*$ (see Theorem \ref{Theorem-the-representation-theorem}), we find that
   $$
   y_n\rightarrow y_T\;\;\mbox{in the weak star topology},\;\;\mbox{as}\;\;n\rightarrow\infty.
   $$
    Hence,
   $$
   \|y_T\|_{\mathcal R_T}\leq\liminf_{n\rightarrow\infty} \|y_n\|_{\mathcal R_T}\leq 1,
   $$
    which yields that  $y_T\in B_{\mathcal R_T}$. This proves (\ref{RT0-RT-4-1}).

    Finally, it follows from (\ref{RT0-RT-4-2}) and (\ref{RT0-RT-4-1}) that
    $B_{\mathcal R_T}=\overline{B}_{\mathcal R_T^0}^{\sigma(\mathcal R_T,Y_T)}$.
    This ends the proof.

 \end{proof}

 The following lemma mainly shows  that  the reachable subspaces $\mathcal R_{T}$ and $\mathcal R_{T}^0$ are independent of  $T\in(0,\infty)$, provided that the condition (H1) holds.

 \begin{Proposition}\label{attinable-invariant}
  Suppose that (H1) holds. Let $0<T_1<T_2<\infty$. Then the following conclusions are valid:

  \noindent (i) The spaces  $\mathcal R_{T_1}$ and $\mathcal R_{T_2}$ are same, and the norms $\|\cdot\|_{\mathcal R_{T_1}}$ and $\|\cdot\|_{\mathcal R_{T_2}}$ are equivalent.

  \noindent (ii) The spaces  $\mathcal R_{T_1}^0$ and $\mathcal R_{T_2}^0$ are same.

 \end{Proposition}

 \begin{proof}
  Suppose that (H1) holds. Arbitrarily fix $0<T_1<T_2<\infty$. We will prove the conclusions (i)-(ii) one by one.

  (i)  Arbitrarily fix $y_{T_1}\in \mathcal R_{T_1}$ and $k\in\mathbb N^+$. Then by (\ref{attainable-space}) and (\ref{attianable-space-norm}), there exists a control $u_{y_{T_1}}\in L^\infty(0,T_1;U)$ so that
  \begin{eqnarray}\label{0102-sec2.3-2}
   y_{T_1}=\hat y(T_1;0,u_{y_{T_1}}) \;\;\mbox{and}\;\;  \|u_{y_{T_1}}\|_{L^\infty(0,T_1;U)}\leq \|y_{T_1}\|_{\mathcal R_{T_1}} + 1/k.
  \end{eqnarray}
  Define another control $\hat u_{y_{T_1}}$  by setting
  \begin{eqnarray}\label{0102-sec2.3-3}
    \hat u_{y_{T_1}}(t)=\begin{cases}
                     0,  ~&t\in(0,T_2-T_1],\\
                     u_{y_{T_1}}(t-T_2+T_1), ~&t\in(T_2-T_1,T_2).
                    \end{cases}
  \end{eqnarray}
  Then  from (\ref{Changshubianyi1.6}), the first equality in (\ref{0102-sec2.3-2}) and (\ref{0102-sec2.3-3}), one can easily check that
  $y_{T_1}=\hat y(T_2;0, \hat u_{y_{T_1}})$,
    which, along with  (\ref{attainable-space}), (\ref{attianable-space-norm}), (\ref{0102-sec2.3-3}) and the second inequality in (\ref{0102-sec2.3-2}), yields that
 $y_{T_1} \in \mathcal R_{T_2}$ and $\|y_{T_1}\|_{\mathcal R_{T_2}} \leq \|y_{T_1}\|_{\mathcal R_{T_1}}+1/k$.
    Since $k$ was arbitrarily taken from $\mathbb N^+$, the above implies that  for each $y_{T_1}\in \mathcal R_{T_1}$,
  \begin{eqnarray}\label{0102-last-two-1}
    y_{T_1} \in \mathcal R_{T_2}
   \;\;\mbox{and}\;\;  \|y_{T_1}\|_{\mathcal R_{T_2}} \leq \|y_{T_1}\|_{\mathcal R_{T_1}}.
  \end{eqnarray}

  Conversely,  arbitrarily fix $y_{T_2}\in \mathcal R_{T_2}$ and $k\in\mathbb N^+$. Then by (\ref{attainable-space}) and (\ref{attianable-space-norm}), there exists a control $u_{y_{T_2}}\in L^\infty(0,T_2;U)$ so that
  \begin{eqnarray}\label{0102-sec2.3-4}
   y_{T_2}=\hat y(T_2;0,u_{y_{T_2}})
   \;\;\mbox{and}\;\;  \|u_{y_{T_2}}\|_{L^\infty(0,T_2;U)}\leq \|y_{T_2}\|_{\mathcal R_{T_2}} + 1/k.
  \end{eqnarray}
  By (H1), we can apply Lemma \ref{Lemma-H3-eq} to get  the conclusion (ii) of Lemma \ref{Lemma-H3-eq} with some $p_1\in [2,\infty)$.
  Because  $\chi_{(0,T_2-T_1)}u_{y_{T_2}}\in L^{p_1}(0,T_2;U)$, it follows from
   (ii) of Lemma \ref{Lemma-H3-eq} (where $T=T_2$ and $t=T_2-T_1$) that there exists a control $\hat v \in L^\infty(0,T_2;U)$ so that
  \begin{eqnarray}\label{0102-sec2.3-5}
   \hat y(T_2;0,\chi_{(0,T_2-T_1)} u_{y_{T_2}}) = \hat y(T_2;0,\chi_{(T_2-T_1,T_2)} \hat v)
  \end{eqnarray}
  and
  \begin{eqnarray}\label{0102-sec2.3-5-1}
   \|\hat v\|_{L^\infty(0,T_2;U)} \leq C_1 \| u_{y_{T_2}} \|_{L^{p_1}(0,T_2;U)}
   \leq C_1 (T_2)^{1/p_1} \| u_{y_{T_2}} \|_{L^\infty(0,T_2;U)},
  \end{eqnarray}
  where $C_1 \triangleq C_1(T_2,T_2-T_1)$ is given by (ii) of Lemma \ref{Lemma-H3-eq}.
  Define a control $\widetilde{v}(\cdot)$ by setting
  \begin{eqnarray*}
   \widetilde{v}(t) \triangleq u_{y_{T_2}}  (t+T_2-T_1)  +  \hat v (t+T_2-T_1),~t\in(0,T_1).
  \end{eqnarray*}
  Then, by the first assertion in (\ref{0102-sec2.3-4}) and (\ref{0102-sec2.3-5}), one can directly check that $ y_{T_2}=\hat y(T_1;0,\widetilde{v})$,
    which, together with (\ref{attainable-space}), (\ref{attianable-space-norm}), (\ref{0102-sec2.3-5-1}) and the  inequality in (\ref{0102-sec2.3-4}), indicates that
  \begin{eqnarray*}
   y_{T_2} \in \mathcal R_{T_1}
   \;\;\mbox{and}\;\;
    \|y_{T_2}\|_{\mathcal R_{T_1}} \leq  (1+C_1(T_2)^{1/p_1}) \big(\|y_{T_1}\|_{\mathcal R_{T_2}}+1/k \big).
  \end{eqnarray*}
  Since $k$ was arbitrarily taken from $\mathbb N^+$, the above implies that  for each $y_{T_2}\in \mathcal R_{T_2}$,
  \begin{eqnarray}\label{0102-last-two-2}
    y_{T_2} \in \mathcal R_{T_1}
   \;\;\mbox{and}\;\;  \|y_{T_2}\|_{\mathcal R_{T_1}} \leq (1+C_1(T_2)^{1/p_1}) \|y_{T_2}\|_{\mathcal R_{T_2}}.
  \end{eqnarray}
   Now,  the conclusion (i) follows from (\ref{0102-last-two-1}) and (\ref{0102-last-two-2}).

  \vskip 3pt
  (ii) By a very  similar way as that used in the proof of the conclusion (i), we can show that $\mathcal R_{T_1}^0 =  \mathcal R_{T_2}^0$.

%

  \vskip 3pt

  In summary, we end the proof of this proposition.

 \end{proof}

\section{Properties of several  functions}

This section presents some properties on functions $N(\cdot,y_0)$ (with $y_0\in X\setminus\{0\}$), $T^0(\cdot)$ and $T^1(\cdot)$, which are defined by
(\ref{NP-0}), (\ref{y0-controllable}) and (\ref{Ty0}), respectively.
 The decompositions of $\mathcal{W}$ and $\mathcal{V}$ (given in (i) of Theorem~\ref{Proposition-NTy0-partition} and (i) of Theorem~\ref{Proposition-TMy0-partition}, respectively) are based on these properties.
  We begin with the following Lemma~\ref{Lemma-t-u-converge}.
    Since  the exactly same result as that in this lemma was not found by us in literatures  (but the proof for the similar result to Lemma~\ref{Lemma-t-u-converge} can be found in, for instance, \cite[Lemma 1.1]{HOF1}), we give its proof in Appendix E, for the sake of the completeness of the paper.

\begin{Lemma}\label{Lemma-t-u-converge}
 Let $\{T_n\}_{n=1}^\infty\subset[0,\infty)$ and $\{u_n\}_{n=1}^\infty\subset L^2(\mathbb R^+;U)$ satisfy that
 \begin{eqnarray}\label{wanggengsheng3.1}
 T_n\rightarrow \widehat T\;\;\mbox{and}\;\;u_n\rightarrow \hat u\;\;\mbox{weakly in}\;\;
 L^2(\mathbb R^+;U),
 \;\;\mbox{as}\;\;   n\rightarrow \infty
 \end{eqnarray}
 for some $\widehat T\in [0, \infty)$ and $\hat u \in L^2(\mathbb R^+;U)$.
   Then for each $y_0\in X$,
 \begin{eqnarray}\label{time-control-converge}
  y(T_n;y_0,u_n) \rightarrow y(\widehat T;y_0,\hat u) \mbox{ weakly in } X,\;\;\mbox{as}\;\;
  n\rightarrow\infty.
 \end{eqnarray}
 \end{Lemma}

The next lemma concerns the monotonicity of the function $N(\cdot,y_0)$.
 \begin{Lemma}\label{Lemma-NP-decreasing}
  Let $y_0\in X\setminus\{0\}$. Then the following conclusions are true:

  \noindent (i) The function $N(\cdot,y_0)$ is decreasing from $(0,\infty)$ to $[0,\infty]$.

   \noindent (ii) The function $N(\cdot,y_0)$, when extended over $[0,\infty]$ via (\ref{N-infty-y0}), is decreasing from $[0,\infty]$ to $[0,\infty]$.

 \end{Lemma}

 \begin{proof}
 (i) We first show that $N(\cdot,y_0)$ is decreasing over $(0,\infty)$. For this purpose, let $T_1$ and $T_2$ satisfy that $0<T_1<T_2<\infty$.
   There are only two possibilities on  $N(T_1,y_0)$: either $N(T_1,y_0)=\infty$ or $N(T_1,y_0)<\infty$.

  In the case that $N(T_1,y_0)=\infty$, it is obvious that
  $  N(T_1,y_0)\geq N(T_2,y_0)$.
        In the case that $N(T_1,y_0)<\infty$, we arbitrarily fix a $\varepsilon>0$.
      It follows from (\ref{NP-0}) that there exists a
      control $v_\varepsilon$ so that
 $\hat y(T_1;y_0,v_\varepsilon)=0$ and $\|v_\varepsilon\|_{L^\infty(0,T_1;U)}\leq N(T_1,y_0)+\varepsilon$.
 Write $\widetilde v_\varepsilon$ for the zero extension of  $v_\varepsilon$   over $(0,T_2)$. Then from the above, we find that
\begin{eqnarray}\label{NP-decrease-2}
  \hat y(T_2;y_0, \widetilde v_\varepsilon)=0 \;\;\mbox{and}\;\; \|\widetilde v_\varepsilon\|_{L^\infty(0,T_2;U)}=\| v_\varepsilon\|_{L^\infty(0,T_1;U)}
  \leq N(T_1,y_0)+\varepsilon.
 \end{eqnarray}
From   the first equality in (\ref{NP-decrease-2}), it follows  that $\widetilde v_\varepsilon$ is an admissible control to $(NP)^{T_2,y_0}$.
This, along with the optimality of $N(T_2,y_0)$ and the second assertion in (\ref{NP-decrease-2}), yields that
$$
  N(T_2,y_0)\leq \|\widetilde v_\varepsilon\|_{L^\infty(0,T_2;U)} \leq N(T_1,y_0)+\varepsilon.
$$
 Since $\varepsilon$ was arbitrarily taken, the above leads to the following inequality in this case:
 $ N(T_1,y_0)\geq N(T_2,y_0)$.
  Hence, the function $N(\cdot,y_0)$ is decreasing over $(0,\infty)$.

  Next,  by (\ref{NP-0}), we see that
  $0\leq N(T,y_0) \leq \infty $ for all $T\in(0,\infty)$.
 Thus, the conclusion (i) of this lemma has been proved.

 (ii) The conclusion (ii)    follows from the conclusion (i) of this lemma and (\ref{N-infty-y0}).

 In summary, we end the proof of this lemma.

 \end{proof}

 The following two lemmas concern with some relations among  $N(\cdot,y_0)$,  $T^0(\cdot)$ and $T^1(\cdot)$.

\begin{Lemma}\label{Lemma-T0-T1}
  Let   $y_0\in X\backslash\{0\}$. Then the following conclusions are true:

  \noindent (i)   $T^0(y_0)\leq T^1(y_0)$. (ii) $T^1(y_0)>0$. (iii)  $N(T,y_0)>0$ for all $T\in \big(0, T^1(y_0)\big)$. (iv) $N(0,y_0)=\infty$.
    (v) If   $T^1(y_0)<\infty$, then $N(T,y_0)=0$ for all $T\in \big[T^1(y_0),\infty\big]$.
       (vi)
       $N(T^1(y_0),y_0)=N(\infty,y_0)$.

  \end{Lemma}

 \begin{proof} (i)  There are only two possibilities on $T^0(y_0)$: either $T^0(y_0)=0$ or $T^0(y_0)>0$.
 In the case that
 $T^0(y_0)=0$, it is clear that $T^0(y_0)\leq T^1(y_0)$. In the case when  $T^0(y_0)>0$, we assume, by contradiction, that $T^0(y_0)>T^1(y_0)$. Fix a $T\in \big(T^1(y_0),T^0(y_0)\big)$. Then by (\ref{y0-controllable}), we would have  that
for all $u\in L^\infty(0,T; U)$, $\hat y(T;y_0,u)\neq 0$; and  by (\ref{Ty0}), we would have that
 $\hat y(T; y_0,0)=S(T)y_0=0$.
 These lead to a contradiction. Hence, $T^0(y_0)\leq T^1(y_0)$.

   (ii)  By contradiction, suppose   that $T^1(y_0)=0$. Then by (\ref{Ty0}), we could have  that for each $\hat t>0$, $S(\hat t)y_0=0$,
     which yields that
    $y_0=\lim_{t\rightarrow 0^+} S(t)y_0=0$.
          This leads to a contradiction, since we assumed that $y_0\in X\setminus\{0\}$. Hence,  $T^1(y_0)>0$.

(iii)
By contradiction, suppose that
$N(T_0,y_0)=0$ for some $T_0\in \big(0, T^1(y_0)\big)$.
Then by (\ref{NP-0}), there would be a sequence $\{v_n\}$ in $L^\infty(0,T_0;U)$ so that
$$
\hat y(T_0; y_0,v_n)=0   \;\;\mbox{for all}\;\; n\in\mathbb N^+;
\;\;\mbox{and}\;\;
\|v_n\|_{L^\infty(0,T_0;U)}\rightarrow 0,
\;\;\mbox{as}\;\;  n\rightarrow\infty.
$$
From these and Lemma \ref{Lemma-t-u-converge}, we find that
$S(T_0)y_0=\hat y(T_0;y_0,0)=0$.
  From the above and (\ref{Ty0}), we see that  $T^1(y_0)\leq T_0$, which leads to a contradiction, since $T_0\in \big(0, T^1(y_0)\big)$. Hence, $N(T,y_0)> 0$ for all $ T\in \big(0, T^1(y_0)\big)$.

  (iv)
   By contradiction,  suppose that $N(0,y_0)<\infty$.  Then by  (ii) of Lemma \ref{Lemma-NP-decreasing},
       we could find  a sequence $\{T_n\} \subset \mathbb R^+$ so that
       \begin{equation}\label{jiangwen1}
       T_n\searrow 0,\;\;\mbox{as}\;\;n\rightarrow\infty
       \end{equation}
       and
   \begin{equation}\label{jiangwen2}
 N(T_n,y_0)\leq N(0,y_0)<\infty\;\;\mbox{for all}\;\; n\in \mathbb{N}^+.
 \end{equation}
  By (\ref{jiangwen2}) and (\ref{NP-0}), we see that  for each $n\in\mathbb N^+$, $(NP)^{T_n,y_0}$ has
 an admissible control  $u_n$ so that
 $\|u_n\|_{L^\infty(0,T_n;U)}\leq N(0,y_0)+1$.
     Write $\widetilde{u}_n$ for the zero extension of $u_n$ over $\mathbb{R}^+$, $n\in\mathbb N^+$. Then we have that
   \begin{equation}\label{jiangwen3}
y(T_n;y_0,\widetilde{u}_n)=0\;\;\mbox{for all}\;\; n\in \mathbb{N}^+
 \end{equation}
    and
     \begin{equation}\label{jiangwen4}
\|\widetilde{u}_n\|_{L^\infty(\mathbb{R}^+;U)}=\|u_n\|_{L^\infty(0,T_n;U)}\leq N(0,y_0)+1\;\;\mbox{for all}\;\; n\in \mathbb{N}^+.
 \end{equation}
 From (\ref{jiangwen1}) and (\ref{jiangwen4}), we see that
    \begin{eqnarray}\label{jiangwen5}
 \chi_{(0,T_n)}\widetilde{u}_n\rightarrow  0 \mbox{ strongly in } L^2(\mathbb R^+;U) \mbox{ as } n\rightarrow\infty.
 \end{eqnarray}
    From (\ref{jiangwen1}), (\ref{jiangwen5}) and Lemma \ref{Lemma-t-u-converge}, we find that
 \begin{equation*}
 y(T_n;y_0,\chi_{(0,T_n)}\widetilde{u}_n)\rightarrow y(0;y_0,0)=y_0\;\; \mbox{ weakly in }\;\; X,\;\; \mbox{ as }\;\; n\rightarrow \infty.
 \end{equation*}
  This, along with (\ref{jiangwen3}), yields that
  $y_0=0$, which leads to
    a contradiction, since it was assumed that   $y_0\in X\setminus\{0\}$. So we have proved that $N(0,y_0)=\infty$.

(v)   Assume that  $T^1(y_0)<\infty$. We first claim that
 \begin{equation}\label{wang3.18}
 N(T, y_0)=0\;\;\mbox{for each}\;\; T\in \big[T^1(y_0), \infty\big).
 \end{equation}
 By contradiction, we suppose that
 $ N(T_1,y_0)\neq 0$ for some $T_1\in[ T^1(y_0),\infty)$.
    Then we would have that
 $ \hat y(T_1; y_0,0)\neq 0$, i.e., $S(T_1)y_0\neq 0$.
 By the continuity of the function $t\rightarrow S(t)y_0$ at  $T_1$, there is a $\delta>0$ so that
$S(T_1+\delta)y_0\neq 0$,
which implies that for each $ t\in [0,T_1+\delta]$,
$ S(t)y_0\neq 0$.
This, together with  (\ref{Ty0}), implies that
\begin{equation}\label{jiangwen4.23}
T_1+\delta\leq T^1(y_0).
\end{equation}
However, we had that
$T_1\geq T^1(y_0)$ and $\delta>0$.
These contradict (\ref{jiangwen4.23}).
 So (\ref{wang3.18}) is proved.

Next,  we see from the first equality in (\ref{N-infty-y0}) and (\ref{wang3.18})  that $N(\infty, y_0)=0$. This, together with (\ref{wang3.18}), proves the conclusion (v).

 (vi) There are only two possibilities on $T^1(y_0)$: either $T^1(y_0)=\infty$ or $T^1(y_0)<\infty$.
 In the case when $T^1(y_0)=\infty$, it is clear  that
 $  N(T^1(y_0),y_0)=N(\infty,y_0)$.
   In the case that   $T^1(y_0)<\infty$, we see   from (v)  in this lemma that
  \begin{equation}\label{huangwang3.5}
   N(T^1(y_0),y_0)=0=N(\infty,y_0).
   \end{equation}
   This implies that   $  N(T^1(y_0),y_0)=N(\infty,y_0)$ in this case.

   In summary, we end the proof of this lemma.

 \end{proof}

 \begin{Lemma}\label{Lemma-NT0-T0-T1}
  Let $y_0\in X\setminus\{0\}$. Then the following conclusions are true:

  \noindent (i) If $N(T^0(y_0),y_0)=\infty$, then either $T^0(y_0)<T^1(y_0)$ or $T^0(y_0)=T^1(y_0)=\infty$.

  \noindent (ii) If $T^0(y_0)=\infty$, then  $N(T^0(y_0),y_0)=\infty$.

  \noindent(iii) If $0<N(T^0(y_0),y_0)<\infty$, then $T^0(y_0)<T^1(y_0)$.

  \noindent(iv) $N(T^0(y_0),y_0)=0$ if and only if $T^0(y_0)=T^1(y_0)<\infty$.

\noindent (v)  If $T^0(y_0)<\infty$, then $N(T^1(y_0),y_0)<\infty$.

 \end{Lemma}

 \begin{proof}
 (i) By contradiction, we suppose that the conclusion (i) was not true. Then, by (i) of Lemma~\ref{Lemma-T0-T1}, we would have that
 \begin{equation}\label{wanghuang3.4}
 N(T^0(y_0),y_0)=\infty\;\;\mbox{and}\;\;T^0(y_0)=T^1(y_0)<\infty.
 \end{equation}
 The second conclusion in (\ref{wanghuang3.4}), along with  (v) of Lemma \ref{Lemma-T0-T1}, yields  that
$N(T^0(y_0),y_0)=N(T^1(y_0),y_0)=0$.
This contradicts the first equality in  (\ref{wanghuang3.4}). So the conclusion (i) is true.

(ii) Assume that $T^0(y_0)=\infty$. Then we find from (\ref{y0-controllable}) that when $T \in(0,\infty)$,
$\hat y(T;y_0,u)\neq 0$ for all $u\in L^\infty(0,T; U)$.
 Thus, for each $T\in(0,\infty)$,  $(NP)^{T,y_0}$ has no any admissible control. So
 we have that $N(T,y_0)=\infty$ for all $T\in(0,\infty)$.
   Since $T^0(y_0)=\infty$, the above, as well as the first equality in (\ref{N-infty-y0}), indicates that
  $$
  N(T^0(y_0),y_0)=N(\infty, y_0)=\lim_{T\rightarrow\infty}N(T, y_0)  =\infty.
  $$
   This ends the proof of the conclusion (ii).

(iii) Assume that $0<N(T^0(y_0),y_0)<\infty$. Suppose, by contradiction, that the conclusion (iii) was not true. Then, by (i) of Lemma \ref{Lemma-T0-T1}, we would have that
 \begin{equation}\label{wanghuang3.5}
 0<N(T^0(y_0),y_0)<\infty\;\;\mbox{and}\;\;T^0(y_0)=T^1(y_0).
  \end{equation}
 These, along with  (ii) of this lemma, yield that
 $ T^1(y_0)=T^0(y_0)<\infty$.
  Then by (v) of
 Lemma \ref{Lemma-T0-T1}, we see that
 $ N(T^0(y_0),y_0)=N(T^1(y_0),y_0)=0$, which contradicts the first conclusion in (\ref{wanghuang3.5}). Hence, the conclusion (iii) is true.

 (iv) We first show that
 \begin{equation}\label{wanghuang3.6}
 T^0(y_0)=T^1(y_0)<\infty\Rightarrow N(T^0(y_0),y_0)=0.
 \end{equation}
 Suppose that the assertion on left  side of (\ref{wanghuang3.6}) holds. Then by (v) of
 Lemma~\ref{Lemma-T0-T1}, we see that
 $ N(T^0(y_0),y_0)=N(T^1(y_0),y_0)=0$,
  which leads to the equality on the right side of (\ref{wanghuang3.6}).

 We next show that
 \begin{equation}\label{wanghuang3.7}
 N(T^0(y_0),y_0)=0\Rightarrow  T^0(y_0)=T^1(y_0).
 \end{equation}
By contradiction, we suppose  that  (\ref{wanghuang3.7}) did not hold. Then by (i) of  Lemma \ref{Lemma-T0-T1},
we would have that
\begin{equation}\label{wanghuang3.8}
N(T^0(y_0),y_0)=0\;\;\mbox{and }\;\;T^0(y_0)<T^1(y_0).
\end{equation}
In the case that $T^0(y_0)=0$, we find from  (iv) of Lemma \ref{Lemma-T0-T1}
that  $N(T^0(y_0),y_0)=\infty$, which contradicts the first equality in (\ref{wanghuang3.8}).
In the case that $T^0(y_0)>0$,  we see from the second inequality of
(\ref{wanghuang3.8}) and  (iii) of Lemma \ref{Lemma-T0-T1}
 that $N(T^0(y_0),y_0)>0$,
which  contradicts the first equality in (\ref{wanghuang3.8}).
Hence, (\ref{wanghuang3.7}) is true.

Finally, the conclusion (iv) follows from (\ref{wanghuang3.6}) and (\ref{wanghuang3.7}).

(v)  Assume that $T^0(y_0)<\infty$. There are
 only two possibilities on
 $T^1(y_0)$: either $T^1(y_0)<\infty$ or $T^1(y_0)=\infty$.
  In the first  the case that
$T^1(y_0)<\infty$,
  we can apply the conclusion
  (v) of Lemma \ref{Lemma-T0-T1} to find that
$ N(T^1(y_0),y_0)=0<\infty$. Hence, the conclusion (v) holds in the first case.
We now consider the second  case that
  $T^1(y_0)=\infty$.
  Because $T^0(y_0)<\infty$, we can take  $\hat t\in\big(T^0(y_0),\infty\big)$. Then by  (\ref{y0-controllable}), we find that
 $  \hat y(\hat t;y_0,\hat u)=0$ for some $\hat u\in L^\infty(0,\hat t;U)$.
  This shows that $\hat u$ is an admissible control to $(NP)^{\hat t,y_0}$,
 from which,  we see that
 \begin{eqnarray}\label{new3.22}
  N(\hat t,y_0)<\infty.
 \end{eqnarray}
Because $T^1(y_0)=\infty$, it follows from
  (ii) of Lemma \ref{Lemma-NP-decreasing} and (\ref{new3.22}) that
  \begin{eqnarray*}
   N(T^1(y_0),y_0)=N(\infty,y_0)\leq N(\hat t,y_0)<\infty.
  \end{eqnarray*}
 Hence,  the conclusion (v) of this Lemma holds in the second case.

 In summary, we finish the proof of this lemma.

 \end{proof}

 \begin{Remark}\label{Remark-0131-intro-1}
 (i) Let $y_0\in X\setminus\{0\}$. From the above lemma, we have the following two observations: (a) $T^0(y_0)<T^1(y_0)$ if and only if either $ 0<N(T^0(y_0),y_0)<\infty$ or $N(T^0(y_0),y_0)=\infty$ and $T^0(y_0)<\infty$; (b) $T^0(y_0)=T^1(y_0)$ if and only if either $ N(T^0(y_0),y_0)=0$ or $N(T^0(y_0),y_0)=\infty$ and $T^0(y_0)=\infty$.

 (ii) From the above two observations and the definitions of $\mathcal W_{2,3}$, $\mathcal W_{3,2}$, $\mathcal V_{2,2}$ and $\mathcal V_{3,2}$ (see (\ref{zhangjinchu7}), (\ref{zhangjinchu9}), (\ref{Lambda-di-2-1}) and (\ref{Lambda-di-3-1}), respectively), one can easily find that
 \begin{eqnarray*}
   \mathcal W_{2,3} \cup \mathcal W_{3,2}  =  \{(T,y_0)\in\mathcal W ~:~  T^0(y_0)<T<T^1(y_0)\}
 \end{eqnarray*}
 and
 \begin{eqnarray*}
   \mathcal V_{2,2} \cup \mathcal V_{3,2}  =  \{(M,y_0)\in\mathcal V ~:~  N(T^1(y_0),y_0)<M<N(T^0(y_0),y_0)\}.
 \end{eqnarray*}

 \end{Remark}

  The next Proposition~\ref{wang-prop3.3} presents the strict monotonicity and the
  continuity for the function $N(\cdot,y_0)$ over $\big(T^0(y_0),T^1(y_0)\big)$. These properties will help us to build up a connection between  minimal time control problems and  minimal norm control problems. This connection
  plays an important role in the studies of the maximum principle for $(TP)^{M, y_0}$. We would like to mention what follows: The properties in Proposition~\ref{wang-prop3.3}
  was proved in \cite{WZ} for the internally controlled heat equation, with the aid of the bang-bang property and the $L^\infty$-null controllability. Here, we have neither the bang-bang property nor the $L^\infty$-null controllability. We prove it under a weaker condition (H1).

 \begin{Proposition}\label{wang-prop3.3}
  Suppose that (H1) holds.  Let $y_0\in X\setminus\{0\}$ satisfy that $T^0(y_0)<T^1(y_0)$.  Then the following conclusions are true:\\
  \noindent(i) The function $N(\cdot,y_0)$ is continuous and strictly decreasing from $\big(T^0(y_0),T^1(y_0)\big)$ onto $\big(N(T^1(y_0),y_0),N(T^0(y_0),y_0)\big)$.

  \noindent(ii) When $T\in\big(T^0(y_0),T^1(y_0)\big)$,
  \begin{eqnarray}\label{NP-strict-decreasing-infty}
   N(t_1,y_0)>N(T,y_0)>N(t_2,y_0) \;\;\mbox{for all}\;\; t_1,\,t_2 \;\;\mbox{with}\;\; 0\leq t_1<T<t_2\leq\infty.
  \end{eqnarray}
 \end{Proposition}

 \begin{proof}
 (i) Arbitrarily fix a  $y_0\in X\setminus\{0\}$ so that
 $ T^0(y_0)<T^1(y_0)$. From (iii) of Lemma \ref{Lemma-T0-T1} and Corollary \ref{Co-NP-duality}, we see that
 \begin{eqnarray}\label{NP-cont-de-key}
  0<N(T,y_0)<\infty \mbox{ for all } T\in \big(T^0(y_0),T^1(y_0)\big).
 \end{eqnarray}
 We organize the rest of the  proof by the following three steps:

 \vskip 3pt
 \noindent \textit{Step 1. To show that the function $N(\cdot,y_0)$ is strictly decreasing over $\big(T^0(y_0),T^1(y_0)\big)$}

  Arbitrarily fix two numbers $T_1$ and $T_2$ so that
  $ T^0(y_0)<T_1<T_2<T^1(y_0)$.
  Because (H1) holds, we can  apply  Lemma \ref{Lemma-H3-eq} to get  the conclusion (ii) of
  Lemma \ref{Lemma-H3-eq}.
  Let  $p_1\in [2,\infty)$
  and $C_1\triangleq C_1(T_2,T_1)$  be given by (ii) of Lemma \ref{Lemma-H3-eq}.
Then by (\ref{NP-cont-de-key}),
  there is a $\delta>0$ so that
     \begin{eqnarray}\label{NP-strict-de-11}
     \lambda\triangleq\frac{2\delta}{N(T_1,y_0)+\delta}\in (0,1)\;\;\mbox{and}\;\;  C_1 \lambda T_1^{1/p_1}\leq \frac{N(T_1,y_0)-\delta}{N(T_1,y_0)+\delta}.
   \end{eqnarray}
    Meanwhile, by  (\ref{NP-cont-de-key}), we have that $N(T_1,y_0)<\infty$. This, along with
      (\ref{NP-0}), yields that there exists an admissible control $v_1$ to $(NP)^{T_1,y_0}$ so that
   \begin{eqnarray}\label{NP-strict-de-12}
    \hat y(T_1;y_0,v_1)=0  \mbox{ and } \|v_1\|_{L^\infty(0,T_1;U)}\leq N(T_1,y_0)+\delta.
   \end{eqnarray}
   Write $\widetilde v_1$ for the zero extension of  $v_1$ over $(0,T_2)$.
   According to       (ii) of Lemma \ref{Lemma-H3-eq},  there is a control  $v_2\in L^\infty(0,T_2;U)$ so that
   \begin{eqnarray}\label{NP-strict-de-13}
     \hat{y}(T_2;0,\chi_{(0,T_1)}\lambda\widetilde v_1)=\hat{y}(T_2;0,\chi_{(T_1,T_2)}v_2)
   \end{eqnarray}
   and so that
   \begin{eqnarray}\label{NP-strict-de-14}
    \|v_2\|_{L^{\infty}(0,T_2;U)}\leq C_1\|\lambda\widetilde v_1\|_{L^{p_1}(0,T_2;U)}\leq C_1\lambda T_1^{1/p_1}\|v_1\|_{L^{\infty}(0,T_1;U)}.
   \end{eqnarray}
   We now define another control:
   \begin{eqnarray}\label{NP-strict-de-14-1}
   v_3(t)\triangleq\chi_{(0,T_1)}(t)(1-\lambda)\widetilde v_1(t)+\chi_{(T_1,T_2)}(t)v_2(t),~t\in (0,T_2).
   \end{eqnarray}
   From (\ref{NP-strict-de-14-1}), (\ref{NP-strict-de-13}) and the first equality in (\ref{NP-strict-de-12}), one can easily check that
   $ \hat y(T_2;y_0,v_3)=S(T_2-T_1)\hat y(T_1;y_0,v_1)=0$,
      which implies that $v_3$ is  an admissible control to $(NP)^{T_2,y_0}$. This, together with  the definition  of $N(T_2,y_0)$ (see  (\ref{NP-0})) and (\ref{NP-strict-de-14-1}), implies that
   $$
    N(T_2,y_0)\leq \|v_3\|_{L^\infty(0,T_2;U)}\leq \max \big\{(1-\lambda)\|v_1\|_{L^\infty(0,T_1;U)},\|v_2\|_{L^{\infty}(0,T_2;U)} \big\}.
   $$
   From this, (\ref{NP-strict-de-12}), (\ref{NP-strict-de-14}) and (\ref{NP-strict-de-11}), after some simple computations,  we deduce that
   \begin{eqnarray*}
    N(T_2,y_0) &\leq& \max \Big\{(1-\lambda)\big(N(T_1,y_0)+\delta\big),C_1 \lambda T_1^{1/p_1}\big(N(T_1,y_0)+\delta\big)\Big\} \nonumber\\
    &=& N(T_1,y_0)-\delta<N(T_1,y_0).
   \end{eqnarray*}
    So   $N(\cdot,y_0)$ is strictly decreasing  over $\big(T^0(y_0),T^1(y_0)\big)$.

 \noindent \textit{Step 2. To show that
 \begin{eqnarray}\label{1210-lowconti-0}
  N(T,y_0) \leq \liminf_{t\in\mathcal A,\,t\rightarrow T} N(T,y_0)
  \;\;\mbox{for all}\;\;   T\in \mathcal A\triangleq \big[T^0(y_0),T^1(y_0)\big)
 \end{eqnarray}}
  Arbitrarily fix a $T_0\in \big[T^0(y_0),T^1(y_0)\big)$. Then arbitrarily
   take a sequence:
   \begin{eqnarray}\label{1210-lowconti-01}
   \{T_n\}_{n=1}^\infty \subset \big(T^0(y_0),T^1(y_0)\big),
   \;\;\mbox{with}\;\; \lim_{n\rightarrow\infty} T_n= T_0.
   \end{eqnarray}
   To show (\ref{1210-lowconti-0}), it suffices to prove that
   \begin{eqnarray}\label{1210-lowconti-1}
    N(T_0,y_0)  \leq  \liminf_{n\rightarrow\infty} N(T_n,y_0).
   \end{eqnarray}
   By contradiction, we suppose that
  $\liminf_{n\rightarrow\infty} N(T_n,y_0)  <  N(T_0,y_0)$.
       Then there would be a subsequence $\{T_{n_k}\}_{k=1}^\infty$ of $\{T_n\}_{n=1}^\infty$ so that
   \begin{eqnarray}\label{1210-lowconti-2}
    \lim_{k\rightarrow\infty} N(T_{n_k},y_0) = \liminf_{n\rightarrow\infty} N(T_n,y_0)  <  N(T_0,y_0).
   \end{eqnarray}
   Thus there is a positive constant $C$ so that
   \begin{eqnarray}\label{1210-lowconti-3}
    N(T_{n_k},y_0)<C<\infty\;\;\mbox{ for all}\;\; k\geq 1.
   \end{eqnarray}
    It is clear that $0<T_{n_k}<\infty$ (see (\ref{1210-lowconti-01})) for each $k\in\mathbb N^+$.    This, along with
    (\ref{NP-0}) and (\ref{1210-lowconti-3}), yields that for each $k\in\mathbb N^+$,  there is a control $u_{n_k}\in L^\infty(0,T_{n_k};U)$ so that
 \begin{eqnarray}\label{NP-cont-de-3}
  \hat{y}(T_{n_k};y_0,u_{n_k})=0\;\;\mbox{and}\;\;\|u_{n_k}\|_{L^\infty(0,T_{n_k};U)}
  <N(T_{n_k},y_0)+1/k.
 \end{eqnarray}
For each $k\in\mathbb N^+$, we let  $\widetilde u_{n_k}$ be the zero extension of $u_{n_k}$ over $\mathbb R^+$.  From (\ref{NP-cont-de-3}) and (\ref{1210-lowconti-3}), it follows  that  $\{\widetilde u_{n_k}\}_{k=1}^\infty$ is bounded in $L^\infty(\mathbb R^+;U)$. Then there is a subsequence $\{\widetilde u_{n_{k_l}}\}_{l=1}^\infty$ of $\{\widetilde u_{n_k}\}_{k=1}^\infty$ and a control $v_0\in L^\infty(\mathbb R^+;U)$ so that
 \begin{eqnarray}\label{wgs3.28}
  \widetilde u_{n_{k_l}} \rightarrow v_0 \;\;\mbox{weakly star in}\;\; L^\infty(\mathbb R^+;U), \;\;\mbox{as}\;\; l\rightarrow\infty,
 \end{eqnarray}
 which implies that
 \begin{eqnarray*}
  \widetilde u_{n_{k_l}} \rightarrow v_0 \;\;\mbox{weakly in}\;\; L^2(\mathbb R^+;U), \;\;\mbox{as}\;\; l\rightarrow\infty.
 \end{eqnarray*}
 Because $\lim_{l\rightarrow\infty} T_{n_{k_l}}=T_0$, the above convergence, together with Lemma \ref{Lemma-t-u-converge}, yields that
 \begin{eqnarray*}
  y(T_{n_{k_l}};y_0,\widetilde u_{n_{k_l}}) \rightarrow y(T_0;y_0,v_0) \;\;\mbox{weakly in}\;\; X, \;\;\mbox{as}\;\; l\rightarrow\infty,
 \end{eqnarray*}
 which, along with the first equality in (\ref{NP-cont-de-3}), implies that
 \begin{eqnarray}\label{1210-lowconti-4}
   y(T_0;y_0,v_0)=0.
 \end{eqnarray}
 Since $y_0\in X\setminus\{0\}$ and $T_0<T^1(y_0)$, the equality (\ref{1210-lowconti-4}) indicates that
 $  0<T_0<\infty$.
    Therefore, the problem $(NP)^{T_0,y_0}$ makes sense. From (\ref{1210-lowconti-4}), we know that
   $v_0|_{(0,T_0)}$ is an admissible control to $(NP)^{T_0,y_0}$. This, along with (\ref{NP-0}), (\ref{wgs3.28})  and the second inequality in (\ref{NP-cont-de-3}),
   yields  that
 \begin{eqnarray*}
  N(T_0,y_0) \leq  \|v_0\|_{L^\infty(0,T_0;U)}
  \leq \liminf_{l\rightarrow\infty}    \|\widetilde u_{n_{k_l}}\|_{L^\infty(\mathbb R^+;U)}
  \leq \liminf_{l\rightarrow\infty} N(T_{n_{k_l}},y_0),
 \end{eqnarray*}
 which contradicts (\ref{1210-lowconti-2}). Thus,  (\ref{1210-lowconti-1}) is true. This ends the proof of
  (\ref{1210-lowconti-0}).

\vskip 3pt
 \noindent \textit{Step 3. To show that
 \begin{eqnarray}\label{1210-supconti-0}
  N(T,y_0) \geq \limsup_{t\in\mathcal B,\,t\rightarrow T} N(T,y_0)
  \;\;\mbox{for all}\;\;   T\in \mathcal B \triangleq \big(T^0(y_0),T^1(y_0)\big]
 \end{eqnarray}}
   Arbitrarily fix a $T_0\in \big(T^0(y_0),T^1(y_0)\big]$. We aim to show that (\ref{1210-supconti-0}) holds for $T=T_0$. There are only two possibilities on $T_0$: either $T_0=\infty$ or $T_0<\infty$.
   In the case that $T_0=\infty$, (\ref{1210-supconti-0}),   with $T=T_0$, follows directly  from the first equality in (\ref{N-infty-y0}).

  The key of  this step is to prove that
    \begin{eqnarray}\label{1210-supconti-1}
  N(T_0,y_0) \geq \limsup_{t\rightarrow T_0} N(t,y_0), \;\;\mbox{when}\;\;T_0<\infty.
 \end{eqnarray}
  To this end, we arbitrarily take
  $\{T_n\}_{n=1}^\infty$ in $\big(T^0(y_0),T^1(y_0)\big)$
  so that $ \lim_{n\rightarrow\infty} T_n = T_0<\infty$.
    According to Corollary~\ref{Co-NP-duality}, there is a sequence $\{z_n\}_{n=1}^\infty\subset D(A^*)$ so that for each $n\geq 1$,
 \begin{eqnarray}\label{NP-cont-de-5}
   \|B^*S^*(T_n-\cdot)z_n\|_{L^1(0,T_n;U)}=1
 \end{eqnarray}
 and
 \begin{eqnarray}\label{NP-cont-de-6}
  N(T_n,y_0)-1/n\leq \langle S(T_n)y_0,z_n\rangle_X\leq N(T_n,y_0).
 \end{eqnarray}
 Arbitrarily fix a sequence:
 \begin{eqnarray}\label{1207-sequence-tk}
 \{t_k\}_{k=1}^\infty\subset \big(T^0(y_0),T_0\big)\;\;\mbox{with}\;\;t_k\nearrow T_0.
 \end{eqnarray}
  The rest of the proof of this step is divided into three parts as follows:

 \noindent \textit{Part 3.1. To prove that  there is  a subsequence  $\{n_l\}_{l=1}^\infty$ in $\mathbb N^+$ and a function $g\in B_{Y_{T_0}}$ so that for each $ k\in\mathbb N^+$,
 \begin{eqnarray}\label{WGS3.32}
  B^*S^*(T_{n_l}-\cdot)z_{n_l} \rightarrow g\; \;\mbox{weakly in }\;  L^1(0,t_k;U),\;\;\mbox{as}\;\;l\rightarrow\infty
 \end{eqnarray}}
   For each $n$, we define a function $\psi_{n}$ over $(0,T_0)$ in the following manner:
   \begin{eqnarray*}
  \psi_n(t)=\left\{\begin{array}{ll}
              B^*S^*(T_n-t)z_n, \;\; &t\in \big(0,\min\{T_n,T_0\}\big),\\
              0, \;\; &t\in \big[\min\{T_n,T_0\},T_0\big).
             \end{array}
      \right.
 \end{eqnarray*}
      For each $k\in\mathbb N^+$,  since $t_k<T_0$ (see (\ref{1207-sequence-tk})) and $\lim_{n\rightarrow\infty} T_n = T_0$, we see that
    there is $N(k)\in\mathbb N^+$ so that
   $t_k <  \min\{T_n,T_0\}$, when $n\geq N(k)$.
       Since $z_n\in D(A^*)$ for all $n$, we have that for each $k\in\mathbb N^+$,
    $S^*(T_n-t_k)z_n\in D(A^*)$, when $n\geq N(k)$.
        Then by (\ref{ob-space}),  we find that when $k\in\mathbb N^+$ and $n\geq N(k)$,
 \begin{eqnarray}
  \psi_n\mid_{(0,t_k)}=B^*S^*(T_n-\cdot)z_n\mid_{(0,t_k)}
  =B^*S^*(t_k-\cdot)\big(S^*(T_n-t_k)z_n\big)\mid_{(0,t_k)}
  \in Y_{t_k}.
 \end{eqnarray}
 This, along with  (\ref{NP-cont-de-5}), yields that for each $k\in\mathbb N^+$,
 $\psi_n\mid_{(0,t_k)}\in B_{Y_{t_k}}$, when $n\geq N(k)$.
  From this, (H1) and   Corollary \ref{Corollary-weakly-star-compact} (with $T=t_k$), we see that  for each $k\in\mathbb N^+$,  there is a function  $g_k\in B_{Y_{t_k}}$ and a subsequence $\{\psi_{k_n}\}_{n=1}^\infty$   so that
  \begin{eqnarray*}
  \{\psi_{k_n}\}_{n=1}^\infty\subset \{\psi_{(k-1)_n}\}_{n=1}^\infty\subset \{\psi_n\}_{n=1}^\infty,\;\;
  \mbox{with}\;\; \{\psi_{0_n}\}_{n=1}^\infty\triangleq \{\psi_n\}_{n=1}^\infty,
  \end{eqnarray*}
  and so that
 \begin{eqnarray*}
  \psi_{k_n}\mid_{(0,t_k)} \rightarrow g_k \mbox{ in the topology } \sigma(Y_{t_k},\mathcal R^0_{t_k}), \mbox{ as } n\rightarrow\infty.
 \end{eqnarray*}
 From these and the diagonal law, the subsequence $\{\psi_{n_n}\}_{n=1}^\infty$ of $\{\psi_n\}_{n=1}^\infty$ satisfies that for each $k\in \mathbb N^+$,
 \begin{eqnarray}\label{NP-cont-de-100}
  \psi_{n_n}\mid_{(0,t_k)} \rightarrow g_k \mbox{ in the topology } \sigma(Y_{t_k},\mathcal R^0_{t_k}),\;\;\mbox{as}\;\; n\rightarrow\infty.
 \end{eqnarray}
 Arbitrarily fix a $k\in\mathbb N^+$ and then arbitrarily take $u_k\in \mathcal{U}_k$ where
 \begin{eqnarray}\label{1207-set-Vk}
  \mathcal{U}_k\triangleq \Big\{u\in L^\infty(0,t_k;U) ~:~ \lim_{s\rightarrow t_k}\|u\|_{L^\infty(s,t_k;U)}=0\Big\}.
 \end{eqnarray}
  By (\ref{R0T}), we have that
  $\hat y(t_k;0,u_k)\in\mathcal R_{t_k}^0$.
      This, along with (\ref{NP-cont-de-100}), yields that for each $k\in\mathbb N^+$,
 \begin{eqnarray}\label{NP-cont-de-100-1}
  \langle \psi_{n_n},\hat y(t_k;0,u_k)\rangle_{Y_{t_k},\mathcal R_{t_k}^0} \rightarrow \langle g_k,\hat y(t_k;0,u_k)\rangle_{Y_{t_k},\mathcal R_{t_k}^0}, \mbox{ as } n\rightarrow\infty.
 \end{eqnarray}
 Since $u_k$ is an admissible control to $(NP)^{y_{t_k}}$, with $y_{t_k}\triangleq\hat y(t_k;0,u_k)$, we can use Theorem \ref{Theorem-the-second-representation-theorem} and (\ref{NP-cont-de-100-1}) to get that for each $k$ and each $u_k\in \mathcal{U}_k$,
 \begin{eqnarray}\label{wgs3.34}
 \int_0^{t_k} \langle \psi_{n_n}(t),u_k(t)\rangle_U \,\mathrm dt
 \rightarrow \int_0^{t_k} \langle g_k(t),u_k(t)\rangle_U \,\mathrm dt, \mbox{ as } n\rightarrow \infty.
 \end{eqnarray}
  We next claim  that
 \begin{eqnarray}\label{NP-cont-de-101}
  g_j = g_{j^\prime}
  \;\;\mbox{over}\;\; [0,t_j]  \;\;\mbox{for all}\;\;  j,\, j^\prime\in\mathbb N^+
  \;\;\mbox{with}\;\;  j< j^\prime.
 \end{eqnarray}
 For this purpose, we arbitrarily fix  $j, j^\prime\in\mathbb N^+$ so that $ j< j^\prime$.
 Let $u_j\in \mathcal{U}_j$. Write $\widetilde u_j$ for the zero extension of $u_j$ over $(0,t_{j'})$.
  It follows from (\ref{1207-set-Vk}) that
  $  \widetilde u_j\in \mathcal{U}_{j^\prime}$.
   This, along with  (\ref{wgs3.34}), indicates that
  \begin{eqnarray*}
   \int_0^{t_j} \langle g_j(t),u_j(t)\rangle_U \,\mathrm dt
   &=&\lim_{n\rightarrow\infty} \int_0^{t_j} \langle \psi_{n_n}(t),u_j(t)\rangle_U \,\mathrm dt
   =\lim_{n\rightarrow\infty} \int_0^{t_{j^\prime}} \langle \psi_{n_n}(t),\widetilde u_j(t)\rangle_U
      \,\mathrm dt
   \nonumber\\
   &=&\int_0^{t_{j^\prime}} \langle g_{j^\prime}(t),\widetilde u_j(t)\rangle_U \,\mathrm dt
   =\int_0^{t_{j}} \langle g_{j^\prime}(t),u_j(t)\rangle_U  \,\mathrm dt.
  \end{eqnarray*}
  Since $u_j$ was arbitrarily taken from $\mathcal{U}_j$ (see (\ref{1207-set-Vk})), the above leads to (\ref{NP-cont-de-101}).

Now, define $g(\cdot): (0,T_0)\rightarrow U$ by
 \begin{eqnarray}\label{NP-cont-de-102}
  g(t) \triangleq g_k(t),~t\in (0,t_k],  \;\;\mbox{for each}\;\; k\in\mathbb N^+.
 \end{eqnarray}
 From (\ref{NP-cont-de-101}), we see that $g$ is well defined.
 By (\ref{wgs3.34}) and (\ref{NP-cont-de-102}), we find that for each $k\in\mathbb N^+$ and each $u_{k+1}\in \mathcal{U}_{k+1}$,
 \begin{eqnarray}\label{wanghuang3.30}
 \int_0^{t_{k+1}} \langle \psi_{n_n}(t),u_{k+1}(t)\rangle_U \,\mathrm dt
  \rightarrow \int_0^{t_{k+1}} \langle g(t),u_{k+1}(t)\rangle_U \,\mathrm dt,
   \mbox{ as } n\rightarrow \infty.
 \end{eqnarray}
Given a $v_k\in L^\infty(0,t_k;U)$, let $\widetilde v_k$ be the zero extension of $v_k$ over $(0,t_{k+1})$. Then
$\widetilde v_k\in \mathcal{U}_{k+1}$.
Replacing $u_{k+1}$ by  $\widetilde v_k$ in (\ref{wanghuang3.30}), we obtain that
for each $k\in\mathbb N^+$ and  each $v_k\in L^\infty(0,t_k;U)$,
\begin{eqnarray*}
 \int_0^{t_{k}} \langle \psi_{n_n}(t),v_{k}(t)\rangle_U \,\mathrm dt
  \rightarrow \int_0^{t_{k}} \langle g(t),v_{k}(t)\rangle_U \,\mathrm dt,
  \mbox{ as } n\rightarrow \infty,
 \end{eqnarray*}
from which, it follows that for each $k\in\mathbb N^+$,
\begin{equation}\label{wanghuang3.31}
\psi_{n_n}\rightarrow g\;\;\mbox{weakly in}\;\; L^1(0, t_{k}; U), \mbox{ as } n\rightarrow\infty\color{black}.
\end{equation}
We now prove that  $g\in B_{Y_{T_0}}$. Indeed, since $g_k\in B_{Y_{t_k}}$ for each $k\in\mathbb N^+$, by (\ref{NP-cont-de-102}) and (i) of Lemma \ref{Lemma-left-continuity-YT}, we deduce that
 $ g|_{(0,s)}\in Y_s $ for all $ s\in (0,T_0)$ and that $\|g\|_{ L^1(0,T_0;U)}\leq 1$.
From these, as well as (H1) and (ii) of Lemma \ref{Lemma-left-continuity-YT}, we see that $g\in B_{Y_{T_0}}$. This, together with (\ref{wanghuang3.31}), leads to   the conclusion of Part 3.1.

\vskip 5pt
 \noindent  \textit{Part 3.2. To show that the subsequence $\{n_l\}_{l=1}^\infty$, obtained in Part 3.1, satisfies that
 \begin{equation}\label{WgS3.38}
 \langle S(T_{n_l})y_0,z_{n_l}\rangle_X\rightarrow \langle S(T_0)y_0,g\rangle_{\mathcal R_{T_0},Y_{T_0}},\;\;\mbox{as}\;\;l\rightarrow\infty
 \end{equation}}

  \noindent Recall (\ref{1207-sequence-tk}) for $\{t_k\}_{k=1}^\infty$.
      Since $t_1>T^0(y_0)$, we see from (\ref{y0-controllable}) that  there is an $u_1\in L^\infty (\mathbb{R}^+;U)$ so that
$ 0={y}(t_1;y_0,\chi_{(0,t_1)}u_1)$,
from which, it follows from (\ref{Changshubianyi1.6}) that for each $T\geq t_1$,
 \begin{eqnarray}\label{NP-cont-de-104}
 0 =  \hat{y}(T;y_0, \chi_{(0,t_1)}u_1|_{(0,T)})
 =S(T)y_0+\int_0^{T} S_{-1}(T-\tau)B \chi_{(0,t_1)}(\tau)u_1(\tau) \,\mathrm d\tau.
 \end{eqnarray}
 Because
 $ \lim_{l\rightarrow\infty} T_{n_l}=T_0>t_1$,
  there exists an $N_0>0$ so that
   $T_{n_l}\geq t_1$ for all $l\geq N_0$.
       This, along with (\ref{NP-cont-de-104}) (with $T=T_{n_l}$) and (\ref{NNNWWW2.1}), yields   that
  for each $l\geq N_0$,
 $$
   \langle S(T_{n_l})y_0,z_{n_l}\rangle_X
      = -\int_0^{t_1} \big\langle \chi_{(0,t_1)}(\tau)u_1(\tau),B^*S^*(T_{n_l}-\tau)z_{n_l} \big\rangle_U \,\mathrm d\tau,
 $$
 which, together with (\ref{WGS3.32}) (where $k=2$),  implies that
 \begin{eqnarray}\label{zero-t1-1}
   \lim_{l\rightarrow\infty} \langle S(T_{n_l})y_0,z_{n_l}\rangle_X
   =  -\int_0^{t_1}  \big\langle \chi_{(0,t_1)}(\tau)u_1(\tau),g(\tau) \big\rangle_U  \,\mathrm d\tau.
 \end{eqnarray}
 Meanwhile, since $T_0>t_1$, it follows by (\ref{attainable-space}) and (\ref{NP-cont-de-104}) (where $T=T_0$) that
 \begin{eqnarray}\label{T0y0-RT0}
 S(T_0)y_0\in\mathcal R_{T_0}.
 \end{eqnarray}
 By (\ref{NP-cont-de-104}), we know that $-\chi_{(0,t_1)}u_1|_{(0,T_0)}$ is an admissible control to $(NP)^{y_{T_0}}$, with $y_{T_0}\triangleq S(T_0)y_0$. Thus, it follows from (\ref{T0y0-RT0}) and (\ref{ob-attain-0}) that
 \begin{eqnarray*}
 \langle S(T_0)y_0,g\rangle_{\mathcal R_{T_0},Y_{T_0}}
    = -\int_0^{t_1} \big\langle \chi_{(0,t_1)}(\tau)u_1(\tau),g(\tau) \big\rangle_U \,\mathrm d\tau .
 \end{eqnarray*}
 This, along with (\ref{zero-t1-1}), yields (\ref{WgS3.38}).

 \vskip 5pt

  \noindent\textit{Part 3.3. To show  (\ref{1210-supconti-1}) }

  It is clear that $T^0(y_0)<T_0<\infty$. Then by (\ref{T0y0-RT0}) and (ii) of Proposition \ref{Lemma-NP-yT-y0-eq}, we see that
  \begin{eqnarray}\label{1209-eq-1}
   N(T_0,y_0)=\|-S(T_0)y_0\|_{\mathcal R_{T_0}}.
  \end{eqnarray}
 From (\ref{T0y0-RT0}) and (\ref{ob-attain-0}), we find that
  \begin{eqnarray*}
    \|-S(T_0)y_0\|_{\mathcal R_{T_0}}\|g\|_{Y_{T_0}}\geq \langle S(T_0)y_0,g\rangle_{\mathcal R_{T_0},Y_{T_0}}.
  \end{eqnarray*}
 This, along with (\ref{1209-eq-1}), implies that
 \begin{equation}\label{part3.3-1}
    N(T_0,y_0)\|g\|_{Y_{T_0}}\geq \langle S(T_0)y_0,g\rangle_{\mathcal R_{T_0},Y_{T_0}}.
 \end{equation}
 Since $g\in B_{Y_{T_0}}$  (see Part 3.1), we have that $\|g\|_{Y_{T_0}}\leq 1$. This, as well as (\ref{part3.3-1}) and (\ref{WgS3.38}), yields that
 \begin{eqnarray}\label{wanghuang3.34}
   N(T_0,y_0)\geq N(T_0,y_0)\|g\|_{Y_{T_0}} \geq \lim_{l\rightarrow\infty} \langle S(T_{n_l})y_0,z_{n_l}\rangle_X.
 \end{eqnarray}
 From (\ref{wanghuang3.34}) and (\ref{NP-cont-de-6}), we obtain that
$ N(T_0,y_0)\geq \lim_{l\rightarrow\infty} N(T_{n_l},y_0)$.
   Since the function $N(\cdot,y_0)$ is decreasing (see (ii) of  Lemma \ref{Lemma-NP-decreasing}), the above leads to  (\ref{1210-supconti-1}) (in the case that $T_0<\infty$).
  \vskip 3pt
  In summary, we conclude that (\ref{1210-supconti-0}) holds. This ends the proof of Step 3.

  \vskip 5pt
 Now, from Lemma \ref{Lemma-NP-decreasing} and the conclusions in Step 2 and Step 3, we see that the function $N(\cdot,y_0)$ is continuous from $\big(T^0(y_0),T^1(y_0)\big)$ onto $\big(N(T^1(y_0),y_0),N(T^0(y_0),y_0)\big)$. This, along with the conclusion in Step 1, proves the conclusion (i)
 of  Proposition~\ref{wang-prop3.3}.

 \vskip 5pt
 (ii) Fix a $y_0\in X\setminus\{0\}$ so that $T^0(y_0)<T^1(y_0)$. Let
  $ T\in \big(T^0(y_0),T^1(y_0)\big)$ and $0\leq s_1<T<s_2\leq\infty$.
   Choose two numbers $s_1^\prime$ and $s_2^\prime$ so that
  \begin{equation}\label{WANG3.53}
  s_1^\prime,\, s_2^\prime\in \big(T^0(y_0),T^1(y_0)\big)\;\;\mbox{and}\;\;
  s_1<s_1^\prime<T<s_2^\prime<s_2.
  \end{equation}
     Because $N(\cdot,y_0)$ is strictly decreasing over $\big(T^0(y_0),T^1(y_0)\big)$ (see the conclusion (i) in  this proposition), it follows from (\ref{WANG3.53}) that
  \begin{eqnarray}\label{WANG3.54}
    N(s_1^\prime,y_0)>N(T,y_0)>N(s_2^\prime,y_0).
  \end{eqnarray}
  Since $N(\cdot,y_0)$ is decreasing over $[0,\infty]$ (see (ii) of Lemma \ref{Lemma-NP-decreasing}), it follows by (\ref{WANG3.53}) and (\ref{WANG3.54}) that
  \begin{eqnarray*}
   N(s_1,y_0)\geq N(s_1^\prime,y_0)>N(T,y_0)>N(s_2^\prime,y_0)\geq N(s_2,y_0),
  \end{eqnarray*}
  which leads to  (\ref{NP-strict-decreasing-infty}). The conclusion (ii) is proved.

  \vskip 5pt

  In summary, we finish the proof of Proposition~\ref{wang-prop3.3}.

 \end{proof}

\begin{Corollary}\label{wangcorollary3.8}
  Suppose that (H1) holds.  Let $y_0\in X\setminus\{0\}$ satisfy that $T^0(y_0)<T^1(y_0)$. Then the following conclusions are valid:

 \noindent (i) When $M\in \big(N(T^1(y_0), y_0), N(T^0(y_0),y_0)\big)$,
   \begin{equation}\label{ggssww3.40}
  T^0(y_0)<T(M,y_0)<T^1(y_0)\;\;\mbox{and}\;\; M=N(T(M,y_0),y_0).
  \end{equation}

  \noindent (ii) When $T\in \big(T^0(y_0), T^1(y_0)\big)$,
   \begin{equation}\label{ggssww3.41}
  N(T^1(y_0),y_0)<N(T,y_0)<N(T^0(y_0),y_0)\;\;\mbox{and}\;\; T=T(N(T,y_0),y_0). \end{equation}
 \end{Corollary}
 \begin{proof}
  (i) Let $y_0\in X\setminus\{0\}$, with $T^0(y_0)<T^1(y_0)$.
 Then by (H1), we can apply (i) of Proposition~\ref{wang-prop3.3} to see that
 $ N(T^1(y_0),y_0)< N(T^0(y_0),y_0)$.
 Let
 \begin{equation}\label{Wanghuang3.41}
  M\in\big(N(T^1(y_0),y_0), N(T^0(y_0),y_0)\big).
  \end{equation}
    According to (i) of Proposition~\ref{wang-prop3.3},  there is  $\hat T$ so that
  \begin{eqnarray}\label{TP-NP-1-1}
   T^0(y_0)<\hat T<T^1(y_0)\;\;\mbox{and}\;\; M=N(\hat T,y_0).
  \end{eqnarray}

  To prove  (\ref{ggssww3.40}), it suffices to show   that
  \begin{eqnarray}\label{TP-NP-1-2}
   \hat T=T(M,y_0).
  \end{eqnarray}
  By contradiction, suppose that (\ref{TP-NP-1-2}) were not true. Then we would have that either $\hat T<T(M,y_0)$ or $\hat T>T(M,y_0)$.
    In the case that $\hat T<T(M,y_0)$, we first observe from (\ref{TP-NP-1-1}) and (\ref{Wanghuang3.41}) that
  $N(\hat T,y_0)=M<N(T^0(y_0),y_0)\leq\infty$.
   Thus, it follows from (\ref{NP-0}) that for each $n\geq 1$, there is a control $v_n$ so that
  \begin{equation}\label{WAng3.44}
  \|v_n\|_{L^\infty(0,\hat T;U)} \leq   N(\hat T,y_0) + 1/n<\infty
  \end{equation}
  and
  \begin{equation}\label{WAng3.45}
  \hat y(\hat T;y_0,v_n)=0.
  \end{equation}
  Write $\widetilde v_n$ for the zero extension of $v_n$ over $\mathbb R^+$, $n\in\mathbb N^+$. From (\ref{WAng3.44}), we see that on a subsequence of $\{\widetilde v_n\}_{n=1}^\infty$, still denoted in the same manner,
    \begin{equation}\label{WAng3.46}
   \widetilde  v_n\rightarrow v_0\;\;\mbox{weakly star in}\;\; L^\infty(\mathbb R^+;U),
   \;\;\mbox{as}\;\;  n\rightarrow\infty.
  \end{equation}
  It is clear that $\widetilde v_n$ converges to $v_0$ weakly in $L^2(\mathbb R^+;U)$. Then by  Lemma \ref{Lemma-t-u-converge} and (\ref{WAng3.45}), we find that
  \begin{equation}\label{WAng3.47}
  y(\hat T; y_0,v_0)=0.
  \end{equation}
  Meanwhile, from (\ref{WAng3.46}), (\ref{WAng3.44}) and (\ref{TP-NP-1-1}),
  we have that
  \begin{equation}\label{WAng3.48}
  \|v_0\|_{L^\infty(\mathbb R^+;U)}
  \leq \liminf_{n\rightarrow\infty}  \|\widetilde v_n\|_{L^\infty(\mathbb R^+;U)}
  \leq N(\hat T, y_0)
  =M.
  \end{equation}
   From (\ref{WAng3.47}) and (\ref{WAng3.48}), we see that $v_0$ is an admissible control to $(TP)^{M,y_0}$.
  Then by (\ref{TP-0}), we see that
  $\hat T\geq  T(M,y_0)$, which leads to a contradiction, since we are in the case that $\hat T<T(M,y_0)$.

  In the case when $\hat T>T(M,y_0)$, we have that $T(M,y_0)<\infty$. This, along with (\ref{TP-0}), yields that for each $n\geq 1$, there is a control $u_n\in \mathcal{U}^M$
  and a number $T_n$ so that
  \begin{equation}\label{WGS3.47}
  T(M,y_0) \leq T_n  \leq T(M,y_0)+1/n<\infty;
  \end{equation}
  \begin{equation}\label{WGS3.48}
  \|u_n\|_{L^\infty(\mathbb{R}^+;U)}\leq M\;\;\mbox{and}\;\; y(T_n;y_0,u_n)=0.
  \end{equation}
  Since $y_0\in X\setminus\{0\}$, these imply that
 $ 0<T_n<\infty$  for all $n\geq 1$.
    From this and the second equality in (\ref{WGS3.48}), it follows that for each $n$, $u_n|_{(0,T_n)}$ is an admissible control to $(NP)^{T_n,y_0}$. This, along with the first inequality in (\ref{WGS3.48}) and the definition of $N(T_n,y_0)$ (see (\ref{NP-0})), yields that for each $n$,
  $M\geq \|u_n\|_{L^\infty(\mathbb{R}^+;U)}\geq N(T_n,y_0)$,
    which, together with the second equality in (\ref{TP-NP-1-1}), implies that
  \begin{eqnarray}\label{N-Tn-M-0}
   N(\hat T,y_0)\geq  N(T_n,y_0)  \;\;\mbox{for each}\;\; n.
  \end{eqnarray}
  Since (H1) holds and $\hat T\in \big(T^0(y_0),T^1(y_0)\big)$, we see from  (\ref{NP-strict-decreasing-infty}) and (\ref{N-Tn-M-0})  that for each $n\in \mathbb{N}^+$,
  $T_n \geq \hat T$
    which, together with (\ref{WGS3.47}), indicates that $T(M,y_0)\geq\hat T$. This leads to a contradiction, because we are in the case   that $\hat T>T(M,y_0)$.
  Thus, the conclusion (i) of this corollary is true.

  (ii) Let $y_0\in X\setminus\{0\}$, with $T^0(y_0)<T^1(y_0)$.
  Arbitrarily fix $T\in (T^0(y_0), T^1(y_0)$.
   Since (H1) holds, we
  can use the conclusion (i) of  Proposition~\ref{wang-prop3.3} to see
  the $T$ satisfies the first inequality in (\ref{ggssww3.41}).
  Then by  this and  (\ref{ggssww3.40}) (where $M=N(T,y_0)$), we find that
  \begin{eqnarray}\label{WANG3.55}
   T^0(y_0)<T(N(T,y_0),y_0)< T^1(y_0)\;\;\mbox{and}\;\;N(T, y_0)=N\big(T(N(T,y_0),y_0),y_0\big).
  \end{eqnarray}
  Since $N(\cdot,y_0)$ is strictly decreasing over $\big(T^0(y_0), T^1(y_0)\big)$ (see (i) of Proposition~\ref{wang-prop3.3}) and because $T\in \big(T^0(y_0), T^1(y_0)\big)$\color{black}, it follows from (\ref{WANG3.55}) that $T$ satisfies the second equality in
(\ref{ggssww3.41}).

 In summary, we finish the proof of this corollary.
 \end{proof}

  We can have the  following property on  $T(M,y_0)$, without assuming (H1). (Compare it with the conclusion (i) of Corollary~\ref{wangcorollary3.8}.)

  \begin{Proposition}\label{Proposition-Range-T}
 Let $y_0\in X\setminus\{0\}$. Then
 $T^0(y_0)\leq T(M,y_0)\leq T^1(y_0)$ for each $M\in(0,\infty)$.
  \end{Proposition}

 \begin{proof}

  Let $y_0\in X\setminus\{0\}$ and  $M\in(0,\infty)$.
  We first show that
  \begin{equation}\label{WANGGENG3.69}
  T(M,y_0)\geq T^0(y_0).
  \end{equation}
   By contradiction,  suppose that
   $
   T(M,y_0)<T^0(y_0)$.
    Then by (\ref{TP-0}), there would be
  $  \hat t\in \big[T(M,y_0),T^0(y_0)\big)$ and $u_1\in \mathcal U^M$
     so that
  $  y(\hat t;y_0,u_1)=0$.
    This contradicts the definition of $T^0(y_0)$ (see (\ref{y0-controllable})). So we have proved (\ref{WANGGENG3.69}).

  We next show that
  \begin{equation}\label{WANGGENG3.70}
  T(M,y_0)\leq T^1(y_0).
  \end{equation}
   By contradiction,  suppose that
 $T^1(y_0)<T(M,y_0)$.
   Then by   (ii) of Lemma \ref{Lemma-T0-T1}, we would have that
  $0<T^1(y_0)<\infty$.
   By this and (\ref{NP-0}), we find that the problem $(NP)^{T^1(y_0),y_0}$ makes sense.
   Since $T^1(y_0)<\infty$,  it follows from (v) of Lemma \ref{Lemma-T0-T1}  that
   $   N(T^1(y_0),y_0)=0$.
         From this and (\ref{NP-0}), we see that there exists a control $v_1$ to $(NP)^{T^1(y_0),y_0}$ so that
  \begin{eqnarray}\label{op-TP-infty-3}
   \hat y(T^1(y_0);y_0,v_1)=0
   \;\;\mbox{and}\;\; \|v_1\|_{L^\infty(0,T^1(y_0);U)}<M.
  \end{eqnarray}
  Let  $\widetilde{v}_1$ be the zero extension of $v_1$ over $\mathbb R^+$.
  Then from (\ref{op-TP-infty-3}), it follows that
  \begin{eqnarray}\label{wanggengsheng3.71}
    y(T^1(y_0);y_0,\widetilde{v}_1)=0
    \;\;\mbox{and}\;\;  \|\widetilde{v}_1\|_{L^\infty(\mathbb{R}^+;U)}<M.
  \end{eqnarray}
  From  (\ref{wanggengsheng3.71}), we see that
      $\widetilde{v}_1$ is an admissible control to $(TP)^{M,y_0}$.
      Then, from the first equation in (\ref{wanggengsheng3.71}) and (\ref{TP-0}), we see  that
      $ T(M,y_0)\leq T^1(y_0)$,
       which leads to a contradiction. Hence,  (\ref{WANGGENG3.70}) is true.

       Finally, by (\ref{WANGGENG3.69}) and (\ref{WANGGENG3.70}), we end
        the proof of Proposition~\ref{Proposition-Range-T}.
 \end{proof}

\bigskip
\section{Existence of minimal time and minimal norm controls}

In this section, we present the existence of minimal time and minimal norm controls
for  $(TP)^{M,y_0}$  and $(NP)^{T,y_0}$, and  the non-existence
of admissible controls for $(TP)^{M,y_0}$  and $(NP)^{T,y_0}$ for all possible cases. These properties play import roles in the proofs of Theorem~\ref{Proposition-NTy0-partition} and Theorem~\ref{Proposition-TMy0-partition}.
We also study the existence of  minimal norm controls
for affiliated minimal norm problems
     $(NP)^{y_T}$, with $y_T\in \mathcal{R}_T$ (given by (\ref{attianable-space-norm}) and  (\ref{attainable-space})). Such existence will be used in the studies of a maximum principle for $(NP)^{y_T}$, with $y_0\in \mathcal{R}_T^0$ (given by (\ref{R0T})). The later is the base of the studies of maximum principles, as well as the bang-bang properties for $(TP)^{M,y_0}$  and $(NP)^{T,y_0}$.
               The first theorem in this section concerns with the existence of
  minimal norm controls to the problem $(NP)^{y_T}$.
\begin{Theorem}\label{wangtheorem4.1}
  Let $T\in (0,\infty)$. The following conclusions are true:

  \noindent (i) For each $y_T\in\mathcal R_T$, $(NP)^{y_T}$ has at least one minimal norm  control.

  \noindent (ii) The null control is the unique minimal norm control to $(NP)^{y_T}$, with $y_T=0$ in $\mathcal R_T$.
\end{Theorem}
\begin{proof} Arbitrarily fix a $T\in(0,\infty)$. We are going to show the conclusions (i)-(ii) one by one.

(i) Let $y_T\in\mathcal R_T$ be arbitrarily given.
According to the definitions of the problem $(NP)^{y_T}$ and the subspace $\mathcal R_T$
(see (\ref{attianable-space-norm}) and (\ref{attainable-space})),
 $(NP)^{y_T}$ has at least one admissible control.  Thus
 there is a minimization sequence  $\{v_n\}_{n=1}^\infty\subset L^\infty(0,T;U)$  for $(NP)^{y_T}$ so that
\begin{equation}\label{wanggeng4.1}
\hat y(T;0,v_n)=y_T\;\;\mbox{for all}\;\; n\in \mathbb N^+
\end{equation}
and
\begin{equation}\label{wanggeng4.2}
 \|v_n\|_{L^\infty(0,T; U)}
 \leq \|y_T\|_{\mathcal{R}_T} + 1/n\;\;\mbox{for all}\;\; n\in \mathbb N^+.
\end{equation}
From (\ref{wanggeng4.2}), we find that  there is a subsequence of $\{v_n\}_{n=1}^\infty$, denoted in the same manner, and a control $v_0\in L^\infty(0,T;U)$ so that
\begin{equation}\label{wanggeng4.3}
v_n\rightarrow v_0\;\;\mbox{weakly star in}\;\; L^\infty(0,T;U),\;\;\mbox{as}\;\;n\rightarrow\infty.
\end{equation}
From  (\ref{wanggeng4.3}), Lemma \ref{Lemma-t-u-converge} and (\ref{wanggeng4.1}), we see that
\begin{equation}\label{wanggeng4.4}
\hat y(T;0,v_0)=y_T.
\end{equation}
 This, along with  (\ref{attianable-space-norm}), (\ref{wanggeng4.3}) and
(\ref{wanggeng4.2}), yields  that
\begin{equation*}
\|y_T\|_{\mathcal{R}_T}\leq \|v_0\|_{L^\infty(0,T; U)}
\leq \liminf_{n\rightarrow\infty}
\|v_n\|_{L^\infty(0,T; U)} \leq \|y_T\|_{\mathcal{R}_T},
\end{equation*}
from which, it follows that
\begin{equation}\label{wanggeng4.5}
\|y_T\|_{\mathcal{R}_T}= \|v_0\|_{L^\infty(0,T; U)}.
\end{equation}
By (\ref{wanggeng4.4}) and (\ref{wanggeng4.5}), we find that
$v_0$ is a minimal norm control to $(NP)^{y_T}$. This ends the proof of the conclusion (i).

 (ii) By (\ref{Changshubianyi1.6}), we see that
$\hat y(T;0,0)=0$.
Meanwhile, since $\|\cdot\|_{\mathcal{R}_T}$ is a norm (see (\ref{attianable-space-norm})), we find that
$\|0\|_{\mathcal{R}_T}=0$.
Therefore, we see that when $y_T= 0$, the null control is a minimal norm control to $(NP)^{y_T}$ and that the minimal norm of $(NP)^{y_T}$ is $0$. The latter shows that $(NP)^{y_T}$, with $y_T=0$, has no non-zero minimal norm control. Thus, the null control is the unique minimal norm control to $(NP)^{y_T}$, with $y_T=0$.

In summary, we  complete the proof of this theorem.
\end{proof}

We now present the following lemma which will be used in the studies on the existence
of minimal norm controls to  $(NP)^{T,y_0}$ and minimal time controls   to  $(TP)^{M,y_0}$.

\begin{Lemma}\label{wanglemma4.1}
Let $y_0\in X\setminus\{0\}$, $T\in (0,\infty)$ and $M\in (0,\infty)$.
Then the following conclusions are true:

\noindent (i) If $(NP)^{T,y_0}$ has an admissible control, then it has at least one minimal norm control.

\noindent (ii) If $(TP)^{M,y_0}$ has an admissible control, then it has at least one minimal time control.

\noindent (iii) If $N(T,y_0)<\infty$, then $(NP)^{T,y_0}$ has at least one minimal norm control.

\noindent (iv) If  $N(T,y_0)=0$, then the null control is the unique minimal norm control to $(NP)^{T,y_0}$.

\noindent (v) If $N(T,y_0)=\infty$, then $(NP)^{T,y_0}$ has no any admissible control.

\end{Lemma}
\begin{proof}
(i) Suppose that $(NP)^{T,y_0}$ has an admissible control. Then it has a minimization sequence $\{v_n\}_{n=1}^\infty \subset L^\infty(0,T;U)$ so that
\begin{equation}\label{Wang4.1}
\hat y(T;y_0,v_n)=0\;\;\mbox{for all}\;\; n\in \mathbb{N}^+
\end{equation}
and
\begin{equation}\label{Wang4.2}
  \|v_n\|_{L^\infty(0,T;U)}\leq N(T,y_0)+1/n\;\;\mbox{for all}\;\; n\in \mathbb{N}^+.
\end{equation}
By (\ref{Wang4.2}), we see that there is a subsequence of $\{v_n\}_{n=1}^\infty$,
denoted in the same manner,  and a control $v_0\in L^\infty(0,T;U)$ so that
\begin{equation}\label{Wang4.3}
v_n\rightarrow v_0\;\;\mbox{weakly star in}\;\; L^\infty(0,T;U),
\;\;\mbox{as}\;\;  n\rightarrow\infty.
\end{equation}
From  (\ref{Wang4.3}), Lemma~\ref{Lemma-t-u-converge} and (\ref{Wang4.1}), we find that \begin{equation}\label{Wang4.4}
\hat y(T;y_0,v_0)=0.
\end{equation}
This, together with (\ref{NP-0}), (\ref{Wang4.3}) and (\ref{Wang4.2}), yields that
\begin{equation*}
N(T,y_0)\leq\|v_0\|_{L^\infty(0,T;U)}\leq \liminf_{n\rightarrow\infty}\|v_n\|_{L^\infty(0,T;U)}\leq N(T,y_0).
\end{equation*}
Hence, we have that
\begin{equation}\label{Wang4.5}
N(T,y_0)=\|v_0\|_{L^\infty(0,T;U)}.
\end{equation}
By (\ref{Wang4.4}) and (\ref{Wang4.5}), we find that $v_0$ is a minimal norm control to $(NP)^{T,y_0}$.

(ii) Suppose that $(TP)^{M,y_0}$ has an admissible control. Then there are two sequences $\{u_n\}_{n=1}^\infty \subset L^\infty(\mathbb{R}^+;U)$ and $\{T_n\}_{n=1}^\infty \subset \mathbb{R}^+$ so that
\begin{equation}\label{Wang4.6}
y(T_n;y_0,u_n)=0\;\;\mbox{for all}\;\; n\in \mathbb{N}^+,
\end{equation}
\begin{equation}\label{Wang4.7}
T_n\searrow T(M,y_0),\;\;\mbox{as}\;\; n\rightarrow\infty
\end{equation}
and
\begin{equation}\label{Wang4.8}
\|u_n\|_{L^\infty(\mathbb{R}^+;U)}\leq M\;\;\mbox{for all}\;\; n\in \mathbb{N}^+.
\end{equation}
By (\ref{Wang4.8}), we see that there are a subsequence of  $\{u_n\}_{n=1}^\infty$, still denoted in the same manner, and an $u_0\in L^\infty(\mathbb{R}^+;U)$ so that
\begin{equation}\label{Wang4.9}
u_n\rightarrow u_0\;\;\mbox{weakly star in}\;\; L^\infty(\mathbb{R}^+;U),
\;\;\mbox{as}\;\;   n\rightarrow\infty.
\end{equation}
From  (\ref{Wang4.7}), (\ref{Wang4.9}), Lemma~\ref{Lemma-t-u-converge} and (\ref{Wang4.6}), it follows  that
\begin{equation}\label{Wang4.10}
y\big(T(M,y_0);y_0,u_0\big)=0.
\end{equation}
Meanwhile, it follows from  (\ref{Wang4.9}) and (\ref{Wang4.8})  that
\begin{equation}\label{Wang4.11}
\|u_0\|_{L^\infty(\mathbb{R}^+;U)}\leq \liminf_{n\rightarrow\infty}\|u_n\|_{L^\infty(\mathbb{R}^+;U)}\leq M.
\end{equation}
By (\ref{Wang4.10}) and (\ref{Wang4.11}), we see that $u_0$ is a minimal time control to $(TP)^{M,y_0}$.

(iii) Suppose that $N(T,y_0)<\infty$. Then it follows by (\ref{NP-0}) that
$(NP)^{T,y_0}$ has an admissible control. Thus, by (i) of this lemma, we find that
$(NP)^{T,y_0}$ has at least one minimal norm control.

(iv) Suppose that $N(T,y_0)=0$. On one hand, by the conclusion (iii) in this lemma, we see that $(NP)^{T,y_0}$ has at least one minimal norm control. On the other hand, if $v^*$ is a minimal norm control to $(NP)^{T,y_0}$, then we have that
$$
\|v^*\|_{L^\infty(0,T;U)}=N(T,y_0)=0,
$$
which yields that $v^*=0$. Hence, the null control is the unique minimal norm control to $(NP)^{T,y_0}$.

\noindent (v) Assume that $N(T,y_0)=\infty$. By contradiction, suppose that  $(NP)^{T,y_0}$ had an admissible control $v^*\in L^\infty(0,T;U)$. Then, by (\ref{NP-0}), we would have that
$$
\infty =N(T,y_0)\leq \|v^*\|_{L^\infty(0,T;U)}<\infty.
$$
This leads to a contradiction. Hence, $(NP)^{T,y_0}$ has no any admissible control.

In summary, we finish the proof of this lemma.
\end{proof}

The next theorem concerns with the existence of   minimal norm controls to the problem $(NP)^{T,y_0}$, in the case when   $T^0(y_0)<\infty$.

 \begin{Theorem}\label{Lemma-existence-optimal-control}
 Let  $y_0\in X\setminus\{0\}$ satisfy that $T^0(y_0)<\infty$.
  Then the following conclusions are true:

 \noindent(i)  If $T^0(y_0)<T<\infty$,
  then $(NP)^{T,y_0}$  has at least one minimal norm control.

 \noindent(ii)  If $T^0(y_0)>0$ and $0<T<T^0(y_0)$,
  then
  $(NP)^{T,y_0}$ has no any admissible control.

 \noindent(iii) If $0<N(T^0(y_0),y_0)<\infty$,
  then
 \begin{equation}\label{huangling4.20}
T^0(y_0)>0
 \end{equation}
 and  $(NP)^{T^0(y_0),y_0}$ has at least one minimal norm control.

 \noindent (iv) If $N(T^0(y_0),y_0)=0$,
  then (\ref{huangling4.20}) holds and
   the null control is the unique minimal norm control to the problem $(NP)^{T^0(y_0),y_0}$.

 \noindent(v) If $N(T^0(y_0),y_0)=\infty$ and $T^0(y_0)>0$,
  then
  $(NP)^{T^0(y_0),y_0}$ has no any admissible control.

  \noindent (vi) If
 $T^0(y_0)=0$, then the problem  $(NP)^{T^0(y_0),y_0}$ does not make sense.

 \end{Theorem}

 \begin{proof} Arbitrarily fix a $y_0\in X\setminus\{0\}$ so that
 \begin{equation}\label{huangling4.23}
T^0(y_0)<\infty.
 \end{equation}

     (i) Suppose that
     \begin{equation}\label{huangling4.17}
 T^0(y_0)<T<\infty.
 \end{equation}
               Then by (\ref{y0-controllable}) and (\ref{huangling4.17}),   there are a  $\hat t\in\big(T^0(y_0),T\big)$ and an $\hat u\in L^\infty(0,\hat t;U)$ so that
 \begin{eqnarray}\label{1009-1}
  \hat y(\hat t;y_0,\hat u)=0.
 \end{eqnarray}
 Extend $\hat u$ over $(0,T)$ by setting it to be zero over $[\hat t,T)$. Denote the extension in the same manner. Then we  see from (\ref{1009-1}) that
  $\hat y(T;y_0,\hat u)=0$,
       from which, it follows that   $\hat u$ is an admissible control to $(NP)^{T,y_0}$.
   This, along with  (i) of Lemma~\ref{wanglemma4.1}, yields  that
   $(NP)^{T,y_0}$ has at least one minimal norm control.

 (ii)  Suppose that
 \begin{equation}\label{huangling4.18}
 T^0(y_0)>0\;\;\mbox{and}\;\; 0<T<T^0(y_0).
 \end{equation}
  Then  it follows from (\ref{y0-controllable}) and (\ref{huangling4.18}) that $(NP)^{T,y_0}$ has no any admissible control.

  (iii) Assume that
  \begin{equation}\label{huangling4.19}
0<N(T^0(y_0),y_0)<\infty.
 \end{equation}
     We first show (\ref{huangling4.20}).
 By contradiction, suppose that (\ref{huangling4.20}) were not true. Then we would have that $T^0(y_0)=0$. This, along with (iv) in Lemma~\ref{Lemma-T0-T1}, yields that
 $$
 N(T^0(y_0),y_0)=N(0,y_0)=\infty,
 $$
  which contradicts
(\ref{huangling4.19}). Hence, we have proved (\ref{huangling4.20}).
 Next, it follows from  (\ref{huangling4.20}) and (\ref{huangling4.23}) that
 $0<T^0(y_0)<\infty$.
   This shows that  the problem $(NP)^{T^0(y_0),y_0}$ makes sense.
  (Notice that in the definition of $(NP)^{T,y_0}$, it is required that  $0<T<\infty$, see (\ref{NP-0}).)
  Finally, by (\ref{huangling4.19}), we can apply
    (iii) of Lemma~\ref{wanglemma4.1} to find that $(NP)^{T^0(y_0),y_0}$ has  at least one minimal norm control.

  (iv) Suppose that
  \begin{equation}\label{huangling4.21}
N(T^0(y_0),y_0)=0.
 \end{equation}
     We first show that (\ref{huangling4.20}) stands in this case. By contradiction, suppose that
  (\ref{huangling4.20}) were not true.
  Then we would have that $T^0(y_0)=0$. This, together with
  (iv) in Lemma~\ref{Lemma-T0-T1}, indicates that
 $$
 N(T^0(y_0),y_0)=N(0,y_0)=\infty,
 $$
 which contradicts (\ref{huangling4.21}). So
   (\ref{huangling4.20}) in this case.
Next, by (\ref{huangling4.20}) and (\ref{huangling4.23}), we see that
 $ 0<T^0(y_0)<\infty$.
  Hence, the problem  $(NP)^{T^0(y_0),y_0}$ makes sense.
Finally, by (\ref{huangling4.21}), we can apply
  (iv) of Lemma~\ref{wanglemma4.1} to find that  the null control is the unique minimal norm control to
 $(NP)^{T^0(y_0),y_0}$.

 (v) Suppose that
 \begin{equation}\label{huangling4.22}
N(T^0(y_0),y_0)=\infty\;\;\mbox{and}\;\;T^0(y_0)>0.
 \end{equation}
  Then it follows from the second inequality in (\ref{huangling4.22}) and (\ref{huangling4.23}) that
  $ 0<T^0(y_0)<\infty$.
         Hence, the problem $(NP)^{T^0(y_0),y_0}$ makes sense.
   Finally, by the first equality in   (\ref{huangling4.22}),
   we can apply (v) of Lemma~\ref{wanglemma4.1} to find that
              $(NP)^{T^0(y_0),y_0}$ has no any admissible
control.

 (vi) Suppose that  $T^0(y_0)=0$. Then
the problem  $(NP)^{T^0(y_0),y_0}$ does not make sense, since in the definition of
$(NP)^{T,y_0}$, it is required that $T\in (0,\infty)$ (see (\ref{NP-0})).

 In summary, we finish the proof of Theorem~\ref{Lemma-existence-optimal-control}.
  \end{proof}

 The following theorem concerns with the  existence of minimal time controls to $(TP)^{M,y_0}$ and   minimal norm controls  to $(NP)^{T,y_0}$, in the case that  $T^0(y_0)=\infty$.
 \begin{Theorem}\label{Theorem-ex-T0y0-infty}
  Let $y_0\in X\setminus\{0\}$ satisfy that $T^0(y_0)=\infty$.
  Then the following conclusions are true:

  \noindent (i) For each $M\in(0,\infty)$, $(TP)^{M,y_0}$
  does not have any admissible control.

  \noindent (ii) For each $T\in(0,\infty)$,
 $(NP)^{T,y_0}$ does not have any admissible control.
 \end{Theorem}

 \begin{proof}
 Arbitrarily fix a $y_0\in X\setminus\{0\}$ so that  $T^0(y_0)=\infty$. First of all, since $T^0(y_0)=\infty$, it follows from (\ref{y0-controllable}) that for each $T\in(0,\infty)$,
 \begin{eqnarray}\label{Th3.7-1}
 \hat y(T;y_0,u)\neq 0 \;\;\mbox{for all}\;\;  u\in L^\infty(0,T; U).
 \end{eqnarray}
We next show the conclusions (i)-(ii) one by one.

   (i)    By contradiction, suppose  that for some $\hat M\in(0,\infty)$,  $(TP)^{\hat M,y_0}$ had an admissible control $\hat u$. Then we would have that
  $y(\hat t;y_0,\hat u)=0$ for some $\hat t\in(0,\infty)$,
  which contradicts (\ref{Th3.7-1}). So for each $M\in(0,\infty)$, $(TP)^{M,y_0}$ has no any admissible control.

     (ii)  By contradiction, we suppose that for some $\hat T\in(0,\infty)$, $(NP)^{\hat T,y_0}$ had an admissible control $\hat v$. Then we would have that  $\hat y(\hat T;y_0,\hat v)=0$, which contradicts (\ref{Th3.7-1}). So  for each $T\in(0,\infty)$, $(NP)^{T,y_0}$ has no any admissible control.

  Thus we complete the proof of this theorem.
 \end{proof}

 The following theorem concerns with  the existence of minimal time controls to $(TP)^{M,y_0}$, in the case when $T^0(y_0)<\infty$.

 \begin{Theorem}\label{Theorem-ex-op-TP}
 Let $y_0\in X\setminus\{0\}$ satisfy that
  $T^0(y_0)<\infty$. Then it holds that
  \begin{eqnarray}\label{jiajinting4.26}
  N(T^1(y_0),y_0)<\infty.
 \end{eqnarray}
 Furthermore,
 the following conclusions are true:
 \vskip 2pt
 \noindent (i) If $N(T^1(y_0),y_0)<M<\infty$,
  then
 $(TP)^{M,y_0}$ has at least one minimal time control.
 \vskip 2pt
  \noindent (ii) If $ N(T^1(y_0),y_0)>0$ and  $0<M<N(T^1(y_0),y_0)$,
    then
  $(TP)^{M,y_0}$ has no any admissible control.
 \vskip 2pt
 \noindent (iii) Suppose that (H1) holds. If $M_0\triangleq N(T^1(y_0),y_0)>0$,
     then $(TP)^{M_0,y_0}$ has no any admissible control.

 \end{Theorem}

 \begin{proof}
 Arbitrarily fix a $y_0\in X\setminus\{0\}$ so that
 $T^0(y_0)<\infty$.
    Then (\ref{jiajinting4.26}) follows from (v) of Lemma~\ref{Lemma-NT0-T0-T1}.
 Next, we are going to show conclusions (i)-(iii) one by one.

    (i) Let $M\in (N(T^1(y_0),y_0), \infty)$.
      Then by   (vi) of Lemma \ref{Lemma-T0-T1}, we see that
   \begin{eqnarray}\label{jiajinting4.34}
  \infty > M>N(T^1(y_0),y_0)=N(\infty,y_0).
 \end{eqnarray}
   Since $T^0(y_0)<\infty$, it follows from (\ref{jiajinting4.34})
    and the first equality in (\ref{N-infty-y0}) that there is a number $T_1$ so that
  \begin{eqnarray}\label{jiajinting4.35}
  T^0(y_0)<T_1<\infty
  \;\;\mbox{and}\;\;  N(T_1,y_0)<M<\infty.
 \end{eqnarray}
 By the first conclusion in (\ref{jiajinting4.35}), we can apply
  (i) of Theorem \ref{Lemma-existence-optimal-control} to find that   $(NP)^{T_1,y_0}$ has a minimal norm control $v^*$. Hence we have that
  \begin{eqnarray}\label{jiajinting4.36}
 \hat y(T_1;y_0,v^*)=0\;\;\mbox{and}\;\; \|v^*\|_{L^\infty(0, T_1;U)}=N(T_1,y_0).
 \end{eqnarray}
  Write  $\widetilde v^*$ for the zero extension of $v^*$ over $\mathbb{R}^+$. Then it follows from (\ref{jiajinting4.36}) and (\ref{jiajinting4.35}) that
 $$
 y(T_1; y_0, \widetilde v^*)=0
 \;\;\mbox{and}\;\;
 \| \widetilde v^*\|_{L^\infty(\mathbb{R}^+;U)}=N(T_1,y_0)<M<\infty.
 $$
These imply that $\widetilde v^*$ is an admissible control to
 $(TP)^{M,y_0}$.
 Then by (ii) of Lemma~\ref{wanglemma4.1}, we find that  $(TP)^{M,y_0}$ has at least one
 minimal time control.

   (ii)  Assume that
   \begin{eqnarray}\label{jiajinting4.28}
 N(T^1(y_0),y_0)>0\;\;\mbox{and}\;\; 0<M<N(T^1(y_0),y_0).
 \end{eqnarray}
      We aim to show that
   $(TP)^{M,y_0}$ has no any admissible control. By contradiction, suppose that
   $(TP)^{M,y_0}$ had an admissible control. Then according to
    (ii) of Lemma~\ref{wanglemma4.1},
         $(TP)^{ M,y_0}$
 would have a   minimal time control  $u^*_1$. Hence, it holds that
 \begin{equation}\label{jiajinting4.37}
 T( M,y_0)<\infty,\;\; \|u^*_1\|_{L^\infty(\mathbb{R}^+;U)}\leq  M
 \;\;\mbox{ and}\;\;y(T(M,y_0);y_0,u^*_1)=0.
 \end{equation}
   Since $y_0\in X\setminus\{0\}$, from the third and the first conclusions in (\ref{jiajinting4.37}), we see that
   \begin{eqnarray}\label{1210-timeestimate-1}
    0< T( M,y_0)<\infty.
   \end{eqnarray}
     Write $\hat u^*_1$ for the restriction of $u^*_1$  over $\big(0,T( M,y_0)\big)$. Then it follows from (\ref{jiajinting4.37}) that
  \begin{equation}\label{jiajinting4.38}
  \|\hat u^*_1\|_{L^\infty(0,T(M,y_0); U)} \leq  M
  \end{equation}
  and
  \begin{eqnarray}\label{jiajinting4.39}
    \hat y(T(M,y_0);y_0, \hat u^*_1)=0.
  \end{eqnarray}
  By (\ref{1210-timeestimate-1}), the problem $(NP)^{T(M,y_0),y_0}$ makes sense (see (\ref{NP-0})). Then by  (\ref{jiajinting4.39}), we find that $\hat u^*_1$ is an admissible control     to $(NP)^{T(M,y_0),y_0}$. This, along with  the definition  of $N(T(M,y_0),y_0)$ (see (\ref{NP-0})) and (\ref{jiajinting4.38}), yields that
   \begin{equation*}
   N(T(M,y_0),y_0)\leq \|\hat u^*_1\|_{L^\infty(0,T(M,y_0); U)} \leq  M,
   \end{equation*}
   which, together with the second inequality in (\ref{jiajinting4.28}), indicates that
   \begin{equation}\label{jiajinting4.40}
   N(T(M,y_0),y_0)<N(T^1(y_0), y_0).
   \end{equation}
   From (\ref{jiajinting4.40}),  (ii) of Lemma \ref{Lemma-NP-decreasing}
   and the first inequality in (\ref{jiajinting4.37}), it follows  that
   \begin{equation}\label{jiajinting4.41}
   T^1(y_0)< T(M,y_0)<\infty.
   \end{equation}
  By (\ref{jiajinting4.41}), we can apply   (v) of Lemma \ref{Lemma-T0-T1} to get  that
  $ N( T^1(y_0),y_0)=0$,
   which  contradicts the first inequality in (\ref{jiajinting4.28}).
   Hence, $(TP)^{M,y_0}$ has no any admissible control in this case.

   (iii)  Suppose that (H1) holds. And assume that
   \begin{eqnarray}\label{jiajinting4.29}
 M_0\triangleq N(T^1(y_0),y_0)>0.
 \end{eqnarray}
    Then by (\ref{jiajinting4.29}) and (\ref{jiajinting4.26}), it follows that
   $    0<M_0<\infty$.
         Hence, the problem $(TP)^{M_0,y_0}$
makes sense. (It is required that $0<M_0<\infty$ in the definition of $(TP)^{M_0,y_0}$, see (\ref{TP-0}).)

 We aim to show that $(TP)^{M_0,y_0}$ has no any admissible control.
   By contradiction, suppose that it had an admissible control.
  Then we could apply (ii) of Lemma~\ref{wanglemma4.1} to get a   minimal time control  $u^*_2$ for
     $(TP)^{M_0,y_0}$. Hence,  we have that
 \begin{equation}\label{jiajinting4.42}
 T(M_0,y_0)<\infty,\;\; \|u^*_2\|_{L^\infty(\mathbb{R}^+;U)}\leq M_0\;\;\mbox{and}\;\;
 y(T(M_0,y_0);y_0,u^*_2)=0.
 \end{equation}
 Since $y_0\in X\setminus\{0\}$, from the third and the first assertions in (\ref{jiajinting4.42}), we see that
 \begin{eqnarray}\label{1210-timeestimate-2}
  0<T(M_0,y_0)<\infty.
 \end{eqnarray}
    Write $\hat u^*_2$ for the restriction of $u^*_2$  over $\big(0,T(M_0,y_0)\big)$. Then it follows from (\ref{jiajinting4.42}) that
  \begin{equation}\label{jiajinting4.43}
  \|\hat u^*_2\|_{L^\infty(0,T(M_0,y_0); U)} \leq M_0
  \end{equation}
  and
  \begin{eqnarray}\label{jiajinting4.44}
    \hat y(T(M_0,y_0);y_0, \hat u^*_2)=0.
  \end{eqnarray}
   By (\ref{1210-timeestimate-2}),
 the problem $(NP)^{T(M_0,y_0),y_0}$ makes sense. Then from (\ref{jiajinting4.44}), we find   that $\hat u^*_2$ is an admissible control to  $(NP)^{T(M_0,y_0),y_0}$.
    This, along with  the definition  of $N(T(M_0,y_0),y_0)$ (see (\ref{NP-0})), (\ref{jiajinting4.43}) and (\ref{jiajinting4.29}), yields that
    \begin{equation}\label{jiajinting4.45}
     N(T(M_0,y_0),y_0)\leq \|\hat u^*_2\|_{L^\infty(0,T(M_0,y_0); U)} \leq M_0 = N(T^1(y_0),y_0).
    \end{equation}
    By (\ref{jiajinting4.45}) and
       (vi) of Lemma \ref{Lemma-T0-T1}, we find that
  \begin{eqnarray}\label{jiajinting4.46}
   N(T( M_0,y_0),y_0)\leq  N(T^1(y_0),y_0)= N(\infty,y_0).
  \end{eqnarray}

  Next, we will use (\ref{jiajinting4.46}) to prove that $T^1(y_0)<\infty$.
   When this is proved, we can apply  (v) of Lemma \ref{Lemma-T0-T1}
  to get  that
  $ N(T^1(y_0),y_0)=0$,
            which contradicts  (\ref{jiajinting4.29}). Hence, $(TP)^{M_0,y_0}$
    has no any admissible control in this case.

  The remainder is to show that $T^1(y_0)<\infty$. By  contradiction,  suppose that it were not true. Then we would have that
  $T^1(y_0)=\infty$. Since we are in the case that $T^0(y_0)<\infty$, it holds that
   \begin{equation}\label{jiajinting4.49}
  T^0(y_0)<\infty=T^1(y_0).
  \end{equation}
  By the first inequality in (\ref{jiajinting4.42}) and (\ref{jiajinting4.49}), we can find a number $\widehat T$ so that
  \begin{equation}\label{jiajinting4.50}
  \max\{T^0(y_0), T(M_0,y_0)\}<\widehat T<\infty.
  \end{equation}
  Meanwhile, by (H1) and (\ref{jiajinting4.49}), we can apply
     (i) of Proposition~\ref{wang-prop3.3} to find that
   $N(\cdot,y_0)$  is strictly decreasing over $\big(T^0(y_0),\infty\big)$.
   This, together with (\ref{jiajinting4.50}) and
         the first equality in  (\ref{N-infty-y0}), yields that
    \begin{equation}\label{jiajinting4.51}
      N(\widehat T,y_0)>N(\infty,y_0).
   \end{equation}
    Since $N(\cdot,y_0)$ is decreasing over $[0,\infty]$ (see (ii) of Lemma \ref{Lemma-NP-decreasing}), we find from the first inequality in (\ref{jiajinting4.50}) and (\ref{jiajinting4.51})
     that
  \begin{equation*}
    N(T(M_0,y_0),y_0)\geq N(\widehat T,y_0)>N(\infty,y_0).
  \end{equation*}
  This contradicts (\ref{jiajinting4.46}). Hence,  we have proved
  that $T^1(y_0)<\infty$.
   This ends the proof of the conclusion (iii)  of this theorem.

    In summary, we complete the proof of this theorem.

 \end{proof}

Theorem \ref{Lemma-existence-optimal-control}, Theorem~\ref{Theorem-ex-T0y0-infty} and Theorem~\ref{Theorem-ex-op-TP}
contain results on the existence of minimal time controls and minimal norm controls and the non-existence of admissible controls  of $(TP)^{M,y_0}$ and $(NP)^{T,y_0}$ for all
possible cases. In order to use them in the proof of our BBP decomposition theorems better, we need  several corollaries as follows:

\begin{Corollary}\label{wangpuchou1}
Let $y_0\in X\setminus\{0\}$ satisfy that $T^0(y_0)<T^1(y_0)$. Then the following conclusions are true:

\noindent (i) If $T^0(y_0)>0$ and $0<T<T^0(y_0)$, then
 $(NP)^{T,y_0}$ has no any admissible control.

\noindent (ii) If $T^1(y_0)<\infty$ and $T^1(y_0)\leq T<\infty$,
then   the null control is the unique minimal norm control to   $(NP)^{T,y_0}$.
 \end{Corollary}

 \begin{proof}
  Arbitrarily fix a  $y_0\in X\setminus\{0\}$
 so that $T^0(y_0)<T^1(y_0)$.
 Then, we have that
 \begin{eqnarray}\label{1102-assum-1102}
 T^0(y_0)<\infty\;\;\mbox{and}\;\; T^1(y_0)>0.
 \end{eqnarray}
  We will prove the conclusions (i)-(ii) one by one.

 \vskip 2pt

  (i) Suppose that
  \begin{equation}\label{WGGSS4.41}
 T^0(y_0)>0\;\;\mbox{and}\;\;0<T<T^0(y_0).
 \end{equation}
    Then we see that $T\in (0,\infty)$. Thus, the problem $(NP)^{T,y_0}$ makes sense. Furthermore, since $T^0(y_0)<\infty$ (see (\ref{1102-assum-1102})), by (\ref{WGGSS4.41}),
  we can apply (ii) of Theorem \ref{Lemma-existence-optimal-control} to find that
    $(NP)^{T,y_0}$ has no any admissible control.
\vskip 2pt
    (ii)  Suppose that
    \begin{equation}\label{WGGSS4.42}
 T^1(y_0)<\infty\;\;\mbox{and}\;\;T^1(y_0)\leq T<\infty.
 \end{equation}
          By  (\ref{WGGSS4.42}) and (v) of Lemma \ref{Lemma-T0-T1}, we find that
          \begin{eqnarray}\label{1210-coro1.6-1}
            N(T,y_0)=0.
          \end{eqnarray}
       Meanwhile, from (\ref{WGGSS4.42}) and the second inequality in (\ref{1102-assum-1102}), it follows  that $T\in (0,\infty)$.
       Hence, we find from (iv) of Lemma~\ref{wanglemma4.1} and (\ref{1210-coro1.6-1})  that
    the null control is the unique minimal norm control to $(NP)^{T,y_0}$.
\vskip 2pt

    In summary, we finish   the proof of this corollary.

 \end{proof}

 \begin{Corollary}\label{wangpuchou1-1}
 Let $y_0\in X\setminus\{0\}$ satisfy that $T^0(y_0)=T^1(y_0)$. Then
 it holds that $T^0(y_0)>0$. Furthermore, the following conclusions are true:

 \noindent (i) If $0<T<T^0(y_0)$,
        then $(NP)^{T,y_0}$ has no any admissible control.

      \noindent (ii) If $ T^0(y_0)<\infty$ and $T^0(y_0)\leq T<\infty$,
      then     the null control is the unique minimal norm control to $(NP)^{T,y_0}$.
  \end{Corollary}
\begin{proof}
  Arbitrarily fix  a  $y_0\in X\setminus\{0\}$ so that
 $T^0(y_0)=T^1(y_0)$.
 Then by   (ii) of Lemma \ref{Lemma-T0-T1}, we have that
 \begin{eqnarray}\label{1102-zhang-1}
 T^0(y_0)>0.
 \end{eqnarray}

 Next, we will  show the conclusions (i)-(ii) one by one.

  (i) Suppose that
  \begin{equation}\label{WGGSS4.43}
   0<T<T^0(y_0).
      \end{equation}
    In the case that $T^0(y_0)<\infty$, by (\ref{1102-zhang-1}) and (\ref{WGGSS4.43}), we can
  apply  (ii) of Theorem \ref{Lemma-existence-optimal-control} to find that $(NP)^{T,y_0}$ has no any admissible control in this situation. In the case that  $T^0(y_0)=\infty$,
   we can apply (ii) of Theorem \ref{Theorem-ex-T0y0-infty} to find that $(NP)^{T,y_0}$ has no any admissible control in this situation. Hence,  $(NP)^{T,y_0}$ has no any admissible control.

 (ii) Suppose  that
 \begin{equation}\label{WGGSS4.44}
   T^0(y_0)<\infty\;\;\mbox{and}\;\;T^0(y_0)\leq T<\infty.
      \end{equation}
Since we are in the case that $T^0(y_0)=T^1(y_0)$, it follows from   (\ref{WGGSS4.44}) that
   $T^1(y_0)\leq T <\infty$.
      Then by   (v) of Lemma \ref{Lemma-T0-T1}, we find that
   \begin{equation}\label{GGSS4.47}
   N(T,y_0)=0.
   \end{equation}
   Meanwhile, it follows from (\ref{1102-zhang-1}) and (\ref{WGGSS4.44}) that $0<T<\infty$.
   By this and (\ref{GGSS4.47}), we can apply (iv) of Lemma~\ref{wanglemma4.1} to see that
    the null control is the unique minimal norm control to $(NP)^{T,y_0}$.

In summary, we end the proof of this corollary.
\end{proof}

\begin{Corollary}\label{wangbuchoutheorem3.13}
  Let $y_0\in X\setminus\{0\}$ satisfy that
      \begin{equation}\label{mama4.52}
    T^0(y_0)<T^1(y_0)\;\;\mbox{and}\;\; N(T^0(y_0),y_0)<\infty.
    \end{equation}
    Then it holds that
    \begin{eqnarray}\label{0228-coro4.8-add-1}
     0< N(T^0(y_0),y_0).
    \end{eqnarray}
    Furthermore,  the following conclusions are true:

    \vskip 2pt
   \noindent  (i) It holds that
    \begin{equation}\label{mama4.53}
    T(M,y_0)=T^0(y_0)\in (0,\infty)
    \;\;\mbox{for each}\;\;
     M \in  \big[N(T^0(y_0),y_0),\infty \big).
    \end{equation}

    \vskip 2pt

    \noindent (ii) For each $M\in \big[N(T^0(y_0),y_0),\infty\big)$,
         $(TP)^{M,y_0}$ has   a minimal time  control $u^*$  so that
         $u^*|_{(0,T^0(y_0))}$ (the restriction of $u^*$ over $(0,T^0(y_0))$) is a minimal norm control to
         $(NP)^{T^0(y_0),y_0}$.

\vskip 2pt
         \noindent (iii) For each $M\in \big[N(T^0(y_0),y_0),\infty\big)$, the null control
         is not a minimal time control to  $(TP)^{M,y_0}$.
    \end{Corollary}

 \begin{proof}
 Arbitrarily fix a $y_0\in X\setminus\{0\}$ satisfying  (\ref{mama4.52}).
  We first prove (\ref{0228-coro4.8-add-1}). By contradiction, suppose that it were not true. Then we would have that $N(T^0(y_0),y_0)=0$. By this and (iv) of Lemma \ref{Lemma-NT0-T0-T1}, we find that
  $T^0(y_0)=T^1(y_0)<\infty$,
    which contradicts the first inequality in (\ref{mama4.52}). So (\ref{0228-coro4.8-add-1}) stands.

    Next, we are going to show conclusions (i)-(iii) one by one.

 (i) We first show that
 \begin{equation}\label{mama4.59}
 0< T^0(y_0)<\infty.
    \end{equation}
Indeed, by the first inequality in  (\ref{mama4.52}), we see that
$T^0(y_0)<\infty$.
    Then by the second inequality in  (\ref{mama4.52}) and by
(iii) and (iv) of Theorem \ref{Lemma-existence-optimal-control},  we find that
$T^0(y_0)>0$. Hence,  (\ref{mama4.59}) stands.

We next show (\ref{mama4.53}).
From (\ref{mama4.59}), we see that the problem $(NP)^{T^0(y_0),y_0}$ makes sense. Since
$T^0(y_0)<\infty$ (see (\ref{mama4.59})), by the second inequality in (\ref{mama4.52}), we can apply (iii) and (iv) of Theorem \ref{Lemma-existence-optimal-control} to find that
$(NP)^{T^0(y_0),y_0}$ has a minimal norm control $v^*$. From this, we have that
\begin{eqnarray}\label{abc3.53}
 \hat y(T^0(y_0);y_0,v^*)=0 \;\;\mbox{ and } \;\; \|v^*\|_{L^\infty(0,T^0(y_0);U)}=N(T^0(y_0),y_0).
\end{eqnarray}
Write $\hat v^*$ for the zero extension of  $v^*$ over $\mathbb R^+$.
Then by (\ref{abc3.53}), it follows that
\begin{eqnarray}\label{mama4.62}
  y(T^0(y_0);y_0,\hat v^*)=0
\end{eqnarray}
and
\begin{eqnarray}\label{mama4.63}
   \|\hat v^*\|_{L^\infty(\mathbb{R}^+;U)}=N(T^0(y_0),y_0) \leq M
   \;\;\mbox{for each}\;\;   M   \in  \big[N(T^0(y_0),y_0),\infty \big).
\end{eqnarray}
Arbitrarily fix $M\in \big[N(T^0(y_0),y_0),\infty \big)$.
  It follows from (\ref{0228-coro4.8-add-1}) that $0<M<\infty$. So the problem
$(TP)^{M,y_0}$ makes sense. (In the definition of $(TP)^{M,y_0}$, it is required that $M\in (0,\infty)$, see (\ref{TP-0}).)
Since $0<T^0(y_0)<\infty$ (see (\ref{mama4.59})), from (\ref{mama4.62}) and (\ref{mama4.63}), it follows that
$\hat v^*$ is an admissible control to $(TP)^{M,y_0}$. This, along with (\ref{TP-0}) and (\ref{mama4.62}), indicates that
\begin{eqnarray}\label{mama4.64}
  T(M,y_0)\leq T^0(y_0).
\end{eqnarray}
 Meanwhile, since $M\in \big[N(T^0(y_0),y_0),\infty \big)$, it follows from
 Proposition \ref{Proposition-Range-T} that
 \begin{eqnarray}\label{mama4.65}
  T(M,y_0)\geq T^0(y_0).
\end{eqnarray}
 By (\ref{mama4.64}) and (\ref{mama4.65}), we see that
$ T(M,y_0)= T^0(y_0)$.
 This, along with (\ref{mama4.59}), leads to (\ref{mama4.53}).

 \vskip 2pt

 (ii) Arbitrarily fix an $M\in \big[N(T^0(y_0),y_0),\infty\big)$. Let $v^*$ and
 $\hat v^*$ be given in the proof of the conclusion (i) of this corollary (see
 (\ref{abc3.53}) and (\ref{mama4.62}), respectively). Write  $u^*\triangleq \hat v^*$. It is clear that
\begin{eqnarray}\label{mama4.66}
 u^*|_{(0,T^0(y_0))}=v^*.
\end{eqnarray}
Then by  (\ref{mama4.53}), (\ref{mama4.62}) and (\ref{mama4.63}), we see that
$$
y(T(M,y_0);y_0,u^*)=0\;\;\mbox{and}\;\; \|u^*\|_{L^\infty(\mathbb{R}^+;U)} \leq M.
$$
These yield that $u^*$ is a minimal time control to $(TP)^{M,y_0}$. Meanwhile, it follows by (\ref{abc3.53}) and (\ref{mama4.66}) that $ u^*|_{(0,T^0(y_0))}$ is a minimal norm control to $(NP)^{T^0(y_0),y_0}$. Hence, in this case,  $(TP)^{M,y_0}$ has a minimal time control whose restriction over $\big(0,T^0(y_0)\big)$ is a minimal norm control to
$(NP)^{T^0(y_0),y_0}$.

  \vskip 2pt

  (iii) By contradiction, suppose that the null control were a minimal time control to $(TP)^{M_0,y_0}$ for some $M_0\in \big[N(T^0(y_0),y_0),\infty\big)$. Then by (\ref{mama4.53}), we would have that

  $$
  S\big(T^0(y_0)\big)y_0= y(T^0(y_0);y_0,0)=y(T(M_0,y_0);y_0,0)=0.
  $$
  This, along with (\ref{Ty0}), implies that
  $T^1(y_0)  \leq T^0(y_0)$, which contradicts the first equality in (\ref{mama4.52}). Hence,  the conclusion (iii) is true.


\vskip 2pt

In summary, we finish the proof of this corollary.
 \end{proof}

\begin{Corollary}\label{wangbuchoutheorem3.13-1}
 Suppose that (H1) holds. Let $y_0\in X\setminus\{0\}$ satisfy  that
\begin{equation}\label{mama4.54}
   T^0(y_0)<T^1(y_0)\;\;\mbox{and}\;\;N(T^1(y_0),y_0)>0.
    \end{equation}
    Then the following conclusions are true:

    \noindent (i) It holds that
    \begin{equation}\label{mama4.55}
   N(T^1(y_0),y_0)<\infty.
    \end{equation}

    \noindent (ii) For each $M\in \big(0, N(T^1(y_0),y_0)\big]$,   $(TP)^{M,y_0}$ has no any admissible control.
\end{Corollary}
\begin{proof}
  Suppose that (H1) holds. Let  $y_0\in X\setminus\{0\}$ satisfy (\ref{mama4.54}).
  We will show the conclusions (i)-(ii) one by one.

  \vskip 2pt

  (i)  We observe from the first inequality in (\ref{mama4.54}) that $T^0(y_0)<\infty$. Then (\ref{mama4.55}) follows  from (\ref{jiajinting4.26}).

\vskip 2pt

(ii) Arbitrarily fix an $M$ so that
   \begin{eqnarray}\label{mama4.67}
 0<M\leq N(T^1(y_0),y_0).
\end{eqnarray}
   By (\ref{mama4.55}) and (\ref{mama4.67}), we see that $M\in (0,\infty)$. Thus
       the problem $(TP)^{M,y_0}$ makes sense.
       Then, by (H1) and (\ref{mama4.67}), we can apply
        (ii) and (iii) of Theorem \ref{Theorem-ex-op-TP} to find that
        $(TP)^{M,y_0}$ has no any admissible control.

        \vskip 2pt

        Thus, we finish the proof of  this corollary.

\end{proof}

\begin{Corollary}\label{wangbuchoutheorem3.13-2}
 Suppose that
 \begin{equation}\label{mama4.56}
  T^0(y_0)=T^1(y_0)=\infty.
    \end{equation}
 Then for each $M\in (0,\infty)$,  $(TP)^{M,y_0}$ does not have any admissible control.
\end{Corollary}
\begin{proof}
    Suppose that (\ref{mama4.56}) holds. Then we can apply (i) of
 Theorem \ref{Theorem-ex-T0y0-infty} to find that
    for each $M\in(0,\infty)$, $(TP)^{M,y_0}$ has no admissible control.
    This ends the proof this corollary.
\end{proof}

\begin{Corollary}\label{wangbuchoutheorem3.13-3}
 Let $y_0\in X\setminus\{0\}$ satisfy  that
 \begin{equation}\label{mama4.57}
  T^0(y_0)=T^1(y_0)<\infty.
    \end{equation}
Then the following conclusions are true:

\noindent (i) It holds that
 \begin{equation}\label{mama4.58}
   T(M,y_0)= T^0(y_0)\in(0,\infty)
    \;\;\mbox{for all}\;\;M\in (0,\infty).
 \end{equation}

    \noindent (ii) For each $M \in (0,\infty)$,  the null control is a minimal time control to
 $(TP)^{M,y_0}$.
\end{Corollary}
\begin{proof}
 Arbitrarily fix a $y_0\in X\setminus\{0\}$ so that (\ref{mama4.57}) holds.
  We now show the conclusions (i)-(ii) one by one.
  \vskip2pt

  (i)
    By (ii) of Lemma~\ref{Lemma-T0-T1} and the first inequality in (\ref{mama4.57}), we have that
    $T^0(y_0)=T^1(y_0)>0$.
         This, together with the second inequality in (\ref{mama4.57}), yields that
    \begin{eqnarray}\label{mama4.68}
     0 < T^0(y_0) <\infty.
    \end{eqnarray}
    Meanwhile, by Proposition \ref{Proposition-Range-T} , we find that
    $$
    T^0(y_0)\leq T(M,y_0)\leq T^1(y_0)\;\;\mbox{for all}\;\; M\in (0,\infty).
    $$
    From the above and  the first equality in (\ref{mama4.57}), we find that
    $T(M,y_0)=T^0(y_0)$  for all $M\in (0,\infty)$, which, along with (\ref{mama4.68}), leads to  (\ref{mama4.58}).

(ii) Because of (\ref{mama4.57}), we can apply
    (iv) of Lemma~\ref{Lemma-NT0-T0-T1} to get that
  $N(T^0(y_0),y_0)=0$.
      Since $T^0(y_0)<\infty$ (see (\ref{mama4.57})), the above, along with    (iv) of  Theorem \ref{Lemma-existence-optimal-control},
   implies that the null control  is the unique  minimal norm control to $(NP)^{T^0(y_0),y_0}$.
   Thus, we have that
   $ y(T^0(y_0);y_0,0)=0$,
     which, together with (\ref{mama4.58}), shows that
  $ y(T(M,y_0);y_0,0)=0$ for all $ M\in(0,\infty)$.
   Form this, we see that the null control is a minimal time control
   to each $(TP)^{M,y_0}$ with $M\in(0,\infty)$.

    In summary, we finish the proof of  this corollary.

\end{proof}

 \bigskip
 \section{Maximum principles and bang-bang properties}

In this section, we  derive maximum principles for  $(NP)^{T,y_0}$, with
$(T,y_0)\in \mathcal{W}_{2,3}\cup \mathcal{W}_{3,2}$,
 and
$(TP)^{M,y_0}$, with $(M,y_0)\in \mathcal{V}_{2,2}\cup \mathcal{V}_{3,2}$, under the assumption (H1).
Here,  $\mathcal{W}_{2,3}$, $\mathcal{W}_{3,2}$, $\mathcal{V}_{2,2}$ and $\mathcal{V}_{3,2}$ are given by (\ref{zhangjinchu7}), (\ref{zhangjinchu9}),
(\ref{Lambda-di-2-1}) and (\ref{Lambda-di-3-1}), respectively. Then we prove
the bang-bang properties for these problems
under assumptions (H1) and (H2).
The key to obtain the above-mentioned results is a maximum principle for
affiliated minimal norm problem $(NP)^{y_T}$, with $y_T\in \mathcal{R}_T^0$.
Recall (\ref{attianable-space-norm}) for the definitions of $(NP)^{y_T}$
and $\|y_T\|_{\mathcal R_T}$; (\ref{attainable-space}) for the definition of $\mathcal R_T$; (\ref{R0T}) for the definition of $\mathcal{R}_T^0$;
(\ref{y0-controllable}) for the definition of $T^0(y_0)$; (\ref{Ty0})
for the definition of $T^1(y_0)$; and (\ref{N-infty-y0}) for the definitions of $N(0,y_0)$ and $N(\infty,y_0)$.

 \subsection{Maximum principle for affiliated  problem}

 This subsection presents a  maximum principle of
 $(NP)^{y_T}$, with $y_T\in \mathcal{R}_T^0\setminus\{0\}$. Write respectively  $B_{\mathcal{R}_T}\big(0,\|y_T\|_{\mathcal R_T}\big)$
 and $B_{\mathcal{R}_T^0}\big(0, \|y_T\|_{\mathcal R_T}\big)$
  for the closed balls in $\mathcal{R}_T$ and $\mathcal{R}_T^0$, centered at the origin and of radius  $\|y_T\|_{\mathcal R_T}$. The way to  build up  the maximum principle of  $(NP)^{y_T}$, with  $y_T\in \mathcal R_T^0\setminus\{0\}$, is as follows: First,  with the aid of Theorem~\ref{Theorem-the-second-representation-theorem}, we use
 the Hahn-Banach separation theorem to separate $y_T$ from $B_{\mathcal{R}_T^0}(0,\|y_T\|_{\mathcal{R}_T})$ in the space $\mathcal{R}^0_T$  by a hyperplane   with a normal vector $f^*\in Y_T$. Then, with the help of Theorem~\ref{Theorem-the-representation-theorem}, Theorem~\ref{Theorem-the-second-representation-theorem}, and Proposition \ref{Proposition-ball-RT0-RT}, we prove that the above-mentioned $f^*$ also separates $y_T$ from $B_{\mathcal R_T}(0,\|y_T\|_{\mathcal{R}_T})$ in the space $\mathcal{R}_T$. Finally, we apply
 Theorem~\ref{Theorem-the-representation-theorem} to the aforementioned separation in $\mathcal{R}_T$ to get the maximum principle for $(NP)^{y_T}$.

 \begin{Theorem}\label{Lemma-maximum-NP-yT}
   Suppose that (H1) holds. Let $T\in (0,\infty)$. Then for each $y_T\in \mathcal R_T^0\setminus\{0\}$, there is an $ f^*\in Y_T\setminus\{0\}$ so that each minimal norm control $v^*$ to $(NP)^{y_T}$ verifies that
  \begin{eqnarray}\label{maximum-NP-yT-1}
   \langle v^*(t),f^*(t) \rangle_U=\max_{\|w\|_U\leq \|y_T\|_{\mathcal R_T}} \langle w,f^*(t) \rangle_U~\mbox{for a.e. }t\in (0,T).
  \end{eqnarray}
 \end{Theorem}

 \begin{proof}
 First of all, we notice that  $\mathcal R_T^0\setminus\{0\}\neq \emptyset$ for all $T\in(0,\infty)$ (see   Lemma \ref{Lemma-non-trivial}).
 Arbitrarily fix a $T\in (0,\infty)$ and then fix a  $y_T\in\mathcal R_T^0\setminus\{0\}$.
  We organize the proof by several steps.

  \vskip 3pt

  \noindent {\it Step 1. To find a vector $f^*\in Y_T\setminus \{0\}$ separating $y_T$ from
  $B_T^0(0, \|y_T\|_{\mathcal{R}_T})$ in $\mathcal{R}_T^0$ in the sense that
  \begin{equation}\label{wang3.4}
\max_{z_T\in B_{\mathcal{R}_T^0}(0, \|y_T\|_{\mathcal R_T})}
\displaystyle\langle f^*, z_T  \displaystyle\rangle_{Y_T, \mathcal R_T^0}=
\langle f^*, y_T  \rangle_{Y_T, \mathcal R_T^0}
\end{equation}
   }

  Since $y_T\neq 0$ in $\mathcal{R}_T^0$, $B_{\mathcal{R}_T^0}(0, \|y_T\|_{\mathcal R_T})$ is a non-degenerating closed ball in $\mathcal{R}_T^0$. Thus, we can apply the Hahn-Banach separation theorem in the space $\mathcal{R}_T^0$ to find a vector
   $\eta_0\in(\mathcal R_T^0)^*\setminus\{0\}$ so that
\begin{eqnarray*}
  \displaystyle\langle \eta_0, z_T  \displaystyle\rangle_{(\mathcal R_T^0)^*, \mathcal R_T^0}\leq  \displaystyle\langle \eta_0, y_T  \displaystyle\rangle_{(\mathcal R_T^0)^*, \mathcal R_T^0}\;\;\mbox{for each}\;\;
  z_T\in B_{\mathcal{R}_T^0}(0, \|y_T\|_{\mathcal R_T}).
\end{eqnarray*}
Since $y_T\in B_{\mathcal{R}_T^0}(0, \|y_T\|_{\mathcal R_T})$, the above yields that
\begin{equation}\label{wang3.2}
\max_{z_T\in B_{\mathcal{R}_T^0}(0, \|y_T\|_{\mathcal R_T})}
\displaystyle\langle \eta_0, z_T  \displaystyle\rangle_{(\mathcal R_T^0)^*, \mathcal R_T^0}=\displaystyle\langle \eta_0, y_T  \displaystyle\rangle_{(\mathcal R_T^0)^*, \mathcal R_T^0}.
\end{equation}
Meanwhile, because (H1) holds, we can apply Theorem \ref{Theorem-the-second-representation-theorem} to find  a vector $f^*\in Y_T$ so that
\begin{eqnarray}\label{WANG3.64}
 \langle f^*,z_T \rangle_{Y_T,\mathcal R_T^0}   =   \langle \eta_0,z_T \rangle_{(\mathcal R_T^0)^*,\mathcal R_T^0} \;\;\mbox{for all}\;\;  z_T\in\mathcal R_T^0;
 \;\;\mbox{and}\;\;
 \|f^*\|_{Y_T}=\|\eta_0\|_{(\mathcal R_T^0)^*}.
\end{eqnarray}
Now, (\ref{wang3.4}) follows from (\ref{wang3.2}) and (\ref{WANG3.64}). Besides, since $\eta_0\neq 0$ in $(\mathcal R_T^0)^*$, it follows from the second equality in (\ref{WANG3.64}) that
$f^*\neq 0$ in $Y_T$.

\vskip 3pt

\noindent {\it Step 2. To show that $f^*$ given in Step 1 also separates $y_T$ from $B_{\mathcal{R}_T}(0, \|y_T\|_{\mathcal R_T})$ in $\mathcal R_T$ in the sense that
\begin{equation}\label{wang3.7}
\sup_{z_T\in B_{\mathcal{R}_T}(0, \|y_T\|_{\mathcal R_T})}
\displaystyle\langle z_T, f^*  \displaystyle\rangle_{\mathcal R_T, Y_T}
=\langle  y_T, f^*  \rangle_{ \mathcal R_T, Y_T}
\end{equation}}
   We first claim that
   \begin{equation}\label{wang3.6}
\displaystyle\langle f^*, z_T  \displaystyle\rangle_{Y_T, \mathcal R_T^0}
=\displaystyle\langle z_T, f^*  \displaystyle\rangle_{ \mathcal R_T, Y_T}
\;\;\mbox{for all}\;\; z_T\in \mathcal{R}^0_T.
\end{equation}
  In fact, for each $z_T\in \mathcal{R}_T^0$, it follows from
(i) of Theorem~\ref{wangtheorem4.1} that $(NP)^{z_T}$ has a minimal norm control $v_{z_T}$. Then by  Theorem~\ref{Theorem-the-representation-theorem} and Theorem~\ref{Theorem-the-second-representation-theorem} (more precisely, by  (\ref{ob-attain-0}) and (\ref{ob-attain-R0})), we have that
$$
\langle z_T, f^*\rangle_{\mathcal{R}_T,Y_T}
=\int_0^T\langle v_{z_T}(t), f^*(t)\rangle_U  \,\mathrm dt
\;\;\mbox{and}\;\;
\langle f^*, z_T\rangle_{Y_T, \mathcal{R}_T^0}
=\int_0^T\langle f^*(t), v_{z_T}(t)\rangle_U   \,\mathrm dt.
$$
These lead to (\ref{wang3.6}).

We next claim that
\begin{equation}\label{wang3.5}
\sup_{z_T\in B_{\mathcal{R}_T^0}(0, \|y_T\|_{\mathcal R_T})}
\displaystyle\langle  z_T, f^*  \displaystyle\rangle_{ \mathcal R_T, Y_T}
=\sup_{z_T\in B_{\mathcal{R}_T}(0, \|y_T\|_{\mathcal R_T})}
\displaystyle\langle  z_T, f^*  \displaystyle\rangle_{\mathcal R_T, Y_T}.
\end{equation}
Indeed, on one hand, since
$$
B_{\mathcal{R}_T^0}(0, \|y_T\|_{\mathcal R_T})\subseteq
B_{\mathcal{R}_T}(0, \|y_T\|_{\mathcal R_T}),
$$
we have that
\begin{equation}\label{wangmax4.86}
\sup_{z_T\in B_{\mathcal{R}_T^0}(0, \|y_T\|_{\mathcal R_T})}
\displaystyle\langle  z_T, f^*  \displaystyle\rangle_{ \mathcal R_T, Y_T}
\leq\sup_{z_T\in B_{\mathcal{R}_T}(0, \|y_T\|_{\mathcal R_T})}
\displaystyle\langle  z_T, f^*  \displaystyle\rangle_{\mathcal R_T, Y_T}.
\end{equation}
 On the other hand, it follows from Proposition \ref{Proposition-ball-RT0-RT} that for each $z_T\in B_{\mathcal{R}_T}(0, \|y_T\|_{\mathcal R_T})$, there is a sequence
$\{z_{T,n}\}_{n=1}^\infty$ in $B_{\mathcal{R}_T^0}(0, \|y_T\|_{\mathcal R_T})$ so that
$$
z_{T,n} \rightarrow z_T\;\;\mbox{in}\;\; \sigma(\mathcal{R}_T,Y_T),\;\;\mbox{as}\;\;
n\rightarrow\infty,
$$
which yields that
$$
\langle z_{T,n}, f^*\rangle_{\mathcal R_T, Y_T}\rightarrow
\langle z_{T}, f^*\rangle_{\mathcal R_T, Y_T},\;\;\mbox{as}\;\;
n\rightarrow\infty.
$$
   From this, one can easily check that
   \begin{equation}\label{wangmax4.87}
\sup_{z_T\in B_{\mathcal{R}_T^0}(0, \|y_T\|_{\mathcal R_T})}
\displaystyle\langle  z_T, f^*  \displaystyle\rangle_{ \mathcal R_T, Y_T}
\geq\sup_{z_T\in B_{\mathcal{R}_T}(0, \|y_T\|_{\mathcal R_T})}
\displaystyle\langle  z_T, f^*  \displaystyle\rangle_{\mathcal R_T, Y_T}.
\end{equation}
   By (\ref{wangmax4.86}) and (\ref{wangmax4.87}),  (\ref{wang3.5}) follows at once.

Finally,  (\ref{wang3.7}) follows from  (\ref{wang3.4}),  (\ref{wang3.6}) and (\ref{wang3.5}) at once.

\vskip 3pt

\noindent {\it Step 3.  To derive from  (\ref{wang3.7}) that
\begin{eqnarray}\label{wangmax4.89}
\sup_{\|v\|_{L^\infty(0,T;U)}\leq\|y_T\|_{\mathcal R_T}} \int_0^T \langle v(t),f^*(t) \rangle_U  \,\mathrm dt
= \int_0^T \langle v^*(t),f^*(t) \rangle_U \,\mathrm dt,
\end{eqnarray}
for any minimal norm control $v^*$ to $(NP)^{y_T}$}

First, according to Theorem~\ref{Theorem-the-representation-theorem} (more precisely, see   (\ref{ob-attain-0})),  any minimal norm control $v^*$ to $(NP)^{y_T}$ (the existence of $v^*$  is guaranteed by
Theorem~\ref{wangtheorem4.1}) satisfies that
  \begin{equation}\label{wang3.9}
  \langle  y_T, f^*  \rangle_{ \mathcal R_T, Y_T} = \int_0^T \langle v^*(t),f^*(t) \rangle_U  \,\mathrm dt.
  \end{equation}
We next claim that
\begin{equation}\label{wang3.10}
 \sup_{\|v\|_{L^\infty(0,T;U)}\leq\|y_T\|_{\mathcal R_T}} \int_0^T \langle v(t),f^*(t) \rangle_U \,\mathrm dt =\sup_{z_T\in B_{\mathcal{R}_T}(0, \|y_T\|_{\mathcal R_T})}
\displaystyle\langle z_T, f^*  \displaystyle\rangle_{\mathcal R_T, Y_T}.
\end{equation}
In fact, on one hand, arbitrarily fix a $v\in L^\infty(0,T;U)$ so that
$\|v\|_{L^\infty(0,T;U)}\leq \|y_T\|_{\mathcal{R}_T}$.
 Then we find from  (\ref{attianable-space-norm}) that
\begin{equation}\label{wangmax4.92}
\|\hat y(T; 0,v)\|_{\mathcal{R}_T}\leq \|v\|_{L^\infty(0,T;U)}\leq\|y_T\|_{\mathcal{R}_T}.
\end{equation}
Meanwhile, since the above-mentioned $v$ is an admissible control to
the problem
$(NP)^{z_T}$, with $z_T\triangleq \hat y(T;0,v)$,
 we see from Theorem~\ref{Theorem-the-representation-theorem} (more precisely, from (\ref{ob-attain-0})) that
\begin{equation}\label{wangmax4.93}
\int_0^T\langle v(t), f^*(t)\rangle_U  \,\mathrm dt =\langle \hat y(T;0,v), f^*\rangle_{\mathcal{R}_T, Y_T}.
\end{equation}
From (\ref{wangmax4.93}) and (\ref{wangmax4.92}), it follows that
\begin{eqnarray*}
\int_0^T \langle v(t),f^*(t) \rangle_U \,\mathrm dt=\langle \hat y(T; 0,v), f^* \rangle_{\mathcal{R}_T, Y_T}\leq \sup_{z_T\in B_{\mathcal{R}_T}(0, \|y_T\|_{\mathcal R_T})}
\displaystyle\langle z_T, f^*  \displaystyle\rangle_{\mathcal R_T, Y_T},
\end{eqnarray*}
which leads to that
\begin{equation}\label{wang3.11}
\sup_{\|v\|_{L^\infty(0,T;U)}\leq\|y_T\|_{\mathcal R_T}} \int_0^T \langle v(t),f^*(t) \rangle_U \,\mathrm dt
\leq \sup_{z_T\in B_{\mathcal{R}_T}(0, \|y_T\|_{\mathcal R_T})}
\displaystyle\langle z_T, f^*  \displaystyle\rangle_{\mathcal R_T, Y_T}.
\end{equation}
On the other hand, arbitrarily fix a $z_T\in B_{\mathcal{R}_T}(0, \|y_T\|_{\mathcal R_T})$.  According to Theorem~\ref{wangtheorem4.1}, $(NP)^{z_T}$ has a minimal norm control $v^*_{z_T}$ satisfying that
$$
z_T=\hat y(T;0,v^*_{z_T})\;\;\mbox{and}\;\; \|v^*_{z_T}\|_{L^\infty(0,T; U)}=\|z_T\|_{\mathcal{R}_T}\leq \|y_T\|_{\mathcal{R}_T}.
$$
Then, by (\ref{ob-attain-0}), we find that
$$
\langle  z_T, f^* \rangle_{\mathcal{R}_T, Y_T}
= \int_0^T \langle v_{z_T}(t),f^*(t) \rangle_U \,\mathrm dt
\leq
\sup_{\|v\|_{L^\infty(0,T;U)}\leq\|y_T\|_{\mathcal R_T}} \int_0^T \langle v(t),f^*(t) \rangle_U \,\mathrm dt.
$$
From this, we see that
\begin{equation}\label{wang3.12}
\sup_{z_T\in B_{\mathcal{R}_T}(0, \|y_T\|_{\mathcal R_T})}
\displaystyle\langle z_T, f^*  \displaystyle\rangle_{\mathcal R_T, Y_T}
\leq
\sup_{\|v\|_{L^\infty(0,T;U)}\leq\|y_T\|_{\mathcal R_T}} \int_0^T \langle v(t),f^*(t) \rangle_U \,\mathrm dt.
\end{equation}
By  (\ref{wang3.11}) and (\ref{wang3.12}), we obtain (\ref{wang3.10}).

Finally, (\ref{wangmax4.89}) follows from (\ref{wang3.7}), (\ref{wang3.9}) and (\ref{wang3.10}) at once.

\vskip 3pt

\noindent{\it Step 4. To get (\ref{maximum-NP-yT-1}) by  dropping the integral in (\ref{wangmax4.89}) }

Arbitrarily fix a minimal norm control $v^*$ to $(NP)^{y_T}$. Since
$f^*\in L^1(0,T;U)$ and $y_T\neq 0$ in $\mathcal R_T$,
we have that
$$
\|f^*\|_{L^1(0,T;U)}=\sup_{\|v\|_{L^\infty(0,T;U)\leq\|y_T\|_{\mathcal{R}_T}}}
\frac{\langle f^*,v\rangle_{L^1(0,T;U), L^\infty(0,T;U)}}{\|y_T\|_{\mathcal{R}_T}},
$$
which, together with (\ref{wangmax4.89}), yields that
\begin{eqnarray}\label{maximum-int-point}
\int_0^T \|y_T\|_{\mathcal R_T} \|f^*(t)\|_U \,\mathrm dt
=\int_0^T \langle v^*(t),f^*(t) \rangle_U \,\mathrm dt.
\end{eqnarray}
Meanwhile, since $v^*$ is a minimal norm control to $(NP)^{y_T}$,
$\|v^*\|_{L^\infty(0,T;U)} = \|y_T\|_{\mathcal{R}_T}$.
This yields that
$\|v^*(t)\|_U\leq \|y_T\|_{\mathcal{R}_T}$ for a.e. $ t\in (0,T)$.
Hence, we have that
\begin{eqnarray}\label{wangmax4.97}
 \langle v^*(t),f^*(t) \rangle_U \leq \|y_T\|_{\mathcal R_T} \|f^*(t)\|_U
 \;\;\mbox{for a.e.}\;\; t\in (0,T).
\end{eqnarray}
From (\ref{wangmax4.97}) and  (\ref{maximum-int-point}), we find that
\begin{eqnarray}\label{maximum-int-point-1}
 \langle v^*(t),f^*(t) \rangle_U = \|y_T\|_{\mathcal R_T} \|f^*(t)\|_U
 \;\;\mbox{for a.e.}\;\; t\in (0,T).
\end{eqnarray}
Meanwhile, we have that
 \begin{eqnarray}\label{wangmax4.99}
  \|y_T\|_{\mathcal R_T} \|f^*(t)\|_U
  = \max_{\|w\|_U\leq\|y_T\|_{\mathcal{R}_T}}
  \langle w,f^*(t)\rangle_U
  \;\;\mbox{for a.e.}\;\; t\in (0,T).
\end{eqnarray}
 From  (\ref{maximum-int-point-1})  and (\ref{wangmax4.99}), we are led to (\ref{maximum-NP-yT-1}).

 \vskip 2pt

 In summary, we finish the proof of this theorem.

 \end{proof}

 \begin{Remark}\label{wang3.3}
 (i) We would like to mention that (\ref{maximum-NP-yT-1}) is not  a standard Pontryagin maximum principle, since we are not sure if  $f^*$ can be  expressed as $B^*\varphi$ with $\varphi$ a solution of the adjoint equation over $(0,T)$, even in the case that $B\in \mathcal{L}(U,X)$.

\noindent (ii) It is natural to ask if we can directly apply the Hahn-Banach separation theorem to separate  $\{y_T\}$ from $B_{\mathcal{R}_T}(0,\|y_T\|_{\mathcal R_T})$ in the state space $X$? By our understanding, the answer seems to be negative in general. However, if we have that
\begin{eqnarray}\label{quasi-solid}
  B_{\mathcal R_T}\big(0,\|y_T\|_{\mathcal R_T}\big)^o\neq\emptyset,
 \end{eqnarray}
 where $ B_{\mathcal R_T}\big(0,\|y_T\|_{\mathcal R_T}\big)^o$ is the interior of the set $B_{\mathcal R_T}\big(0,\|y_T\|_{\mathcal R_T}\big)$ in the space:
 $$
 \widetilde{X}\triangleq  \overline{\mbox{span\,}\mathcal R_T}^{\|\cdot\|_X} ,\;\;\mbox{with the norm}\;\;\|\cdot\|_{X},
 $$
 then  the answer to the above question is positive. Indeed,
 we first notice that $\widetilde{X}$ is a closed subspace of $X$. Next,
  since  $\{y_T\}$ lies at the boundary of $B_{\mathcal R_T}\big(0,\|y_T\|_{\mathcal R_T}\big)$, by the assumption (\ref{quasi-solid}),  we can apply the Hahn-Banach separation theorem in the space $\widetilde{X}$ to separate $\{y_T\}$ from $B_{\mathcal R_T}\big(0,\|y_T\|_{\mathcal R_T}\big)$ via a normal vector $\eta^*\in X\setminus\{0\}$, i.e.,
 \begin{eqnarray}\label{0920-1}
  \langle z_T,\eta^* \rangle_X \leq \langle y_T,\eta^* \rangle_X\;\;\mbox{for all}\;\;z_T\in B_{\mathcal R_T}\big(0,\|y_T\|_{\mathcal R_T}\big).
 \end{eqnarray}
 Meanwhile,  from the first assertion in (\ref{ob-attain-11}), (\ref{ob-attain-0}) and (\ref{wang1.15}), one can easily check that
 \begin{eqnarray*}
   \langle z_T,\eta \rangle_X = \langle z_T, \widetilde{B^*S^*}(T-\cdot)\eta  \displaystyle\rangle_{\mathcal R_T, Y_T}
   \;\;\mbox{for all}\;\; z_T\in \mathcal R_T
   \;\;\mbox{and}\;\; \eta\in X.
 \end{eqnarray*}
  This, along with (\ref{0920-1}), yields that
 \begin{equation*}
\sup_{z_T\in B_{\mathcal{R}_T}(0, \|y_T\|_{\mathcal R_T})}
\displaystyle\langle z_T, f^*  \displaystyle\rangle_{\mathcal R_T, Y_T}
=\langle  y_T, f^*  \rangle_{ \mathcal R_T, Y_T},
\end{equation*}
where $f^*(\cdot)\triangleq \widetilde{B^*S^*}(T-\cdot)\eta^*$. Then by the similar arguments as those used in (\ref{wang3.7})-(\ref{wangmax4.99}), we can obtain the standard Pontryagin maximum principle.

 Unfortunately, the condition  (\ref{quasi-solid}) does not hold in general. In fact,  consider the inclusion map $i_{\mathcal R_T}:\,(\mathcal R_T,\|\cdot\|_{\mathcal R_T})\hookrightarrow \widetilde{X}(\subset X)$. If  (\ref{quasi-solid}) holds, then one can easily show that this map  is surjective. By the open mapping theorem, we find that  $i_{\mathcal R_T}$ is isomorphic from $(\mathcal R_T,\|\cdot\|_{\mathcal R_T})$ to $(\widetilde{X},\|\cdot\|_X)$. Hence, $\mathcal R_T$($=\widetilde X$) is closed in $X$ and  norms $\|\cdot\|_{\mathcal R_T}$ and $\|\cdot\|_X$ are equivalent. However,  these fail for general controlled system $(A,B)$, such as the internally controlled heat equations. (It is well known that the reachable subspace at time $T$ for the internally controlled heat equations over $\Omega\times (0,T)$
 is  not closed in $L^2(\Omega)$, where $\Omega\subset\mathbb R^n$ is an open bounded domain of $C^2$.)

 \end{Remark}

\subsection{Maximum principles for minimal norm and time controls}

We first present a maximum principle for  $(NP)^{T,y_0}$, with
$(T,y_0)\in \mathcal{W}_{2,3}\cup \mathcal{W}_{3,2}$ in  next Theorem~\ref{Lemma-maximum-NP-y0}.
We would like to mention two facts as follows: First, it is not  obvious, at the first sight, that the region
of pairs $(T,y_0)$ described in  Theorem~\ref{Lemma-maximum-NP-y0}, is the same as $\mathcal{W}_{2,3}\cup \mathcal{W}_{3,2}$. However, from (ii) of Remark~\ref{Remark-0131-intro-1}, we know that they are the same. Second, the proof of Theorem~\ref{Lemma-maximum-NP-y0} is based on Theorem \ref{Lemma-maximum-NP-yT} and the connection between $(NP)^{y_T}$ and $(NP)^{T,y_0}$
built up in Proposition \ref{Lemma-NP-yT-y0-eq}.

 \begin{Theorem}\label{Lemma-maximum-NP-y0}
  Suppose that (H1) holds.
   Let  $y_0\in X\setminus\{0\}$ satisfy that  $T^0(y_0)<T^1(y_0)$. Then for each    $T\in \big(T^0(y_0),T^1(y_0)\big)$, there is an $f^*\in Y_T\setminus\{0\}$ so that every minimal norm control $v^*$ to $(NP)^{T,y_0}$ satisfies that
  \begin{eqnarray}\label{maximum-NP-y0-1}
   \langle v^*(t),f^*(t) \rangle_U=\max_{\|w\|_U\leq N(T,y_0)} \langle w,f^*(t) \rangle_U\;\;\mbox{for a.e.}\;\;t\in (0,T).
  \end{eqnarray}
%
   \end{Theorem}
 \begin{proof}
    Arbitrarily fix a $y_0\in X\setminus\{0\}$ so that  $T^0(y_0)<T^1(y_0)$, and then fix  a $T\in \big(T^0(y_0),T^1(y_0)\big)$. Write
   \begin{equation}\label{whenhen4.104}
   \hat y_T\triangleq -S(T)y_0.
   \end{equation}

    First, we claim that
     \begin{equation}\label{whenhen4.105}
   \hat y_T\in \mathcal{R}_T^0\setminus\{0\}.
   \end{equation}
   In fact, since $T>T^0(y_0)$, it follows from
     (\ref{y0-controllable}) that there is a  $\hat t\in \big[T^0(y_0),T\big)$ so that
     $\hat y(\hat t;y_0,\hat v)=0$ for some $\hat v\in L^\infty(0,\hat t;U)$.
        Write  $\widetilde{v}$ for the zero extension of $\hat v$  over $(0,T)$.
      It is clear that
      $\hat y(T;y_0,\widetilde{v})=0$ and $\lim_{s\rightarrow T}\|\widetilde{v}\|_{L^\infty(s,T;U)}=0$.
      These,
      together with (\ref{Changshubianyi1.6}), (\ref{R0T}) and (\ref{attainable-space}), yield that
 \begin{eqnarray}\label{NP-T-RT0}
 -S(T)y_0=\int_0^T S_{-1}(T-t)B\widetilde{v}(t) \,\mathrm dt
 =\hat y(T;0,\widetilde v) \in\mathcal R_T^0
 \subset \mathcal R_T.
 \end{eqnarray}
 By (\ref{whenhen4.104}) and  (\ref{NP-T-RT0}), we can apply  (ii) of Proposition \ref{Lemma-NP-yT-y0-eq}
 to get that
 \begin{equation}\label{whenhen4.107}
  \|\hat y_T\|_{\mathcal{R}_T}=\|-S(T)y_0\|_{\mathcal{R}_T}=N(T,y_0).
   \end{equation}
   Meanwhile, since
   $T\in \big(T^0(y_0),T^1(y_0)\big)\subseteq \big(0, T^1(y_0)\big)$,
 we can apply  (iii) of Lemma \ref{Lemma-T0-T1} to find that
 \begin{equation}\label{Wgs3.22}
 N(T,y_0)>0.
  \end{equation}
  From (\ref{whenhen4.107}) and (\ref{Wgs3.22}), we obtain that
  $S(T)y_0\neq 0$ in $\mathcal R_T$, which along with  (\ref{NP-T-RT0}), leads to
  (\ref{whenhen4.105}).

  Next, by (H1) and (\ref{whenhen4.105}), we can apply Theorem \ref{Lemma-maximum-NP-yT}
   (where $y_T=\hat y_T$ is given by (\ref{whenhen4.104})) to find an $f^*\in Y_T\setminus\{0\}$ so that
  for each minimal norm control $\hat v^*$ to $(NP)^{\hat y_T}$,
  \begin{eqnarray}\label{whenhen4.109}
   \langle \hat v^*(t),f^*(t) \rangle_U=\max_{\|w\|_U\leq \|\hat y_T\|_{\mathcal{R}_T}} \langle w,f^*(t) \rangle_U\;\;\mbox{for a.e.}\;\;t\in (0,T).
  \end{eqnarray}
 ~~~\, Finally, we arbitrarily fix a minimal norm control $v^*$ to $(NP)^{T,y_0}$.
  (The existence of $v^*$ is guaranteed by (i) of Theorem~\ref{Lemma-existence-optimal-control}, since
  $T\in (T^0(y_0), T^1(y_0))$.)
   Because of  (\ref{NP-T-RT0}), we can apply  (iii) of Proposition \ref{Lemma-NP-yT-y0-eq} to see that $v^*$ is also a minimal norm control to
   $(NP)^{\hat y_T}$. This, together with (\ref{whenhen4.109}) and   (\ref{whenhen4.107}), indicates that $v^*$ satisfies (\ref{maximum-NP-y0-1}) with $f^*$ given by (\ref{whenhen4.109}). This ends the proof of this theorem.

%

     \end{proof}

To get the maximum principle for $(TP)^{M,y_0}$, we need the following lemma.
 \begin{Lemma}\label{the-equivalence}
  Suppose that (H1) holds.  Let $y_0\in X\setminus\{0\}$, with $T^0(y_0)<T^1(y_0)$. Then it holds that
  \begin{equation}\label{wanghenhen4.110}
  N(T^1(y_0),y_0)<N(T^0(y_0),y_0).
  \end{equation}
  Furthermore, the following conclusions are true:
  \vskip 2pt

  \noindent  (i) If $M\in\big(N(T^1(y_0),y_0), N(T^0(y_0),y_0)\big)$ and
  $u^*$ is a minimal time control to  $(TP)^{M,y_0}$, then $u^*|_{(0, T(M,y_0))}$
  (the restriction of $u^*$ over $(0, T(M,y_0))$) is a minimal norm control to $(NP)^{T(M,y_0),y_0}$.

  \vskip 2pt

  \noindent (ii) If  $T\in \big(T^0(y_0),T^1(y_0)\big)$ and $v^*$
  is a
  minimal norm control
    to $(NP)^{T,y_0}$, then the zero extension of $v^*$ over  $\mathbb{R}^+$ is a minimal time control to  $(TP)^{N(T,y_0),y_0}$.
 \end{Lemma}
 \begin{proof}
  Since  (H1) holds, we can apply  (i) of  Proposition~\ref{wang-prop3.3} to get
  (\ref{wanghenhen4.110}). Next we will prove the conclusions (i)-(ii) one by one.

  \vskip 2pt

  (i) Arbitrarily fix an $M$ so that
  \begin{equation}\label{wanghenhen4.111}
  N(T^1(y_0),y_0)<M<N(T^0(y_0),y_0).
  \end{equation}
  Suppose that $u^*$ is a minimal time control to $(TP)^{M,y_0}$. (Since $T^0(y_0)<T^1(y_0)$, the existence of $u^*$ is guaranteed by (i) of  Theorem~\ref{Theorem-ex-op-TP}, as well as (\ref{wanghenhen4.111}).) Then we have that
  \begin{equation}\label{wanghenhen4.112}
  \|u^*\|_{L^\infty(\mathbb{R}^+;U)}\leq M\;\;\mbox{and}\;\;
  y(T(M,y_0);y_0,u^*)=0.
  \end{equation}
  Meanwhile, since $T^0(y_0)<T^1(y_0)$, by using (H1), we can apply
   (i) of
  Corollary~\ref{wangcorollary3.8} to see that
 \begin{equation}\label{wanghenhen4.113}
 T(M,y_0)\in (0,\infty)\;\;\mbox{and}\;\; M=N(T(M,y_0),y_0).
  \end{equation}
 By (\ref{wanghenhen4.112}) and (\ref{wanghenhen4.113}), we see that the problem
 $(NP)^{T(M,y_0),y_0}$ makes sense, and find  that
 \begin{equation}\label{wanghenhen4.114}
\|u^*|_{(0, T(M,y_0))}\|_{L^\infty(0, T(M,y_0);U)}\leq N(T(M,y_0),y_0)
   \end{equation}
  and
  \begin{equation}\label{wanghenhen4.115}
    \hat y(T(M,y_0);y_0,u^*|_{(0, T(M,y_0))})=0.
  \end{equation}
   From (\ref{wanghenhen4.114}), (\ref{wanghenhen4.115}) and (\ref{NP-0}),
   it follows that  $u^*|_{(0,T(M,y_0))}$ is a minimal norm control to the problem
 $(NP)^{T(M,y_0),y_0}$.

\vskip 2pt

\noindent (ii) Arbitrarily fix a $T$ so that
  \begin{equation}\label{wanghenhen4.116}
T^0(y_0)<T<T^1(y_0).
  \end{equation}
  Suppose that $v^*$ is a minimal norm control to $(NP)^{T,y_0}$. ( The existence of
  $v^*$ is guaranteed by (i) of Theorem~\ref{Lemma-existence-optimal-control}, because of  (\ref{wanghenhen4.116}).)
  Write $\widetilde{v}^*$ for the zero extension of $v^*$ over $\mathbb{R}^+$. Then we have that
  \begin{equation}\label{wanghenhen4.117}
y(T; y_0,\widetilde{v}^*)=0\;\;\mbox{and}\;\; \|\widetilde{v}^*\|_{L^\infty(\mathbb{R}^+;U)}\leq N(T,y_0).
  \end{equation}
  Meanwhile, by (H1) and (\ref{wanghenhen4.116}), we can apply (ii) of
Corollary~\ref{wangcorollary3.8} to find that
  \begin{equation}\label{wanghenhen4.118}
N(T,y_0)\in (0,\infty)\;\;\mbox{and}\;\; T=T(N(T,y_0),y_0).
  \end{equation}
  From (\ref{wanghenhen4.117}) and (\ref{wanghenhen4.118}), it follows that the problem
  $(TP)^{N(T,y_0),y_0}$ makes sense and that
  $$
  y\big(T(N(T,y_0),y_0);y_0,\widetilde{v}^*\big)=0
  \;\;\mbox{and}\;\;
  \|\widetilde{v}^*\|_{L^\infty(\mathbb{R}^+;U)}\leq N(T,y_0).
  $$
  These imply that $\widetilde{v}^*$ is a minimal time control to $(TP)^{N(T,y_0),y_0}$.

  \vskip 2pt

  In summary, we finish the proof of this lemma.

\end{proof}

Now, we will present a maximum principle for  $(TP)^{M,y_0}$, with
$(M,y_0)\in \mathcal{V}_{2,2}\cup \mathcal{V}_{3,2}$ in  next Theorem~\ref{Theorem-time-variant-maximum-principle}.
Two facts deserve to be mentioned: First, it is not  obvious, at the first sight, that the region
of pairs $(M,y_0)$ described in  Theorem~\ref{Theorem-time-variant-maximum-principle}, is the same as $\mathcal{V}_{2,2}\cup \mathcal{V}_{3,2}$. However, from (ii) of Remark~\ref{Remark-0131-intro-1} and the definitions of $\mathcal{V}_{2,2}$ and $\mathcal{V}_{3,2}$ (see (\ref{Lambda-di-2-1}) and (\ref{Lambda-di-3-1})), we can easily verify that they are the same. Second, the proof of Theorem~\ref{Theorem-time-variant-maximum-principle} is based on Theorem~\ref{Lemma-maximum-NP-y0} and the connections between $(NP)^{T,y_0}$
and $(TP)^{M,y_0}$ built up in Corollary~\ref{wangcorollary3.8} and Lemma~\ref{the-equivalence}.

\begin{Theorem}\label{Theorem-time-variant-maximum-principle}
  Suppose that (H1) holds.
   Let  $y_0\in X\setminus\{0\}$ satisfy that  $T^0(y_0)<T^1(y_0)$. Then
     $$
   N(T^1(y_0),y_0)<N(T^0(y_0),y_0).
   $$
   Furthermore, for each  $M\in \big(N(T^1(y_0),y_0), N(T^0(y_0),y_0)\big)$,
   the following conclusions are true:

   \vskip 2pt

   \noindent (i) It holds  that
   \begin{equation}\label{maxthree4.119}
   T^0(y_0)<T(M,y_0)<T^1(y_0).
   \end{equation}

   \vskip 2pt

   \noindent (ii) There is a vector $f^*\in Y_{T(M,y_0)}\setminus\{0\}$
      so that each minimal time control $u^*$ to $(TP)^{M,y_0}$ satisfies that
   \begin{eqnarray}\label{max-TP}
    \langle u^*(t),f^*(t) \rangle_U  = \max_{\|w\|_U\leq M} \langle w,f^*(t) \rangle_U\;\; \mbox{for a.e. } t\in \big(0,T(M,y_0)\big).
   \end{eqnarray}

 \end{Theorem}

 \begin{proof}
 Arbitrarily fix a $y_0\in X\setminus\{0\}$ so that
 \begin{equation}\label{maxthree4.121}
   T^0(y_0)<T^1(y_0).
   \end{equation}
 By (H1) and (\ref{maxthree4.121}), we can see from (\ref{wanghenhen4.110}) that
 \begin{equation}\label{maxthree4.122}
   N(T^1(y_0),y_0)< N(T^0(y_0),y_0).
   \end{equation}
  Arbitrarily fix a number $M$ so that
 \begin{eqnarray}\label{1008-11}
 N(T^1(y_0),y_0)<M< N(T^0(y_0),y_0).
 \end{eqnarray}

 We now are going to show the conclusions (i)-(ii) in this theorem one by one.

 \vskip 2pt

 (i) By (H1) and  (\ref{maxthree4.121}), we can apply (i) of
Corollary~\ref{wangcorollary3.8} (more precisely, apply (\ref{ggssww3.40})) to get
 both (\ref{maxthree4.119}) and the fact that
 \begin{equation}\label{maxthree4.124}
   M=N(T(M,y_0),y_0).
   \end{equation}

 \vskip 2pt

 (ii) By (H1), (\ref{maxthree4.121}) and  (\ref{maxthree4.119}), we can apply
 Theorem~\ref{Lemma-maximum-NP-y0}  to get a vector
 $f^*\in Y_{T(M,y_0)}\setminus\{0\}$ so that every minimal norm control $v^*$
 to $(NP)^{T(M,y_0),y_0}$ satisfies that
 \begin{equation}\label{maxthree4.125}
  \langle v^*(t),f^*(t) \rangle_U  = \max_{\|w\|_U\leq N(T(M,y_0),y_0)} \langle w,f^*(t) \rangle_U\;\; \mbox{for a.e. } t\in \big(0,T(M,y_0)\big).
   \end{equation}

 Next, we suppose that $u^*$ is a minimal time control to $(TP)^{M,y_0}$.
 (The existence of $u^*$ is guaranteed by (i) of  Theorem~\ref{Theorem-ex-op-TP},
 because of
 (\ref{maxthree4.121}) and (\ref{1008-11}).) Then by (H1), (\ref{maxthree4.121}) and (\ref{1008-11}), we can use (i) of
 Lemma~\ref{the-equivalence} to see that $u^*|_{(0,T(M,y_0))}$
is a minimal norm control to  $(NP)^{T(M,y_0),y_0}$. This, along with
(\ref{maxthree4.125}) and (\ref{maxthree4.124}), leads to (\ref{max-TP}).

\vskip 2pt

In summary, we finish the proof of this theorem.

   \end{proof}

 \subsection{Bang-bang properties of minimal time and norm controls}

In this section, we will present the bang-bang properties for
$(NP)^{T,y_0}$, with
$(T,y_0)\in \mathcal{W}_{2,3}\cup \mathcal{W}_{3,2}$,
 and
$(TP)^{M,y_0}$, with $(M,y_0)\in \mathcal{V}_{2,2}\cup \mathcal{V}_{3,2}$, under the assumptions (H1) and (H2). Their proof are based on Theorem~\ref{Lemma-maximum-NP-y0} and Theorem~\ref{Theorem-time-variant-maximum-principle}.

\begin{Theorem}\label{yubiaotheorem5.6}
Suppose that (H1) and (H2) hold. Let $y_0\in X\backslash\{0\}$ satisfy that $T^0(y_0)<T^1(y_0)$. Then for each $T\in \big(T^0(y_0), T^1(y_0)\big)$,   $(NP)^{T,y_0}$ has the bang-bang property.

\end{Theorem}
\begin{proof}
 Arbitrarily fix  $y_0\in X\backslash\{0\}$ so that $T^0(y_0)<T^1(y_0)$. Let $T\in \big(T^0(y_0),T^1(y_0)\big)$. Then according to (i) of
Theorem~\ref{Lemma-existence-optimal-control},  $(NP)^{T,y_0}$
has at least one minimal norm control. Arbitrarily fix a minimal norm control
$v^*$ to this problem. By (H1), we can apply  Theorem \ref{Lemma-maximum-NP-y0} to find a vector
   $f^*\in Y_{T}\setminus\{0\}$ so that
\begin{eqnarray}\label{whoareyou4.34}
   \langle v^*(t),f^*(t) \rangle_U=\max_{\|w\|_U\leq N(T,y_0)} \langle w,f^*(t) \rangle_U\;\;\mbox{for a.e.}\;\;t\in (0,T).
  \end{eqnarray}
   Meanwhile, since $f^*\neq 0$ in $Y_T$, we can derive from  (H2) that
   $  f^*(t)\neq 0$ for a.e. $t\in (0,T)$.
       This, along with (\ref{whoareyou4.34}), yields that
  $\|v^*(t)\|_U=N(T,y_0)$ for a.e. $t\in (0, T)$.
  Hence,
  $(NP)^{T,y_0}$ has the bang-bang property. We end the proof of this theorem.

%
%
\end{proof}

 \begin{Theorem}\label{Theorem-time-variant-bangbang}
   Suppose that (H1) holds.  Let $y_0\in X\backslash\{0\}$ satisfy that $T^0(y_0)<T^1(y_0)$. Then
    $N(T^1(y_0),y_0)<N(T^0(y_0),y_0)$.
         If further assume that   (H2) holds, then for each $M\in \big(N(T^1(y_0),y_0),N(T^0(y_0),y_0)\big)$,
    $(TP)^{M,y_0}$ has the bang-bang property.
 \end{Theorem}

 \begin{proof}
  Arbitrarily fix a $y_0\in X\setminus\{0\}$ so that
 $T^0(y_0)<T^1(y_0)$.
By  (H1), we can apply (i) of Proposition~\ref{wang-prop3.3} to find that
 $ N(T^1(y_0),y_0)< N(T^0(y_0),y_0)$.
    Arbitrarily fix an
 $ M\in \big(N(T^1(y_0),y_0), N(T^0(y_0),y_0)\big)$.
  Then we can use
    (i) of Theorem~\ref{Theorem-ex-op-TP} to find that $(TP)^{M,y_0}$ has at least one minimal time control. Next, we arbitrarily fix   a minimal time control
    $u^*$
    to $(TP)^{M,y_0}$. Then by (H1), we can apply (ii) of
          Theorem~\ref{Theorem-time-variant-maximum-principle}
          to find a vector $f^*$ in $Y_{T(M,y_0)}\setminus\{0\}$ so that
          \begin{eqnarray}\label{whoareyou4.130}
    \langle u^*(t),f^*(t) \rangle_U  = \max_{\|w\|_U\leq M} \langle w,f^*(t) \rangle_U\;\; \mbox{for a.e. } t\in \big(0,T(M,y_0)\big).
   \end{eqnarray}
          Meanwhile, since $f^*\neq 0$ in $Y_{T(M,y_0)}$, it follows from (H2) that
          $$
          f^*(t)\neq 0\;\;\mbox{for a.e.}\;\; t\in \big(0,T(M,y_0)\big).
          $$
          This, along with (\ref{whoareyou4.130}), yields that
          \begin{equation}\label{whoareyou4.31}
          \|u^*(t)\|_U=M\;\;\mbox{for a.e.}\;\;t\in \big(0,T(M,y_0)\big).
          \end{equation}
        Thus, $(TP)^{M,y_0}$ has at least one minimal time control and each minimal time control $u^*$ to this problem satisfies (\ref{whoareyou4.31}).
         Hence, $(TP)^{M,y_0}$ has
          the bang-bang property. this ends the proof of this theorem.

 \end{proof}

\section{Proofs of main results}

  This section is devoted to prove the main theorems of this paper. They are
   Theorem \ref{Proposition-NTy0-partition}, Theorem~\ref{Proposition-TMy0-partition} and
  Theorem \ref{Proposition-Corollary1.7}.

  \subsection{Some preliminaries }

  Before proving the main theorems of this paper, we introduce the two theorems (Theorem~\ref{Theorem-TP-infinite-op} and Theorem~\ref{Theorem-TP-infinite-op-1}), which concern with the conclusions (iii) and (iv) in Theorem~\ref{Proposition-TMy0-partition}. The proofs of these two theorems are based on the next Lemma~\ref{infinite-T}.

\begin{Lemma}\label{infinite-T}
 Suppose that (H1) holds. Let
 \begin{eqnarray*}
  \mathcal O_T\triangleq \big\{u\in L^\infty(0,T;U) ~:~\hat y(T;0,u)=0\big\},\;\;\mbox{with}\;\;T\in(0,\infty).
 \end{eqnarray*}
 Then  $\mathcal O_T$ is a closed and infinitely dimensional subspace  in $L^\infty(0,T;U)$.
 \end{Lemma}

 \begin{proof}
 Let $0<T<\infty$. It is clear that $\mathcal O_T$ is a closed subspace  in $L^\infty(0,T;U)$. It remains to  show that $\mathcal O_T$ is of infinite dimension. To this end,  we define
 \begin{eqnarray}\label{def-Ot1t2}
 \mathcal O_{t_1,t_2}\triangleq \big\{u\in \mathcal O_T ~:~\mbox{supp}(u)\subset(t_1,t_2)\big\},  \mbox{ $0<t_1<t_2<T$}.
 \end{eqnarray}
  The rest of the proof is organized by two steps.

  \textit{Step 1. To show that when $0<t_1<t_2<T$, $\mathcal O_{t_1,t_2}$ is a closed subspace of  $\mathcal O_T$ with $\mbox{dim\,}O_{t_1,t_2}\geq 1$}

  Define
 \begin{eqnarray}\label{def-YTt1t2}
 Y_{T,t_1,t_2}\triangleq\{f\in L^1(0,T;U) ~:~ f|_{(t_1,t_2)}=g|_{(t_1,t_2)} \mbox{ for some } g\in Y_{t_2}\}.
 \end{eqnarray}
We claim that
 \begin{eqnarray}\label{OT-11}
   Y_{T,t_1,t_2} \mbox{ is a closed proper subspace in } L^1(0,T;U).
 \end{eqnarray}
 To this end, we first show that $Y_{T,t_1,t_2} $ is closed in $L^1(0,T;U)$.
 For this purpose, let $\{f_n\}_{n=1}^\infty\subset Y_{T,t_1,t_2}$ satisfy  that
 \begin{eqnarray}\label{OT-21}
  f_n \rightarrow \widehat f \mbox{ in } L^1(0,T;U),\;\;\mbox{as}\;\;n\rightarrow\infty.
 \end{eqnarray}
 Since $\{f_n\}_{n=1}^\infty\subset Y_{T,t_1,t_2}$, from (\ref{def-YTt1t2}),  there exists a sequence $\{g_n\}_{n=1}^\infty\subset Y_{t_2}$ so that for all $n\geq 1$,
       $ f_n=g_n $ over $ (t_1,t_2)$.
  This, as well as (\ref{OT-21}), yields that
  \begin{eqnarray}\label{OT-31}
 g_n\rightarrow \widehat f \mbox{  in } L^1(t_1,t_2;U) \;\;\mbox{as}\;\;n\rightarrow\infty.
 \end{eqnarray}
  Meanwhile, by (H1), we can use Lemma \ref{Lemma-H3-eq} to get the conclusion  (iii) in Lemma \ref{Lemma-H3-eq}. This, as well as  (\ref{OT-31}), indicates that $\{g_n\}_{n=1}^\infty$ is a Cauchy sequence in $L^1(0,t_2;U)$. Since $Y_{t_2}$ is closed in $L^1(0,t_2;U)$ (see (\ref{ob-space})), we have that  $g_n$ converges to a function $\hat g$ in $Y_{t_2}$.
 This, along with  (\ref{OT-31}), shows that
 $\widehat f=\hat g$ over $(t_1,t_2)$,
   which, combined with (\ref{def-YTt1t2}), implies that $\widehat f\in Y_{T,t_1,t_2}$. Hence,  the subspace $Y_{T,t_1,t_2}$ is closed in $L^1(0,T;U)$.

 We next show that $Y_{T,t_1,t_2}$ is a proper subspace of $L^1(0,T;U)$.
 In fact, for each   $f\in Y_{T,t_1,t_2}$, we obtain from (\ref{def-YTt1t2}) and (iii) of Lemma \ref{Lemma-H3-eq}  that there is $p_2>1$ so that
 \begin{eqnarray}\label{property-p2}
   f|_{(t_1,s)}\in L^{p_2}(t_1,s;U) \;\; \mbox{ for all }\;\;s\in (t_1,t_2).
 \end{eqnarray}
 However, it is clear that not every function in $L^1(0,T;U)$ holds the property (\ref{property-p2}). Hence,  $Y_{T,t_1,t_2}$ is strictly contained in $L^1(0,T;U)$.
 This finishes the proof of  (\ref{OT-11}).

 Now by  (\ref{OT-11}), there is an $h\in L^1(0,T;U)\setminus Y_{T,t_1,t_2}$.
 Since $Y_{T,t_1,t_2}$ is closed subspace of  $L^1(0,T;U)$, we can apply
   the Hahn-Banach separation theorem  to find a function $ u_h$ in $\big(L^1(0,T;U)\big)^*$
   (which is $L^{\infty}(0,T;U)$)
    so that
 \begin{eqnarray}\label{infinite-2}
  0= \int_0^T \langle u_h(t),f(t)\rangle_{U} \,\mathrm dt
   < \int_0^T \langle u_h(t),h(t)\rangle_{U} \,\mathrm dt
  \;\;\mbox{for all}\;\;   f\in Y_{T,t_1,t_2}.
 \end{eqnarray}
 For each $g\in L^1\big((0,t_1)\cup(t_2,T);U\big)$, let $\widetilde g(\cdot)$ be the zero extension of $g$ over $(0,T)$. Clearly, it follows from  (\ref{def-YTt1t2}) that $\widetilde g\in Y_{T,t_1,t_2}$. Then by the first equality in (\ref{infinite-2}), we find that
 \begin{eqnarray*}
  0=\int_0^T \langle u_h(t),\widetilde g(t) \rangle_U \,\mathrm dt\;\;\mbox{for all}\;\;g\in L^1\big((0,t_1)\cup(t_2,T);U\big).
 \end{eqnarray*}
 This yields that
 \begin{eqnarray}\label{OT-41}
  u_h=0 \mbox{ over } (0,t_1)\cup(t_2,T).
 \end{eqnarray}

 Meanwhile,  for each $z\in D(A^*)$, we define $\psi_z:\,(0,T)\rightarrow U$ by
 \begin{eqnarray*}
  \psi_z(t)=\left\{\begin{array}{ll}
             B^*S^*(t_2-t)z,  ~&t\in(t_1,t_2),\\
             0,   ~&t\in(0,t_1] \cup [t_2,T).
            \end{array}
            \right.
 \end{eqnarray*}
 It follows from (\ref{def-YTt1t2}) and (\ref{ob-space}) that for all $ z\in D(A^*)$,
 $\psi_z \in Y_{T,t_1,t_2}$.
  Then we see from (\ref{NNNWWW2.1}), (\ref{OT-41}) and the first equality in (\ref{infinite-2}) that for each $z\in D(A^*)$,
 \begin{eqnarray*}
  \langle \hat y(t_2;0, u_h),z \rangle_X &=& \int_0^{t_2} \langle u_h(t),B^*S^*(t_2-t)z\rangle_{U} \,\mathrm dt  \nonumber\\
  &=& \int_0^T \langle u_h(t),\psi_z(t)\rangle_{U} \,\mathrm dt =0.
 \end{eqnarray*}
 Since $D(A^*)$ is dense in $X$, the above, as well as (\ref{OT-41}), yields that
 $$\hat y(T;0,u_h)=\hat y(t_2;0,u_h)=0,$$
 which leads to that $u_h\in\mathcal O_{T}$. This, along with (\ref{def-Ot1t2}) and (\ref{OT-41}), implies that
  $u_h\in\mathcal O_{t_1,t_2}$.

  Finally, we see from the second equality in (\ref{infinite-2}) that $u_h\neq 0$ in $L^\infty(0,T;U)$. Hence, we have that $\mbox{dim\,}O_{t_1,t_2}\geq 1$.

 \textit{Step 2. To show that $\mbox{dim\,}\mathcal O_T=+\infty$}

\vskip 5pt
By the conclusion in Step 1, we find  that
 \begin{eqnarray*}
  \{0\}\neq\mathcal O_{T/2^{k+1},T/2^k}\subset\mathcal O_T
  \;\;\mbox{for all}\;\;    k\in\mathbb N^+.
 \end{eqnarray*}
 From (\ref{def-Ot1t2}), we see that
 $$
 \mathcal O_{T/2^{i+1},T/2^i} \cap  \mathcal O_{T/2^{j+1},T/2^j}=\{0\}\;\;\mbox{for all}\;\; i,j\in\mathbb N^+,\;\;\mbox{with}\;\; i\neq j.
 $$
  Take a sequence $\{u_k\}$ so that for each $ k\in\mathbb  N^+$,
$u_k\in O_{T/2^{k+1},T/2^k}$.
  Arbitrarily take a finite subsequence $\{u_{k_n}\}_{n=1}^N$ from
 $\{u_k\}_{k=1}^\infty$. Let $\{\alpha_n\}_{n=1}^N\subset\mathbb R$ be so that
$\sum_{n=1}^N  \alpha_n u_{k_n}=0$.
  Since for each $k$, the support of $u_{k}$ belongs to $(T/2^{k+1},T/2^k)$, we can easily derive from the above equality  that
 $\alpha_n=0$ for all $n\in\{1,\cdots,N\}$.
   So $u_{k_1},u_{k_2},\cdots,u_{k_N}$ are linearly independent. Thus, we conclude that dim\,$\mathcal O_T=\infty$.
  \vskip 5pt

  In summary, we finish the proof of this lemma.
 \end{proof}

\begin{Theorem}\label{Theorem-TP-infinite-op}
 Let $y_0\in X\setminus\{0\}$ satisfy that
 \begin{eqnarray}\label{infinite-1201-1}
  T^0(y_0)<T^1(y_0)
  \;\;\mbox{and}\;\;  N(T^0(y_0),y_0)<\infty.
 \end{eqnarray}
 Suppose that (H1) holds and that
 \begin{eqnarray}\label{infinite-1201-2}
  N(T^0(y_0),y_0)<M<\infty.
 \end{eqnarray}
 Then  $(TP)^{M,y_0}$ has infinitely many different minimal time controls so that among them, any finite number of controls are linearly independent in $L^\infty(\mathbb R^+;U)$.
\end{Theorem}

\begin{proof}
Arbitrarily fix a $y_0\in X\setminus\{0\}$ so that (\ref{infinite-1201-1}) holds. Then fix an $M$ so that (\ref{infinite-1201-2}) holds. By (\ref{infinite-1201-1}) and (\ref{infinite-1201-2}), we can use (i) and (ii) of  Corollary \ref{wangbuchoutheorem3.13} to see that
\begin{equation}\label{wwggss3.59}
   T^0(y_0) = T(M,y_0)\in(0,\infty),
\end{equation}
and to find a minimal time control $u^*$ so that $v^*\triangleq u^*|_{(0,T^0(y_0))}$ is a minimal norm control to $(NP)^{T^0(y_0),y_0}$. The latter, along with  (\ref{wwggss3.59}) and (\ref{infinite-1201-2}), yields that
\begin{equation}\label{wwggss3.59-1}
  y(T(M,y_0);y_0,u^*)= y(T^0(y_0);y_0,u^*)=\hat y(T^0(y_0);y_0,v^*)=0
\end{equation}
and
\begin{eqnarray}\label{TP-NP-11}
  \|u^*\|_{L^\infty(0,T(M,y_0);U)} =\|v^*\|_{L^\infty(0,T^0(y_0);U)}  = N(T^0(y_0),y_0)<M.
\end{eqnarray}

Next, since $0<T^0(y_0)<\infty$ (see (\ref{wwggss3.59})), by (H1), we can use Lemma  \ref{infinite-T} to find a sequence
$\{u_k\}_{k=1}^\infty  \subset  L^\infty(0, T^0(y_0);U)$ so that
\begin{equation}\label{wwggss3.60}
   \hat y(T^0(y_0); 0,u_k)=0\;\;\mbox{for all}\;\;    k\in\mathbb N^+,
\end{equation}
and so that any finite number of elements in $\{u_k\}_{k=1}^\infty$ are linearly independent in the space
$L^\infty(0, T^0(y_0);U)$. Write $\hat u_k$, $k=1,2,\cdots$, for the zero extension of $u_k$ over $\mathbb{R}^+$. Then any finite number of elements in $\{\hat u_k\}_{k=1}^\infty$ are linearly independent in
$L^\infty(\mathbb R^+;U)$.
Arbitrarily fix a $k\in\mathbb N^+$. It follows from (\ref{wwggss3.59}) and (\ref{wwggss3.60}) that
\begin{equation}\label{wwggss3.61}
   y(T(M,y_0); 0, \hat u_k)=0.
\end{equation}
Because of (\ref{infinite-1201-2}), we can take  $\varepsilon_k>0$ so that
  \begin{equation}\label{wwggss3.62}
  \varepsilon_k \|\hat u_k\|_{L^\infty(\mathbb R^+;U)}< M-N(T^0(y_0),y_0).
  \end{equation}
 Define a control $u^*_k$ as follows:
 \begin{eqnarray}\label{wwggss3.62-1}
  u^*_k\triangleq \varepsilon_k \hat u_k + \chi_{(0,T(M,y_0))}u^*
  \;\;\mbox{over}\;\;   \mathbb R^+.
 \end{eqnarray}
 This, along with (\ref{wwggss3.59-1}) and (\ref{wwggss3.61}), yields that
 \begin{eqnarray}\label{wwggss3.62-2}
  y(T(M,y_0);y_0,u^*_k)
  = y(T(M,y_0);y_0,u^*) + \varepsilon_k y(T(M,y_0);0,\hat u_k) =0.
 \end{eqnarray}
 At same time, it follows from (\ref{wwggss3.62-1}), (\ref{wwggss3.62}) and (\ref{TP-NP-11}) that
 \begin{eqnarray}\label{wwggss3.62-3}
  \|u^*_k\|_{L^\infty(\mathbb R^+;U)}   < M.
 \end{eqnarray}
   Since $k$ was arbitrarily taken from $\mathbb N^+$, by (\ref{wwggss3.62-2}) and (\ref{wwggss3.62-3}),   $\{u^*_k\}_{k=1}^\infty$ is a sequence of  minimal time controls to $(TP)^{M,y_0}$. (Each $u^*_k$  is not a bang-bang control, see (\ref{wwggss3.62-3}).)

    Finally, we will show that any finite number of controls in $\{u^*_k\}_{k=1}^\infty$ are linearly independent
    in $L^\infty(\mathbb{R}^+;U)$. Suppose that there are a finite subsequence  $\{u^*_{k_j}\}_{j=1}^N$ of
    $\{u^*_{k}\}_{k=1}^\infty$ and a sequence $\{\alpha_j\}_{j=1}^N \subset \mathbb R$ so that
    \begin{equation}\label{ggssww3.64}
   \sum_{j=1}^N \alpha_j u^*_{k_j}=0.
    \end{equation}
    We aim to show that
    \begin{eqnarray}\label{ggssww3.64-1}
     \alpha_j=0   \;\;\mbox{for each}\;\;    j\in \{1,2,\cdots,N\}.
    \end{eqnarray}
    By (\ref{ggssww3.64}) and (\ref{wwggss3.62-1}), it follows that
    \begin{equation}\label{ggssww3.64-2}
   \sum_{j=1}^N \alpha_j\varepsilon_{k_j} \hat u_{k_j} +\Big(\sum_{j=1}^N \alpha_j\Big)  \chi_{(0,T(M,y_0))} u^*=0.
    \end{equation}
    Since $\hat u_{k_1},\dots, \hat u_{k_N}$ are linearly independent, we see from (\ref{ggssww3.64-2}) that, to show (\ref{ggssww3.64-1}), it suffices to prove that
    \begin{eqnarray}\label{ggssww3.64-3}
     \sum_{j=1}^N \alpha_j  =0.
    \end{eqnarray}
    By contradiction, suppose that (\ref{ggssww3.64-3}) were not true. Then we would have
    \begin{eqnarray}\label{ggssww3.64-4}
     \sum_{j=1}^N \alpha_j\neq 0.
    \end{eqnarray}
    By (\ref{ggssww3.64-4}) and (\ref{ggssww3.64-2}), we know that $\chi_{(0,T(M,y_0))}u^*$ is a linear combination of $\hat u_{k_1},\cdots,\hat u_{k_N}$. This, along with (\ref{wwggss3.59}) and (\ref{wwggss3.61}), yields that
    \begin{eqnarray}
     y(T^0(y_0); 0,u^*)=y(T(M,y_0); 0,u^*)=0,
    \end{eqnarray}
    which, together with (\ref{wwggss3.59-1}), implies that
    \begin{eqnarray}\label{ggssww3.64-5}
     y(T^0(y_0);y_0,0)
     =  y(T^0(y_0);y_0,u^*) - y(T^0(y_0);0,u^*)=0.
    \end{eqnarray}
    Notice that $T^0(y_0)\in (0,\infty)$ (see (\ref{wwggss3.59})). So the problem $(NP)^{T^0(y_0),y_0}$ is well defined. Then by (\ref{NP-0}) and (\ref{ggssww3.64-5}), we see that
   $ N(T^0(y_0),y_0)=0$.
       By this, we can use (iv) of Lemma \ref{Lemma-NT0-T0-T1} to find that
   $T^0(y_0)=T^1(y_0)$,
       which contradicts (\ref{infinite-1201-1}). So (\ref{ggssww3.64-3}) is true and then any finite number of controls in $\{u^*_k\}_{k=1}^\infty$ are linearly independent
    in $L^\infty(\mathbb{R}^+;U)$. We end the proof of this theorem.

\end{proof}

\begin{Theorem}\label{Theorem-TP-infinite-op-1}
Suppose that (H1) holds. Let $y_0\in X\setminus\{0\}$ satisfy that
 \begin{eqnarray}\label{infinite-1203-1}
  T^0(y_0)=T^1(y_0)<\infty.
 \end{eqnarray}
Then for each $M\in(0,\infty)$, $(TP)^{M,y_0}$  has infinitely many different minimal time controls so that among them, any finite number of controls are linearly independent in $L^\infty(\mathbb R^+;U)$.

\end{Theorem}

\begin{proof}
Arbitrarily fix a $y_0\in X\setminus\{0\}$ so that (\ref{infinite-1203-1}) holds. Let  $M\in(0,\infty)$. Then by (\ref{infinite-1203-1}), we can use Corollary \ref{wangbuchoutheorem3.13-3} to see that
\begin{eqnarray}\label{infinite-1203-2}
   0<T^1(y_0)=T(M,y_0)=T^0(y_0)<\infty
 \end{eqnarray}
 and to find that the null control is a minimal time control to $(TP)^{M,y_0}$, i.e.,
 \begin{eqnarray}\label{infinite-1203-2-1}
  y(T(M,y_0);y_0,0)=0.
 \end{eqnarray}

 Meanwhile,  since $0<T^0(y_0)<\infty$ (see (\ref{infinite-1203-2})), by (H1), we can use Lemma  \ref{infinite-T} to find a sequence
$\{u_k\}_{k=1}^\infty  \subset  L^\infty(0, T^0(y_0);U)$ so that
\begin{equation}\label{infinite-1203-3}
   \hat y(T^0(y_0); 0,u_k)=0\;\;\mbox{for all}\;\;    k\in\mathbb N^+,
\end{equation}
and so that any finite number of elements in $\{u_k\}_{k=1}^\infty$ are linearly independent in the space
$L^\infty(0, T^0(y_0);U)$. Write $\hat u_k$, $k=1,2,\cdots$, for the zero extension of $u_k$ over $\mathbb{R}^+$. Then any finite number of elements in $\{\hat u_k\}_{k=1}^\infty$ are linearly independent in
$L^\infty(\mathbb R^+;U)$.
Arbitrarily fix a $k\in\mathbb N^+$. It follows from (\ref{infinite-1203-2}) and (\ref{infinite-1203-3}) that
\begin{equation}\label{infinite-1203-4}
   y(T(M,y_0); 0, \hat u_k)=0.
\end{equation}
Since $M>0$, we can take  $\varepsilon_k>0$ so that
  \begin{equation}\label{infinite-1203-5}
  \varepsilon_k \|\hat u_k\|_{L^\infty(\mathbb R^+;U)}< M.
  \end{equation}

Next, we define a control $u^*_k$ in the following manner:
 \begin{eqnarray}\label{infinite-1203-6}
  u^*_k\triangleq \varepsilon_k \hat u_k
  \;\;\mbox{over}\;\;   \mathbb R^+.
 \end{eqnarray}
 Then by (\ref{infinite-1203-6}),  (\ref{infinite-1203-2-1}) and (\ref{infinite-1203-4}), we find that
 \begin{eqnarray}\label{infinite-1203-7}
  y(T(M,y_0);y_0,u^*_k)= y(T(M,y_0);y_0,0)  +  \varepsilon_k y(T(M,y_0);0,\hat u_k)=0.
 \end{eqnarray}
 Meanwhile, by (\ref{infinite-1203-6})  and (\ref{infinite-1203-5}), we see that
 \begin{eqnarray}\label{infinite-1203-8}
  \|u^*_k\|_{L^\infty(\mathbb R^+;U)}   < M.
 \end{eqnarray}
   Since $k$ was arbitrarily taken from $\mathbb N^+$, it follows by (\ref{infinite-1203-7}) and (\ref{infinite-1203-8}) that
    for each $k\in \mathbb N^+$, $u^*_k$ is a   minimal time control to $(TP)^{M,y_0}$ and has no the  bang-bang property.

 Finally, we will show that any finite number of controls in $\{u^*_k\}_{k=1}^\infty$ are linearly independent
    in $L^\infty(\mathbb{R}^+;U)$. Here is the argument: Suppose that  there are a finite subsequence  $\{u^*_{k_j}\}_{j=1}^N$ of
    $\{u^*_{k}\}_{k=1}^\infty$ and a sequence $\{\alpha_j\}_{j=1}^N \subset \mathbb R$ so that
    \begin{equation}\label{infinite-1203-9}
   \sum_{j=1}^N \alpha_j u^*_{k_j}=0.
    \end{equation}
    Then we will have  that
    \begin{eqnarray}\label{infinite-1203-10}
     \alpha_j=0   \;\;\mbox{for each}\;\;    j\in \{1,2,\cdots,N\}.
    \end{eqnarray}
    Indeed, by (\ref{infinite-1203-9}) and (\ref{infinite-1203-6}), it follows that
    \begin{equation}\label{infinite-1203-11}
   \sum_{j=1}^N \alpha_j\varepsilon_{k_j} \hat u_{k_j}   =0.
    \end{equation}
    Since $\hat u_{k_1},\dots, \hat u_{k_N}$ are linearly independent,
    we find that for each $j\in \{1,2,\cdots,N\}$, $\alpha_j\varepsilon_{k_j}=0$.
    Because $\{\varepsilon_{k_j}\}_{j=1}^N \subset (0,\infty)$,   we see that
    (\ref{infinite-1203-10}) holds. So any finite number of controls in $\{u^*_k\}_{k=1}^\infty$ are linearly independent   in $L^\infty(\mathbb{R}^+;U)$. This ends the proof.

\end{proof}

\subsection{Proofs of the main theorems}

We begin with the proof of Theorem \ref{Proposition-NTy0-partition}, which gives the BBP decomposition for $(NP)^{T,y_0}$.

\begin{proof}[Proof of Theorem \ref{Proposition-NTy0-partition}]
(i) First of all, we observe from (\ref{zhangjinchu1}) and (\ref{zhangjinchu5})-(\ref{zhangjinchu3}) that
\begin{eqnarray}\label{zhangjinchu2}
 \mathcal{W}=\mathcal{W}_1 \cup \mathcal{W}_2 \cup \mathcal{W}_3,
\end{eqnarray}
\begin{eqnarray}\label{zhangjinchu4}
\mathcal{W}_1=\mathcal{W}_{1,1} \cup \mathcal{W}_{1,2},
\end{eqnarray}
\begin{eqnarray}\label{zhangjinchu6}
\mathcal{W}_2=\mathcal{W}_{2,1} \cup \mathcal{W}_{2,2} \cup \mathcal{W}_{2,3} \cup \mathcal W_{2,4},
\end{eqnarray}
and
\begin{eqnarray}\label{zhangjinchu8}
\mathcal{W}_3=\mathcal{W}_{3,1} \cup \mathcal{W}_{3,2} \cup \mathcal{W}_{3,3} \cup \mathcal W_{3,4}.
\end{eqnarray}
To prove the conclusion (i), it suffices to show that
 \begin{eqnarray}\label{zhangjinchu9-1}
 \mathcal{W}=\big(\cup_{j=1}^2 \mathcal{W}_{1,j}\big)  \cup \big(\cup_{j=1}^4 \mathcal{W}_{2,j}\big)
  \cup \big(\cup_{j=1}^4 \mathcal{W}_{3,j}\big)
\end{eqnarray}
and
\begin{equation}\label{zhangjinchu10}
\mathcal{W}_{i,j}\cap \mathcal{W}_{i^\prime,j^\prime}=\emptyset,
\;\;\mbox{when}\;\;(i,j)\neq(i^\prime,j^\prime).
\end{equation}
     The  equality  (\ref{zhangjinchu9-1}) follows from    (\ref{zhangjinchu2}),  (\ref{zhangjinchu4}), (\ref{zhangjinchu6}) and  (\ref{zhangjinchu8}) at once. To show  (\ref{zhangjinchu10}),  three observations are given in order:  First, from (\ref{zhangjinchu3-1}), (\ref{zhangjinchu3-2}) and (\ref{zhangjinchu3}), we see
   that  $\mathcal{W}_1$, $\mathcal{W}_2$ and $\mathcal{W}_3$  are pairwise disjoint; Second,
    from (\ref{zhangjinchu5}), it follows that $\mathcal{W}_{1,1}$ and  $\mathcal{W}_{1,2}$
    are disjoint; Third,
              by  (\ref{zhangjinchu7}) and (\ref{zhangjinchu9}),
      we see respectively that all $\mathcal{W}_{2,j}$, $j=1,2,3,4$ are pairwise disjoint, and that all
    $\mathcal{W}_{3,j}$, $j=1,2,3,4$ are pairwise disjoint.
    The above three observations, together with
    (\ref{zhangjinchu4}), (\ref{zhangjinchu6}) and  (\ref{zhangjinchu8}), leads to
    (\ref{zhangjinchu10}). Thus, we end the proof of the conclusion (i).

 (ii)  First, we let $(T,y_0)\in \mathcal{W}_{1,2}$. Then by the definitions of $\mathcal{W}_{1,2}$ and $\mathcal{W}_{1}$ (see (\ref{zhangjinchu5}) and (\ref{zhangjinchu3-1}), respectively), we have that
 \begin{equation}\label{zhangjinchu12}
   T^0(y_0) \leq T<\infty  \;\;\mbox{and}\;\; N(T^0(y_0), y_0)=0.
 \end{equation}
 By the last equation in (\ref{zhangjinchu12}), we can use
 (iv) of Lemma \ref{Lemma-NT0-T0-T1} to obtain that
 $T^0(y_0)= T^1(y_0)<\infty$.
 From this  and the first inequality in (\ref{zhangjinchu12}),
  we can apply
  (ii) of Corollary \ref{wangpuchou1-1}  to see that the null control is the unique minimal norm control to
   $(NP)^{T,y_0}$.
   This, along with (\ref{NP-0}), yields that $N(T,y_0)=0$.
   Hence,
  $(NP)^{T,y_0}$  has the bang-bang property.

     Next, we let $(T,y_0)\in \mathcal{W}_{2,4}$. Then by the definitions of $\mathcal{W}_{2,4}$ and $\mathcal{W}_{2}$ (see (\ref{zhangjinchu7}) and (\ref{zhangjinchu3-2}), respectively), we have that
 \begin{equation}\label{zhangjinchu12-1}
   T^1(y_0) \leq T<\infty  \;\;\mbox{and}\;\; 0<N(T^0(y_0), y_0)<\infty.
 \end{equation}
 By the last equation in (\ref{zhangjinchu12-1}), we can apply the conclusion (iii) in Lemma \ref{Lemma-NT0-T0-T1} to obtain that $T^0(y_0)< T^1(y_0)$.
From this and the first inequality in (\ref{zhangjinchu12-1}),
  we can apply
  (ii) of Corollary \ref{wangpuchou1} to see that the null control is the unique minimal norm control to
   $(NP)^{T,y_0}$. This, along with (\ref{NP-0}), yields that $N(T,y_0)=0$.
   Hence,
  $(NP)^{T,y_0}$  has the bang-bang property.

   Finally, we let $(T,y_0)\in \mathcal{W}_{3,3}$. Then by the definitions of $\mathcal{W}_{3,3}$ and $\mathcal{W}_{3}$ (see (\ref{zhangjinchu9}) and (\ref{zhangjinchu3}), respectively), we have that
 \begin{equation}\label{zhangjinchu12-2}
   T^1(y_0) \leq T<\infty  \;\;\mbox{and}\;\; N(T^0(y_0), y_0)=\infty.
 \end{equation}
 By  (\ref{zhangjinchu12-2}), we can use
 (i) of Lemma \ref{Lemma-NT0-T0-T1} to obtain that
$T^0(y_0)< T^1(y_0)$. From this and  the first inequality in (\ref{zhangjinchu12-2}),
  we can apply
  (ii) of Corollary \ref{wangpuchou1} to see that the null control is the unique minimal norm control to
   $(NP)^{T,y_0}$.  This, along with (\ref{NP-0}), yields that $N(T,y_0)=0$.
   Hence,
  $(NP)^{T,y_0}$  has the bang-bang property. This ends the proof of the conclusion (ii).

(iii)  First, we let  $(T,y_0)\in \mathcal{W}_{2,3}$. Then by  the definition of $\mathcal{W}_{2,3}$  (see (\ref{zhangjinchu7})), we have that
  $T^0(y_0) < T<T^1(y_0)$. From this and  and the assumptions (H1)-(H2), we can apply   Theorem \ref{yubiaotheorem5.6}
         to find that  $(NP)^{T,y_0}$ has the bang-bang property. The remainder is to show that the null control is not a minimal norm control to $(NP)^{T,y_0}$.
          In fact, since $T^0(y_0) < T<T^1(y_0)$, it follows from (iii) of Lemma \ref{Lemma-T0-T1} that  $N(T,y_0)>0$, from which, we see that the null control is not a minimal norm control to $(NP)^{T,y_0}$.

         Next, we let  $(T,y_0)\in \mathcal{W}_{3,2}$. By the definition of $\mathcal{W}_{3,2}$  (see (\ref{zhangjinchu9})), we find that
         $T\in (T^0(y_0), T^1(y_0))$. Then by the same way as that used for the above case that $(T,y_0)\in \mathcal{W}_{2,3}$, we see that
         $(NP)^{T,y_0}$ has the bang-bang property and the null control is not its minimal norm control.
         This ends the proof of the conclusion (iii).

 \vskip 5pt
 (iv) First we let $(T,y_0)\in \mathcal{W}_{1,1}$. Then by the definitions of $\mathcal{W}_{1,1}$ and $\mathcal{W}_{1}$ (see (\ref{zhangjinchu5}) and
 (\ref{zhangjinchu3-1}), respectively), we have that
 \begin{equation}\label{zhangjinchu18}
 0<T<T^0(y_0)\;\;\mbox{and}\;\; N(T^0(y_0),y_0)=0.
 \end{equation}
 From the last equation in (\ref{zhangjinchu18}), we can apply
   (iv) of Lemma \ref{Lemma-NT0-T0-T1} to see  that $T^0(y_0)=T^1(y_0)<\infty$.
   This, together with the first inequality in (\ref{zhangjinchu18}), yields that
   \begin{equation}\label{zhangjinchu19}
 T^0(y_0)=T^1(y_0)\;\;\mbox{and}\;\; T\in \big(0, T^0(y_0)\big).
 \end{equation}
   From (\ref{zhangjinchu19}), we can use  (i) of Corollary \ref{wangpuchou1-1} to find that $(NP)^{T,y_0}$ has no any admissible control and so does not hold the  bang-bang property.

   Next we let $(T,y_0)\in \mathcal{W}_{2,1}$. Then by the definitions of $\mathcal{W}_{2,1}$ and $\mathcal{W}_{2}$ (see (\ref{zhangjinchu7}) and (\ref{zhangjinchu3-2}), respectively), we have that
    \begin{equation}\label{zhangjinchu20}
 0<T<T^0(y_0)\;\;\mbox{and}\;\; 0<N(T^0(y_0),y_0)<\infty.
 \end{equation}
    By the second inequality in (\ref{zhangjinchu20}), we can use
         (iii) of Lemma \ref{Lemma-NT0-T0-T1} to get that $T^0(y_0)<T^1(y_0)$.
      This, along with the first inequality in    (\ref{zhangjinchu20}), yields that
         \begin{equation}\label{zhangjinchu21}
 0<T^0(y_0)<T^1(y_0)\;\;\mbox{and}\;\; 0<T<T^0(y_0).
 \end{equation}
     From    (\ref{zhangjinchu21}), we can use  (i) of Corollary \ref{wangpuchou1}
     to get that $(NP)^{T,y_0}$ has no any admissible control and so does not hold   bang-bang property.

     We now let  $(T,y_0)\in \mathcal{W}_{3,1}$. Then by the definitions of $\mathcal{W}_{3,1}$ and $\mathcal{W}_{3}$ (see (\ref{zhangjinchu9}) and (\ref{zhangjinchu3}), respectively), we see that
      \begin{equation}\label{zhangjinchu22}
 T^0(y_0)<\infty,\;\; 0<T\leq T^0(y_0)\;\;\mbox{and}\;\; N(T^0(y_0),y_0)=\infty.
 \end{equation}
      By the first inequality and the last equality in (\ref{zhangjinchu22}), we can use (i) of Lemma \ref{Lemma-NT0-T0-T1} to find that $T^0(y_0)<T^1(y_0)$. This, together with (\ref{zhangjinchu22}), indicates that
      \begin{equation}\label{zhangjinchu23}
 0<T^0(y_0)<T^1(y_0),\;\; 0<T\leq T^0(y_0)\;\;\mbox{and}\;\; N(T^0(y_0),y_0)=\infty.
 \end{equation}
  In the case that $T=T^0(y_0)$, from the first inequality in (\ref{zhangjinchu23}), we can use (i) of Corollary \ref{wangpuchou1} to see that $(NP)^{T,y_0}$ has no any admissible control and so does not have the bang-bang property. In the case when $T<T^0(y_0)$, from the last equality in (\ref{zhangjinchu23}), we can apply
  (v) of Theorem \ref{Lemma-existence-optimal-control}  to find that
   $(NP)^{T,y_0}$ has no any admissible control and so does not have the bang-bang property.

   Finally, we let $(T,y_0)\in \mathcal{W}_{3,4}$. Then by the definitions of $\mathcal{W}_{3,4}$ and $\mathcal{W}_{3}$  (see (\ref{zhangjinchu9})
   and (\ref{zhangjinchu3}), respectively), we have that
    \begin{equation}\label{zhangjinchu24}
 0<T<\infty,\;\;T^0(y_0)=\infty\;\;\mbox{and}\;\; N(T^0(y_0),y_0)=\infty.
 \end{equation}
   By the last two equalities in (\ref{zhangjinchu24}), we can  use (i) of Lemma \ref{Lemma-NT0-T0-T1} to see that
   $T^0(y_0)=T^1(y_0)=\infty$,
       which,
   along with the first inequality in (\ref{zhangjinchu24}), yields that
    $T^0(y_0)=T^1(y_0)$ and $0<T<T^0(y_0)$.
 From these, we can apply  (i) of Corollary \ref{wangpuchou1-1} to find that $(NP)^{T,y_0}$ has no any admissible control and so does not hold the    bang-bang property. This ends the proof of the conclusion (iv).

(v) Let $(T,y_0)\in \mathcal W_{2,2}$. Then by the definitions of $\mathcal W_{2,2}$ and $\mathcal W_{2}$ (see (\ref{zhangjinchu7}) and (\ref{zhangjinchu3-2}), respectively), we see that
$0<T=T^0(y_0)<\infty$ and $0<N(T^0(y_0),y_0)<\infty$.
From these, we can use (iii) of Theorem \ref{Lemma-existence-optimal-control} to see that $(NP)^{T,y_0}$ has at least one minimal norm control. This ends the proof of the conclusion (v).

   In summary, we finish the proof of Theorem \ref{Proposition-NTy0-partition}.

\end{proof}

Next, we prove Theorem \ref{Proposition-TMy0-partition}, which gives the BBP decompositions for $(TP)^{M,y_0}$.

\begin{proof}[Proof of Theorem \ref{Proposition-TMy0-partition}]
(i) First of all, we observe from (\ref{TP-divide-domain}) and (\ref{Lambda-di-0-1-1})-(\ref{Lambda-di-0-1}) that
\begin{eqnarray}\label{Lambda-di-0}
\mathcal V  =    \mathcal V_1 \cup \mathcal V_2 \cup \mathcal V_3,
\end{eqnarray}
\begin{eqnarray}\label{Lambda-di-2}
\mathcal V_2  =    \mathcal V_{2,1} \cup \mathcal V_{2,2} \cup \mathcal V_{2,3}
 \cup   \mathcal V_{2,4},
\end{eqnarray}
and
\begin{eqnarray}\label{Lambda-di-3}
\mathcal V_3  =    \mathcal V_{3,1} \cup \mathcal V_{3,2} \cup \mathcal V_{3,3}.
\end{eqnarray}
To show the conclusion (i), it suffices to verify that
\begin{eqnarray}\label{TP-1204-1}
 \mathcal V  =   \mathcal V_{1}  \cup  (\cup_{j=1}^4 \mathcal V_{2,j})
 \cup (\cup_{j=1}^3 \mathcal V_{3,j})
\end{eqnarray}
 and
\begin{equation}\label{TP-1204-2}
\mathcal V_1 \cap \mathcal V_{i,j}=\emptyset,\;\;
\mathcal{V}_{i^\prime,j^\prime}\cap \mathcal{V}_{i^{\prime\prime},j^{\prime\prime}}=\emptyset
\;\;\mbox{when}\;\;(i^\prime,j^\prime)\neq(i^{\prime\prime},j^{\prime\prime}).
\end{equation}
First of all, the equality   (\ref{TP-1204-1}) follows from    (\ref{Lambda-di-0}),   (\ref{Lambda-di-2}) and  (\ref{Lambda-di-3}) at once. To prove (\ref{TP-1204-2}),  three observations are given in order: First, from (\ref{Lambda-di-0-1-1}), (\ref{Lambda-di-0-1-2}) and (\ref{Lambda-di-0-1}), we see
   that  $\mathcal{V}_1$, $\mathcal{V}_2$ and $\mathcal{V}_3$ are pairwise disjoint. Second,
    from (\ref{Lambda-di-2-1}), we find that all $\mathcal{V}_{2,j}$, $j=1,2,3,4$, are pairwise disjoint. Third,  from (\ref{Lambda-di-3-1}), we find that all
    $\mathcal{V}_{3,j}$,  $j=1,2,3$, are pairwise disjoint. The above three observations, along with
     (\ref{Lambda-di-2}) and  (\ref{Lambda-di-3}), yield
    (\ref{TP-1204-2}). Thus, we end the proof of the conclusion (i).

 (ii) First we let $(M,y_0)\in \mathcal{V}_{2,2}$. By the definitions of $\mathcal{V}_{2,2}$ and $\mathcal V_2$ (see (\ref{Lambda-di-2-1}) and (\ref{Lambda-di-0-1-2})), we have that
 \begin{equation}\label{newyear6.64}
  N(T^1(y_0), y_0) < M < N(T^0(y_0), y_0)
 \;\;\mbox{and}\;\;
  0< N(T^0(y_0), y_0) <\infty.
 \end{equation}
 By the second inequality in (\ref{newyear6.64}), we can use  (iii) of Lemma \ref{Lemma-NT0-T0-T1} to see that
 $T^0(y_0)<T^1(y_0)$.
 By this, the first inequality in (\ref{newyear6.64}) and the assumptions  (H1)-(H2), we can apply
   Theorem \ref{Theorem-time-variant-bangbang} to see that
  $(TP)^{M,y_0}$ has the bang-bang property.

  Next, we let  $(M,y_0)\in \mathcal{V}_{3,2}$. By the definitions of $\mathcal{V}_{3,2}$ and $\mathcal{V}_{3}$ (see (\ref{Lambda-di-3-1}) and (\ref{Lambda-di-0-1})), we find that
    \begin{equation}\label{newyear6.65}
    T^0(y_0)<\infty
    \;\;\mbox{and}\;\;   N(T^1(y_0), y_0) < M < \infty=N(T^0(y_0), y_0).
    \end{equation}
    From (\ref{newyear6.65}), we can use (i) of Lemma \ref{Lemma-NT0-T0-T1} to get that $T^0(y_0)<T^1(y_0)$.
    By this, the second conclusion in (\ref{newyear6.65}) and the assumptions (H1) and (H2),
         we can apply
   Theorem \ref{Theorem-time-variant-bangbang} to see that
  $(TP)^{M,y_0}$ has the bang-bang property.  This ends the proof of the conclusion (ii).

 (iii)  Let $(M,y_0)\in \mathcal V_{2,4}$. By the definitions of  $\mathcal V_{2,4}$ and $\mathcal V_2$ (see (\ref{Lambda-di-2-1}) and (\ref{Lambda-di-0-1-2})), we find that
    \begin{equation}\label{newyear6.66}
    N(T^0(y_0), y_0) < M < \infty
 \;\;\mbox{and}\;\;
  0< N(T^0(y_0), y_0) <\infty.
    \end{equation}
    From the second inequality in (\ref{newyear6.66}), we can use (iii) of Lemma \ref{Lemma-NT0-T0-T1} to see that
 \begin{eqnarray}\label{TP-1204-6}
  T^0(y_0)<T^1(y_0).
 \end{eqnarray}
 By  (\ref{TP-1204-6}) and  (\ref{newyear6.66}), we can use
  (i) and (ii) of Corollary \ref{wangbuchoutheorem3.13} to find  respectively
   that
   \begin{equation}\label{Newnew0202}
   T(M,y_0)=T^0(y_0)\in(0,\infty),
   \end{equation}
   and  that
  $(TP)^{M,y_0}$ has a minimal time control $u^*$ so that
  $u^*|_{(0,T^0(y_0))}$ is a minimal norm control to $(NP)^{T^0(y_0),y_0}$.
  The  later, together with (\ref{Newnew0202}) and the first inequality in (\ref{newyear6.66}), indicates that
       $$
     \|u^*\|_{L^\infty(0,T(M,y_0);U)}=\|u^*\|_{L^\infty(0,T^0(y_0);U)}
     =
    N(T^0(y_0),y_0)<M.
    $$
    This implies that $(TP)^{M,y_0}$ does not hold the bang-bang property.

   Meanwhile, according to  (iii) of Corollary \ref{wangbuchoutheorem3.13},
   the null control is not a minimal time control to    $(TP)^{M,y_0}$.

 The remainder is to show that $(TP)^{M,y_0}$ has infinitely many different minimal time controls. Fortunately, this follows from Theorem \ref{Theorem-TP-infinite-op},
 since we already have  (\ref{TP-1204-6}), (\ref{newyear6.66}) and (H1).
  This ends the proof of the conclusion (iii).

 (iv)  Let $(M,y_0)\in \mathcal V_{1}$. By the definition of  $\mathcal V_{1}$ (see (\ref{Lambda-di-0-1-1})), we find that
    \begin{equation}\label{TP-1204-7}
      N(T^0(y_0), y_0)=0 < M < \infty.
    \end{equation}
 Since $N(T^0(y_0), y_0)=0$, it follows from (iv) of Lemma \ref{Lemma-NT0-T0-T1} that
 \begin{equation}\label{TP-1204-7-1}
      T^0(y_0)=T^1(y_0)<\infty.
 \end{equation}
 By (\ref{TP-1204-7-1}) and (\ref{TP-1204-7}), we can use (ii) of Corollary \ref{wangbuchoutheorem3.13-3} to see that the null control is a minimal time control to $(TP)^{M,y_0}$.
  From this, we see that $(TP)^{M,y_0}$ does not hold the bang-bang property, since  $M>0$.

  The remainder is to show that $(TP)^{M,y_0}$ has infinitely many different minimal time controls. Fortunately, this follows from Theorem \ref{Theorem-TP-infinite-op-1}, since we already have (\ref{TP-1204-7-1}) and (H1).
 This ends the proof of the conclusion (iv).

 (v)  First, we let  $(M,y_0)\in \mathcal{V}_{3,3}$. Then by the definitions of $\mathcal{V}_{3,3}$  and $\mathcal{V}_{3}$ (see (\ref{Lambda-di-3-1}) and (\ref{Lambda-di-0-1})), we find that
      \begin{equation}\label{newyear6.73}
      T^0(y_0)=\infty
      \;\;\mbox{and}\;\; N(T^0(y_0),y_0)=\infty.
      \end{equation}
      From (\ref{newyear6.73}), we can use  (i) of Lemma \ref{Lemma-NT0-T0-T1}
      to see  that
      $       T^0(y_0)=T^1(y_0)=\infty$.
             By this,  we can apply   Corollary \ref{wangbuchoutheorem3.13-2} to find that $(TP)^{M,y_0}$ has no any admissible control and so does not hold the  bang-bang property.

 Next, we let $(M,y_0)\in \mathcal{V}_{2,1}$. By the definitions of  $\mathcal{V}_{2,1}$ and $\mathcal{V}_{2}$ (see (\ref{Lambda-di-2-1}) and
 (\ref{Lambda-di-0-1-2})), we have that
 \begin{equation}\label{newyear6.71}
 0< M \leq N(T^1(y_0), y_0)
 \;\;\mbox{and}\;\;
 0< N(T^0(y_0), y_0)<\infty.
 \end{equation}
 By the second inequality in (\ref{newyear6.71}), we can use (iii) of Lemma \ref{Lemma-NT0-T0-T1} to see that
 $ T^0(y_0)<T^1(y_0)$.
 From this, the first inequality in (\ref{newyear6.71}) and the assumption (H1), we can apply  (ii) of Corollary \ref{wangbuchoutheorem3.13-1} to find that $(TP)^{M,y_0}$ has no any admissible control and so does not hold the bang-bang property.

 Finally, we let $(M,y_0)\in \mathcal{V}_{3,1}$. By the definitions of   $\mathcal{V}_{3,1}$ and $\mathcal{V}_{3}$ (see
 (\ref{Lambda-di-3-1}) and (\ref{Lambda-di-0-1})), we have that
 \begin{equation}\label{newyear6.72}
 T^0(y_0)<\infty,\;\;
  0< M \leq N(T^1(y_0), y_0)
  \;\;\mbox{and}\;\; N(T^0(y_0), y_0)=\infty.
 \end{equation}
 By the last equality and the first inequality in (\ref{newyear6.72}), we can use (i) of Lemma \ref{Lemma-NT0-T0-T1} to get that
 $ T^0(y_0)<T^1(y_0)$.
 From this, the second inequality in (\ref{newyear6.72}) and the assumption (H1),
   we can use  (ii) of Corollary \ref{wangbuchoutheorem3.13-1} to find that $(TP)^{M,y_0}$ has no any admissible control and so does not hold the bang-bang property.
     This ends the proof of the conclusion (v).

(vi) Let $(T,y_0)\in \mathcal V_{2,3}$. Then by the definitions of $\mathcal V_{2,3}$ and $\mathcal V_{2}$ (see (\ref{Lambda-di-2-1}) and
 (\ref{Lambda-di-0-1-2})), we see that
$ 0<M=N(T^0(y_0),y_0)<\infty$.
This, along with (iii) of Lemma \ref{Lemma-NT0-T0-T1}, yields that
$T^0(y_0)<T^1(y_0)$ and    $N(T^0(y_0),y_0)=M<\infty$.
From these, we can use (ii) of Corollary \ref{wangbuchoutheorem3.13} to find that $(TP)^{M,y_0}$ has at least one minimal time control. This ends the proof of the conclusion (vi).

   In summary, we finish the proof of Theorem \ref{Proposition-TMy0-partition}.

\end{proof}

We end this section with proving  Theorem \ref{Proposition-Corollary1.7}.
To do it, we need three propositions.
 The first one is the following Proposition~\ref{newhuangproposition2.15}. It presents some equivalent conditions for the $L^\infty$-null controllability of $(A,B)$. Though there have been many literatures on such issue,
 we do not find the exactly same version of Proposition~\ref{newhuangproposition2.15} in literatures.
 For the sake of the completeness of the paper,  we provide the detailed proof in Appendix F.
 \begin{Proposition}\label{newhuangproposition2.15}
The following conclusions are equivalent:

\noindent(i) The pair $(A^*, B^*)$ is  $L^1$-observable, i.e., the condition (H3) holds, i.e., for each $T\in(0,\infty)$, there exists a positive constant $C_1(T)$
  so that
\begin{eqnarray}\label{ob-control-eq}
  \|S^*(T)z\|_X\leq C_1(T) \int_0^T \|B^*S^*(T-t)z\|_U  \,\mathrm dt
  \;\;\mbox{for all}\;\; z\in D(A^*).
 \end{eqnarray}

\noindent(ii) The pair $(A,B)$ has the $L^\infty$-null controllability with a cost,  i.e., for each $T\in(0,\infty)$, there is a positive constant $C_2(T)$ so that for each $y_0\in X$, there exists a control $v\in L^\infty(0,T;U)$ satisfying that
 \begin{equation}\label{fangyan2.90}
 \hat y(T;y_0,v)=0\;\;\mbox{and}\;\;  \|v\|_{L^\infty(0,T;U)} \leq C_2(T)\|y_0\|_X.
 \end{equation}

\noindent(iii) The pair $(A,B)$ is $L^\infty$-null controllable, i.e., for each $T\in(0,\infty)$ and each $y_0\in X$, there exists a control $v\in L^\infty(0,T;U)$ so that $\hat y(T;y_0,v)=0$.

Furthermore, when one of the above three conclusions is valid,  the  constants $C_1(T)$ in (\ref{ob-control-eq}) and $C_2(T)$ in (\ref{fangyan2.90}) can be
taken as the same number.

\end{Proposition}
The next two propositions concern some connections among assumptions (H1)-(H4).
 \begin{Proposition}\label{WGSWGSlemma2.15}
Suppose that (H3)  holds. Then (H1) is true.
\end{Proposition}
\begin{proof}
 Suppose that (H3) holds. Arbitrarily fix $T$ and $t$ so that  $0<t<T<\infty$. Then by (H3), there exists a positive number $C_1(T-t)$ (depending on $(T-t)$) so that
\begin{eqnarray*}
 \|S^*(T-t)z\| \leq C_1(T-t) \int_0^{T-t} \|B^*S^*(T-t-s)z\|_U \,\mathrm ds
 \;\;\mbox{for all}\;\;   z\in D(A^*),
\end{eqnarray*}
which implies that
\begin{eqnarray*}
 \|S^*(T-t)z\| \leq C_1(T-t) \int_t^T \|B^*S^*(T-s)z\|_U \,\mathrm ds
 \;\;\mbox{for all}\;\;   z\in D(A^*).
\end{eqnarray*}
This, together with (\ref{admissible-observable}), yields that for each $z\in D(A^*)$,
\begin{eqnarray*}
 \|B^* S^*(T-\cdot)z\|_{L^2(0,t;U)}  &=&  \big\|B^* S^*(t-\cdot)\big(S^*(T-t)z\big) \big\|_{L^2(0,t;U)}
 \nonumber\\
 &\leq& \sqrt{C(t)} \|S^*(T-t)z\|_X
 \nonumber\\
 &\leq& \sqrt{C(t)} C_1(T-t) \|B^* S^*(T-\cdot)z\|_{L^1(t,T;U)},
\end{eqnarray*}
where $C(t)$ is given by (\ref{admissible-observable}). Then by the definition of $Y_T$ (see (\ref{ob-space})), the above yields that
\begin{eqnarray}\label{louhongwei2.111}
 \|g\|_{L^2(0,t;U)} \leq \sqrt{C_1(t)} C(T-t)  \|g\|_{L^1(t,T;U)}
  \;\;\mbox{for all}\;\; g\in Y_T.
\end{eqnarray}
Notice that (\ref{louhongwei2.111}) is exactly the statement (iii) in Lemma \ref{Lemma-H3-eq}, where  $p_2=2$.
 Thus we can apply  Lemma \ref{Lemma-H3-eq} to get  the conclusion  (i) of Lemma \ref{Lemma-H3-eq} which is exactly the condition (H1).
 Hence,  (H1) follows from (H3). This ends the proof of this proposition.

\end{proof}

 \begin{Proposition}\label{Lemma-YT-bounded}
   Suppose that (H3) and (H4) are true. Then   (H2) holds.
 \end{Proposition}

 \begin{proof}
  Let $T\in (0,\infty)$. Suppose that  $f\in Y_T$ satisfies that
  \begin{equation}\label{wang2.62}
  f=0\;\;\mbox{over}\;\; E,
  \end{equation}
  where  the subset $E\subset (0,T)$ is  of positive measure.
  We are going to use (H3) and (H4) to  show that
  \begin{equation}\label{hongwei2.113}
 f=0\;\;\mbox{over}\;\; (0,T).
 \end{equation}
  When this is done, we obtain (H2) from (H3) and (H4).

  The rest is to show (\ref{hongwei2.113}).
  By    (\ref{ob-space}), there exists a sequence $\{z_n\}\subset D(A^*)$ so that
  \begin{eqnarray}\label{wang2.63}
   B^*S^*(T-\cdot)z_n \rightarrow f(\cdot) \mbox{ in } L^1(0,T;U),~\mbox{as}\;\;n \rightarrow \infty.
  \end{eqnarray}
 In particular, $\left\{B^*S^*(T-\cdot)z_n\right\}$ is a Cauchy sequence in $L^1(0,T;U)$.  Take a sequence $\{T_k\}\subset (0, T)$ so that $T_k\nearrow T$. Then by (H3), we find that for each $k$, $\{S^*(T-T_k)z_n\}$ is a Cauchy sequence in $X$. Hence, for each $k$, there is a $\hat z_k\in X$ so that
 \begin{equation}\label{wang2.64}
 S^*(T-T_k)z_n\rightarrow \hat z_k\;\;\mbox{strongly in}\;\; X,\;\;\mbox{as}\;\; n\rightarrow\infty.
 \end{equation}
 By (\ref{wang2.64}) and (\ref{admissible-observable}), we see that for each $k$, $\{B^*S^*(T-\cdot)z_n\}$ is a Cauchy sequence in $L^2(0,T_k;U)$. This, along with (\ref{wang2.64}) and (\ref{wang1.15}),
 indicates that for each $k$,
 \begin{equation}\label{wang2.65}
 B^*S^*(T-\cdot)z_n\rightarrow \widetilde{B^*S^*}(T_k-\cdot)\hat z_k\;\;\mbox{in}\;\; L^2(0,T_k;U), \;\;\mbox{as}\;\; n\rightarrow\infty.
 \end{equation}
 By (\ref{wang2.63}) and (\ref{wang2.65}), we find that for each $k$,
 \begin{equation}\label{wang2.66}
 f(\cdot)=\widetilde{B^*S^*}(T_k-\cdot)\hat z_k\;\;\mbox{over}\;\; (0,T_k).
 \end{equation}
 Since $T_k\nearrow T$, we see that for each $k$ large enough, $E_k\triangleq
 E\cap (0,T_k)$ has a positive measure. Then from (\ref{wang2.66}) and
 (\ref{wang2.62}), we observe that for each $k$ large enough,
 $$
 \widetilde{B^*S^*}(T_k-\cdot)\hat z_k=0\;\;\mbox{over}\;\; E_k.
 $$
 This, along with (H4), yields that for all $k$ large enough,
 \begin{equation}\label{wang2.67}
 \widetilde{B^*S^*}(T_k-\cdot)\hat z_k=0\;\;\mbox{over}\;\; (0,T_k).
 \end{equation}
Now, (\ref{hongwei2.113}) follows from  (\ref{wang2.66}) and (\ref{wang2.67}).
  This ends the proof.

 \end{proof}

 \begin{Remark}\label{YT-characterization}
  Since $Y_T$ is the completion of the space $X_T$ in the norm $\|\cdot\|_{L^1(0,T;U)}$ (see (\ref{ob-space})), it is hard to characterize elements of $Y_T$ in general. However, when the assumption (H3) holds, we have that  $Y_T= \mathcal Y_T$, where
  \begin{eqnarray*}
   \mathcal Y_T\triangleq\Big\{f\in L^1(0,T;U)~:~\forall\,t\in (0,T),\,\exists\,z^t\in X \mbox{ s.t. } f(\cdot)|_{(0,t)}=\widetilde{B^*S^*}(t-\cdot)z^t \Big\}.
  \end{eqnarray*}
  Indeed, on one hand, by (H3), we get (\ref{wang2.66}), from which, it follows  that $Y_T\subset\mathcal Y_T$. On the other hand, from (H1) and (ii) of Lemma \ref{Lemma-left-continuity-YT}, we find that $\mathcal Y_T\subset Y_T$. (Notice that (H1) is ensured by (H3), see Proposition \ref{WGSWGSlemma2.15}.)    For time varying systems, we do not know if these two spaces are the same in general.  (In the proof of Lemma \ref{Lemma-left-continuity-YT}, we used the time-invariance of the system.)

 \end{Remark}

\vskip 5pt

We now are on the position to show Theorem~\ref{Proposition-Corollary1.7}.

\begin{proof}[Proof of  Theorem \ref{Proposition-Corollary1.7}]

(i)  We first claim that
\begin{eqnarray}\label{no-critical-point-1}
 T^0(y_0)=0 \;\;\mbox{and}\;\; N(T^0(y_0),y_0)=\infty\;\;\mbox{for all}\;\;y_0\in X\setminus\{0\}.
\end{eqnarray}
Indeed, by (H3), we can use Proposition~\ref{newhuangproposition2.15} to get the $L^\infty$-null
controllability for  $(A,B)$, which, along with the definition of $T^0(\cdot)$ (see (\ref{y0-controllable})), yields
the first equality in (\ref{no-critical-point-1}). This, together with (iv) of Lemma \ref{Lemma-T0-T1}, leads to the second equality in (\ref{no-critical-point-1}).

We next claim that
\begin{equation}\label{newyear6.105}
\mathcal{W}=\mathcal{W}_{3,2} \cup \mathcal{W}_{3,3}.
\end{equation}
In fact, by the second equality in (\ref{no-critical-point-1}) and the definition of $\mathcal{W}_1$ and $\mathcal{W}_2$ (see (\ref{zhangjinchu3-1}) and (\ref{zhangjinchu3-2})), we find that $\mathcal{W}_1\cup\mathcal{W}_2=\emptyset$. Meanwhile, by the first  equality in (\ref{no-critical-point-1}) and the definitions of  $\mathcal W_{3,1}$ and $\mathcal W_{3,4}$ (see (\ref{zhangjinchu9})), we find that $\mathcal W_{3,1}  \cup  \mathcal W_{3,4}=\emptyset$. These, along with (i) of Theorem \ref{Proposition-NTy0-partition}, lead to (\ref{newyear6.105}).

We then claim that
 \begin{equation}\label{newyear6.106}
\mathcal{V}=\mathcal{V}_{3,1}\cup\mathcal{V}_{3,2}.
\end{equation}
Indeed, by the second equality in (\ref{no-critical-point-1}) and the definitions of $\mathcal{V}_1$ and $\mathcal{V}_2$ (see (\ref{Lambda-di-0-1-1}) and (\ref{Lambda-di-0-1-2})), we see that $\mathcal{V}_1\cup \mathcal{V}_2=\emptyset$.
Meanwhile, the first  equality in (\ref{no-critical-point-1}) and the definition of  $\mathcal V_{3,3}$ (see (\ref{Lambda-di-3-1})), we find
that $\mathcal V_{3,3}=\emptyset$. These,  along with (i) of Theorem \ref{Proposition-TMy0-partition}, lead to (\ref{newyear6.106}).

Now,   (\ref{no-critical-point-2}) follows from (\ref{newyear6.105}) and (\ref{newyear6.106}) at once.

Finally, we  verify (\ref{no-critical-point-3}). On one hand, by the definitions of $\gamma_1$ and $\mathcal W_{2,2}$ (see (\ref{zhangjinchu1-5}) and (\ref{zhangjinchu7})), we see that $\gamma_1=\mathcal W_{2,2}$. On the other hand, from (i) of  Lemma \ref{Lemma-T0-T1} and (ii) of Lemma \ref{Lemma-NP-decreasing}, it follows that
\begin{eqnarray*}
 N(T^0(y_0),y_0)  \geq N(T^1(y_0),y_0)
 \;\;\mbox{for all}\;\;
 y_0 \in X\setminus\{0\}.
\end{eqnarray*}
Then by the definitions of $\gamma_2$ and $\mathcal V_{2,3}$  (see (\ref{TP-divide-domain-left}) and  (\ref{Lambda-di-2-1})), one can directly check  that
$\gamma_2=\mathcal V_{2,3}$.
Since we already knew that $\mathcal{W}_2=\emptyset$, $\mathcal{V}_2=\emptyset$,
$\mathcal W_{2,2}\subset \mathcal{W}_2$ and $\mathcal V_{2,3}\subset \mathcal{V}_2$,
(\ref{no-critical-point-3}) follows at once. Thus we end the proof of the conclusion (i) of Theorem \ref{Proposition-Corollary1.7}.

\vskip 5pt
 (ii) Since (H3) and (H4) hold, we find from  Proposition~\ref{WGSWGSlemma2.15} and Proposition~\ref{Lemma-YT-bounded} that
   both (H1) and (H2) hold. Then by the conclusions (ii) and (v) of Theorem \ref{Proposition-TMy0-partition}, as well as   the second equality in (\ref{no-critical-point-2}), we get  the conclusion (ii) of Theorem \ref{Proposition-Corollary1.7}.

\vskip 5pt
 (iii) By (H3) and (H4), we can use   Proposition~\ref{WGSWGSlemma2.15} and Proposition~\ref{Lemma-YT-bounded} to get (H1) and (H2). Then by (ii) and (iii)  of Theorem \ref{Proposition-NTy0-partition}, as well as  the first equality in (\ref{no-critical-point-2}), we are led to the conclusion (iii) of  Theorem \ref{Proposition-Corollary1.7}.

\vskip 5pt

In summary, we finish the proof of   Theorem \ref{Proposition-Corollary1.7}.
 \end{proof}

 \section{Applications}

Two applications of the main theorems of this paper will be given in this section. The first one is an application of  Theorem \ref{Proposition-Corollary1.7}, while the second one is an application of
Theorem \ref{Proposition-NTy0-partition}, Theorem~\ref{Proposition-TMy0-partition}.

  \subsection{Application to boundary controlled   heat equations}

  In this subsection, we will use  Theorem \ref{Proposition-Corollary1.7} to study the BBP decompositions for minimal time and minimal norm control problems
  for
  boundary controlled heat equations.  We begin with introducing the controlled equations.
  Let $\Omega\subset\mathbb R^n$, $n\geq 1$, be a bounded domain with a smooth boundary $\partial\Omega$. Let $\Gamma$ be a nonempty open subset of $\partial\Omega$. Consider the following two  controlled equations:
  \begin{eqnarray}\label{heat-boundary}
   \left\{\begin{array}{lll}
           \partial_t y - \Delta y =0 &\mbox{ in } &\Omega\times(0,\infty),\\
           y=u   &\mbox{ on } &\Gamma\times(0,\infty),\\
           y=0   &\mbox{ on } &(\partial\Omega\setminus\Gamma)\times(0,\infty),\\
           y(0)=y_0 &\mbox{ in } &\Omega
          \end{array}
   \right.
  \end{eqnarray}
  and
  \begin{eqnarray}\label{heat-boundary-finite}
   \left\{\begin{array}{lll}
           \partial_t y - \Delta y =0 &\mbox{ in } &\Omega\times(0,T),\\
           y=v   &\mbox{ on } &\Gamma\times(0,T),\\
           y=0   &\mbox{ on } &(\partial\Omega\setminus\Gamma)\times(0,T),\\
           y(0)=y_0 &\mbox{ in } &\Omega.
          \end{array}
   \right.
  \end{eqnarray}
 Here, $y_0\in H^{-1}(\Omega)$, $0<T<\infty$,  $u\in L^\infty(\mathbb R^+;L^2(\Gamma))$ and
 $v\in L^\infty(0,T;L^2(\Gamma))$. Write $y_1(\cdot;y_0,u)$ and $\hat y_1(\cdot;y_0,v)$ for the solutions of (\ref{heat-boundary}) and (\ref{heat-boundary-finite}), respectively.

We will put the above systems in our framework where  $X\triangleq H^{-1}(\Omega)$, $U\triangleq L^2(\Gamma)$, $A\triangleq A_1$ and $B\triangleq B_1$. Here,  $A_1=\Delta$, with $D(A_1)=H_0^1(\Omega)$, and $B_1$ is defined in the following manner: Let $D\,:\,L^2(\partial\Omega)\rightarrow L^2(\Omega)$ be  defined by
$Dv\triangleq f_v$, for all  $v\in L^2(\partial\Omega)$, where $f_v$ solves the equation
 \begin{eqnarray}\label{0107-Dirichlet}
  \begin{cases}
   -\Delta f =0  ~&\mbox{in}\;\;\Omega,\\
   f=v   ~&\mbox{on}\;\;  \partial\Omega.
  \end{cases}
 \end{eqnarray}
Then let
$B_1\triangleq-\Delta D$. We regard  $L^2(\Gamma)$  as a subspace of $L^2(\partial\Omega)$.
 Let  $X_{-1} \triangleq (D(A_1^*))^\prime$
 be the dual of $D(A_1^*)$  with respect to the pivot space $X$.

 To prove that the above $X$, $U$ and $(A_1,B_1)$ are in our framework,
 we will use some results in \cite{TW} where both state and control spaces are assumed to be complex Hilbert spaces. Thus, we will consider the complexifications
 of our spaces.
  Write $\mathcal H^{-1}(\Omega)$ and $\mathcal  H_0^1(\Omega)$ for the complexifications of $H^{-1}(\Omega)$ and $H_0^1(\Omega)$, respectively. Write  $\mathcal X\triangleq\mathcal  H^{-1}(\Omega)$ and $\mathcal U\triangleq L^2(\Gamma;\mathbb C)$.  Let $\mathcal A_1\triangleq\Delta$, with $D(\mathcal A_1)=\mathcal H_0^1(\Omega)$. Define $\mathcal D\,:\,L^2(\partial\Omega;\mathbb C)\rightarrow L^2(\Omega;\mathbb C)$ given by $\mathcal Dw=g_w$, for all $w\in L^2(\partial\Omega;\mathbb C)$, where $g_w$ solves (\ref{0107-Dirichlet}) with $v=w$. Then let  $\mathcal B_1\triangleq-\Delta\mathcal D$.
  The space $L^2(\Gamma;\mathbb C)$ is regarded as a subspace of $L^2(\partial\Omega;\mathbb C)$.
 Let  $\mathcal X_{-1} \triangleq (D(\mathcal A_1^*))^\prime$
 be the dual of $D(\mathcal A_1^*)$  with respect to the pivot space $\mathcal X$.
  Then, from \cite[Proposition 10.7.1]{TW}, it follows that $\mathcal A_1$  generates a $C_0$-semigroup $\{\mathcal S_1(t)\}_{t\in\mathbb R^+}$ over $\mathcal H^{-1}(\Omega)$; $\mathcal B_1\in \mathcal L(\mathcal U;\mathcal X_{-1})\setminus\{0\}$ is
an admissible control operator for the semigroup $\{\mathcal S_1(t)\}_{t\in\mathbb R^+}$.

Several observations are given in order: First,
 $\mathcal A_1|_{D(A_1)}=A_1$ and
  $\mathcal B_1|_{L^2(\Gamma)}=B_1$;
   Second, $\{\mathcal S_1(t)|_{X}\}_{t\in\mathbb R^+}$ is a $C_0$-semigroup over $H^{-1}(\Omega)$, with its generator  $A_1$; Third,  $B_1\in \mathcal L(U,X_{-1})\setminus\{0\}$ is
an admissible control operator for the semigroup $\{\mathcal S_1(t)|_{X}\}_{t\in\mathbb R^+}$.
From these observations, we see that if $S_1(t)\triangleq\mathcal S_1(t)|_{X}$, $t\in\mathbb R^+$, then the systems (\ref{heat-boundary}) and (\ref{heat-boundary-finite}) can be rewritten respectively
as
\begin{eqnarray*}
           y^\prime(t) = A_1y(t)+B_1 u(t),~t>0;\;\;  y(0)=y_0;
          \end{eqnarray*}
 \begin{eqnarray*}
            y^\prime(t) = A_1y(t)+B_1 v(t),~0<t\leq T;\;\;
          y(0)=y_0.
          \end{eqnarray*}
 The corresponding two optimal control problems are as follows:
The first one is
 the  minimal time control problem $(TP)^{M,y_0}_1$, with $y_0\in H^{-1}(\Omega)\setminus\{0\}$ and $M\in(0,\infty)$:
 \begin{eqnarray*}
  T_1(M,y_0)  \triangleq  \{ \hat t>0~:~\exists\,u\in\mathcal U_1^M
 \;\;\mbox{s.t.}\;\;  y_1(\hat t;y_0,u)=0\},
 \end{eqnarray*}
 where
 \begin{eqnarray*}
  \mathcal U_1^M \triangleq \{u \in L^\infty(\mathbb R^+, L^2(\Gamma)) ~:~\|u(t)\|_{L^2(\Gamma)}\leq M
   \;\;\mbox{a.e.}\;\;  t\in\mathbb R^+\}.
 \end{eqnarray*}
The second one is the  minimal norm control problem $(NP)^{T,y_0}_1$, (with $y_0\in H^{-1}(\Omega)\setminus\{0\}$ and $T\in(0,\infty)$) as follows:
 \begin{eqnarray*}
   N_1(T,y_0)  \triangleq  \inf\{ \|v\|_{L^\infty(0,T;L^2(\Gamma))}~:~\hat y_1(T;y_0,v)=0\}.
 \end{eqnarray*}

 \begin{Lemma}\label{wanglemma8.9w}
 The conditions (H3) and (H4) hold for the pair $(A_1,B_1)$. Furthermore, $N_1(T^1(y_0),y_0)=0$ for each $y_0\in H^{-1}(\Omega)\setminus\{0\}$,
 where $T^1(y_0)$ is given by (\ref{Ty0}) where $\{S(t)\}_{t\in\mathbb R^+}$ is replaced by $\{S_1(t)\}_{t\in \mathbb R^+}$.
 \end{Lemma}

 \begin{proof}
 First, the condition (H3)  follows from Proposition \ref{newhuangproposition2.15} and the $L^\infty$-null boundary controllability of the heat equation (see, for instance, subsection 3.2.1 in \cite{WCZ}).

 Next, we prove that (H4) holds for $(A_1,B_1)$. For this purpose, let $0<T<\infty$ and $E\subset(0,T)$ be a measurable subset of positive measure. Then fix a $\hat z\in X$ so that
 \begin{eqnarray}\label{0107-14}
  \widetilde{B_1^*S_1^*}(T-\cdot)\hat z =0  \;\;\mbox{over}\;\; E,
 \end{eqnarray}
 where $\widetilde{B_1^*S_1^*}(T-\cdot)\hat z$ is given by (\ref{wang1.15}). We will use the real analyticity of $\{S_1(t)\}_{t\in\mathbb R}$ to show that
 \begin{eqnarray}\label{0107-15}
  \widetilde{B_1^*S_1^*}(T-\cdot)\hat z =0  \;\;\mbox{over}\;\; (0,T).
 \end{eqnarray}
 Indeed,  from subsection 3.2.1 in \cite{WCZ}, it follows that the semigroup $\{\mathcal S_1(t)\}_{t\in\mathbb R^+}$ can be extended to an analytic semigroup. Thus, the semigroup $\{\mathcal S^*_1(t)\}_{t\in\mathbb R^+}$ is also analytic.
 Then by \cite[Theorem 5.2 in Chapter 2]{Pazy}, we find that
 \begin{eqnarray}\label{0107-11}
  \mathcal S^*_1(\cdot)\;\;\mbox{is real analytic over}\;\; (0,\infty);
  \;\;\mbox{and}\;\;
  \|\mathcal S_1^*(t)\|_{\mathcal L(\mathcal X,D(\mathcal A_1^*))} \leq \widehat C/t,~t>0,
 \end{eqnarray}
 where the constant $\widehat C$ is independent of $t>0$. Since $\mathcal S_1(\cdot)|_{X}=S_1(\cdot)$ over $\mathbb R^+$, we have that
 $\mathcal S_1^*(\cdot)|_{X}=S_1^*(\cdot)$ over $\mathbb R^+$,
  which, along with (\ref{0107-11}), implies that
 \begin{eqnarray}\label{0107-12}
  S^*_1(\cdot)\;\;\mbox{is real analytic over}\;\; (0,\infty);
  \;\;\mbox{and}\;\; \|S_1^*(t)\|_{\mathcal L(X,D(A_1^*))} \leq \widehat C/t,~t>0.
 \end{eqnarray}
 Arbitrarily fix an $\varepsilon\in(0,T)$ so that
 \begin{eqnarray}\label{0107-13}
  |E\cap (0,T-\varepsilon)| >0.
 \end{eqnarray}
 Let $\{z_n\}\subset D(A_1^*)$ so that $\lim_{n\rightarrow\infty} z_n=\hat z$ in $X$. Because $B_1\in \mathcal L(D(A^*_1),U)$, we find from the second conclusion in (\ref{0107-12}) that when $n$ goes to $\infty$,
 \begin{eqnarray*}
  \|B_1^*S_1^*(\cdot)z_n - B_1^*S_1^*(\cdot)\hat z\|_{L^2(\varepsilon,T;U)}
  &=& \|B_1^*S_1^*(\cdot-\varepsilon)S_1^*(\varepsilon)(z_n - \hat z)\|_{L^2(\varepsilon,T;U)}
  \rightarrow 0.
 \end{eqnarray*}
 This, along with  (\ref{wang1.15}), yields that
 \begin{eqnarray*}
   \widetilde{B_1^*S_1^*}(T-\cdot)\hat z
   =B_1^*S_1^*(T-\cdot)\hat z
   \;\;\mbox{over}\;\; (0,T-\varepsilon),
 \end{eqnarray*}
 which, together with the first conclusion in (\ref{0107-12}), shows that $\widetilde{B_1^*S_1^*}(T-\cdot)\hat z$ is real analytic over $(0,T-\varepsilon)$. Then, by  (\ref{0107-13}) and (\ref{0107-14}), we see that
  \begin{eqnarray*}
  \widetilde{B_1^*S_1^*}(T-\cdot)\hat z =0  \;\;\mbox{over}\;\; (0,T-\varepsilon).
 \end{eqnarray*}
 Sending  $\varepsilon\rightarrow 0$ in the above leads to
 (\ref{0107-15}). Hence,  (H4) holds  for $(A_1,B_1)$.

 Finally,  we will prove that
 \begin{equation}\label{newyear7.9}
  N_1(T^1(y_0),y_0)=0  \;\;\mbox{for all}\;\; y_0\in H^{-1}(\Omega)\setminus\{0\}.
 \end{equation}
 According to (vi) of  Lemma~\ref{Lemma-T0-T1},
  (\ref{newyear7.9}) is equivalent to that
  \begin{eqnarray}\label{0106-302}
  N_1(\infty,y_0)=0  \;\;\mbox{for all}\;\; y_0\in H^{-1}(\Omega)\setminus\{0\}.
 \end{eqnarray}
 To prove (\ref{0106-302}), we arbitrarily fix a $y_0\in H^{-1}(\Omega)\setminus\{0\}$ and then fix a $\hat t\in(0,\infty)$. Notice that  the semigroup $\{S_1(t)\}_{t\geq 0}$ has the following property: there exist $C>0$ and $\delta>0$, independent of $\hat t$, so that
 \begin{eqnarray}\label{0106-301}
  \|S_1(\hat t)y_0\|_{H^{-1}(\Omega)}
  \leq C e^{-\delta \hat t} \|y_0\|_{H^{-1}(\Omega)}.
 \end{eqnarray}
 Meanwhile, according to the $L^\infty$-null controllability of the boundary controlled
 heat equation, there exist a positive constant $C^\prime$ (independent of $\hat t$) and a control $u_{\hat t}\in L^\infty(0,1;L^2(\Gamma))$ so that
 \begin{eqnarray}\label{0106-300}
  \hat y_1(1;S_1(\hat t)y_0,u_{\hat t})=0
  \;\;\mbox{and}\;\;
  \|u_{\hat t}\|_{L^\infty(0,1;L^2(\Gamma))} \leq C^\prime \|S_1(\hat t)y_0\|_{H^{-1}(\Omega)}.
 \end{eqnarray}
 Define another control
 \begin{eqnarray*}
  v_{\hat t} (\tau) =\begin{cases}
                     0,~&\tau\in(0,\hat t],\\
                     u_{\hat t}(\tau-\hat t), ~&\tau \in (\hat t, \hat t+1).
                     \end{cases}
 \end{eqnarray*}
 From this and (\ref{0106-300}), we find that
 \begin{eqnarray*}
  \hat y_1(\hat t+1;y_0,v_{\hat t})  =  \hat y_1(1;S_1(\hat t)y_0,u_{\hat t})=0;
 \end{eqnarray*}
 \begin{eqnarray*}
  \|v_{\hat t}\|_{L^\infty(0,\hat t+1;L^2(\Gamma)) }
  =  \|u_{\hat t}\|_{L^\infty(0,1;L^2(\Gamma))} \leq C^\prime \|S_1(\hat t)y_0\|_{H^{-1}(\Omega)}.
 \end{eqnarray*}
 These, along with the optimality of $N_1(\hat t+1,y_0)$ and (\ref{0106-301}), yield that
 \begin{eqnarray*}
  N_1(\hat t+1,y_0) \leq \|v_{\hat t}\|_{L^\infty(0,\hat t+1;L^2(\Gamma)) }
  \leq C^\prime \|S_1(\hat t)y_0\|_{H^{-1}(\Omega)}
  \leq C^\prime C e^{-\delta \hat t} \|y_0\|_{H^{-1}(\Omega)}.
 \end{eqnarray*}
By this and the first equality in (\ref{N-infty-y0}), we obtain (\ref{0106-302}).
Hence, (\ref{newyear7.9}) has been proved.
  This ends the proof of this lemma.

 \end{proof}

  The BBP decompositions for $(A_1,B_1)$ are presented in the following Theorem~\ref{newyeartheorem7.2}:

  \begin{Theorem}\label{newyeartheorem7.2}
  Let $\mathcal{W}$, $\mathcal W_{3,2}$,  $\mathcal{V}$ and  $\mathcal{V}_{3,2}$
    be respectively given by (\ref{zhangjinchu1}), (\ref{zhangjinchu9}),   (\ref{TP-divide-domain})  and  (\ref{Lambda-di-3-1}), where $(A,B)=(A_1,B_1)$.
  Then the following conclusions are true:

\noindent (i) $\mathcal{W}=\mathcal W_{3,2}$ and $\mathcal{V}=\mathcal{V}_{3,2}$.

    \noindent (ii) For each $(M,y_0)\in (0,\infty)\times (H^{-1}(\Omega)\setminus\{0\})$, the problem $(TP)_1^{M,y_0}$   has the bang-bang property.

   \vskip 2pt

   \noindent (iii) For each $(T,y_0)\in (0,\infty)\times (H^{-1}(\Omega)\setminus\{0\})$, the problem $(NP)_1^{T,y_0}$   has the bang-bang property and the null control is not a minimal norm control to this problem.
  \end{Theorem}

  \begin{proof}
  (i) By Lemma~\ref{wanglemma8.9w}, we see that (H3) and (H4) holds for $(A_1,B_1)$.
  Then we can use Theorem~\ref{Proposition-Corollary1.7} to find that
  \begin{equation}\label{newyear7.13}
  \mathcal{W}=\mathcal{W}_{3,2} \cup \mathcal W_{3,3}\;\;\mbox{and}\;\;
   \mathcal{V}=\mathcal{V}_{3,1}\cup\mathcal{V}_{3,2}.
  \end{equation}
  On one hand, by the backward uniqueness property for $\{S_1(t)\}_{t\in\mathbb R^+}$, we have that $T^1(y_0)=\infty$ for all $y_0\in X\setminus\{0\}$. On the other hand,
  by Lemma~\ref{wanglemma8.9w}, we also have that
   $N_1(T^1(y_0),y_0)=0$ for all $y_0\in H^{-1}(\Omega)\setminus\{0\}$. These,
   along with the definitions of $\mathcal W_{3,3}$ and $\mathcal{V}_{3,1}$ (see (\ref{zhangjinchu9}) (\ref{Lambda-di-3-1})), yield that $\mathcal W_{3,3}=\emptyset$ and $\mathcal{V}_{3,1}=\emptyset$ in this case. From this and (\ref{newyear7.13}), we get the conclusion (i) of this theorem.

   \vskip 3pt
   (ii) Notice that $\mathcal{V}=(0,\infty)\times (H^{-1}(\Omega)\setminus\{0\})$
   in this case. (For the definition of $\mathcal{V}$, see (\ref{TP-divide-domain}).) Then by the second equality in the conclusion (i) of this theorem  and the assumptions (H3) and (H4), we can apply (ii) of Theorem~\ref{Proposition-Corollary1.7} to get the conclusion (ii) of this theorem.

   \vskip 3pt

   (iii) Notice that $\mathcal{W}=(0,\infty)\times (H^{-1}(\Omega)\setminus\{0\})$
   in this case. (For the definition of $\mathcal{W}$, see (\ref{zhangjinchu1})) Then by the first equality in the conclusion (i) of this theorem and the assumptions (H3) and (H4), we can apply (iii) of Theorem~\ref{Proposition-Corollary1.7} to get the conclusion (iii) of this theorem.

   \vskip 3pt
  In summary, we finish the proof of this theorem.

  \end{proof}

 \begin{Remark}\label{newyearremark7.3}
 (i) From Theorem~\ref{newyeartheorem7.2}, we see that the BBP decomposition for
 $(NP)_1^{T,y_0}$ has only one part which is $\mathcal{W}=(0,\infty)\times (H^{-1}(\Omega)\setminus\{0\})$ and that for each $(T,y_0)$ in $\mathcal{W}$, the corresponding  $(NP)_1^{T,y_0}$ has the bang-bang property. The reason to cause such decomposition is that $(A_1,B_1)$ is $L^\infty$-null controllable. The same can be said about
  the BBP decomposition for $(NP)^{T,y_0}$ built up in Theorem~\ref{Proposition-Corollary1.7}.

 (ii) From Theorem~\ref{newyeartheorem7.2}, we see that the BBP decomposition for
 $(TP)_1^{M,y_0}$ has only one part which is $\mathcal{V}=(0,\infty)\times (H^{-1}(\Omega)\setminus\{0\})$ and that for each $(M,y_0)$ in $\mathcal{V}$, the corresponding  $(TP)_1^{M,y_0}$ has the bang-bang property. The reasons to cause such decomposition are that $(A_1,B_1)$ is $L^\infty$-null controllable and $N_1(T^1(y_0), y_0)=0$ for all $y_0\in (0,\infty)\times (H^{-1}(\Omega)\setminus\{0\})$. (Compare this BBP decomposition with the BBP decomposition (P1) given by  (\ref{0127-intro-P1}).)
 The above-mentioned  second property (i.e., $N_1(T^1(y_0), y_0)=0$ for all $y_0\in (0,\infty)\times (H^{-1}(\Omega)\setminus\{0\})$)  holds, because solutions of the controlled system (governed by $(A_1, B_1)$), with the null control, tend to zero as time goes to infinity.
   \end{Remark}

\bigskip
  \subsection{Application to some special controlled evolution   systems}

In this subsection, we will use  Theorem~\ref{Proposition-NTy0-partition} and
 Theorem~\ref{Proposition-TMy0-partition} to study the BBP decompositions for minimal time and minimal norm control problems in a special setting. The controlled system in this setting  is taken from  \cite{AK-2}.

   Let $X$ and $U$ be two real separable Hilbert spaces. Let $A\triangleq A_2$
    and $B\triangleq B_2$, where $A_2$ and $B_2$ are defined in the following manner: Arbitrarily fix  a Riesz basis $\{\phi_j\}_{j\geq 1}$  in $X$ and  a biorthogonal sequence
    $\{\psi_j\}_{j\geq 1}$ of the aforementioned  Riesz basis.
       Take a sequence  $\Lambda\triangleq\{\lambda_j\}_{j\geq 1}\subset \mathbb R^+$  so that
\begin{equation}\label{intro-seq}
  0<\lambda_1<\lambda_2<\cdots<\lambda_j<\cdots;
  \;\;\mbox{and}\;\;\displaystyle\Sigma_{j\geq 1} 1/\lambda_j <\infty.
\end{equation}
Write
$X_1 \triangleq \{y\in X~:~\|y\|_1<\infty\}$ with the norm
$\|y\|_{X_1} \triangleq \sqrt{\sum_{j\geq 1} \lambda_j^2  \langle y,\psi_j \rangle_X^2}$.     Define  $A_2:~ D(A_2)\triangleq X_1\subset X\rightarrow X$  by
setting
\begin{equation}\label{intro-A}
   A_2 x \triangleq -\sum_{j\geq 1} \lambda_j \langle x,\psi_j \rangle_X\phi_j
   \;\;\mbox{for each}\;\;x\in D(A_2).
\end{equation}
Write $X_{-1} \triangleq (D(A_2^*))^\prime$ (the dual of $D(A_2^*)$  with respect to the pivot space $X$). Then let  $B_2\in \mathcal L(U,X_{-1})\setminus\{0\}$.


 One can directly check the following facts: First,  the operator $A_2$ generates a $C_0$-semigroup $\{S_2(t)\}_{t\in\mathbb R^+}$ over $X$; Second, the semigroup $\{S_2(t)\}_{t\in\mathbb R^+}$ has the expression:
 \begin{eqnarray}\label{wsiaoxiao7.22}
 S_2(t)x=\sum_{j=1}^\infty x_j e^{-\lambda_j t}\phi_j,\; t\geq 0,\;\;\mbox{for each}\;\;  x=\sum_{j=1}^\infty x_j\phi_j\in X.
 \end{eqnarray}
 Third, the dual semigroup $\{S_2^*(t)\}_{t\geq 0}$ has the expression:
 \begin{eqnarray}\label{wsiaoxiao7.23}
 S_2^*(t)x=\sum_{j=1}^\infty \hat x_j e^{-\lambda_j t}\psi_j,\; t\geq 0,\;\;\mbox{for each}\;\;  x=\sum_{j=1}^\infty \hat x_j\psi_j\in X.
 \end{eqnarray}
In this setting, the systems (\ref{system-0}) and (\ref{system-1})  read respectively as follows:
  \begin{eqnarray}\label{intro-equation-infty}
             y^\prime(t) = A_2y(t)+B_2u(t), ~ t>0;\;\; y(0)=y_0;
          \end{eqnarray}
 \begin{eqnarray}\label{intro-equation-finite}
             y^\prime(t) = A_2y(t)+B_2v(t), ~ 0<t\leq T;\;\; y(0)=y_0.
         \end{eqnarray}
 Here, $y_0\in X$, $0<T<\infty$,  $u\in L^\infty(\mathbb R^+;U)$ and
 $v\in L^\infty(0,T;U)$. Write $y_2(\cdot;y_0,u)$ and $\hat y_2(\cdot;y_0,v)$ for the solutions of (\ref{intro-equation-infty}) and (\ref{intro-equation-finite}), respectively. There are many controlled PDEs governed by $(A_2,B_2)$, we refer the readers to \cite{AK-2}, \cite{AK-3} and \cite{AK-1}.

For each $y_0\in X\setminus\{0\}$ and each $M\in(0,\infty)$, we consider the minimal time control problem:
 \begin{equation*}\label{intro-TP}
 (TP)^{M,y_0}_2\;\;\;\;\;\; T_2(M,y_0) \triangleq \inf \{\hat t>0~:~\exists\,u\in \mathcal U_{3}^M
  \;\;\mbox{s.t.}\;\;y(\hat t;y_0,u)=0\},
 \end{equation*}
 where
 \begin{equation*}\label{intro-U}
  \mathcal U_{3}^M  \triangleq \{u\in L^\infty(\mathbb R^+;U) ~:~ \|u(t)\|_U\leq M \;\;\mbox{a.e.}\;\;t\in\mathbb R^+\}.
 \end{equation*}
 For each $y_0\in X\setminus\{0\}$ and  each $T\in(0,\infty)$, we consider the  minimal norm control problem:
 \begin{eqnarray*}
 (NP)^{T,y_0}_2\;\;\;\;\;\; N_2(T,y_0)  \triangleq  \inf\{ \|v\|_{L^\infty(0,T;U)}~:~\hat y_2(T;y_0,v)=0\}.
 \end{eqnarray*}

We will prove that $(A_2,B_2)$ satisfies (H1) and (H2). To do this,  we need three lemmas.  The first one is very similar to \cite[Lemma 4.6]{AK-1}. We will give its proof in Appendix G of this paper. To state it, we define
\begin{equation}\label{space-2-1}
     \mathcal P \triangleq\Big\{z\rightarrow\sum_{j=1}^N c_j e^{-\lambda_j z},\,z\in\mathbb C^+   ~:~
     \{c_j\}_{j=1}^N \subset \mathbb C,  \,N\in\mathbb N^+ \Big\},
  \end{equation}
 where $\mathbb C^+ \triangleq \{x+iy \in\mathbb C ~:~ x\geq 0\}$. And then for
  each  $\theta_0\in (0,\frac{\pi}{2})$ and $\varepsilon>0$, define
  \begin{equation}\label{space-2}
     S_{\varepsilon,\theta_0}
     \triangleq \Big\{z=x+iy \in\mathbb C
     ~:~x\geq\varepsilon,\,\frac{|y|}{x}\leq \frac{1}{2}\cot\theta_0 \Big\}.
  \end{equation}

\begin{Lemma}\label{lemma-key-inequality}
 For each  $\theta_0\in (0,\frac{\pi}{2})$, $\varepsilon>0$,
 and each $T>0$, there exist two positive constants $C_1\triangleq C_1(\theta_0,\varepsilon,T)$ and $C_2\triangleq C_2(\theta_0)$ so that
  \begin{equation}\label{space-3}
   |p(z)| \leq C_1e^{-C_2 Re\, z}
   \|p|_{(0,T)}\|_{L^1(0,T;\mathbb C)}
   \;\;\mbox{for all}\;\; p\in\mathcal P
   \;\;\mbox{and}\;\; z\in S_{\varepsilon,\theta_0}.
  \end{equation}
 Here, $p|_{(0,T)}$ denotes the restriction of $p$ on $(0,T)$.

 \end{Lemma}

 To  state the second lemma, we write $\widetilde{U}$ for the complexification of $U$ and then define
\begin{equation}\label{space-U}
     \mathcal P_{\widetilde{U}} \triangleq\Big\{ z\rightarrow\sum_{j=1}^N c_j e^{-\lambda_j z} B_2^*\psi_j,\,z\in\mathbb C^+
     ~:~ \{c_j\}_{j=1}^N\subset \mathbb C,\,N\in\mathbb N^+  \Big\}.
  \end{equation}
Notice that each element in $\mathcal P_{\widetilde{U}}$ is a vector-valued function, with its domain $\mathbb C$ and its range $\widetilde{U}$.
  With the aid of  Lemma \ref{lemma-key-inequality}, we build up an estimate
  in the second lemma as follows:
  \begin{Lemma}\label{ob-estimate-YT}
   For each $\theta_0\in (0,\frac{\pi}{2})$,  $\varepsilon>0$ and each $T>0$, there exist two positive constants $C_1\triangleq C_1(\theta_0,\varepsilon,T)$ and $C_2\triangleq C_2(\theta_0)$ so that
  \begin{equation}\label{space-U-estimate}
   \|f(z)\|_{\widetilde{U}} \leq C_1e^{-C_2 Re\, z}   \|f|_{(0,T)}\|_{L^1(0,T;\widetilde{U})}
   \;\;\mbox{for all}\;\;
   f\in\mathcal P_{\widetilde{U}}   \;\;\mbox{and}\;\;  z\in S_{\varepsilon,\theta_0},
  \end{equation}
where, $S_{\varepsilon,\theta_0}$ and $\mathcal P_{\widetilde{U}}$
are defined by  (\ref{space-2}) and (\ref{space-U}), respectively, and $f|_{(0,T)}$
denotes the restriction of $f$ on $(0,T)$.
 \end{Lemma}

 \begin{proof}
  Arbitrarily fix  $f\in \mathcal P_{\widetilde{U}}$.  Then by (\ref{space-U}), there is $N\in\mathbb N^+$
 and $ \{c_j\}_{j=1}^N\subset \mathbb C$ so that
  \begin{eqnarray*}
    f(z)=\sum_{j=1}^N c_j e^{- \lambda_j z} B_2^*\psi_j\;\;\mbox{for all}\;\;z\in \mathbb C^+.
  \end{eqnarray*}
  Arbitrarily fix a $v\in \widetilde{U}$.  Since
  \begin{eqnarray*}
   f_v(z) \triangleq \langle f(z),v\rangle_{\widetilde{U}}
    = \sum_{j=1}^N c_j \langle B_2^*\psi_j,v \rangle_{\widetilde{U}} e^{- \lambda_j z},
   ~z\in \mathbb C^+,
  \end{eqnarray*}
    it follows from  (\ref{space-2-1}) that   $f_v \in\mathcal P$.
     Then according to Lemma \ref{lemma-key-inequality},  for each $\theta_0\in (0,\frac{\pi}{2})$, each $\varepsilon>0$ and each $T>0$, there are two positive constants $C_1(\theta_0,\varepsilon,T)$ and $C_2(\theta_0)$ (independent of
     $f$ and $v$) so that
     \begin{eqnarray*}
   |\langle f(z),v \rangle_{\widetilde{U}}|
   \leq  C_1(\theta_0,\varepsilon,T) e^{-C_2(\theta_0) Re \,z} \int_0^T |\langle f|_{(0,T)}(t),v \rangle_{\widetilde{U}}|  \,\mathrm dt
   \;\;\mbox{for each}\;\; z\in  S_{\varepsilon,\theta_0}.
  \end{eqnarray*}
       Since for each  $z \in S_{\varepsilon,\theta_0}$, the above inequality holds for all $v\in \widetilde{U}$, we find that for each $z \in S_{\varepsilon,\theta_0}$,
  \begin{eqnarray*}
   \|f(z)\|_{\widetilde{U}}  = \sup_{\|v\|_{\widetilde{U}} \leq 1 }  |\langle f(z),v \rangle_{\widetilde{U}} |
   \leq  C_1(\theta_0,\varepsilon,T) e^{-C_2(\theta_0) Re \,z} \int_0^T \|f|_{(0,T)}(t)\|_{\widetilde{U}} \,\mathrm dt.
  \end{eqnarray*}
  Since $f$ was arbitrarily taken from $\mathcal P_{\widetilde{U}}$, the above
  inequality leads to  (\ref{space-U-estimate}).  This ends  the proof of this lemma.
 \end{proof}

 With the aid of Lemma~\ref{ob-estimate-YT}, we obtain the third lemma which will play a key role in the proof of the conclusion that (H1) and (H2) hold for $(A_2,B_2)$.

 \begin{Lemma}\label{lemmazhang7.10}
 Let  $\theta_0\in (0,\frac{\pi}{2})$. Then for each $T\in (0,\infty)$, each  $\varepsilon\in (0,T)$ and each $f\in Y_T$ (which is defined by (\ref{ob-space}) with $(A^*,B^*)$ being replaced by $(A^*_2,B^*_2)$), there is a continuous and
  weakly analytic function $\widetilde g_{\varepsilon,f}: S_{\varepsilon,\theta_0}\rightarrow
  \widetilde{U}$ so that
 \begin{eqnarray}\label{zhang7.31}
  \widetilde g_{\varepsilon,f}|_{(\varepsilon,T)}(T-t)=f(t) \;\;\mbox{for each}\;\; t\in (0,T-\varepsilon),
  \end{eqnarray}
  and so that
 \begin{eqnarray}\label{zhang7.32}
  \|\widetilde g_{\varepsilon,f}\|_{L^\infty(S_{\varepsilon,\theta_0};\widetilde{U})} \leq C_1(\theta_0,\varepsilon,\varepsilon) \|f\|_{L^1(T-\varepsilon,T;U)},
 \end{eqnarray}
 where $C_1(\theta_0,\varepsilon,\varepsilon)$ is given
  by (\ref{space-U-estimate}).

 \end{Lemma}

 \begin{proof}
  Let $\theta_0\in (0,\frac{\pi}{2})$ be given. Arbitrarily fix
  $T\in (0,\infty)$,  $\varepsilon\in (0,T)$ and $f\in Y_T$. First of all, since  $\{\psi_j\}_{j\geq 1}$ is a biorthogonal sequence of the Riesz basis $\{\phi_j\}_{j\geq 1}$  in $X$,
  it follows by (\ref{intro-A}) that
  each element $w\in D(A_2^*)$ can be expressed by
 $w=\sum_{j=1}^\infty \alpha_j \psi_j$, with $\{\alpha_j\}_{j=1}^\infty \subset\mathbb R$,
and satisfies that
  \begin{eqnarray}\label{zhang7.33}
   \Big\|\sum_{j=1}^N\alpha_j \psi_j-w\Big\|_{D(A_2^*)} = \sqrt{\sum_{j\geq N} \lambda_j^2 \alpha_j^2}
   \rightarrow 0,  \;\;\mbox{as}\;\; N\rightarrow\infty,
 \end{eqnarray}
   Since $B_2^*\in \mathcal L(D(A_2^*),U)$, it follows from (\ref{zhang7.33}) that
 \begin{eqnarray}\label{zhang7.34}
  B_2^*S_2^*(T-\cdot)\sum_{j=1}^N\alpha_j \psi_j  \rightarrow B_2^*S_2^*(T-\cdot)w
  \;\;\mbox{in}\;\;  L^1(0,T;U),  \;\;\mbox{as}\;\; N\rightarrow\infty.
 \end{eqnarray}
  Since $f\in Y_T$, according to (\ref{ob-space}) and (\ref{zhang7.34}), there is a sequence $\{w_N\}_{N=1}^\infty$ in $D(A_2^*)$ so that for each $N\in \mathbb{N}^+$,
  \begin{eqnarray}\label{zhang7.35}
 w_N=\sum_{j=1}^{K_N}\alpha_j(w_N)\psi_j,\;\;\mbox{with}\;\;K_N\in \mathbb{N}^+
 \;\;\mbox{and}\;\;\{\alpha_j(w_N)\}_{j=1}^{K_N}\subset \mathbb{R},
 \end{eqnarray}
  and so that
  \begin{eqnarray}\label{zhang7.36}
  B_2^*S_2^*(T-\cdot)w_N  \rightarrow f(\cdot)
  \;\;\mbox{in}\;\;  L^1(0,T;U),  \;\;\mbox{as}\;\; N\rightarrow\infty.
 \end{eqnarray}

  Next, for each $N\in \mathbb{N}^+$,
  define $g_N: \mathbb C^+\rightarrow \widetilde{U}$ by
 \begin{eqnarray}\label{zhang7.37}
  g_N(z) \triangleq   \sum_{j=1}^{K_N} \alpha_j(w_N) e^{-\lambda_j z} B_2^*\psi_j,~ z\in\mathbb C^+.
 \end{eqnarray}
   By (\ref{zhang7.37}),  (\ref{zhang7.35}) and (\ref{wsiaoxiao7.23}), we see  that
 \begin{eqnarray}\label{zhang7.39}
    g_N|_{(0,T)}(t) = B_2^*S_2^*(t)w_N \;\;\mbox{for each}\;\; t\in(0,T).
 \end{eqnarray}
 Meanwhile, from (\ref{zhang7.37}) and (\ref{space-U}), we see that $g_N\in \mathcal P_{\widetilde{U}}$ for all $N\in \mathbb{N}^+$.
  This, along with  Lemma~\ref{ob-estimate-YT}, yields that for each $N\in \mathbb{N}^+$,
 \begin{eqnarray*}
  \|g_N|_{S_{\varepsilon,\theta_0}}\|_{L^\infty(S_{\varepsilon,\theta_0};
  \widetilde{U})} \leq C_1(\theta_0,\varepsilon,\varepsilon)  \|g_N|_{(0,\varepsilon)}\|_{L^1(0,\varepsilon;\widetilde{U})},
 \end{eqnarray*}
 where $C_1(\theta_0,\varepsilon,\varepsilon)$ is given by (\ref{space-U-estimate}).
  Since for each $t\in \mathbb{R}^+$, we have that $g_N(t)\in U$ (see (\ref{zhang7.37}) and (\ref{zhang7.35})), the above inequality can be rewritten as:
  \begin{eqnarray}\label{zhang7.38}
  \|g_N|_{S_{\varepsilon,\theta_0}}\|_{L^\infty(S_{\varepsilon,\theta_0};
  \widetilde{U})} \leq C_1(\theta_0,\varepsilon,\varepsilon)  \|g_N|_{(0,\varepsilon)}\|_{L^1(0,\varepsilon;{U})},
 \end{eqnarray}
    By (\ref{zhang7.39}) and (\ref{zhang7.36}), we see that
  \begin{eqnarray}\label{zhang7.40}
    g_N|_{(0,\varepsilon)}(\cdot)\rightarrow f(T-\cdot)\;\;\mbox{in}\;\; L^1(0,\varepsilon; U).
 \end{eqnarray}
  Hence, $\{g_N|_{(0,\varepsilon)}\}_{N=1}^\infty$ is a Cauchy sequence in $L^1(0,\varepsilon;U)$. From this and (\ref{zhang7.38}), we can easily see that
  there exists a function $\widetilde g_{\varepsilon,f}\in  L^\infty(S_{\varepsilon,\theta_0}; \widetilde{U})$ so that
 \begin{eqnarray}\label{zhang7.41}
  g_N|_{S_{\varepsilon,\theta_0}} \rightarrow \widetilde g_{\varepsilon,f}
  \;\;\mbox{in}\;\;  L^\infty(S_{\varepsilon,\theta_0}; \widetilde{U}),
  \;\;\mbox{as}\;\;  N\rightarrow\infty.
 \end{eqnarray}
  We  claim that
  \begin{eqnarray}\label{zhang7.42}
   \widetilde g_{\varepsilon,f}: S_{\varepsilon,\theta_0}\rightarrow \widetilde{U}
   \;\;\mbox{is continuous and weakly analytic  over}\;\;
   S_{\varepsilon,\theta_0}.
  \end{eqnarray}
  First,  by (\ref{zhang7.37}), we see that  for each $N\in \mathbb{N}^+$,  the function $ g_N|_{S_{\varepsilon,\theta_0}}$ is continuous. This, along with    (\ref{zhang7.41}), yields
  that the function $\widetilde g_{\varepsilon,f}$ is continuous over $S_{\varepsilon,\theta_0}$, and that
   \begin{eqnarray}\label{zhang7.43}
  g_N|_{S_{\varepsilon,\theta_0}} \rightarrow \widetilde g_{\varepsilon,f}
  \;\;\mbox{in}\;\;  C(S_{\varepsilon,\theta_0}; \widetilde{U}),
  \;\;\mbox{as}\;\;  N\rightarrow\infty.
 \end{eqnarray}

 Next, we prove the weak analyticity of the function $\widetilde g_{\varepsilon,f}$. Arbitrarily fix a $v\in \widetilde{U}$. By (\ref{zhang7.43}), we find that
  \begin{eqnarray}\label{zhang7.44}
  \langle g_N|_{S_{\varepsilon,\theta_0}},v \rangle_{\widetilde{U}}
  \rightarrow \langle \widetilde g_{\varepsilon,f},v \rangle_{\widetilde{U}}
  \;\;\mbox{in}\;\;  C(S_{\varepsilon,\theta_0};\mathbb{C}),
  \;\;\mbox{as}\;\;  N\rightarrow\infty.
 \end{eqnarray}
  Meanwhile, by (\ref{zhang7.37}), we see that  for each $N\in \mathbb N^+$, the function $z\rightarrow
  \langle g_N|_{S_{\varepsilon,\theta_0}}(z),v \rangle_{\widetilde{U}}$
  is  analytic over $S_{\varepsilon,\theta_0}$. By this and (\ref{zhang7.44}),   we can use  \cite[Theorem 10.28] {Rudin} to see that
   the function $z\rightarrow \langle \widetilde g_{\varepsilon,f}(z),v \rangle_{\widetilde{U}}$ is  analytic over $S_{\varepsilon,\theta_0}$. Since $v$ was arbitrarily taken from $\widetilde{U}$, $\widetilde g_{\varepsilon,f}$ is weakly analytic over $S_{\varepsilon,\theta_0}$.
   Hence, conclusions in  (\ref{zhang7.42}) are true.

 We now show that the above function $\widetilde g_{\varepsilon,f}$
satisfies (\ref{zhang7.31}). Indeed, by (\ref{space-2}), we see that $(\varepsilon, T)\subset S_{\varepsilon,\theta_0}$. This, together with (\ref{zhang7.43}), yields that
\begin{eqnarray}\label{zhang7.45}
  g_N|_{(\varepsilon, T)} \rightarrow \widetilde g_{\varepsilon,f}|_{(\varepsilon, T)}
  \;\;\mbox{in}\;\;  C((\varepsilon, T); \widetilde{U}),
  \;\;\mbox{as}\;\;  N\rightarrow\infty.
 \end{eqnarray}
From (\ref{zhang7.45}) and  (\ref{zhang7.39}), it follows that
\begin{eqnarray}\label{zhang7.46}
   B_2^*S_2^*(T-\cdot)w_N \rightarrow \widetilde g_{\varepsilon,f}|_{(\varepsilon, T)}(T-\cdot)
  \;\;\mbox{in}\;\;  C((0,T-\varepsilon); \widetilde{U}),
  \;\;\mbox{as}\;\;  N\rightarrow\infty.
 \end{eqnarray}
From (\ref{zhang7.36}) and (\ref{zhang7.46}), the desired equality (\ref{zhang7.31}) follows at once.

  Finally, since
  $$
  \int_0^\varepsilon\|f(T-t)\|_U \,\mathrm dt =\int_{T-\varepsilon}^T\|f(t)\|_U \,\mathrm dt,
  $$
 by (\ref{zhang7.41}) and (\ref{zhang7.40}),
 we can pass to the limit for $N\rightarrow\infty$ in (\ref{zhang7.38}) to see that
 the above function $\widetilde g_{\varepsilon,f}$ satisfies (\ref{zhang7.32}).
 This ends the proof.
 \end{proof}

\begin{Proposition}\label{zhangproposition7.11}
The condition (H1), with $p_0=2$, and the condition (H2) hold for $(A_2,B_2)$.
\end{Proposition}
\begin{proof}
From Lemma \ref{Lemma-H3-eq}, we see that in order to  show  the condition (H1) (with $p_0=2$) for $(A_2,B_2)$, it suffices to prove the property (iii) in Lemma \ref{Lemma-H3-eq} (with $p_2=2$) for $(A_2,B_2)$. To prove the later, we arbitrarily fix $\hat t$ and $T$ so that $0<\hat t<T<\infty$. Let $f\in Y_T$, which is defined by (\ref{ob-space}) with $(A^*,B^*)$ being replaced by $(A^*_2,B^*_2)$. Then by Lemma~\ref{lemmazhang7.10} (where $\varepsilon=T-\hat t$), we see that $f$ satisfies (\ref{zhang7.31}) and (\ref{zhang7.32}) (with  $\varepsilon=T-\hat t$) for some
continuous and
  weakly analytic function $\widetilde g_{\varepsilon,f}: S_{\varepsilon,\theta_0}\rightarrow
  \widetilde{U}$ with some  $\theta_0\in (0,\frac{\pi}{2})$. By (\ref{zhang7.31}),
  one can easily check that
\begin{eqnarray*}
  \|\widetilde g_{\varepsilon,f}(\cdot)\|_{L^\infty(S_{\varepsilon,\theta_0};\widetilde{U})}
  &\geq& \|\widetilde g_{\varepsilon,f}|_{(\varepsilon,T)}(\cdot)\|_{L^\infty(\varepsilon, T;\widetilde{U})}
  \geq\|\widetilde g_{\varepsilon,f}|_{(\varepsilon,T)}(T-\cdot)\|_{L^\infty(0,\hat t;\widetilde{U})}\\
  &=& \|f(\cdot)\|_{L^\infty(0,\hat t; {U})}
  \geq {\hat{t}}^{-1/2} \|f(\cdot)\|_{L^2(0,\hat t; {U})}.
    \end{eqnarray*}
    This, along with (\ref{zhang7.32}) (where  $\varepsilon=T-\hat t$), yields that
    $$
    \|f\|_{L^2(0,\hat t; {U})}\leq
    {\hat{t}}^{1/2} C_1(\theta_0,\varepsilon,\varepsilon)\|f\|_{L^1(\hat t, T; {U})}\triangleq C(T,\hat t, \theta_0)\|f\|_{L^1(\hat t, T; {U})},
    $$
  which leads to the property (iii) in Lemma \ref{Lemma-H3-eq} (with $p_2=2$) for $(A_2,B_2)$. Hence, (H1) with $p_0=2$ holds for $(A_2,B_2)$.

We next show that (H2) holds for $(A_2,B_2)$. Arbitrarily fix $T\in (0,\infty)$. Assume that there is $f\in Y_T$ and a subset $E\subset (0,T)$ with a positive measure so that
\begin{eqnarray}\label{zhangyu7.47}
  f=0\;\;\mbox{over}\;\; E.
 \end{eqnarray}
 We will show that
 \begin{eqnarray}\label{zhangyu7.48}
  f=0 \;\;\mbox{over}\;\; (0,T).
 \end{eqnarray}
In fact, since $|E|>0$, we can arbitrarily take $\varepsilon\in (0,|E|)$. It is clear that
\begin{eqnarray}\label{zhanguu7.49}
  |E\cap(0,T-\varepsilon)|\geq |E|-\varepsilon>0.
 \end{eqnarray}
Since $f\in Y_T$,  by Lemma~\ref{lemmazhang7.10}, we see that $f$ satisfies (\ref{zhang7.31}) and (\ref{zhang7.32}) for some
continuous and
  weakly analytic function $\widetilde g_{\varepsilon,f}: S_{\varepsilon,\theta_0}\rightarrow
  \widetilde{U}$ with some  $\theta_0\in (0,\frac{\pi}{2})$.
  Then by (\ref{zhang7.31}) and the weak analyticity of $\widetilde g_{\varepsilon,f}$, we find that for each $v\in U$, the function $t\rightarrow
  \langle f(t),v\rangle_U$ is real analytic on $(0,T-\varepsilon)$. This, along with (\ref{zhangyu7.47}) and (\ref{zhanguu7.49}), yields that for each $v\in U$,
$$
 \langle f(t),v\rangle_U=0\;\;\mbox{for each}\;\; t\in (0,T-\varepsilon).
 $$
 Sending $\varepsilon\rightarrow 0$ in the above leads to (\ref{zhangyu7.48}).
 Hence, (H2) holds for $(A_2,B_2)$.
 This ends the proof.
\end{proof}

To get the BBP decompositions for $(TP)_2^{M,y_0}$ and $(NP)_2^{T,y_0}$, we also need the following lemma:
 \begin{Lemma}\label{Lemma-ex4-T1-NT1}
   Let functions $T^0(\cdot)$ and $T^1(\cdot)$ be given  respectively by (\ref{y0-controllable}) and (\ref{Ty0}) where $(A,B)=(A_2,B_2)$.  Then the following conclusions are true:

 \noindent (i) For each $y_0\in X\setminus\{0\}$,  $T^1(y_0)=\infty$.

 \noindent (ii) If $y_0\in X\setminus\{0\}$ satisfies that $T^0(y_0)<\infty$, then $N_2(T^1(y_0),y_0)=0$.
 \end{Lemma}

\begin{proof}
 (i) By contradiction, suppose that $T^1(y_0)<\infty$ for some $y_0\in X\setminus\{0\}$. Then from  (\ref{Ty0}), we see that
 \begin{eqnarray}\label{0106-200}
  S_2(T)y_0=0   \;\;\mbox{for each}\;\; T\in \big( T^1(y_0),\infty \big).
 \end{eqnarray}
 Arbitrarily fix a $w_0\in X$. Then we see from (\ref{0106-200}) that
 \begin{eqnarray}\label{0106-201}
  \langle S_2(T)y_0,w_0 \rangle_X=0   \;\;\mbox{for each}\;\; T\in \big( T^1(y_0),\infty \big).
 \end{eqnarray}
 Since  $\{\psi_j\}_{j\geq 1}$ is a biorthogonal sequence of the Riesz basis $\{\phi_j\}_{j\geq 1}$  in $X$, we can write $y_0$ and $w_0$ in the following manner:
 \begin{eqnarray}\label{0106-202}
  y_0=\sum_{i=1}^\infty y_{0,i} \phi_i  \;\;\mbox{and}\;\; w=\sum_{j=1}^\infty w_{0,j} \psi_j.
 \end{eqnarray}
 It is clear that
  $\sum_{i=1}^\infty y_{0,i}^2 < \infty$  and $\sum_{j=1}^\infty w_{0,j}^2 < \infty$.
     These, along with the Cauchy-Schwarz inequality, yield that
 \begin{eqnarray}\label{0106-203}
  \sum_{k=1}^\infty | y_{0,k}| | w_{0,k}|
  \leq \Big({\sum_{k=1}^\infty y_{0,k}^2}\Big)^{1/2} \Big({\sum_{k=1}^\infty w_{0,k}^2}\Big)^{1/2}
  < \infty.
 \end{eqnarray}
 Meanwhile, from (\ref{0106-202}) and (\ref{wsiaoxiao7.22}), it follows that
 \begin{eqnarray}\label{yubiao8.25}
   \langle S_2(t)y_0,w_0 \rangle_X
   = \sum_{k=1}^\infty e^{-\lambda_k t} y_{0,k}w_{0,k}
   \;\;\mbox{for all}\;\;  t\in \mathbb R^+.
 \end{eqnarray}
 Since $\lambda_k>0$ for all $k\geq1$, and because   the  function
  $z\rightarrow\sum_{k=1}^N e^{-\lambda_k z} y_{0,k}w_{0,k}$  ($N\in \mathbb{N}^+$)
  is analytic over $\mathbb C^+$,
    it follows from (\ref{yubiao8.25}) and (\ref{0106-203}) that the function
    $t\rightarrow\langle S_2(t)y_0,w \rangle_X$
         is real analytic over  $(0,\infty)$. This, along with  (\ref{0106-201}), yields that
 $\langle S_2(T)y_0,w_0 \rangle_X=0$ for each $T\in (0,\infty)$.
  Because $w_0$ was arbitrarily taken from $X$, we conclude from the above that
 $ S_2(T)y_0=0 $ for each $T\in(0,\infty)$.
  This implies that
 $ y_0=\lim_{T\rightarrow 0^+} S_2(T)y_0=0$,
     which contradicts the assumption that $y_0 \in X\setminus\{0\}$. Hence, $T^1(y_0)=\infty$.

 \vskip 3pt
 (ii) Suppose that $y_0\in X\setminus\{0\}$ satisfy that $T^0(y_0)<\infty$. Arbitrarily fix a $\hat t\in \big(T^0(y_0),\infty\big)$. Then it follows from Corollary \ref{Co-NP-duality}  that
 \begin{eqnarray}\label{N-infty-both-00}
 N_2(\hat t,y_0)=\sup_{w\in D(A_2^*),B_2^*S_2^*(\hat t-\cdot)w\neq 0} \frac{\langle S_2(\hat t)y_0,w\rangle_X}{\|B_2^*S_2^*(\hat t-\cdot)w\|_{L^1(0,\hat t;U)}}<\infty.
 \end{eqnarray}
 Write
 \begin{eqnarray}\label{1219-y0-expand}
  y_0 =  \sum_{j\geq 1} y_{0,j}\phi_k  \;\;\mbox{for some}\;\;  \{y_{0,j}\}\subset \mathbb R.
 \end{eqnarray}
 Arbitrarily fix such a  $w\in D(A_2^*)$  that
  \begin{eqnarray}\label{1219-w-expand}
  w=\sum_{j=1}^N w_j\psi_k
  \;\;\mbox{for some}\;\;  \{w_j\}\subset \mathbb R
  \;\;\mbox{and}\;\; N\in\mathbb N^+.
  \end{eqnarray}

  The rest of the proof is organized by  three steps.

  \noindent \textit{Step 1. To show that there are positive constants $C_1$ and $C_2$ so that for each $s\in (2\hat t,\infty)$,
  \begin{eqnarray}\label{1219-step1-estimate}
    |\langle S_2(s)y_0,w\rangle_X|
    \leq C_1e^{-C_2 s}  \int_{\hat t}^{2\hat t}  |\langle S_2(t)y_0,w\rangle_X|\,\mathrm dt
  \end{eqnarray}}
  ~~~\,Observe from (\ref{1219-y0-expand}), (\ref{1219-w-expand}) and (\ref{wsiaoxiao7.22}) that
  \begin{eqnarray}\label{1219-step1-1}
   \langle S_2(t)y_0,w\rangle_X = \sum_{j=1}^N y_{0,j} w_j e^{-\lambda_j t}
   \;\;\mbox{for each}\;\; t\in \mathbb R^+.
  \end{eqnarray}
  Define a function $g_1$ over $\mathbb C_+$ in the following manner:
 $g_1(z) \triangleq   \sum_{j=1}^N y_{0,j} w_j e^{-\lambda_j z}$ for each $z\in\mathbb C_+$.
  Then by  (\ref{space-2-1}) and (\ref{1219-step1-1}), we find that
  \begin{eqnarray}
   g_1(\cdot+\hat t)|_{\mathbb C_+}\in \mathcal P;
   \;\;\mbox{and}\;\; g_1(\cdot)= \langle S_2(\cdot)y_0,w\rangle_X  \;\;\mbox{over}\;\; \mathbb R^+.
  \end{eqnarray}
  These, together with (\ref{space-3}), yield that there exist two positive constants $C_1$ and $C_2$, independent of $w$,  so that for each $s\in (2\hat t,\infty)$,
  \begin{eqnarray*}
    |\langle S_2(s)y_0,w\rangle_X| &=& \big|g_1\big((s-\hat t)+\hat t\big)\big|
    \leq C_1e^{-C_2(s-\hat t)} \int_0^{\hat t} |g_1(t+\hat t)| \,\mathrm dt
    \nonumber\\
    &=& C_1e^{-C_2(s-\hat t)}\int_{\hat t}^{2\hat t} |g_1(t)| \,\mathrm dt
    = C_1e^{-C_2(s-\hat t)}\int_{\hat t}^{2\hat t} |\langle S_2(t)y_0,w\rangle_X|\,\mathrm dt,
  \end{eqnarray*}
  which implies (\ref{1219-step1-estimate}).

  \noindent \textit{Step 2. To show that There are positive constants $C_1^\prime$ and $C_2^\prime$ so that for each $s\in (\hat t,\infty)$,
  \begin{eqnarray}\label{1219-step2-estimate}
    \int_0^s \|B_2^*S_2^*(t)w\|_U \,\mathrm dt \geq (1-C_1^\prime e^{-C_2^\prime s}) \int_0^\infty \|B_2^*S_2^*(t)w\|_U \,\mathrm dt
  \end{eqnarray}}
  ~~~\,From  (\ref{1219-w-expand}) and (\ref{wsiaoxiao7.23}), we find that
  \begin{eqnarray}\label{1219-step2-1}
   B_2^*S_2^*(t)w = \sum_{j=1}^N  w_j e^{-\lambda_j t} B_2^*\psi_j
   \;\;\mbox{for each}\;\; t\in \mathbb R^+.
  \end{eqnarray}
 Write $\widetilde{U}$ for the complexification of $U$.
  Define a function $g_2: \mathbb C^+\rightarrow \widetilde{U}$ in the following manner:
  \begin{eqnarray*}
   g_2(z) \triangleq   \sum_{j=1}^N  w_j e^{-\lambda_j z} B_2^*\psi_j,~z\in\mathbb C_+,
  \end{eqnarray*}
     This, along with  (\ref{space-U}) and (\ref{1219-step2-1}), yields that
  \begin{eqnarray}
   g_2(\cdot)\in \mathcal P_{\widetilde{U}}
   \;\;\mbox{and}\;\; g_2(t)= B_2^*S_2^*(t)w\in U  \;\;\mbox{for each}\;\; t\in\mathbb R^+.
  \end{eqnarray}
   These, together with Lemma~\ref{ob-estimate-YT}, yield that there exist two positive constants $C_1^\prime$ and $C_2^\prime$, independent of $w$,  so that for each $t\in (\hat t,\infty)$,
  \begin{eqnarray*}
    \|B_2^*S_2^*(t)w\|_U &=& \|g_2(t)\|_U
    \leq C_1^\prime e^{-C_2^\prime t} \|g_2(\cdot)\|_{L^1(0,\hat t;U)}
    =  C_1^\prime e^{-C_2^\prime t} \|B_2^*S_2^*(\cdot)w\|_{L^1(0,\hat t;U)}
    \nonumber\\
    &\leq& C_1^\prime e^{-C_2^\prime t} \|B_2^*S_2^*(\cdot)w\|_{L^1(\mathbb R^+;U)}.
  \end{eqnarray*}
  Thus, we find that for each $s\in (\hat t,\infty)$,
  \begin{eqnarray*}
    \int_0^s \|B_2^*S_2^*(t)w\|_U \,\mathrm dt
    &=& \int_0^\infty \|B_2^*S_2^*(t)w\|_U \,\mathrm dt
        -  \int_s^\infty \|B_2^*S_2^*(t)w\|_U \,\mathrm dt
    \nonumber\\
    &\geq& \|B_2^*S_2^*(\cdot)w\|_{L^1(\mathbb R^+;U)}
    - \int_s^\infty \big(C_1^\prime e^{-C_2^\prime t} \|B_2^*S_2^*(\cdot)w\|_{L^1(\mathbb R^+;U)} \big) \,\mathrm dt
    \nonumber\\
    &\geq& \big(1-C_1^\prime e^{-C_2^\prime s}/C_2^\prime\big) \|B_2^*S_2^*(\cdot)w\|_{L^1(\mathbb R^+;U)},
  \end{eqnarray*}
  which implies (\ref{1219-step2-estimate}).

 \noindent \textit{Step 3. To show that $N_2(T^1(y_0),y_0)=0$}

 We first claim that for each $t\in (\hat t,2\hat t)$,
 \begin{eqnarray}\label{1219-last}
   |\langle S_2(t)y_0,w\rangle_X|  \leq N_2(t,y_0) \|B_2^*S_2^*(t-\cdot)w\|_{L^1(0,t;U)}.
 \end{eqnarray}
 To this end, fix a $t\in (\hat t,2\hat t)$. There are only two possibilities on $B_2^*S_2^*(t-\cdot)w$: either $B_2^*S_2^*(t-\cdot)w\neq 0$ in $L^1(0,t;U)$ or $B_2^*S_2^*(t-\cdot)w=0$ in $L^1(0,t;U)$.

 In first case, since $\hat t>T^0(y_0)$, we see from Corollary \ref{Co-NP-duality} that (\ref{1219-last}) holds. In the second case, it follows from (ii) of Lemma \ref{Lemma-NP-decreasing} and (\ref{N-infty-both-00}) that
 \begin{eqnarray*}
  N_2(t,y_0)\leq N_2(\hat t,y_0)<\infty.
 \end{eqnarray*}
 So $(NP)^{t,y_0}_2$ has at least one admissible control. Then   there exists a control $u\in L^\infty(0,t;U)$ so that $\hat y(t;y_0,u)=0$. Thus, from
 (\ref{NNNWWW2.1}), we obtain that
 \begin{eqnarray*}
  \langle S_2(t)y_0,w \rangle_X = -\int_0^t \langle u(\tau),B_2^*S_2^*(t-\tau)w \rangle_U \,\mathrm d\tau =0,
 \end{eqnarray*}
 which  implies (\ref{1219-last}) in the case that $B_2^*S_2^*(t-\cdot)w=0$. So (\ref{1219-last}) is proved.

 Next, by (\ref{1219-step1-estimate}), (\ref{1219-last}) and (\ref{1219-step2-estimate}), we  find  that for each large enough $s\in (2\hat t,\infty)$,
\begin{eqnarray*}
  \langle S_2(s)y_0,w\rangle_X
  &\leq& C_1e^{-C_2s}  \int_{\hat t}^{2\hat t} |\langle S_2(t)y_0,w\rangle_X|  \,\mathrm dt
  \nonumber\\
  &\leq& C_1e^{-C_2s}  \int_{\hat t}^{2\hat t} N_2(t,y_0) \|B_2^*S_2^*(t-\cdot)w\|_{L^1(0,t;U)}  \,\mathrm dt
  \nonumber\\
  &\leq& C_1e^{-C_2s}   \|B_2^*S_2^*(\cdot)w\|_{L^1(\mathbb R^+;U)}
     \int_{\hat t}^{2\hat t} N_2(t,y_0)   \,\mathrm dt
     \nonumber\\
  &\leq&
  \Big(\frac{C_1e^{-C_2s}}{1-C_1^\prime e^{-C_2^\prime s}}    \int_{\hat t}^{2\hat t} N_2(t,y_0)   \,\mathrm dt \Big)
  \|B_2^*S_2^*(s-\cdot)w\|_{L^1(0,s;U)}.
\end{eqnarray*}
 Since $w$ was arbitrarily taken as in (\ref{1219-w-expand}), the above, along with  Corollary \ref{Co-NP-duality} and (ii) of Lemma \ref{Lemma-NP-decreasing}, yields that for each large enough $s\in (2\hat t,\infty)$,
 \begin{eqnarray*}
  N_2(s,y_0)\leq \frac{C_1e^{-C_2s}}{1-C_1^\prime e^{-C_2^\prime s}}
  \int_{\hat t}^{2\hat t} N_2(t,y_0)\,\mathrm dt
  \leq \frac{C_1e^{-C_2s}}{1-C_1^\prime e^{-C_2^\prime s}} \big(  \hat t  N_2(\hat t,y_0) \big),
 \end{eqnarray*}
 which, together with (\ref{N-infty-both-00}) and the first equality in (\ref{N-infty-y0}), implies that $N_2(\infty,y_0)=0$. This, along with  (vi) of Lemma \ref{Lemma-T0-T1}, yields that the conclusion (ii) is true.

 In summary, we end the proof of Lemma \ref{Lemma-ex4-T1-NT1}.

 \end{proof}

The BBP decompositions for $(A_2,B_2)$ are presented in the following
Theorem~\ref{theorem8.15wang}.

 \begin{Theorem}\label{theorem8.15wang}
 Let $\mathcal{W}$, $\mathcal{W}_{2,j}$ ($j=1,2,3$), $\mathcal{W}_{3,j}$ ($j=1,2,4$), $\mathcal V$,  $\mathcal V_{2,j}$ ($j=2,3,4$) and $\mathcal V_{3,j}$ ($j=2,3$)
 be respectively given by (\ref{zhangjinchu1}), (\ref{zhangjinchu7}), (\ref{zhangjinchu9}), (\ref{TP-divide-domain}), (\ref{Lambda-di-2-1}) and (\ref{Lambda-di-3-1})
  where $(A,B)=(A_2,B_2)$.
 Then the following conclusions are valid:

 \noindent (i) The set $\mathcal W$ is the disjoint union of the above mentioned subsets $\mathcal W_{i,j}$, and
 $\mathcal V$ is the disjoint union of the above mentioned subsets $\mathcal V_{i,j}$.

 \noindent (ii) For each $(T,y_0)\in \mathcal W_{2,1} \cup \mathcal W_{3,1} \cup \mathcal W_{3,4}$,  $(NP)_2^{T,y_0}$ has no any admissible control and does not hold the bang-bang property; For each $(T,y_0)\in \mathcal W_{2,3} \cup \mathcal W_{3,2}$, $(NP)_2^{T,y_0}$ has  the bang-bang property and the null control is not a minimal norm control to this problem; For each $(T,y_0)\in \mathcal W_{2,2}$,  $(NP)_2^{T,y_0}$ has at least one minimal norm control.

 \noindent (iii) For each $(M,y_0)\in \mathcal V_{3,3}$, $(TP)_2^{M,y_0}$ has no any admissible control and does not hold the bang-bang property; For each $(M,y_0)\in \mathcal V_{2,2} \cup \mathcal V_{3,2}$, $(TP)_2^{M,y_0}$ has  the bang-bang property; For each $(M,y_0)\in \mathcal V_{2,4}$, $(TP)_2^{M,y_0}$ has infinitely many different minimal time controls (not including the null control), and does not hold the bang-bang property; For each $(M,y_0)\in \mathcal V_{2,3}$, $(TP)_2^{M,y_0}$ has at least one minimal time control.
  \end{Theorem}
\begin{proof} By Proposition~\ref{zhangproposition7.11}, we see that (H1) and (H2) hold for $(A_2,B_2)$. Thus, all conclusions in Theorem~\ref{Proposition-NTy0-partition} and
 Theorem~\ref{Proposition-TMy0-partition} are true. From these conclusions, we see that to prove this theorem, it suffices to show that
 \begin{equation}\label{newyear7.59}
 \mathcal{W}_{1,1} \cup \mathcal{W}_{1,2} \cup \mathcal{W}_{2,4} \cup \mathcal{W}_{3,3} =\emptyset; ~
 \mathcal{V}_1 \cup \mathcal{V}_{2,1} \cup \mathcal{V}_{3,1}=\emptyset.
 \end{equation}
 Here, $\mathcal{W}_{1,j}$ ($j=1,2$), $\mathcal{W}_{2,4}$, $\mathcal{W}_{3,3}$, $ \mathcal{V}_1$,  $\mathcal{V}_{2,1}$ and $\mathcal{V}_{3,1}$ are respectively given by (\ref{zhangjinchu5}), (\ref{zhangjinchu7}), (\ref{zhangjinchu9}), (\ref{Lambda-di-0-1-1}),  (\ref{Lambda-di-2-1}) and (\ref{Lambda-di-3-1}),
 where $(A,B)=(A_2,B_2)$.

  To show  (\ref{newyear7.59}), we use  Lemma~\ref{Lemma-ex4-T1-NT1} to get that
  \begin{eqnarray}\label{0228-th7.9-1}
   T^1(y_0)=\infty  \;\;\mbox{and}\;\;   N_2(T^1(y_0),y_0)=0
   \;\;\mbox{for all}\;\;   y_0\in X\setminus\{0\}.
  \end{eqnarray}
  By the first equality in (\ref{0228-th7.9-1}) and (iv) of Lemma \ref{Lemma-NT0-T0-T1}, we deduce that
  \begin{eqnarray}\label{0228-th7.9-2}
   N_2(T^0(y_0),y_0)>0
   \;\;\mbox{for all}\;\;   y_0\in X\setminus\{0\}.
  \end{eqnarray}
  ~~~\,We now show the first equality in  (\ref{newyear7.59}). On one hand, by the definitions of $\mathcal{W}_{1,j}$ ($j=1,2$) (see (\ref{zhangjinchu5})), we find from (\ref{0228-th7.9-2}) that $\mathcal{W}_{1,1} \cup \mathcal{W}_{1,2}$ is empty. On the other hand, by contradiction, suppose that $\mathcal{W}_{2,4} \cup \mathcal{W}_{3,3}$  were not empty. Then there would be a pair $(\hat T,\hat y_0)\in \mathcal{W}_{2,4} \cup \mathcal{W}_{3,3}$. Hence, by the definitions of $\mathcal{W}_{2,4} \cup \mathcal{W}_{3,3}$ (see (\ref{zhangjinchu7}) and (\ref{zhangjinchu9})), it follows that
  \begin{eqnarray*}
   T^1(\hat y_0) \leq \hat T <\infty,
  \end{eqnarray*}
  which contradicts the first equality in (\ref{0228-th7.9-1}). So $\mathcal{W}_{2,4} \cup \mathcal{W}_{3,3}$ is empty. Thus, we have proved  the first equality in  (\ref{newyear7.59}).

  Finally, we  prove the second equality in  (\ref{newyear7.59}). On one hand, by the definitions of $\mathcal{V}_{1}$  (see (\ref{Lambda-di-0-1-1})), we find from (\ref{0228-th7.9-2}) that $\mathcal{V}_{1}$ is empty. On the other hand, by contradiction, suppose that $\mathcal{V}_{2,1} \cup \mathcal{V}_{3,1}$  were not empty. Then there would be a pair $(\hat M,\hat y_0)\in \mathcal{V}_{2,1} \cup \mathcal{V}_{3,1}$. So by the definitions of $\mathcal{V}_{2,1} \cup \mathcal{V}_{3,1}$ (see (\ref{Lambda-di-2-1}) and (\ref{Lambda-di-3-1})), it follows that
  \begin{eqnarray*}
   0< \hat M \leq N_2(T^1(\hat y_0),\hat y_0),
  \end{eqnarray*}
  which contradicts the second equality in (\ref{0228-th7.9-1}). Therefore, $\mathcal{V}_{2,1} \cup \mathcal{V}_{3,1}$ is empty. Thus, we have proved the second equality in  (\ref{newyear7.59}).

  \vskip 3pt
 In summary, we end the proof of this theorem.

%
%

\end{proof}

%
%
%

\vskip 5pt
We end this subsection with  presenting such  phenomenon that {\it for some pairs $(A_2,B_2)$, the corresponding function $T^0(\cdot)$ (given by (\ref{y0-controllable}) with $(A,B)$ being replaced by $(A_2,B_2)$) has the following property:  $T^0(y_0)\in (0,\infty)$ for some $y_0\in X$.} To see it,
some preliminaries are needed. First we notice that the  operator $A_2$ depends on the choices of  $\{\phi_j\}_{j\geq 1}$,     $\{\psi_j\}_{j\geq 1}$  and $\Lambda$; the operator $B_2$ can be arbitrarily taken from $\mathcal L(U,X_{-1})\setminus\{0\}$. For each pair $(A_2,B_2)$, we define
 \begin{eqnarray}\label{0104-T2}
   T_2(A_2,B_2) \triangleq \inf\{T\in(0,\infty)  ~:~   (A_2,B_2)   \mbox{ has }
   \mbox{ the }  L^2\mbox{-null controllability at }  T  \}.
  \end{eqnarray}
  (By  the $L^2$-null controllability at $T$ for  $(A_2,B_2)$, we mean that for each
  $y_0\in X$, there is a control $v\in L^2(0, T;U)$ so that
  $\hat y_2(T ;y_0,v)=0$.)
  Sometimes, we will use $T_2$ to denote $T_2(A_2,B_2)$, if there is no risk causing any confusion.
It is proved in   \cite{AK-2} and \cite{AK-1}  that
 $ T_2\in (0,\infty)$ for some pairs $(A_2,B_2)$. One such example (taken from  \cite{AK-1}) is as follows:
 \begin{Example}\label{Example-0104-1}
   Consider the following controlled system
  \begin{eqnarray*}
   \left\{\begin{array}{lll}
           \partial_t y-\partial_{xx}y=\delta_{x_0}v  &\mbox{in}  &(0,\pi)\times(0,\infty),\\
           y(0,\cdot)=y(\pi,\cdot)=0 &\mbox{in}  &(0,\infty),\\
           y(\cdot,0)\in L^2(0,\pi).
          \end{array}
   \right.
  \end{eqnarray*}
  One can directly check that this example can be put into the framework $(A_2,B_2)$.   According to Corollary 6.4 and Theorem 6.5 in \cite{AK-1}, there are many $x_0\in (0,\pi)$ so that  the corresponding $T_2\in (0,\infty)$.
  \end{Example}

\noindent  In the current paper, controls are taken from $L^\infty$ spaces. Thus, we define for each pair $(A_2,B_2)$,
  \begin{eqnarray}\label{0104-T-infty}
   T_\infty(A_2,B_2) \triangleq \inf\{T\in(0,\infty)  ~:~   (A_2,B_2)   \mbox{ has the }
   L^\infty\mbox{-null controllability at }  T  \}.
  \end{eqnarray}
  Also, we simply use $T_\infty$ to denote $T_\infty(A_2,B_2)$, if there is no risk to cause any confusion.

 \begin{Lemma}\label{Corollary-0104-minimal-time}
  For each pair $(A_2,B_2)$, the corresponding $T_2$ and $T_\infty$ (defined by (\ref{0104-T2}) and (\ref{0104-T-infty}), respectively)
  are the same.
   \end{Lemma}

 \begin{proof}
 It suffices to show  that
 \begin{eqnarray}\label{0104-night-1}
  T_\infty\leq T_2.
 \end{eqnarray}
 By contradiction, suppose that it was not true. Then there would be  two numbers $\hat t$ and $\hat t^\prime$ so that
 \begin{eqnarray}\label{0104-night-2}
  T_2< \hat t<\hat t^\prime<T_\infty.
 \end{eqnarray}
 Arbitrarily fix a $y_0\in X$. According to the definition of $T_2$,  there exists a control $u\in L^2(0,\hat t;U)$ so that
 \begin{eqnarray}\label{0104-night-3}
  \hat y_2(\hat t;y_0,u)=0.
 \end{eqnarray}
 Write $\widetilde u$ for the zero extension of $u$ over $(0,\hat t^\prime)$.
  According to Proposition~\ref{zhangproposition7.11}, the pair $(A_2,B_2)$
  satisfies the condition  (H1) with $p_0=2$. Thus, we   apply (H1), where $p_0=2$ and   $T=\hat t^\prime$ and $t=\hat t$,  to find a control $v_u \in L^\infty(0,\hat t^\prime;U)$ so that
 $\hat y_2(\hat t^\prime;0,\widetilde u)   =   \hat y_2(\hat t^\prime;0, v_u)$,
 which implies that
 \begin{eqnarray*}
  \hat y_2(\hat t^\prime;y_0,\widetilde u) = \hat y_2(\hat t^\prime;y_0,0) + \hat y_2(\hat t^\prime;0,\widetilde u)=\hat y_2(\hat t^\prime;y_0, v_u).
   \end{eqnarray*}
 This, along with (\ref{0104-night-3}), yields that
 \begin{eqnarray*}
   \hat y_2(\hat t^\prime;y_0, v_u)=S_2(\hat t^\prime-\hat t) \hat y_2(\hat t;y_0,u)=0.
 \end{eqnarray*}
 Since $y_0$ was arbitrarily taken from $X$, the above implies that the pair $(A_2,B_2)$ has $L^\infty$-null controllability at time $\hat t^\prime$. By this and the definition of $T_\infty$, we deduce that
 $ T_\infty \leq \hat t^\prime$,
  which contradicts (\ref{0104-night-2}). So (\ref{0104-night-1}) holds. We end the proof of this lemma.

 \end{proof}

 \begin{Remark}\label{Remark-minimal-time-infty}
  There are  systems (under the framework $(A_2,B_2)$) so that $0<T_\infty<\infty$ (see Example \ref{Example-0104-1} and  Lemma \ref{Corollary-0104-minimal-time}). With the aid of this, we can prove that for some pair
  $(A_2,B_2)$, the corresponding  function $T^0(\cdot)$, defined by (\ref{y0-controllable}), satisfies that  $T^0(y_0)\in (0,\infty)$ for some $y_0\in X$.

  Here is the argument: Suppose that for some $(A_2,B_2)$,
  \begin{equation}\label{wang8.42h}
  0<T_\infty(A_2,B_2) =  T_\infty<\infty.
  \end{equation}
  On one hand, by the first inequality in (\ref{wang8.42h}) and the definition of $T_\infty$, we can find $T\in (0, T_\infty)$ so that the pair $(A_2,B_2)$ is not $L^\infty$-null controllable. Thus there is $\hat y_0\in X$ so that for any $v\in L^\infty(0,T;U)$,
  $\hat y_2(T;\hat y_0,v)\neq 0$. Then by the definition of $T^0(\hat y_0)$ (see (\ref{y0-controllable})), we see that $T\leq T^0(\hat y_0)$, which leads to that $T^0(\hat y_0)>0$.

  On the other hand, by the last inequality in (\ref{wang8.42h})
  and the definition of $T_\infty$,  we can find
  $\hat T\in (T_\infty,\infty)$ so that the pair $(A_2,B_2)$ is the $L^\infty$-null controllable at $\hat T$. Thus,  for each $y_0\in X$ there is a control $v\in L^\infty(0,\hat T;U)$ so that $\hat y(\hat T; y_0,v)=0$. This, along with  the definition of $T^0(y_0)$ (see (\ref{y0-controllable})), yields that
  $T_0(y_0)\leq \hat T<\infty$ for all $y_0\in X$.

  In summary, we conclude that  $T^0(\hat y_0)\in (0,\infty)$.

 \end{Remark}

\section{Appendix}

\subsection{Appendix A}

 In Appendix A, we will use  the Kalman controllability decomposition to prove the following Proposition:

\begin{Proposition}\label{wangnewyearproposition8.1}
For each pair of matrices $(A,B)$ in $\mathbb R^{n\times n}\times (\mathbb R^{n\times m}\setminus\{0\})$ (with $n, m\geq 1$), the corresponding  decompositions (P1) and (P2) (given by (\ref{0127-intro-P1}) and (\ref{0127-intro-P2}), respectively) hold.
\end{Proposition}

\begin{proof}
Arbitrarily fix $(A,B)\in\mathbb R^{n\times n}\times (\mathbb R^{n\times m}\setminus\{0\})$. Let  $\mathcal R$ be given by (\ref{0125-intro-attain}).
Since $B\neq 0$,  we have that
 \begin{equation}\label{newyearyuan1.12}
 p \triangleq \mbox{dim}\; \mathcal R>0\;\;\mbox{and}\;\;\mathcal{R}\setminus\{0\}\neq\emptyset.
 \end{equation}
    We now recall the Kalman controllability decomposition of $(A,B)$ (see, for instance, Lemma 3.3.3 and Lemma 3.3.4 in \cite{E.D.Sontag}):
There exist  $K\in GL(n)$, $A_1\in\mathbb R^{p\times p}$, $A_2\in\mathbb R^{p\times(n-p)}$, $A_3\in\mathbb R^{(n-p)\times(n-p)}$ and $B_1\in\mathbb R^{p\times m}$
  so that
   \begin{eqnarray}\label{0123-intro-100}
    K^{-1} A K=  \left(
                   \begin{array}{cc}
                     A_1 & A_2 \\
                     0 & A_3 \\
                   \end{array}
                 \right)
    \;\;\mbox{and}\;\;
     K^{-1} B=\left(
                    \begin{array}{c}
                      B_1 \\
                      0 \\
                    \end{array}
                  \right),
   \end{eqnarray}
    where the pair $(A_1,B_1)$ is controllable, which is equivalent to
  \begin{eqnarray}\label{0127-A1B1-controllable}
   \mbox{rank } (B_1,A_1B_1,\cdots,A_1^pB_1)=p.
  \end{eqnarray}
  Notice that when $p=n$, the decomposition is trivial. In this case,  $A_1=A$, $B_1=B$ and $A_2$ and $A_3$ are not there.

We organize the proof by two steps as follows:

\noindent{\it Step 1. The proof of  (P2)}

For each $z_0\in \mathbb R^n\setminus\{0\}$ and $T\in(0,\infty)$, we  define
 an affiliated  minimal norm control problem:
 \begin{eqnarray}\label{0123-intro-NP-K}
 (\mathcal{NP})^{T,z_0}_K \;\;\;\;\;\;  \mathcal{N}_K(T,z_0)  \triangleq  \inf\{ \|v\|_{L^\infty(0,T;\mathbb R^m)}~:~\hat z(T;z_0,v)=0\},
 \end{eqnarray}
 where  $\hat z(\cdot;z_0,v)$ is the solution to the equation:
\begin{eqnarray}
  \left\{\begin{array}{l}\label{0123-intro-4}
         z^\prime(t)=\left(
                   \begin{array}{cc}
                     A_1 & A_2 \\
                     0 & A_3 \\
                   \end{array}
                 \right)z(t)
                 +\left(
                    \begin{array}{c}
                      B_1 \\
                      0 \\
                    \end{array}
                  \right)v(t),~0<t\leq T,\\
         z(0)=z_0.
        \end{array}
 \right.
 \end{eqnarray}
  By the invertibility of   $K$, one can easily show that   when $z_0=K^{-1}y_0$,
the problems $(\mathcal{NP})^{T,y_0}$ and $(\mathcal{NP})^{T,z_0}_K$
(given by (\ref{0123-intro-NP}) and (\ref{0123-intro-NP-K}), respectively)
 are equivalent,
 i.e., either they  have the same minimal norm controls or  both of them have no any admissible control. From (\ref{0125-intro-attain}), (\ref{0123-intro-100}) and (\ref{0127-A1B1-controllable}), it follows that
   \begin{eqnarray}\label{0125-intro-attain-K}
    \mathcal R= \mbox{span}\, (B,AB,\cdots,A^nB)
    =\mbox{span}\,K\left(
                          \begin{array}{c}
                            B_1,A_1B_1,\cdots,A_1^nB_1 \\
                            0 \\
                          \end{array}
                        \right)
    =K(\widetilde{\mathbb R}^p),
   \end{eqnarray}
   where the span of a matrix denotes the subspace generated by all columns of the matrix, and   $\widetilde{\mathbb R}^p$ is the following subspace:
\begin{eqnarray}\label{0127-Rp-tilde}
 \widetilde{\mathbb R}^p  \triangleq  \big\{(z_1,z_2,\cdots,z_n)\in\mathbb R^n
 ~:~  z_{p+1}=\cdots=z_n=0 \big\}.
\end{eqnarray}
By (\ref{newyearyuan1.12}), we see that $\widetilde{\mathbb R}^p\setminus\{0\}\neq\emptyset$.
From the equivalence of $(\mathcal{NP})^{T,y_0}$ and $(\mathcal{NP})^{T,z_0}_K$ (with  $z_0=K^{-1}y_0$), (\ref{0125-intro-attain-K}) and (\ref{0127-intro-P2}), we see that to prove (P2),  it suffices to show the following BBP decomposition  for $(\mathcal{NP})^{T,z_0}_K$:
\begin{eqnarray*}
  \mbox{\textbf{(Q2)}} \begin{array}{ll}
       &\bullet  ~\,\mbox{For each } (T,z_0)\in (0,\infty) \times (\widetilde{\mathbb R}^p\setminus\{0\}), \;\;(\mathcal{NP})^{T,z_0}_K
                  \mbox{ has the bang-bang property.}  \\
             &\bullet ~\,\mbox{For each } (T,z_0)\in (0,\infty) \times (\mathbb R^n\setminus\widetilde{\mathbb R}^p), \;\;(\mathcal{NP})^{T,z_0}_K
                  \mbox{ has no  admissible control. }\\
              \end{array}
 \end{eqnarray*}
 To show the first conclusion in (Q2), we let
\begin{equation}\label{newyear8.3}
(T,z_0)\in (0,\infty) \times (\widetilde{\mathbb R}^p\setminus\{0\}).
\end{equation}
  Write   $z_{0,1}$ for the first $p$ components  of $z_0$. Since $z_0\in\widetilde{\mathbb R}^p$, it follows that $z_0=(z_{0,1},0)$, if $p<n$; and $z_0=z_{0,1}$, if $p=n$.
  Thus, for each $v\in L^\infty(0,T;\mathbb R^m)$, the solution $\hat z(\cdot;z_0,v)$  of the equation (\ref{0123-intro-4}) satisfies
  that
  \begin{eqnarray*}
  \hat z(t;z_0,v) =
  \left\{\begin{array}{l}
  \big(\hat z_1(t;z_{0,1},v) ,0 \big)
  \;\;\mbox{for all}\;\;
  t\in [0,T],\;\;\mbox{when}\;\; p<n,\\
   \hat z_1(t;z_{0,1},v)
  \;\;\mbox{for all}\;\;
  t\in [0,T],\;\;\mbox{when}\;\; p=n,
     \end{array}
  \right.
 \end{eqnarray*}
 where $\hat z_1(\cdot;z_{0,1},v)$ solves the following equation:
 $$
  z_1^\prime(t)=A_1z_1(t)+B_1v(t), ~0<t\leq T;\;\; z_1(0)=z_{0,1}.
 $$
  This, along with the controllability of  $(A_1,B_1)$ (which follows from (\ref{0127-A1B1-controllable}), see, for instance, Theorem 3 on Page 89 in \cite{E.D.Sontag}), indicates  that $(\mathcal{NP})_K^{T,z_0}$ has an admissible control.  Then by a standard way (see for instance \cite[Lemma 1.1]{HOF1}),  we can deduce that $(\mathcal{NP})_K^{T,z_0}$ has a minimal norm control.

 Meanwhile, according to the Pontryagin maximum principle for $(\mathcal{NP})_K^{T,z_0}$ (see, for instance, \cite[Theorem 1.1.1]{HOF}),
  there is $\eta_1$ in $\mathbb R^p\setminus\{0\}$  so that each minimal norm control $v^*$ to $(\mathcal{NP})_K^{T,z_0}$ verifies that
 \begin{eqnarray}\label{0131-last-2}
  \big\langle v^*(t),B_1^*e^{A_1^*(T-t)}\eta_1 \big\rangle_{\mathbb R^m}
  = \max_{\|w\|_{\mathbb R^m} \leq \mathcal{N}_K(T,z_0)}
  \big\langle w,B_1^*e^{A_1^*(T-t)}\eta_1 \big\rangle_{\mathbb R^m}
  \;\;\mbox{a.e.}\;\;  t\in \big(0,T\big).
 \end{eqnarray}
where   $\mathcal{N}_K(T,z_0)$ is given by (\ref{0123-intro-NP-K}).
 Besides, since $\eta_1\neq 0$ and the  function  $t\rightarrow B_1^*e^{A_1^*(T-t)}$  is real analytic over $\mathbb{R}$, it follows from  (\ref{0127-A1B1-controllable}) that the set $\big\{t\in(0,T)~:~B_1^*e^{A_1^*(T-t)} \eta_1=0\big\}$  has measure zero. From this and  (\ref{0131-last-2}), we see that  $\mathcal{(NP)}^{T,z_0}_K$ has the bang-bang property.   So  the first conclusion in (Q2) is true.

 To verify the second conclusion in (Q2), we first notice that when $p=n$,
 $\mathbb R^n\setminus\widetilde{\mathbb R}^p$ is empty.
 Thus, we can assume, without loss of generality,  that $p<n$. Arbitrarily fix $(T,z_0)\in (0,\infty) \times (\mathbb R^n\setminus\widetilde{\mathbb R}^p)$.
  Then from the equation (\ref{0123-intro-4}), we see that any control $v$ has no any influence to the last $(n-p)$ components of the solution $\hat z(\cdot;z_0,v)$.
Thus,  for each control $v$ in $L^\infty(0,T;\mathbb R^m)$, the solution
  $\hat z(\cdot;z_0,v)$ of the equation (\ref{0123-intro-4}) satisfies that
$\hat z(T;z_0,v) \neq 0$. Hence, $\mathcal{(NP)}^{T,z_0}_K$  has no any admissible control.  This proves the second conclusion in (Q2). Hence, the decomposition  (Q2) holds. Consequently, (P2) is true.

\vskip 3pt

\noindent{\it Step 2. The proof of  (P1)}

  For each $z_0\in \mathbb R^n\setminus\{0\}$ and $M\in(0,\infty)$, we define an affiliated  minimal time control problem:
 \begin{eqnarray}\label{0123-intro-TP-K}
  (\mathcal{TP})^{M,z_0}_K \;\;\;\;\;\; \mathcal{T}_K(M,z_0)  \triangleq  \{ \hat t>0~:~\exists\,u\in\mathbb U^M
 \;\;\mbox{s.t.}\;\;  z(\hat t;z_0,u)=0\},
 \end{eqnarray}
 where $\mathbb U^M$ is given by (\ref{0125-UM}), and $z(\cdot;z_0,u)$ is the solution to the equation:
     \begin{eqnarray}
  \left\{\begin{array}{l}\label{0123-intro-3}
         z^\prime(t)=\left(
                   \begin{array}{cc}
                     A_1 & A_2 \\
                     0 & A_3 \\
                   \end{array}
                 \right)z(t)
                 +\left(
                    \begin{array}{c}
                      B_1 \\
                      0 \\
                    \end{array}
                  \right)u(t),~t>0,\\
         z(0)=z_0.
        \end{array}
 \right.
\end{eqnarray}
 Two observations are given in order: First, by the invertibility of  $K$, one can easily see that
the problems $(\mathcal{TP})^{M,y_0}$ and $(\mathcal{TP})^{M,z_0}_K$
 (given by (\ref{0123-intro-TP}) and (\ref{0123-intro-TP-K}), respectively)
  are equivalent, i.e., either they  have the same minimal time controls or  both
   of them have no any admissible control.
Second,  from (\ref{0123-intro-NP}), one can easily check that when $y_0\in \mathcal{R}\setminus\{0\}$,
the function $\mathcal{N}(\cdot,y_0)$ has the properties: it is decreasing over $(0,\infty)$; for each $T\in (0,\infty)$, $\mathcal{N}(T,y_0)\in (0,\infty)$.
Hence, for each $y_0\in \mathcal{R}\setminus\{0\}$,
$\lim_{T\rightarrow\infty} \mathcal{N}(T,y_0)$ exists and is a finite and non-negative number. Meanwhile,
by the equivalence between $(\mathcal{NP})^{T,y_0}$ and $(\mathcal{NP})^{T,z_0}_K$  (with $z_0=K^{-1}y_0$), it follows  that for each $T>0$, $\mathcal{N}(T,y_0)=\mathcal{N}_K(T,z_0) $. These imply that
\begin{eqnarray}\label{Appendixwang8.13}
 \lim_{T\rightarrow\infty} \mathcal{N}(T,y_0)
 = \lim_{T\rightarrow\infty} \mathcal{N}_K(T,z_0)<\infty
 \;\;\mbox{when}\;\;  z_0=K^{-1}y_0\;\;\mbox{and} \;\; y_0\in \mathcal{R}\setminus\{0\}.
\end{eqnarray}
From the above-mentioned two observations, as well as (\ref{0125-intro-attain-K}) and (\ref{0127-intro-P1}),
we find that to prove (P1),  it suffices to show the following BBP decomposition  for $(\mathcal{TP})^{M,z_0}_K$:
\begin{eqnarray*}
  \mbox{\textbf{(Q1)}} \begin{array}{ll}
       &\bullet  ~\,\mbox{For each } (M,z_0)\in   \mathcal D_{bbp}^K, \;\;(\mathcal{TP})^{M,z_0}_K
                  \mbox{ has the bang-bang property}.  \\
       &\bullet ~\,\mbox{For each } (M,z_0)\in \big((0,\infty) \times (\mathbb R^n\setminus\{0\})\big) \setminus \mathcal D_{bbp}^K, \;\;(\mathcal{TP})^{M,z_0}_K
                  \mbox{ has } \mbox{ no }\\
       &  ~\,  \mbox{ any } \mbox{ admissible } \mbox{ control }.
       \end{array}
 \end{eqnarray*}
Here,
\begin{eqnarray}\label{0127-D-K}
  \mathcal D_{bbp}^K \triangleq  \big\{(M,z_0)\in (0,\infty)\times (\widetilde{\mathbb R}^p\setminus\{0\})
  ~:~  M> \lim_{T\rightarrow\infty} \mathcal{N}_K(T,z_0)\big\},
 \end{eqnarray}
 where $\widetilde{\mathbb R}^p$ and $\mathcal{N}_K(T,z_0)$ are given by (\ref{0127-Rp-tilde}) and (\ref{0123-intro-NP-K}), respectively.
 From (\ref{0127-D-K}),  (\ref{newyearyuan1.12}), (\ref{wangyu1.8}) and (\ref{Appendixwang8.13}), one can easily check that
  \begin{equation}\label{Appendixwang8.17}
 \mathcal D_{bbp}^K\neq\emptyset\;\;\mbox{ and}\;\;\mathcal D_{bbp}\neq\emptyset.
 \end{equation}

Before proving the decomposition (Q1), we  observe that by the first conclusion in  (Q2), we can use the same way used in the proof of \cite[Proposition 4.4]{WXZ} to get the following conclusion: When
  $z_0\in\widetilde{\mathbb R}^p\setminus\{0\}$,
 \begin{eqnarray}\label{0127-existence-TP}
 (\mathcal{TP})_K^{M,z_0}\;\;\mbox{has a minimal time control}\;\;\Longleftrightarrow  \infty>M >  \lim_{T\rightarrow\infty} \mathcal{N}_K(T,z_0).
 \end{eqnarray}

  To show the first conclusion in (Q1), we let $(M,z_0)\in   \mathcal D_{bbp}^K$.
  Then, it follows from (\ref{0127-existence-TP}) and (\ref{0127-D-K})  that  $(\mathcal{TP})_K^{M,z_0}$ has at least one minimal time control.

  Write   $z_{0,1}$ for the first $p$ components  of $z_0$. Since $z_0\in\widetilde{\mathbb R}^p$, it follows that $z_0=(z_{0,1},0)$ when $p<n$; while $z_0=z_{0,1}$ when $p=n$. Then by (\ref{0123-intro-3}), we can easily check that
  \begin{eqnarray*}
  z(t;z_0,v) =
  \left\{\begin{array}{l}
  \big(z_1(t;z_{0,1},u) ,0 \big)
  \;\;\mbox{for all}\;\;
  t\geq 0,\;\;\mbox{when}\;\; p<n,\\
    z_1(t;z_{0,1},u)
  \;\;\mbox{for all}\;\;
  t\geq 0,\;\;\mbox{when}\;\; p=n,
     \end{array}
  \right.
 \end{eqnarray*}
 where $ z_1(\cdot;z_{0,1},u)$ solves the following equation:
 $$
  z_1^\prime(t)=A_1z_1(t)+B_1u(t), ~0<t<\infty,\;\;  z_1(0)=z_{0,1}.
 $$
   From this, we can use the Pontryagin maximum principle for $(\mathcal{TP})_K^{M,z_0}$ (see, for instance,
  \cite[Theorem 1.1.1]{HOF}) to find  $\eta_2\in\mathbb R^p\setminus\{0\}$  so that each minimal time control $u^*$ to $(\mathcal{TP})_K^{M,z_0}$ verifies that for a.e. $t\in \big(0,\mathcal{T}_K(M,z_0)\big)$,
 \begin{eqnarray}\label{0131-last-1}
  \big\langle u^*(t),B_1^*e^{A_1^*(\mathcal{T}_K(M,z_0)-t)}\eta_2 \big\rangle_{\mathbb R^m}
  = \max_{\|w\|_{\mathbb R^m} \leq M}
  \big\langle w,B_1^*e^{A_1^*(\mathcal{T}_K(M,z_0)-t)}\eta_2 \big\rangle_{\mathbb R^m}.
 \end{eqnarray}
  Meanwhile, since $\eta^*\neq 0$ and the  function  $t\rightarrow B_1^*e^{A_1^*(T-t)}$  is real analytic over $\mathbb{R}$,
  the set $\big\{t\in\big(0,\mathcal{T}_K(M,z_0)\big)~:~B_1^*e^{A_1^*(\mathcal{T}_K(M,z_0)-t)} \eta_2=0\big\}$ has measure zero. This, along with (\ref{0131-last-1}), yields that $(\mathcal{TP})_K^{M,z_0}$ has  the bang-bang property.  Hence,  the first conclusion in (Q1) is true.

 To show the second conclusion in (Q1), we let
\begin{equation}\label{newyear8.15}
 (M,z_0)\in \big((0,\infty) \times (\mathbb R^n\setminus\{0\})\big) \setminus \mathcal D_{bbp}^K.
\end{equation}
 Then, there are only two possibilities on the pair $(M,z_0)$ as follows: First,   $(M,z_0)$ verifies that $z_0\in \widetilde{\mathbb R}^p\setminus\{0\}$ and $0<M\leq \lim_{T\rightarrow\infty} \mathcal{N}_K(T,z_0)$; Second,  $(M,z_0)\in (0,\infty)\times (\mathbb R^n\setminus \widetilde{\mathbb R}^p)$. In the first case, it   follows from (\ref{0127-existence-TP}) that $(\mathcal{TP})_K^{M,z_0}$ has no any admissible control. In the second case, we have that $p<n$ and the last $(n-p)$ components of $z_0$ are not all zero. Then by
       (\ref{0123-intro-3}), we find that
$z(T;z_0,u) \neq 0$ for all $u\in L^\infty(\mathbb R^+;\mathbb R^m)$ and
 $T\in(0,\infty)$.
 This implies that $(\mathcal{TP})^{M,z_0}_K$  has no any admissible control. Hence,  the second conclusion in (Q1) is also true. So the BBP decomposition   (Q1) holds. Consequently, (P1) stands.

\vskip 3pt
 In summary, we end the proof of (P1) and (P2), through using the Kalman controllability decomposition.

\end{proof}

\subsection{Appendix B}

In Appendix B, we will show that each pair of matrices $(A,B)$ in $\mathbb R^{n\times n}\times (\mathbb R^{n \times m}\setminus\{0\})$ (with $n,m\geq 1$)
holds the properties (H1) and (H2).

\begin{Proposition}\label{Lemma-8ex1-1}
 Any pair of matrices  $(A,B)\in\mathbb R^{n\times n}\times (\mathbb R^{n \times m}\setminus\{0\})$ (with $n,m\geq 1$) satisfies  (H1) (with $p_0=2$) and (H2).

 \end{Proposition}

 \begin{proof}
 Arbitrarily fix $(A,B)\in\mathbb R^{n\times n}\times (\mathbb R^{n \times m}\setminus\{0\})$. We organize the proof by two steps.

 In Step 1, we show  that (H1) (with $p_0=2$) holds  for the pair $(A,B)$. For this purpose, we will show that $(A,B)$ satisfies the conclusion (iii) of Lemma \ref{Lemma-H3-eq} (with $p_2=2$). When the later is done, it follows from Lemma \ref{Lemma-H3-eq} that (H1) (with $p_0=2$) holds  for the pair $(A,B)$.

  The remainder of this step is to  show that (iii) of Lemma \ref{Lemma-H3-eq} (with $p_2=2$) holds for the pair $(A,B)$.   Arbitrarily fix  $0<t<T<\infty$. Define  the following two spaces:
 \begin{eqnarray*}\label{1215-ex1-O1}
  \mathcal O_1 \triangleq \{B^* e^{A^*(T-\cdot)}z|_{(0,t)} \in L^2(0,t;\mathbb R^m) ~:~ z\in\mathbb R^n\},\;\;\mbox{with the norm}\;\;\|\cdot\|_{L^2(0,t;\mathbb R^m)},
 \end{eqnarray*}
 and
 \begin{eqnarray*}\label{1215-ex1-O2}
  \mathcal O_2 \triangleq \{B^* e^{A^*(T-\cdot)}z|_{(t,T)} \in L^1(t,T;\mathbb R^m) ~:~ z\in\mathbb R^n\},
 \;\;\mbox{with the norm}\;\;\|\cdot\|_{L^1(t,T;\mathbb R^m)}.
 \end{eqnarray*}
 It is clear that they are finitely dimensional spaces.
 Then define a map $\mathcal F:~\mathcal O_2\rightarrow \mathcal O_1$ by setting
 \begin{eqnarray}\label{1215-ex1-def-F}
  \mathcal F \big(B^* e^{A^*(T-\cdot)}z|_{(t,T)}\big)  \triangleq   B^* e^{A^*(T-\cdot)}z|_{(0,t)}
  \;\;\mbox{for each}\;\;   z\in \mathbb R^n.
 \end{eqnarray}
 By the analyticity of the function  $t \mapsto B^*e^{A^* t}$, $t\in \mathbb{R}$, one can easily check that the map $\mathcal F$ is well defined.
   It is clear that  $\mathcal F$ is linear (from the finitely dimensional space $\mathcal O_2$ to the finitely dimensional space $\mathcal O_1$). Thus,  $\mathcal F$ is bounded. Then it follows by (\ref{1215-ex1-def-F}) that there is a positive constant $C(T,t)$ so that
 \begin{eqnarray*}
       \|B^* e^{A^*(T-\cdot)}z\|_{L^2(0,t;\mathbb R^m)}
  \leq C(T,t)   \|B^* e^{A^*(T-\cdot)}z\|_{L^1(t,T;\mathbb R^m)}
  \;\;\mbox{for each}\;\;   z\in \mathbb R^n.
 \end{eqnarray*}
 This, along with the definition of $Y_T$ (see (\ref{ob-space})), yields that
 \begin{eqnarray*}
  \|g\|_{L^2(0,t;\mathbb R^m)}
  \leq C(T,t)  \|g\|_{L^1(t,T;\mathbb R^m)}
  \;\;\mbox{for each}\;\;    g\in  Y_T,
 \end{eqnarray*}
 which leads to the conclusion (iii) of Lemma \ref{Lemma-H3-eq} (with $p_2=2$).

In Step 2,  we will prove that (H2) holds for the pair $(A,B)$. To this end, we first show that (H4) holds for the pair $(A,B)$.
 In the finitely dimensional setting, we have that for each $z\in \mathbb{R}^n$ and each $T>0$, the function $\widetilde{B^*S^*}(T-\cdot)z$ (defined by (\ref{wang1.15})) is the same as $B^*e^{A^* (T-\cdot)}$ over  $[0,T]$.
 From this and the analyticity of the function  $t \mapsto B^*e^{A^* t}$, $t\in \mathbb{R}$, one can easily check that (H4) holds for the pair $(A,B)$.
Next, we claim that
 for each $T\in(0,\infty)$, the space $X_T$ (defined by (\ref{assumptionspace3})) is the same as $Y_T$. In fact,  it follows from (\ref{assumptionspace3}) that for each $T>0$, $X_T$ is a finitely
  dimensional  subspace in $L^1(0,T;\mathbb R^m)$.  Thus, for each $T>0$, $X_T$  is closed in $L^1(0,T;\mathbb R^m)$. Then we find from (\ref{ob-space}) that
 $X_T=Y_T$ for all $T\in(0,\infty)$.
   From this, it follows that  the conditions (H4) and (H2) are the same. Therefore, (H2) holds for the pair $(A,B)$. This ends the proof of this proposition.

 \end{proof}

\subsection{Appendix C}

In Appendix C, we will explain that the BBP decompositions (P1) and (P2) (given by (\ref{0127-intro-P1}) and (\ref{0127-intro-P2}), respectively) are consequences of
 Theorem \ref{Proposition-NTy0-partition} and Theorem~\ref{Proposition-TMy0-partition}. To see these, we need one lemma. In the proof of  this lemma, the following well known result  (see, for instance,
\cite[Section 3.3, Chapetr 3] {E.D.Sontag}) is used.

  \begin{Lemma}\label{1217-Lemma-controllable-subspace}
   Let $(A,B)\in\mathbb R^{n\times n}\times (\mathbb R^{n \times m}\setminus\{0\})$ (with $n,m\geq 1$). Let $\mathcal{R}_T$ and $\mathcal{R}_T^0$ (with $T>0$)
   be given respectively by (\ref{attainable-space}) and (\ref{R0T}). Let $\mathcal R$ be given by (\ref{0125-intro-attain}).
Define the following subspace
     \begin{eqnarray}\label{1217-controllable-subspace}
     \mathcal C_T  \triangleq  \{y_0\in\mathbb R^n  ~:~ \exists\,v\in L^\infty(0,T;\mathbb R^m)
     \;\;\mbox{s.t.}\;\;  \hat y_1(T;y_0,v)=0\},  ~T>0,
   \end{eqnarray}
   where $\hat y_1(\cdot;y_0,v)$ denotes the solution of (\ref{0123-intro-2}).
       Then it holds that
       $$
       \mathcal C_T=\widehat{\mathcal R}=\mathcal R_T=\mathcal R_T^0
       \;\;\mbox{for all}\;\;  T>0.
       $$

  \end{Lemma}

The following lemma  concern some special properties on the functions $T^0(\cdot)$ and $T^1(\cdot)$ (defined respectively by (\ref{y0-controllable}) and
   (\ref{Ty0})).

 \begin{Lemma}\label{Lemma-8ex1-2}
  Let $(A,B)\in\mathbb R^{n\times n}\times (\mathbb R^{n \times m}\setminus\{0\})$. Let $\mathcal R$   be given by (\ref{0125-intro-attain}).
   Then the functions $T^0(\cdot)$ and $T^1(\cdot)$ (defined respectively by (\ref{y0-controllable}) and
   (\ref{Ty0}))
  have the following properties:

  \noindent (i) For each $y_0\in \mathcal R$, $T^0(y_0)=0$,
    while for each $y_0\in\mathbb R^n\setminus \mathcal R$, $T^0(y_0)=\infty$.

  \noindent (ii) For each $y_0 \in\mathbb R^n\setminus\{0\}$, $N(T^0(y_0),y_0)=\infty$.

  \noindent (iii) For each $y_0\in\mathbb R^n\setminus\{0\}$, $T^1(y_0)=\infty$.

  \noindent (iv) For each $y_0\in \mathcal R\setminus\{0\}$, $N(T^1(y_0),y_0)<\infty$.

 \end{Lemma}

 \begin{proof}
  (i) We first prove that $T^0(y_0)=0$ for each $y_0\in \mathcal R$. Arbitrarily fix $y_0\in\mathcal R$ and $\hat t \in (0,\infty)$. According to  Lemma \ref{1217-Lemma-controllable-subspace}, there is $v\in L^\infty(0,\hat t; \mathbb R^m)$ so that
  $\hat y_1(\hat t;y_0,v)=0$
  From this and the definition of $T^0(y_0)$ (see (\ref{y0-controllable})), we deduce that $T^0(y_0)\leq \hat t$. Since $\hat t$ was arbitrarily taken from $(0,\infty)$, it follows that $T^0(y_0)=0$.

  Next, we verify that $T^0(y_0)=\infty$ for each $y_0\in\mathbb R^n\setminus \mathcal R$. By contradiction, suppose that $T^0(\hat y_0)<\infty$ for some $\hat y_0\in\mathbb R^n\setminus \mathcal R$. Then from  the definition of $T^0(\hat y_0)$ (see (\ref{y0-controllable})), there would be  $\hat t^\prime\in \big(T^0(\hat y_0),\infty\big)$ and  $\hat v\in L^\infty(0,\hat t^\prime; \mathbb R^m)$ so that
  $\hat y_1(\hat t^\prime;\hat y_0, \hat v)=0$.
    This, along with the definition of $\mathcal C_{\hat t^\prime}$ (given by (\ref{1217-controllable-subspace}) with $T=\hat t^\prime$), implies that $\hat y_0 \in \mathcal C_{\hat t^\prime}$.  Then by  Lemma \ref{1217-Lemma-controllable-subspace}, we find that $\hat y_0 \in \mathcal R$, which contradicts the assumption that $\hat y_0\in\mathbb R^n\setminus \mathcal R$.  This ends the proof of the conclusion (i).

 (ii) Let $y_0\in\mathbb R^n\setminus\{0\}$. There are only two possibilities on $y_0$: either $y_0\in\mathcal R\setminus\{0\}$ or $y_0\in \mathbb R^n\setminus\mathcal R$.
 In the case that $y_0\in\mathcal R\setminus\{0\}$, we see from (i) of this lemma that $T^0(y_0)=0$. Then by (iv) of Lemma \ref{Lemma-T0-T1}, we have that
 $N(T^0(y_0),y_0)=N(0,y_0)=\infty$.
  In the case that $y_0 \in \mathbb R^n\setminus\mathcal R$, we find from (i) of this lemma that  $T^0(y_0)=\infty$. Then by (ii) of Lemma \ref{Lemma-NT0-T0-T1}, it follows that $N(T^0(y_0),y_0)=\infty$.

 (iii) Let $y_0\in \mathbb R^n\setminus\{0\}$. Since $\{e^{At}\}_{t\in\mathbb R^+}$ has the backward uniqueness property, we find from the definition of $T^1(y_0)$ (see (\ref{Ty0})) that the conclusion (iii) holds.

 (iv) Let $y_0\in\mathcal R\setminus\{0\}$. Then it follows by the conclusion (i) of this lemma that $T^0(y_0)=0$. This, along with (v) of Lemma \ref{Lemma-NT0-T0-T1}, yields that $N(T^1(y_0),y_0)<\infty$.

 \vskip 3pt
 In summary, we finish  the proof of this lemma.

 \end{proof}

 \begin{Proposition}\label{NEWYearprop8.4} For each pair $(A,B)$
 in $\mathbb R^{n\times n}\times (\mathbb R^{n \times m}\setminus\{0\})$ (with $n,m\geq 1$), the BBP decompositions (P1) and (P2) (given respectively by (\ref{0127-intro-P1}) and (\ref{0127-intro-P2})),  are the consequences of
 Theorem \ref{Proposition-NTy0-partition} and Theorem~\ref{Proposition-TMy0-partition} respectively.
 \end{Proposition}
 \begin{proof} Arbitrarily fix a pair $(A,B)$
 in $\mathbb R^{n\times n}\times (\mathbb R^{n \times m}\setminus\{0\})$. By
 Proposition~\ref{Lemma-8ex1-1}, $(A,B)$ satisfies (H1) and (H2). Then all conclusions in Theorem \ref{Proposition-NTy0-partition} and Theorem~\ref{Proposition-TMy0-partition} hold.
 By (i)-(iii) of Lemma~\ref {Lemma-8ex1-2}, (vi) of
 Lemma~\ref{Lemma-T0-T1}, the first conclusion in Theorem \ref{Proposition-NTy0-partition} and the first conclusion in  Theorem~\ref{Proposition-TMy0-partition}, we can easily check that
 $$
 \mathcal{W}=\mathcal{W}_{3,2}\cup \mathcal{W}_{3,4},\; \mathcal{V}=\mathcal{V}_{3,1}\cup\mathcal{V}_{3,2}\cup\mathcal{V}_{3,3};
 $$
$$
\mathcal{W}_{3,2}=\mathcal{R}\setminus\{0\},\; \mathcal{W}_{3,4}=\mathbb{R}^n\setminus\mathcal{R},\; \mathcal{V}_{3,2}=
\mathcal D_{bbp},\; \mathcal{V}_{3,1}\cup\mathcal{V}_{3,3}=\mathcal X_1\setminus \mathcal D_{bbp}.
$$
(Here, $\mathcal{R}$, $\mathcal D_{bbp}$ and $\mathcal X_1$ are respectively given by (\ref{0125-intro-attain}),
(\ref{wangyu1.8})
and (\ref{NEWYEAR1.6})).  These, along with the conclusions (iii) and (iv) in Theorem \ref{Proposition-NTy0-partition} and the conclusions (ii) and (v) in Theorem~\ref{Proposition-TMy0-partition}, yields that
the BBP decompositions (P1) and (P2) holds for the pair $(A,B)$.
This ends the proof.

\end{proof}

\subsection{Appendix D}

In Appendix D, we provide the proofs of Proposition~\ref{huangwanghenproposition2.1} and Lemma~\ref{ad-control-ob}, respectively.

 \begin{proof}[Proof of Proposition~\ref{huangwanghenproposition2.1}]
      Arbitrarily fix  $T\in(0,\infty)$,  $v\in L^\infty(0,T;U)$ and $z\in D(A^*)$.
      Since $X_{-1}$
  is the dual of $D(A^*)$ with respect to the pivot space $X$, we have that
   \begin{equation}\label{waidarenwu2.4}
\big \langle \int_0^T S_{-1}(T-t)Bv(t) \,\mathrm dt, z \big\rangle _X
=\big \langle \int_0^T S_{-1}(T-t)Bv(t) \,\mathrm dt, z \big\rangle _{X_{-1}, D(A^*)}.
\end{equation}
Because
$S_{-1}(T-\cdot)Bv(\cdot)  \in   L^1(0,T;X_{-1})$,
we have  that
\begin{equation}\label{waidarenwu2.5}
\big \langle \int_0^T S_{-1}(T-t)Bv(t)\, \mathrm dt, z \big\rangle _{X_{-1}, D(A^*)}
=\int _0^T\langle S_{-1}(T-t)Bv(t), z \rangle _{X_{-1}, D(A^*)}  \,\mathrm dt.
\end{equation}
We next claim that
\begin{equation}\label{whynot2.5}
(S_{-1})^*(T-t)z=S^*(T-t)z\;\;\mbox{in}\;\;D(A^*),\;\;\mbox{for all}\;\;t\in[0,T].
\end{equation}
Indeed,  since  $\{S_{-1}(t)\}_{t\in \mathbb{R}^+}$ is  the extension of $\{S(t)\}_{t\in \mathbb{R}^+}$ on $X_{-1}$, and because $X_{-1}$ is the dual of $D(A^*)$ with respect to the pivot space $X$ ,
we find that for each $s\geq 0$ and $w\in X$,
\begin{eqnarray*}
 \langle S^*_{-1}(s)z,w \rangle_{D(A^*),X_{-1}}
 &=&  \langle z,S_{-1}(s) w \rangle_{D(A^*),X_{-1}} =\langle z, S(s)w \rangle_{D(A^*),X_{-1}}
 \nonumber\\
 &=& \langle z, S(s) w \rangle_{X} = \langle S^*(s) z,w \rangle_{X}
  = \langle S^*(s)z,w \rangle_{D(A^*),X_{-1}}.
\end{eqnarray*}
Since $X$ is dense in $X_{-1}$, the above implies that for all $s\geq 0$ and $\hat w\in X_{-1}$,
\begin{eqnarray*}
 \langle S^*_{-1}(s)z, \hat w \rangle_{D(A^*),X_{-1}}=\langle S^*(s)z, \hat w \rangle_{D(A^*),X_{-1}}.
\end{eqnarray*}
This leads to (\ref{whynot2.5}).
From (\ref{whynot2.5}),  we find that
\begin{equation}\label{waidarenwu2.6}
\int _0^T\langle S_{-1}(T-t)Bv(t), z \rangle _{X_{-1}, D(A^*)} \,\mathrm dt
=
\int _0^T\langle v(t), B^*S^*(T-t)z \rangle _{X_{-1}, D(A^*)} \,\mathrm dt.
\end{equation}

Now, (\ref{NNNWWW2.2}) follows from (\ref{waidarenwu2.4}),
(\ref{waidarenwu2.5}) and (\ref{waidarenwu2.6}) immediately.
 This ends the proof of Proposition~\ref{huangwanghenproposition2.1}.
  \end{proof}

 \begin{proof}[Proof of Lemma~\ref{ad-control-ob}]
  Arbitrarily fix $0<T<\infty$ and $z\in D(A^*)$. Then it follows from (\ref{NNNWWW2.2}) that
  \begin{eqnarray*}
  \|B^* S^*(T-\cdot)z\|_{L^2(0,T;U)}
  &=&  \sup_{\|u\|_{L^2(0,T;U)\leq 1}} \Big\langle z , \int_0^T S_{-1}(T-\cdot)Bu(t) \,\mathrm dt  \Big\rangle_X
  \nonumber\\
  &\leq& \sup_{\|u\|_{L^2(0,T;U)\leq 1}} \|z\|_X  \big\| \int_0^T S_{-1}(T-\cdot)Bu(t) \,\mathrm dt \big\|_X,
  \end{eqnarray*}
  which, along with (\ref{admissible-control}), leads to (\ref{admissible-observable}). This ends the proof of  Lemma~\ref{ad-control-ob}.

 \end{proof}

\subsection{Appendix E}

In Appendix E, we give the proof of Lemma~\ref{Lemma-t-u-converge}.

\begin{proof}[Proof of Lemma~\ref{Lemma-t-u-converge}]
  Suppose that  (\ref{wanggengsheng3.1}) holds for some $\{T_n\}_{n=1}^\infty$, $\widehat T$ in $[0,\infty)$, some $\{u_n\}_{n=1}^\infty$ and $\hat u$ in $L^2(\mathbb R^+;U)$.
 Arbitrarily fix a $y_0\in X$.
  We will prove  (\ref{time-control-converge}) by  two steps as follows.

 \noindent \textit{\noindent Step 1. To show that there is a positive constant $C$ so that
 \begin{eqnarray}\label{WHHZ3.2}
  \|y(T_n;y_0,u_n)\|_X \leq C \;\;\mbox{for all}\;\; n
 \end{eqnarray}}
 We first claim that there is a positive constant $C_1$ so that for each $s\in (0,\widehat T+1)$
 and each $u_s\in L^2(0,s;U)$,
  \begin{eqnarray}\label{t-u-1}
  \big\|\int_0^{s} S_{-1}(s-\tau)B u_s(\tau) \,\mathrm  d\tau \big\|_X \leq C_1 \|u_s\|_{L^2(0, s;U)}.
 \end{eqnarray}
  To this end, we arbitrarily fix  $s\in (0,\widehat T+1)$ and  $u_s\in L^2(0,s;U)$. Let
 \begin{eqnarray*}
  v_{u_s,s}(t)=\left\{\begin{array}{ll}
              0, \;\; &t\in (0,\widehat T+1-s],\\
              u_s(t+s-\widehat T-1), \;\; &t\in(\widehat T+1-s,\widehat T+1).
             \end{array}
      \right.
 \end{eqnarray*}
  Then, we have that
  $\|v_{u_s,s}\|_{L^2(0, \widehat T+1;U)} = \|u_s\|_{L^2(0, s;U)}$ and
  \begin{eqnarray*}
  \int_0^{\widehat T+1} S_{-1}(\widehat T+1-\tau) B v_{u_s,s}(\tau) \,\mathrm d\tau
  = \int_0^{s} S_{-1}(s-\tau)B u_s(\tau) \,\mathrm  d\tau.
 \end{eqnarray*}
  These, along with (\ref{admissible-control}), yield that
 \begin{eqnarray*}
  \big\|\displaystyle\int_0^{s} S_{-1}(s-\tau)B u_s(\tau) \,\mathrm  d\tau \big\|_X
  &=&   \big\|\displaystyle\int_0^{\widehat T+1} S_{-1}(\widehat T+1-\tau)B v_{u_s,s}(\tau) \,\mathrm  d\tau \big\|_X
  \nonumber\\
  &\leq&   C_1  \|v_{u_s,s}\|_{L^2(0, \widehat T+1;U)}
  =    C_1   \|u_s\|_{L^2(0, s;U)},
 \end{eqnarray*}
 where $C_1\triangleq C_1(\widehat T+1)$ is given by (\ref{admissible-control}). Hence,  (\ref{t-u-1}) is true.

Next, it follows from (\ref{Changshubianyi1.6}) that
 \begin{eqnarray}\label{wanggengsheng3.5}
  y(T_n;y_0,u_n)=S(T_n)y_0 + \int_0^{T_n}  S_{-1}(T_n-t)Bu_n(t) \,\mathrm dt
  \;\;\mbox{for all}\;\;  n\in\mathbb N^+.
 \end{eqnarray}
 Because of the first convergence in  (\ref{wanggengsheng3.1}), we  can assume, without loss of generality, that
   $T_{n}\leq \widehat T+1$ for all $n$. This, along with (\ref{t-u-1}) and (\ref{wanggengsheng3.5}), yields that
  \begin{eqnarray}\label{t-u-2}
  \|y(T_n;y_0,u_n)\|_X \leq \sup_{0\leq t\leq \widehat T+1} \|S(t)\|_{\mathcal L(X,X)}\|y_0\|_X + C_1\|u_n\|_{L^2(0,T_n;U)}\;\;\mbox{for all}\;\;n.
 \end{eqnarray}
 Meanwhile, it follows from the second convergence in  (\ref{wanggengsheng3.1})
 that there is a  $\widehat C>0$ so that
  $\|u_n\|_{L^2(\mathbb R^+;U)} \leq \widehat C$ for all $n$,
  which, along with (\ref{t-u-2}), implies (\ref{WHHZ3.2}).

 \noindent \textit{\noindent Step 2. To show (\ref{time-control-converge})}

Arbitrarily fix a $z\in D(A^*)$. Define two functions  $\psi_{n}^z(\cdot)$   and $\widehat\psi^z(\cdot)$  over $(-1,\widehat T+1)$  in the following manners:
\begin{eqnarray*}
  \psi_{n}^z(t)\triangleq 0\;\;\mbox{for all}\;\; t\in (T_n, \widehat T+1)
  \;\;\mbox{and}\;\; \psi_{n}^z(t)\triangleq B^*S^*(T_{n}-t)z\;\;\mbox{for all}\;\;
   t\in (-1,  T_n];
              \end{eqnarray*}
\begin{eqnarray*}
  \widehat\psi^z(t)\triangleq 0\;\;\mbox{for all}\;\; t\in (\widehat T, \widehat T+1)
  \;\;\mbox{and}\;\; \widehat\psi^z(t)\triangleq B^*S^*(\widehat T-t)z\;\;\mbox{for all}\;\; t\in (-1, \widehat T].
              \end{eqnarray*}
We  claim that for a.e. $t\in (-1,\widehat T+1)$,
\begin{equation}\label{wanghuang3.20-0}
  \lim_{n\rightarrow\infty}\psi_{n}^z(t)= \widehat\psi^z(t)\;\; \mbox{in}\;\;U.
 \end{equation}
In fact, by the first convergence in  (\ref{wanggengsheng3.1}), we see that
for each $t\in (\widehat T,\widehat T+1)$, there is $N_1(t)\geq 1$ so that
$t\in(T_n,\widehat T+1)$ for all $n\geq N_1(t)$.
 Thus,  we see that for each $t\in (\widehat T,\widehat T+1)$,
\begin{eqnarray}\label{a.e.-latter}
  \psi_n^z(t)-\widehat\psi^z(t)=0 \;\;\mbox{for all}\;\; n\geq N_1(t).
 \end{eqnarray}
Meanwhile, given $t\in (-1,\widehat T)$, there is  $N_2(t)\geq 1$ so that
$t\in(-1,T_n)$ for all $ n\geq N_2(t)$.
 This yields that for each $n\geq N_2(t)$,
\begin{eqnarray}\label{wanggengsheng3.9}
  \|\psi_n^z(t)-\widehat\psi^z(t)\|_U
     &\leq& \|B^*\|_{\mathcal L(D(A^*),U)}  \Big(\|S^*(T_n-t)z-S^*(\widehat T-t)z\|_X
  \nonumber\\
  & & +\|S^*(T_n-t)A^*z-S^*(\widehat T-t)A^*z\|_X\Big).
 \end{eqnarray}
(Here, we used that $B^*\in \mathcal L(D(A^*),U)$.)
Since $\{S^*(t)\}_{t\in\mathbb R^+}$ is a $C_0$-semigroup in $X$, it follows from
(\ref{wanggengsheng3.9}) that for each $t\in (-1,\widehat T)$,
$\psi_n^z(t)\rightarrow \widehat\psi^z(t)$ in
$U$, as $n\rightarrow
\infty$.
This, along with (\ref{a.e.-latter}), leads to
 (\ref{wanghuang3.20-0}).

 Next,  since $B^*\in \mathcal L(D(A^*),U)$ and $0\leq T_n\leq \widehat T+1$, $n\in\mathbb N^+$, one can easily check that for all $n\in\mathbb N^+$ and $t\in(-1,\widehat T+1)$,
 \begin{eqnarray}\label{limit-0916-3-0}
  \|\psi_{n}^z(t)\|_U
        \leq \|B^*\|_{\mathcal L(D(A^*),U)} \max_{0\leq s\leq \widehat T+2}
    \|S^*(s)\|_{\mathcal L(X,X)}\|z\|_{D(A^*)}.
 \end{eqnarray}
  By (\ref{wanghuang3.20-0}) and   (\ref{limit-0916-3-0}), we can use  the Lebesgue dominated convergence theorem to get that
 $\psi_{n}^z \rightarrow \widehat\psi^z$  in $L^2(-1,\widehat T+1;U)$, as $n\rightarrow\infty$.
This, along with
  (\ref{NNNWWW2.1}), yields  that for each $z\in D(A^*)$,
  \begin{eqnarray}\label{DA-weak-1}
 \langle  y(T_{n};y_0,u_{n}),z\rangle_X
  \rightarrow
  \langle  y(\widehat T;y_0, \hat u),z\rangle_X,  \;\;\mbox{as}\;\;  n\rightarrow\infty.
 \end{eqnarray}
Since $D(A^*)$ is dense in $X$, (\ref{time-control-converge})
follows from (\ref{DA-weak-1}) at once.
 This ends the proof of   Lemma~\ref{Lemma-t-u-converge}.
 \end{proof}

\subsection{Appendix F}

In Appendix F, we provide the proof of Proposition~\ref{newhuangproposition2.15}.

\begin{proof}[Proof of  Proposition~\ref{newhuangproposition2.15} ]
We divide the proof into the following several steps.

\vskip 5pt
\noindent\textit{Step 1. To show that (i)$\Rightarrow$(ii)}

 Suppose that (i) holds. Let $T\in(0,\infty)$ and let $C_1(T)$ be given by  (\ref{ob-control-eq}). Arbitrarily fix  $y_0\in X$. Define
 a map
 $ \mathcal {F}_{T,y_0}:\, X_T\rightarrow \mathbb R $ (where
 $X_T$ is given by  (\ref{assumptionspace3})) in the following manner:
\begin{eqnarray}\label{fangyan1}
\mathcal {F}_{T,y_0} \big(B^*S^*(T-\cdot)z|_{(0,T)}\big)
=  \langle y_0,S^*(T)z\rangle_X
 \;\;\mbox{for each}\;\;   z\in D(A^*).
\end{eqnarray}
We first claim that $\mathcal {F}_{T,y_0}$ is well defined. In fact, if
$$
z_1, z_2\in D(A^*)\;\;\mbox{s.t.}\;\; B^*S^*(T-\cdot)z_1=B^*S^*(T-\cdot)z_2
\;\;\mbox{over}\;\; (0,T),
$$
then by (\ref{ob-control-eq}), it follows that
$S^*(T)z_1=S^*(T)z_2$ in $ X$.
Hence, $\mathcal {F}_{T,y_0}$ is well defined.
Besides, one can easily check that $\mathcal {F}_{T,y_0}$ is linear.
By (\ref{ob-control-eq}), we can also find that
\begin{eqnarray*}
 \big| \mathcal F_{T,y_0} \big(B^*S^*(T-\cdot)z|_{(0,T)}\big) \big|
 \leq C_1(T)\|y_0\|_X \|B^*S^*(T-\cdot)z\|_{L^1(0,T;U)}\;\;\mbox{for all}\;\;z\in D(A^*).
\end{eqnarray*}
From this, we see that
\begin{eqnarray}\label{fangyan2}
 \|\mathcal F_{T,y_0}\|_{\mathcal L(X_T,\mathbb R)} \leq C_1(T)\|y_0\|_X.
\end{eqnarray}
 Since $X_T$ is a subspace of $L^1(0,T;U)$ (see (\ref{ob-space})), we can apply  the Hahn-Banach theorem
to find a functional  $\widetilde{\mathcal F}_{T,y_0}\in (L^1(0,T;U))^*$ so that
\begin{eqnarray*}
 \|\mathcal F_{T,y_0}\|_{\mathcal L(X_T,\mathbb R)}
 = \|\widetilde{\mathcal F}_{T,y_0}\|_{(L^1(0,T;U))^*}
 \;\;\mbox{and}\;\;
 \mathcal F_{T,y_0}(g)=\widetilde{\mathcal F}_{T,y_0} (g)\;\;\mbox{for all}\;\; g\in X_T.
\end{eqnarray*}
From these, we can apply  the Riesz representation theorem to find a function $v\in L^\infty(0,T;U)$ so that
\begin{equation}\label{fangyan3}
\|\mathcal F_{T,y_0}\|_{\mathcal L(X_T,\mathbb R)}
 = \|v\|_{L^\infty(0,T;U)}
\end{equation}
and so that
\begin{eqnarray}\label{fangyan4}
 \mathcal F_{T,y_0} (g)=\int_0^T \langle g(t),v(t) \rangle_U \,\mathrm dt
 \;\;\mbox{for all}\;\;  g\in X_T.
\end{eqnarray}

From (\ref{fangyan1}), (\ref{fangyan4}), (\ref{assumptionspace3}) and (\ref{NNNWWW2.2}) in Proposition~\ref{huangwanghenproposition2.1}, we see that for each $z\in D(A^*)$,
\begin{eqnarray*}
 \langle S(T)y_0,z\rangle_X
 &=&  \mathcal F_{T,y_0} \big(B^*S^*(T-\cdot)z|_{(0,T)}\big)
 = \int_0^T \langle v(t), B^*S^*(T-t) z \big\rangle_U \,\mathrm dt
 \nonumber\\
 &=&\big \langle \int_0^T S_{-1}(T-t)B v(t)\,\mathrm dt, z \big\rangle_X.
\end{eqnarray*}
This, along with (\ref{Changshubianyi1.6}), indicates that
$\langle \hat y(T;y_0,-v),z \rangle_X=0$ for all $z\in D(A^*)$.
 Since $D(A^*)$ is dense in $X$, the above leads to that
$\hat y(T;y_0,-v)=0$.
Meanwhile, it follows from (\ref{fangyan3}) and (\ref{fangyan2}) that
$\|v\|_{L^\infty(0,T;U)}   \leq   C_1(T) \|y_0\|_X$.
From these,
(\ref{fangyan2.90}) (with $C_2(T)=C_1(T)$) follows at once.

\vskip 5pt
\noindent\textit{Step 2. To prove that (ii)$\Rightarrow$(i)}

Suppose that (ii) holds. Let  $T\in(0,\infty)$ and let $C_2(T)$ be given by (ii). Arbitrarily fix  $y_0\in X$. By  (ii),  there is  $v \in L^\infty(0,T;U)$ so that
\begin{eqnarray}\label{fangyan6}
 \hat y(T;y_0,v)=0\;\;\mbox{and}\;\;\|v\|_{L^\infty(0,T;U)}\leq C_2(T)\|y_0\|_X.
\end{eqnarray}
By the first equality in (\ref{fangyan6}) and
(\ref{NNNWWW2.1}), we find that
\begin{eqnarray*}
 \langle y_0, S^*(T)z\rangle_X =-\int_0^T\langle v(t), B^*S^*(T-t)z\rangle_U  \,\mathrm dt
 \;\;\mbox{for all}\;\; z\in D(A^*).
\end{eqnarray*}
This, along with the second inequality in (\ref{fangyan6}), yields that
\begin{eqnarray*}
 \langle y_0, S^*(T)z\rangle_X
 \leq
 C_2(T)  \|y_0\|_X   \|B^*S^*(T-\cdot)z\|_{L^1(0,T;U)}
 \;\;\mbox{for all}\;\; z\in D(A^*).
\end{eqnarray*}
Since $y_0$ was arbitrarily taken from $X$, the above implies that
\begin{eqnarray*}
 \|S^*(T)z\|_X   =   \sup_{y_0 \in X\setminus\{0\}} \frac{\langle y_0,S^*(T)z \rangle_X} {\|y_0\|_X}
               \leq  C_2(T)  \|B^*S^*(T-\cdot)z\|_{L^1(0,T;U)} \;\;\mbox{for all}\;\; z\in D(A^*),
\end{eqnarray*}
which leads to  (\ref{ob-control-eq}) with  $C_1(T)=C_2(T)$.

\vskip 5pt
\textit{Step 3. To show that (ii)$\Leftrightarrow$(iii)}

It is clear that (ii)$\Rightarrow$(iii). We now show the reverse. Suppose that (iii) holds. Let $T \in (0,\infty)$. Define a linear operator
 $\mathcal G_T:\, L^\infty(0,T;U)\rightarrow X$ by setting
\begin{eqnarray}\label{fangyan8}
 \mathcal G_T(v)=\int_0^T S_{-1}(T-t)Bv(t) \,\mathrm dt
 \;\;\mbox{for each}\;\;     v\in L^\infty(0,T;U).
\end{eqnarray}
Then it follows from (\ref{admissible-control}) that $\mathcal G_T$ is  bounded. By (iii), we know that for each $y_0\in X$, there is $v\in L^\infty(0,T;U)$ so that $\hat y(T;y_0,v)=0$. This, along with (\ref{Changshubianyi1.6}), yields that
\begin{eqnarray}\label{fangyan9}
 0= S(T)y_0   +   \int_0^T S_{-1}(T-t)Bv(t) \,\mathrm dt.
\end{eqnarray}
From (\ref{fangyan8}) and (\ref{fangyan9}), we see that
\begin{eqnarray}\label{fangyan10}
 \mbox{Range\,} S(T) \subset  \mbox{Range\,} \mathcal G_T.
\end{eqnarray}
Write $Q_T$ for the quotient space of $L^\infty(0,T;U)$ with respect to $\mbox{Ker\,}\mathcal G_T$, i.e.,
$$
Q_T\triangleq L^\infty(0,T;U)/\mbox{Ker\,}\mathcal G_T.
$$
 Let   $\pi_T:~L^\infty(0,T;U)\rightarrow  Q_T$ be the  quotient map. Then $\pi_T$ is surjective and it holds that
  \begin{equation}\label{fangyan11}
  \|\pi_T(v)\|_{Q_T}=\inf\big\{\|w\|_{L^\infty(0,T;U)}\; :\; \,w\in v+\mbox{Ker\,}\mathcal G_T\big\}\;\;\mbox{for each}\;\; v\in L^\infty(0,T;U).
  \end{equation}
           Define a map $ \hat{\mathcal G}_T\,:$ $Q_T\rightarrow X$ in the following manner:
  \begin{eqnarray}\label{fangyan12}
  \hat{\mathcal G}_T(\pi_T(v))=\mathcal G_T(v)
  \;\;\mbox{for each}\;\;  \pi_T(v)\in Q_T.
  \end{eqnarray}
 One can easily check that $\hat{\mathcal G}_T$ is  linear and bounded. By (\ref{fangyan12}) and (\ref{fangyan10}),   we see that $\hat{\mathcal G}_T$ is injective and that
 $$
 \mbox{Range\,}S(T) \subset \mbox{Range\,}\hat{\mathcal G}_T.
  $$
  From these, we find that  given $y_0\in X$, there is a unique $\pi_T(v_{y_0})\in Q_T$ so that
 \begin{equation}\label{fangyan13}
 S(T)y_0=\hat{\mathcal G}_T \big(\pi_T(v_{y_0})\big).
 \end{equation}

 We next define another map $\mathcal T_T  \,:~  X \rightarrow Q_T$ by
  \begin{eqnarray}\label{fangyan14}
     \mathcal T_T(y_0) = \pi_T(v_{y_0})   \;\;\mbox{for each}\;\;   y_0\in X.
  \end{eqnarray}
 One can easily check that  $\mathcal T_T$ is well defined and linear. We will use the closed graph theorem to show that $\mathcal T_T$ is bounded. For this purpose, we let
  $\{y_n\}\subset X$ satisfy  that
   \begin{eqnarray}\label{fangyan15}
   y_n\rightarrow \hat y  \;\;\mbox{in}\;\; X
   \;\;\mbox{and}\;\; \mathcal T_T(y_n)\rightarrow  \hat h  \;\;\mbox{in}\;\;Q_T,
      \;\;\mbox{as}\;\; n\rightarrow\infty.
    \end{eqnarray}
    Because $\hat{\mathcal G}_T$ and $S(T)$  are linear and  bounded,  it follows from (\ref{fangyan15}), (\ref{fangyan14}) and (\ref{fangyan13}) that
       \begin{eqnarray}\label{fangyan16}
   \hat{\mathcal G}_T (\hat h)  =
   \lim_{n\rightarrow \infty}   \hat{\mathcal G}_T \big(\mathcal T_T(y_n)\big)
   =\lim_{n\rightarrow \infty}    \hat{\mathcal G}_T \big(\pi_T(v_{y_n})\big)
   =\lim_{n\rightarrow \infty}   S(T)y_n
   =S(T) \hat y.
  \end{eqnarray}
  Meanwhile, by (\ref{fangyan13}) and (\ref{fangyan14}), we find that
 $ S(T)\hat y  =  \hat{\mathcal G}_T  \big(\pi_T(v_{\hat y})\big)
  =\hat{\mathcal G}_T \big(\mathcal T_T(\hat y)\big)$.
    This, together with (\ref{fangyan16}), yields that
  $\hat{\mathcal G}_T (\hat h)  = \hat{\mathcal G}_T  \big(\mathcal T_T(\hat y)\big)$,
  which, together with the injectivity of  $\hat{\mathcal G}_T $, indicates  that
  $\hat h =\mathcal T_T(\hat y)$.
    So the graph of $\mathcal T_T$ is closed. Now we can apply
     the closed graph theorem to see that
          $\mathcal T_T$ is bounded. Hence, there is a  constant $C(T)>0$ so that
          $\|\mathcal T_T(y_0)\|_{Q_T}\leq C(T)\|y_0\|_X$ for all $y_0\in X$.
                    This, along with (\ref{fangyan14}), indicates that
   \begin{eqnarray}\label{fangyan18}
    \|\pi_T(v_{y_0})\|_{Q_T}\leq C(T)\|y_0\|_X
    \;\;\mbox{for each}\;\;    y_0\in X.
   \end{eqnarray}
  Meanwhile,  by (\ref{fangyan11}), we see that
       for each $y_0\in X$, there is   $v_{y_0}^\prime$ so that
  \begin{eqnarray}\label{fangyan19}
   v_{y_0}^\prime\in v_{y_0}+\mbox{Ker\,}\mathcal G_T\;\;\mbox{and}\;\;
   \|v_{y_0}^\prime\|_{L^\infty(0,T;U)}\leq 2 \|\pi_T(v_{y_0})\|_{Q_T}.
  \end{eqnarray}

  From (\ref{fangyan13}), (\ref{fangyan12}), (\ref{fangyan19}) and
   (\ref{fangyan18}) , we find that for each $y_0\in X$,
       there is a control $v_{y_0}^\prime\in L^\infty(0,T;U)$
   so that
   \begin{eqnarray}\label{fangyan20}
   S(T)y_0={\mathcal G}_T(v_{y_0}^\prime)\;\;\mbox{and}\;\;\|v_{y_0}^\prime\|_{L^\infty(0,T;U)}\leq 2C(T)\|y_0\|_X.
  \end{eqnarray}
   Then by (\ref{Changshubianyi1.6}), (\ref{fangyan8}) and (\ref{fangyan20}), we see that
   for each $y_0\in X$,
       there is a control $v_{y_0}^\prime\in L^\infty(0,T;U)$
   so that
   $$
   \hat y(T;y_0,-v_{y_0}^\prime)=0\;\;\mbox{and}\;\; \|v_{y_0}^\prime\|_{L^\infty(0,T;U)}\leq 2C(T)\|y_0\|_X.
   $$
   These lead to (\ref{fangyan2.90}) with $C_2(T)=2C(T)$.

  \vskip 5pt

  \noindent \textit{Step 4. About the constants  $C_1(T)$ and $C_2(T)$}

  From  the proofs in Step 1-Step 3, we find that the  constants $C_1(T)$ in (\ref{ob-control-eq}) and $C_2(T)$ in (\ref{fangyan2.90}) can be
taken as the same number, provided that one of the conclusions (i)-(iii) holds.

  \vskip 5pt

  In summary, we end the proof of this proposition.

\end{proof}

\subsection {Appendix G}

In Appendix G, we provide the proof of Lemma \ref{lemma-key-inequality}.

\begin{proof}[Proof of Lemma \ref{lemma-key-inequality}]
Recall $\mathcal P$ is given by  (\ref{space-2-1}), where $\Lambda\triangleq
\{\lambda_j\}_{j=1}^\infty\subset \mathbb{R}^+$ satisfies
  (\ref{intro-seq}). Arbitrarily fix $\theta_0\in (0, \frac{\pi}{2})$. By \cite[Proposition 4.5]{AK-1}, there is a sequence
       $\{r_n\}_{n=1}^\infty\subset\mathbb (0,\infty)$ so   that
 \begin{equation}\label{yubiaozhang8.1}
    r_n\nearrow\infty
  \;\;\mbox{and}\;\;  \lim_{n\rightarrow\infty} r_n^{-1}\log|W(r_ne^{i\theta})|=0 \mbox{ uniformly in } |\theta|\leq \theta_0,
 \end{equation}
 where $W(\lambda)$ is given by
 \begin{eqnarray}\label{appendix-2}
  \left\{\begin{array}{l}
          W(\lambda) = \prod_{k\geq 1} \delta_k  \frac{1-\lambda/\lambda_k}{1+\lambda/\overline{\lambda}_k},
          ~\lambda\in\mathbb C^+,\\
          \mbox{with}\;\;\delta_k = \frac{\lambda_k}{\overline{\lambda}_k}  \frac{|\lambda_k-1|}{|\lambda_k+1|}
          \frac{\overline{\lambda}_k+1}{\overline{\lambda}_k-1} \;\;\mbox{if}\;\;\lambda_k\neq1; \;\;\delta_k=1 \;\;\mbox{if}\;\;\lambda_k=1.
         \end{array}
  \right.
 \end{eqnarray}
 (Notice that in \cite{AK-1}, $\lambda_j$ was a complex number, while in the current case, we take it as a real number. So $\lambda_j=\overline{\lambda}_j$ in the current case. To avoid the inconformity, we still use the notation
 $\overline{\lambda}_j$.)
  Since $W(\lambda_k)=0$ for each $k\geq 1$, and because of (\ref{intro-seq})
  and (\ref{yubiaozhang8.1}), we can select a subsequence from $\{r_n\}_{n=1}^\infty$ (denoted in the same manner,) having two properties as follows: First,
  $\{\lambda_j\}_{j=1}^\infty\bigcap \{r_n\}_{n=1}^\infty=\emptyset$. Second,  for each $n\in \mathbb{N}^+$, the set
  \begin{equation*}\label{appendix-3}
    G_n\triangleq \{z=re^{i\theta} ~:~r_n<|z|<r_{n+1},\,|\theta|<\theta_0\}
  \end{equation*}
  contains at least an element of $\Lambda\triangleq\{\lambda_j\}_{j=1}^\infty$.  The sequence $\{G_n\}_{n\geq 1}$ and the function  $W(\cdot)$, as well as their properties,  will be used later.

  Let $J$ be a function defined by
    \begin{eqnarray}\label{1218-def-J}
   J(\lambda)= \frac{W(\lambda)}{(1+\lambda)^2}, ~ \lambda\in \mathbb C^+.
  \end{eqnarray}
  For each $j\geq 1$, define a function $J_j$ by
  \begin{eqnarray}\label{1218-def-J-j}
   J_j(\lambda)  =  \frac{J(\lambda)}{J^\prime(\lambda_j)(\lambda-\lambda_j)},~\lambda\in\mathbb C^+.
  \end{eqnarray}
   According to \cite[Theorem 4.1]{AK-1} (see also the proof of \cite[Theorem 4.1]{AK-1}), there exists a biorthogonal family $\{q_j\}_{j\geq 1}$    to   $\{e^{-\lambda_j t}\}$   in $L^2(\mathbb R^+;\mathbb C)$ so that the Laplace transform of $\bar q_j$ is    $J_j$  for each $j\in \mathbb{N}^+$.

   \vskip 3pt

   To prove the desired inequality (\ref{space-3}), we will build up two inequalities for $p\in\mathcal P$. The first one reads: For each $\varepsilon>0$,
   there is $C(\theta_0,\varepsilon)>0$ so that for each $p\in\mathcal P$,
\begin{eqnarray}\label{yubiao8.5}
    |p(z)|
        \leq C(\theta_0,\varepsilon) e^{-\frac{1}{8}|\lambda_1|\cos\theta_0  Re\,z} \|p\|_{L^1(\mathbb R^+;\mathbb C)}\;\;\mbox{for all}\;\; z\in S_{\varepsilon,\theta_0},
 \end{eqnarray}
 where $S_{\varepsilon,\theta_0}$ is given by (\ref{space-2}).
   The second one reads:  For each $T\in(0,\infty)$, there exists  $C\triangleq C(T)>0$ so that
  \begin{equation}\label{1218-bad-1}
   \|p\|_{L^1(\mathbb R^+;\mathbb C)} \leq   C   \|p\|_{L^1(0,T;\mathbb C)}
   \;\;\mbox{for all}\;\; p\in\mathcal P.
  \end{equation}

  We now show (\ref{yubiao8.5}). Let  $p\in\mathcal P$. By  (\ref{space-2-1}), we can express $p$  in the following manner:
  \begin{equation}\label{appendix-4}
    p(z)=\sum_{j=1}^N c_j e^{- \lambda_jz},~z\in \mathbb C^+,
    \;\;\mbox{with}\;\;N\in\mathbb N^+ \;\mbox{and}\;\{c_j\}_{j=1}^N\subset\mathbb C.
  \end{equation}
  Since each  $\{ G_n\}_{n\geq 1}$ contains at least an element of $\Lambda\triangleq\{\lambda_j\}_{j=1}^\infty$ and $\lambda_j\nearrow \infty$, there is
   an $m\triangleq m(N)\in \mathbb{N}^+$ so that
   $\{\lambda_j\}_{j=1}^N\subset \bigcup_{k=1}^m G_k$.
      This, along with (\ref{appendix-4}), yields that
      \begin{equation}\label{appendix-5}
   p(z)=\sum_{k=1}^m \sum_{\lambda_j\in G_k} c_j e^{-\lambda_j z}\triangleq \sum_{k=1}^m g_k(z),
   ~z\in \mathbb C^+.
  \end{equation}
Meanwhile, since  $\{q_j\}_{j=1}^\infty$ is a biorthogonal family  to $\{e^{-\lambda_j t}\}$   in $L^2(\mathbb R^+;\mathbb C)$, it follows from  (\ref{appendix-4}) that
  \begin{equation*}
     c_j=\int_0^\infty p(t)\overline{q}_j(t) \,\mathrm dt,\;\;\mbox{with}\;\;1\leq j\leq N.
  \end{equation*}
  From this and (\ref{appendix-5}), we have that for each $k\in\{1,\cdots,m\}$,
  \begin{equation*}
    g_k(z)=\int_0^\infty p(t) \Big(\sum_{\lambda_j\in G_k} \overline{q}_j(t) e^{-\lambda_j z}\Big) \,\mathrm dt,
   ~z\in \mathbb C^+.
  \end{equation*}
  This yields that for each $k\in\{1,\cdots,m\}$ and each $z\in \mathbb C^+$,
  \begin{eqnarray}\label{appendix-8}
    |g_k(z)|\leq
            \|p(\cdot)\|_{L^1(\mathbb R^+;\mathbb C)} \|\mathcal G_k(\cdot,z)\|_{L^\infty(\mathbb R^+;\mathbb C)}.
  \end{eqnarray}
  where
  \begin{eqnarray}\label{yubiao8.10}
  \mathcal G_k(t,z)\triangleq \sum_{\lambda_j\in G_k} \overline{q}_j(t) e^{-\lambda_j z},\;\; t\in \mathbb{R}^+.
  \end{eqnarray}
    Arbitrarily fix a  $k\in\{1,\cdots,m\}$. Since for each $j\in \mathbb{N}^+$, the Laplace transform of $\bar q_j$ is $J_j$,
    we see that for each $z\in \mathbb C^+$, the Laplace transform of $\mathcal G_k(t,z)$ is given by
  \begin{equation}\label{appendix-9}
    \int_0^\infty \mathcal G_k(t,z) e^{-\lambda t} \,\mathrm dt
    =\sum_{\lambda_j\in G_k} J_j(\lambda)e^{-\lambda_j z},
    ~\lambda\in \mathbb C^+,
  \end{equation}
   Since $q_j(t)=0$ for all $t<0$ and $j\in \mathbb{N}^+$, we see from
   (\ref{yubiao8.10}) that for each $z\in \mathbb C^+$, $\mathcal G_k(t,z)=0$ for all $t<0$.  This, along with  (\ref{appendix-9}), yields  that for each $z\in \mathbb C^+$,
   the function $\tau\rightarrow\sum_{\lambda_j\in G_k} J_j(i\tau)e^{-\lambda_j z}$, $\tau\in\mathbb R$,
    is the Fourier transform of $\mathcal G_k(\cdot,z)$.  Then  by the inverse Fourier transform, we see that for each $z\in \mathbb C^+$,
  \begin{eqnarray}\label{appendix-10}
   \|\mathcal G_k(\cdot,z)\|_{L^\infty(\mathbb R^+;\mathbb C)}
   &=& \sup_{t\in\mathbb R^+} \Big| \frac{1}{2\pi} \int_{\mathbb R} \Big(\sum_{\lambda_j\in G_k} J_j(i\tau)e^{-\lambda_j z}\Big) e^{i\tau t} \,\mathrm d\tau\Big|
   \nonumber\\
   &\leq&  \frac{1}{2\pi} \int_{\mathbb R} \Big|\sum_{\lambda_j\in G_k} J_j(i\tau)e^{-\lambda_j z} \Big| \,\mathrm d\tau.
  \end{eqnarray}
  Meanwhile, by (\ref{1218-def-J}), (\ref{appendix-2}) and (\ref{space-2-1}), we find that each   $\lambda_j$ is a simple root of $J$. Thus, by
  (\ref{1218-def-J-j}), we can use  the residue theorem to  see that
  \begin{equation}\label{appendix-11}
    \sum_{\lambda_j\in G_k} J_j(i\tau)e^{-\lambda_j z}=\frac{J(i\tau)}{2\pi i} \int_{\Gamma_k} \frac{e^{-\xi z}}{J(\xi)(i\tau-\xi)} \,\mathrm d\xi,
  \end{equation}
  where  $\Gamma_k$ denotes the boundary of $G_k$.
  From (\ref{appendix-11}) and  (\ref{appendix-10}), it follows  that
  for each $k\in\{1\dots,m\}$ and each $z\in\mathbb C^+$,
  \begin{eqnarray}\label{yubiao8.14}
    \|\mathcal G_k(\cdot,z)\|_{L^\infty(\mathbb R^+;\mathbb C)}
    &\leq&  \frac{1}{4\pi^2} \int_{\mathbb R} \Big|\int_{\Gamma_k}  J(i\tau)\frac{e^{-\xi z}}{J(\xi)(i\tau-\xi)} \,\mathrm d\xi \Big| \,\mathrm d\tau  \nonumber\\
      &\leq& \frac{\|J\|_{L^1(i\mathbb R;\mathbb C)}}{4\pi^2\rho}  \int_{\Gamma_k}  \Big|\frac{e^{-\xi z}}{J(\xi)} \Big|
      \,|\mathrm d\xi|,
  \end{eqnarray}
 where $\rho=\displaystyle\min_{k\geq 1} d(i\mathbb R,\Gamma_k)>0$.
 From  (\ref{appendix-5}), (\ref{appendix-8}) and (\ref{yubiao8.14}), we get that
 \begin{eqnarray}\label{appendix-13}
    |p(z)| \leq \frac{\|J\|_{L^1(i\mathbb R;\mathbb C)}}{4\pi^2\rho} \|p\|_{L^1(\mathbb R^+;\mathbb C)} \Big(\sum_{k= 1}^m \int_{\Gamma_k}
    \Big|\frac{e^{-\xi z}}{J(\xi)} \Big| \,|\mathrm d\xi| \Big),~\forall\, z\in\mathbb C^+.
     \end{eqnarray}
 Starting from (\ref{appendix-13}), using the same way as that used in  the proof of estimating (4.12) in \cite[Lemma 4.6]{AK-1} (see \cite[Pages 2113-2115]{AK-1}), we can get the inequality (\ref{yubiao8.5}).

  Now we prove the second inequality (\ref{1218-bad-1}).
  By contradiction,  suppose that it were not true. Then there would be
   a $T>0$ and a sequence $\{p_n\}_{n=1}^\infty\subset \mathcal P$ so that
  \begin{eqnarray}\label{1218-bad-2}
   \|p_n\|_{L^1(\mathbb R^+;\mathbb C)}=1 \;\;\mbox{and}\;\;  \|p_n\|_{L^1(0,T;\mathbb C)}< 1/n
   \;\;\mbox{for each}\;\; n\geq 1.
  \end{eqnarray}
  Arbitrarily fix $\varepsilon_0\in(0,T/2)$.  Then choose a $s_0\in(T,\infty)$ so that
  \begin{eqnarray}\label{1218-bad-3}
    \int_{s_0}^\infty C(\theta_0,\varepsilon_0) e^{-\frac{1}{8}|\lambda_1|\cos\theta_0 t} \,\mathrm dt <1/2,
  \end{eqnarray}
  where $C(\theta_0,\varepsilon_0)$ is given by (\ref{yubiao8.5}).
From   (\ref{yubiao8.5}), we find that for all  $m,n\in\mathbb N^+$,
  \begin{eqnarray*}
   \int_0^\infty |(p_n-p_m)(t)| \,\mathrm dt
      &\leq& \int_0^{s_0} |(p_n-p_m)(t)| \,\mathrm dt
   +
    \nonumber\\
    & & \int_{s_0}^\infty \big( C(\theta_0,\varepsilon_0) e^{-\frac{1}{8}|\lambda_1|\cos\theta_0 t}    \int_0^\infty |(p_n-p_m)(s)| \,\mathrm ds \big) \,\mathrm dt.
  \end{eqnarray*}
 This, along with (\ref{1218-bad-3}), implies that for all $m,n\in\mathbb N^+$,
  \begin{eqnarray}\label{1218-bad-4}
    \int_0^\infty |(p_n-p_m)(t)| \,\mathrm dt \leq 2 \int_0^{s_0} |(p_n-p_m)(t)| \,\mathrm dt.
  \end{eqnarray}
    Two observations are given in order: First, by (\ref{yubiao8.5}) and the first equality in (\ref{1218-bad-2}), we find that $\{\|p_n\|_{C(S_{\varepsilon_0,\theta_0},\mathbb{C})}\}_{n=1}^\infty$ is bounded. Second, each $p_n$ (with $n\in \mathbb{N}^+$) is analytic over $S_{\varepsilon_0,\theta_0}$.
    From these observations, we can use the Montel theorem to find  a subsequence $\{p_{n_k}\}_{k=1}^\infty$ of $\{p_n\}_{n=1}^\infty$ and an analytic function $\hat p$ over $S_{\varepsilon_0,\theta_0}$ so that
  \begin{eqnarray}\label{1218-bad-4-1}
   p_{n_k} \rightarrow \hat p
   \;\;\mbox{uniformly on each compact set of}\;\;  S_{\varepsilon_0,\theta_0},
   \;\;\mbox{as}\;\;k\rightarrow\infty.
  \end{eqnarray}
  Since $0<2\varepsilon_0<T<s_0$, it follows from (\ref{1218-bad-4-1}) and   the second inequality in (\ref{1218-bad-2}) that
  \begin{eqnarray*}
   p_{n_k} \rightarrow 0 \;\;\mbox{in}\;\; L^1(0,T;\mathbb C)
   \;\;\mbox{and}\;\; p_{n_k} \rightarrow \hat p \;\;\mbox{in}\;\; L^1(T,s_0;\mathbb C),
   \;\;\mbox{as}\;\;  k\rightarrow\infty.
  \end{eqnarray*}
  These, along with (\ref{1218-bad-4}), (\ref{1218-bad-4-1}) and the first equality in   (\ref{1218-bad-2}), indicates that
  \begin{eqnarray}\label{1218-bad-5}
   \|\hat p\|_{L^1(T,\infty;\mathbb C)}=1
    \;\;\mbox{and}\;\;
    \|\hat p\|_{L^1(2\varepsilon_0,T;\mathbb C)}=0.
  \end{eqnarray}
  Since $\hat p$ is analytic over $S_{\varepsilon_0,\theta_0}$, from the second assertion in (\ref{1218-bad-5}), we get that $\hat p \equiv 0$ over $S_{\varepsilon_0,\theta_0}$. This contradicts the first assertion in (\ref{1218-bad-5}). So (\ref{1218-bad-1}) is true.

  \vskip 3pt
  Finally, the desired inequality (\ref{space-3}) follows from (\ref{yubiao8.5}) and (\ref{1218-bad-1}) at once. This ends the proof of Lemma \ref{lemma-key-inequality}.

 \end{proof}

\vskip 5pt
 \textbf{Acknowledgements.} The authors gratefully thank Professor Xu Zhang for  his valuable suggestions.


\begin{thebibliography}{1}


\bibitem{AEWZ} J. Apraiz, L. Escauriaza, G. Wang and C. Zhang, Observability inequalities and measurable sets, J. Eur. Math. Soc.,  16 (2014), 2433-2475.

\bibitem{Arada-Raymond} N. Arada and J.-P. Raymond, Time optimal problems with Dirichlet boundary controls,
Discrete Contin. Dyn. Syst., 9  (2003), 1549-1570.

\bibitem{Barbu}  V. Barbu, Analysis and Control of Nonlinear Infinite-dimensional systems, Academic Press, Boston, 1993.

\bibitem{carja} O. C\^{a}rj\v{a}, On continuity of the minimal time function for distributed control systems, Boll. Un. Mat. Ital. A (6), 4 (1985), 293-302.

\bibitem {Coron} J. M. Coron,  Control and Nonlinearity, Math. Surveys Monogr.,  American Mathematical Society, Providence, RI, 2007.



\bibitem{HOF} H. O. Fattorini, INFINITE DIMENSIONAL LINEAR CONTROL SYSTEMS: The Time Optimal and Norm Optimal Problems, North-Holland Mathematics Studies 201, Elsevier, Amsterdam, 2005.

\bibitem{HOF1} H. O. Fattorini, Time-optimal control of solutions of operational differential equations, SIAM J. Control, 2 (1964), 54-59.

\bibitem{HOF2} H. O. Fattorini,   The time-optimal control problem in Banach spaces, Appl. Math. Optim., 2 (1974), 163-188.

\bibitem{HOF3} H. O. Fattorini, Some remarks on the time optimal control problem in infinite dimension, Calculus of Variations and Optimal Control, Chapman \& Hall/CRC Research Notes in Mathmatics Series Vol. 411, CRC Press, Boca Raton (1999), 77-96.

\bibitem{HOF4} H. O. Fattorini,   Existence of singular extremals and singular functionals in reachable spaces,  J. Evol. Equ., 1 (2001),  325-347.



\bibitem{HOF5} H. O. Fattorini,    Time and norm optimal controls: a survey of recent results and open problems, Acta Math. Sci. Ser. B Engl. Ed., 31 (2011),  2203-2218.

\bibitem{GY} W. Gong and N. Yan, finite elemnt method and its error estimates for the time optimal controls of heat equation, INTERNATIONAL JOURNAL OF
NUMERICAL ANALYSIS AND MODELING, 13 (2016), 261-275.

\bibitem{GL} F. Gozzi and P. Loreti, Regularity of the minimum time
function and minimal energy problems: the linear case, SIAM J.
Control Optim., 37  (1999), 1195-1221.

\bibitem{ITO} K. Ito and K. Kunisch, Semismooth Newton methods for
time-optimal control for a class of ODEs, SIAM J. Control Optim., 48
(2010), 3997-4013.

\bibitem{AK-2} F. A. Khodja, A. Benabdallah, M. Gonz\'{a}lez-Burgos, L. Teresa, A new relation between the condensation index of complex sequences and the null controllability of parabolic systems, C. R. Acad. Sci. Paris., Ser. I, 351 (2013), 743-746.

\bibitem{AK-3} F. A. Khodja, A. Benabdallah, M. Gonz\'{a}lez-Burgos, L. Teresa, Minimal time of controllability of two parabolic equations with disjoint control and coupling domains, C. R. Acad. Sci. Paris, Ser. I, 352 (2014), 391-396.

\bibitem{AK-1} F. A. Khodja, A. Benabdallah, M. Gonz\'{a}lez-Burgos, L. Teresa, Minimal time for the null controllability of parabolic systems:
the effect of the condensation index of complex sequences, J. Funct. Anal., 267 (2014), 2077-2151.



\bibitem{KL1} K. Kunisch and L. Wang, Time optimal control of the heat equation with pointwise control constraints,
 ESAIM Control Optim. Calc. Var., 19  (2013),  460-485.

\bibitem{KL2}  K. Kunisch and L. Wang, Time optimal controls of the linear Fitzhugh-Nagumo equation with pointwise control constraints,
J. Math. Anal. Appl., 395  (2012), 114-130.

\bibitem{LY} X. Li and J. Yong, OPTIMAL CONTROL THEORY FOR INFINITE DIMENSIONAL SYSTEMS, \copyright Birkh\"{a}user Boston 1995.



\bibitem{LW} P. Lin  and G. Wang, Some properties for blowup parabolic equations and their application,
 J. Math. Pures Appl., 101  (2014), 223-255.

 \bibitem{LM} J. Loh\'{e}ac and M. Tucsnak, Maximum principle and bang-bang property of time optimal controls for Schr\"{o}dinger type systems, SIAM J. Control Optim., 51 (2013), 4016-4038.

\bibitem{LZ} J. Loh\'{e}ac and E. Zuazua. Norm saturating property of time optimal controls for wave-type
equations, 2016, $<$hal-01258878$>$.

 \bibitem{LWX} H. Lou, J. Wen and Y. Xu, Time optimal control problems for some non-smooth systems,
Math. Control Relat. Fields, 4  (2014), 289-314.

 \bibitem{QL} Q. L$\ddot{u}$, Bang-bang principle of time optimal
controls and null controllability of fractional order parabolic
equations, Acta Math. Sin. (Engl. Ser.), 26  (2010), 2377-2386.

 \bibitem{MRT} S. Micu, I. Roventa and M. Tucsnak, Time optimal
boundary  controls for the heat equation,  J. Funct. Anal., 263
(2012), 25-49.

\bibitem{MZ}A. M\"{u}nch, E. Zuazua, Numerical approximation of null controls for the
heat equation: ill-posedness and remedies, Inverse Problems, 8 (2010), 085018.



\bibitem{MS} V. Mizel and T. Seidman, An abstract bang-bang principle and time-optimal boundary control of the heat equation, SIAM J. Control Optim., 35 (1997), 1204-1216.

\bibitem{Pazy} A. Pazy, Semigroups of Linear Operatots and Applications to Partial Differential Equations, \copyright 1983 Springer-Verlag New York, Inc.

\bibitem{Petrov} N. V. Petrov, The Bellman problem for a time-optimality problem, Prikl. Mat. Meh., 34
(1970),   820-826.

\bibitem{PW} K. D. Phung and G. Wang, An observability estimate for parabolic equations from a measurable set in time and its applications,
J. Eur. Math. Soc., 15  (2013), 681-703.

\bibitem{PW1} K. D.  Phung and G. Wang, Quantitative unique continuation for the semilinear heat equation in a convex domain,
 J. Funct.  Anal., 259  (2010), 1230-1247.

\bibitem{PWCZ} K. D. Phung, L. Wang and C. Zhang, Bang-bang property for time optimal control of semilinear heat
equation,
 Ann. Inst. H. Poincar\'{e} Anal. Non Lin\'{e}aire, 31  (2014), 477-499.




\bibitem{PWZ} K. D. Phung, G. Wang, and X. Zhang, On the existence of time optimal controls for linear
evolution equations, Discrete Contin. Dyn. Syst. Ser. B, 8 (2007),  925-941.

\bibitem{Rudin} W. Rudin, REAL AND COMPLEX ANALYSIS, Third Edition,  McGRAW-HILL
INTERNATIONAL EDITIONS, Mathematics Series, 1987.

\bibitem{EJPGS} E. J. P. G. Schmidt, The ``bang-bang" principle for the time-optimal
problem in boundary control of the heat equation, SIAM J. Control
Optim., 18  (1980), 101-107.

\bibitem{C.S.E.T} C. Silva and E.  Tr$\acute{e}$lat,
Smooth regularization of bang-bang optimal control problems, IEEE
Trans. Automat. Control, 55  (2010), 2488-2499.

\bibitem{E.D.Sontag} E. D. Sontag, Mathematical Control Theory: Deterministic Finite-Dimensional Systems, second edition, Texts Appl. Math., Vol. 6, Springer-Verlag, New York, 1998.





\bibitem{TW} M. Tucsnak and G. Weiss, OBSERVATION AND CONTROL FOR OPERATOR SEMIGROUPS, \copyright\,2009 Birkh\"{a}user Verlag AG.

\bibitem{Wang} G. Wang, $L^\infty$-null controllability for the heat equation and its consequences for the time optimal control problem, SIAM J. Control Optim., 47  (2008), 1701-1720.

\bibitem{WX-1}  G. Wang and Y. Xu, Equivalence of three different kinds of optimal
control problems for heat equations and its applications, SIAM J.
Control Optim., 51  (2013), 848-880.

\bibitem{WX-2}  G. Wang and Y. Xu, Advantages for controls imposed in a proper subset. Discrete Contin. Dyn. Syst. Ser. B, 18 (2013),  2427-2439.

\bibitem{WXZ}  G. Wang, Y. Xu and Y. Zhang, Attianable subspaces and the bang-bang property of time optimal controls for heat equations, SIAM J. Control Optim., 53 (2015), 592-621.

\bibitem{WCZ} G. Wang and C. Zhang, Observability inequalities from measurable sets for some evolution equations, submitted.

\bibitem{WZheng} G. Wang and G. Zheng, An approach to the optimal time for a time optimal control problem of an internally controlled heat equation, SIAM J. Control Optim., 50 (2012), 601-628.

\bibitem{WZ}  G. Wang and E. Zuazua, On the equivalence of minimal time and minimal norm controls for internally controlled heat equations, SIAM J. Control Optim., 50  (2012), 2938-2958.

\bibitem{Y} H. Yu, Approximation of time optimal controls for heat equations with perturbations in the system potential,  SIAM J. Control Optim., 52  (2014), 1663-1692.

\bibitem{CZ-1} C. Zhang, An observability estimate for the heat
equation from a product of two measurable sets, J. Math. Anal.
Appl., 396  (2012), 7-12.

\bibitem{CZ-2} C. Zhang, The time optimal control with constraints
of the rectangular type for linear time-varying ODEs, SIAM J. Control
Optim., 51  (2013), 1528-1542.

\bibitem{Guo-Zh} G. Zheng and B. Ma, A time optimal control problem
of some switching controlled ordinary differential equations,
Advance in Difference Equations, 1 (2012), 1-7.



\end{thebibliography}
\end{document}